\documentclass[12pt]{amsart}
\usepackage{amsmath, amsthm, amssymb, eqnarray}
\usepackage{accents}

\input xy
\xyoption{all}

\usepackage[shortlabels]{enumitem}
\usepackage{amsfonts, calligra, amscd, bbm, eucal, latexsym, extarrows, mathrsfs}
\usepackage[utf8]{inputenc}
\usepackage[T1]{fontenc}
\usepackage[bookmarks=false]{hyperref}
\usepackage[usenames,dvipsnames]{color}
\usepackage{mathtools}
\usepackage{relsize}

\usepackage{fullpage}
\usepackage[all,cmtip]{xy}
\usepackage{color}
\usepackage{microtype}
\usepackage{cancel}
\usepackage{fourier}
\usepackage{graphicx}

\usepackage{mathtools}

\theoremstyle{plain}
\newtheorem{theorem}{Theorem}[section]
\newtheorem{proposition-definition}{Proposition/Definition}[section]

\newtheorem*{proposition-definition-intro}{Proposition/Definition}
\newtheorem*{definition-intro}{Definition}
\newtheorem*{cor-intro}{Corollary}
\newtheorem*{prop-intro}{Proposition}
\newtheorem{proposition}[theorem]{Proposition}		
\newtheorem{corollary}[theorem]{Corollary}
\newtheorem{lemma}[theorem]{Lemma}

\newtheorem*{BCOV-new}{Refined BCOV conjecture at genus one}
\newtheorem*{conjecture-intro}{Conjecture}

\newtheorem{definition}[theorem]{Definition}

\theoremstyle{remark}
\newtheorem{remark}[theorem]{Remark}
\newtheorem*{remark-intro}{Remark}
\newtheorem{notation}[theorem]{Notation}

\newcommand\op[1]{\operatorname{#1}}

\newcounter{tmp}

\newcommand{\Dbold}{{\bf D}}

\newcommand{\Mbold}{{\bf M}}

\newcommand{\Rbold}{{\bf R}}

\newcommand{\ov}{\overline}

\newcommand{\ebold}{{\bf e}}

\newcommand{\gbsymb}{{\boldsymbol{g}}}

\newcommand{\ubsymb}{{\boldsymbol{u}}}
\newcommand{\vbsymb}{{\boldsymbol{v}}}

\newcommand{\CBbb}{\mathbb C}

\newcommand{\HBbb}{\mathbb H}

\newcommand{\PBbb}{\mathbb P}

\newcommand{\RBbb}{\mathbb R}

\newcommand{\VBbb}{\mathbb V}

\newcommand{\ZBbb}{\mathbb Z}

\newcommand{\Acal}{\mathcal A}
\newcommand{\Bcal}{\mathcal B}
\newcommand{\Ccal}{\mathcal C}
\newcommand{\Dcal}{\mathcal D}
\newcommand{\Ecal}{\mathcal E}
\newcommand{\Fcal}{\mathcal F}
\newcommand{\Gcal}{\mathcal G}
\newcommand{\Hcal}{\mathcal H}
\newcommand{\Ical}{\mathcal I}
\newcommand{\Jcal}{\mathcal J}
\newcommand{\Kcal}{\mathcal K}
\newcommand{\Lcal}{\mathcal L}
\newcommand{\Mcal}{\mathcal M}
\newcommand{\Ncal}{\mathcal N}
\newcommand{\Ocal}{\mathcal O}
\newcommand{\Pcal}{\mathcal P}
\newcommand{\Qcal}{\mathcal Q}

\newcommand{\Scal}{\mathcal S}
\newcommand{\Tcal}{\mathcal T}
\newcommand{\Ucal}{\mathcal U}
\newcommand{\Vcal}{\mathcal V}
\newcommand{\Wcal}{\mathcal W}
\newcommand{\Xcal}{\mathcal X}

\newcommand{\Sfrak}{\mathfrak S}

\newcommand{\uG}{\underset{\widetilde{\hspace{0.4cm}}}{\Gcal}}
\newcommand{\uOs}{\underset{\widetilde{\hspace{0.4cm}}}{\Ocal_S}}

\newcommand{\SL}{\mathsf{SL}}
\newcommand{\PSL}{\mathsf{PSL}}
\newcommand{\GL}{\mathsf{GL}}

\newcommand{\SU}{\mathsf{SU}}

\newcommand{\U}{\mathsf{U}}

\newcommand{\CSU}{\scriptscriptstyle{\mathsf{CSU}}}
\newcommand{\CS}{\scriptscriptstyle{\mathsf{CS}}}
\newcommand{\dR}{\scriptscriptstyle{\mathsf{dR}}}
\newcommand{\un}{\scriptscriptstyle{\mathsf{un}}}
\newcommand{\ir}{\scriptscriptstyle{\mathsf{ir}}}
\newcommand{\sst}{\scriptscriptstyle{\mu}} 
\newcommand{\Bet}{\scriptscriptstyle{\mathsf{B}}}
\newcommand{\na}{\scriptscriptstyle{na}}
\newcommand{\chmini}{\scriptscriptstyle{ch}}
\newcommand{\ICmini}{\scriptscriptstyle{IC_2}}

\newcommand{\Fu}{\scriptscriptstyle{\mathsf{F}}}
\newcommand{\Sch}{\scriptscriptstyle{\mathsf{S}}}
\newcommand{\QF}{\scriptscriptstyle{\mathsf{QF}}}
\newcommand{\Be}{\scriptscriptstyle{\mathsf{B}}}
\newcommand{\Bers}{\scriptscriptstyle{\mathsf{Bers}}}

\newcommand{\slfrak}{\mathfrak{sl}}
\newcommand{\glfrak}{\mathfrak{gl}}

\newcommand{\amini}{\scriptscriptstyle{a}}
\newcommand{\bmini}{\scriptscriptstyle{b}}

\newcommand{\hmini}{\scriptscriptstyle{h}}

\newcommand{\Emini}{\scriptscriptstyle{E}}
\newcommand{\Fmini}{\scriptscriptstyle{F}}

\newcommand{\Lmini}{\scriptscriptstyle{L}}
\newcommand{\Mmini}{\scriptscriptstyle{M}}

\providecommand{\rk}{\operatorname{rk}}
\providecommand{\Gr}{\operatorname{Gr}}

\DeclareMathOperator{\utimes}{\underline{\otimes}}
\DeclareMathOperator{\bugtimes}{\underline{\bigotimes}}
\DeclareMathOperator{\End}{End}
\DeclareMathOperator{\Hom}{Hom}
\DeclareMathOperator{\SheafHom}{\mathscr{H}\text{\kern -3pt {\calligra\large om}}\,}

\DeclareMathOperator{\Aut}{Aut}

\DeclareMathOperator{\id}{id}

\DeclareMathOperator{\Ad}{Ad}
\DeclareMathOperator{\tr}{tr}

\DeclareMathOperator{\Div}{div}

\DeclareMathOperator{\an}{\scriptscriptstyle{an}}

\newcommand{\Pic}{{\rm Pic}}

\DeclareMathOperator{\Real}{Re}
\DeclareMathOperator{\Imag}{Im}

\DeclareMathOperator{\codim}{codim}
\DeclareMathOperator{\hol}{\mathsf{hol}}

\DeclareMathOperator{\hyp}{\scriptscriptstyle{\mathsf{hyp}}}
\DeclareMathOperator{\CM}{\scriptscriptstyle{\mathsf{CM}}}
\DeclareMathOperator{\Qu}{\scriptscriptstyle{\mathsf{Q}}}
\DeclareMathOperator{\Ltwo}{\scriptscriptstyle{\mathsf{L}^{2}}}
\DeclareMathOperator{\De}{\scriptscriptstyle{\mathsf{QF}}}
\DeclareMathOperator{\WP}{\scriptscriptstyle{\mathsf{WP}}}
\DeclareMathOperator{\TZ}{\scriptscriptstyle{\mathsf{TZ}}}
\DeclareMathOperator{\TT}{\scriptscriptstyle{\mathsf{TT}}}

\newcommand{\doubleslash}{\bigr/ \negthinspace\negthinspace \bigr/}

\newcommand{\isorightarrow}{\xrightarrow{
   \,\smash{\raisebox{-0.5ex}{\ensuremath{\sim}}}\,}}

\newcommand{\Jac}{{\rm Jac}}

\newcommand\T[1]{{\langle {#1}\rangle}}
\numberwithin{equation}{section}

\newenvironment{dedication}
  {
   \itshape             
   \raggedleft          
  }
  {\par 
  }

\begin{document}

\setcounter{tocdepth}{2}
\title{Complex Chern--Simons bundles in the relative setting}
\author{Dennis Eriksson}
\author{Gerard Freixas i Montplet}
\author{Richard A. Wentworth}
\address{Dennis Eriksson \\ Department of Mathematics \\ Chalmers University of Technology and  University of Gothenburg}
\email{dener@chalmers.se}

\address{Gerard Freixas i Montplet \\ C.N.R.S. -- Institut de Math\'ematiques de Jussieu - Paris Rive Gauche}
\email{gerard.freixas@imj-prg.fr}

\address{Richard A. Wentworth\\ Department of Mathematics\\ University of Maryland}
\email{raw@umd.edu}
\maketitle
\begin{dedication}
\small{To Scott Wolpert, on the occasion of his 70th birthday.}
\end{dedication}

\begin{abstract}
Complex Chern--Simons bundles are line bundles with connection, originating in the study of quantization of moduli spaces of flat connections with complex gauge groups. In this paper we introduce and study these bundles in the families setting.

The central object is a functorial direct image of characteristic classes of vector bundles with connections, for which we develop a formalism. Our strategy elaborates on Deligne--Elkik's intersection bundles, and a refined Chern--Simons theory which parallels the use of Bott--Chern classes in Arakelov geometry. In the context of moduli spaces, we are confronted with flat relative connections on families of Riemann surfaces. To be able to rely on the functorial approach, we prove canonical extension results to global connections, inspired by the deformation theory of harmonic maps in non-abelian Hodge theory. 

The relative complex Chern--Simons bundle $\Lcal_{\CS}$ is then defined as a functorial direct image of the second Chern class on the relative moduli space of flat vector bundles. We establish the crystalline nature of $\Lcal_{\CS}$, and the existence of a holomorphic extension of natural metrics from Arakelov geometry. The curvature of $\Lcal_{\CS}$ can be expressed in terms of the Atiyah--Bott--Goldman form, in agreement with the classical topological approach.

To highlight a few applications, we first mention a characterization of projective structures of Riemann surfaces in terms of connections on intersection bundles. In particular, we settle a conjecture of Bertola--Korotkin--Norton on the comparison between the Bergman and the Bers projective structures. In classical terminology referring to work of Klein and Poincar\'e on the uniformization of Riemann surfaces, this is the problem of determining the accessory parameters of quasi-Fuchsian uniformizations. A conjecture of Cappell--Miller is also established, to the effect that their holomorphic torsion satisfies a Riemann--Roch formula. 
 
\end{abstract}
\tableofcontents
\newpage
\section{Introduction}
\begingroup
\setcounter{tmp}{\value{theorem}}
\setcounter{theorem}{0}
\renewcommand\thetheorem{\Alph{theorem}}

Chern--Simons bundles originate with the work of Witten \cite{Witten, WittenChernSimons} and Ramadas--Singer--Weitsman \cite{RSW} in the quantization of moduli spaces of flat connections, and rely on secondary invariants of vector bundles with connections on three-dimensional manifolds, developed by Chern and Simons \cite{Chern-Simons}. It was further systematically studied for compact gauge groups by Freed in \cite{FreedCS1, FreedCS2}. The purpose of this article is to introduce a formalism of complex Chern--Simons bundles, \emph{i.e.} for $\GL_{r}(\CBbb)$, $\SL_{r}(\CBbb)$ or $\PSL_{r}(\CBbb)$, in the setting of families of compact Riemann surfaces. We adopt the perspective of a functorial intersection theory of Elkik \cite{Elkikfib}, which was initiated by Deligne in \cite{Deligne-determinant} in order for him to obtain a refined understanding of the Riemann--Roch theorem in Arakelov geometry. Our approach is motivated by a set of problems where the family point of view arises naturally. The first main application is the proof of a conjecture of Cappell--Miller, to the effect that their analytic torsion for flat holomorphic vector bundles satisfies a Riemann--Roch formula, that we state as an extension of Deligne's Riemann--Roch isometry. The second main application is the proof of a conjecture of Bertola--Korotkin--Norton on the relationship between the Bergman and Bers projective structures of Riemann surfaces, which was the missing piece in the list of comparison formulas between classical projective structures. For this purpose, we develop a new tool of independent interest, called the Chern--Simons transform. It converts families of projective structures into connections on Deligne pairings, with the complex Chern--Simons bundle acting as the kernel of the transform. This theory has further implications, such as a new proof of Wolpert's result relating the first tautological class to the Weil--Petersson form on the Teichm\"uller space, or a simple construction of the classical Liouville action of Takhtajan--Zograf on the Schottky space.

To present a simplified form of our central construction, consider a family of compact Riemann surfaces $X \to S$ of genus $g\geq 3$, over a complex manifold $S$. In this setting, Simpson \cite{Simpson:moduli-2} introduces a moduli space $\Mbold_{\Bet}^{\ir}(X/S,\SL_{r}) \to S$, whose fiber over $s \in S$ is given by the irreducible part of the $\SL_{r}(\CBbb)$-character variety of $X_s$. This relative moduli space is a smooth complex analytic space over $S$. If $S$ is simply connected with a base point $0$, there is a natural holomorphic retraction $p_{0} \colon \Mbold_{\Bet}^{\ir}(X/S,\SL_{r})\to \Mbold_{\Bet}^{\ir}(X_{0},\SL_{r})$ providing an isomorphism $\Mbold_{\Bet}^{\ir}(X/S,\SL_{r})\simeq\Mbold_{\Bet}^{\ir}(X_{0},\SL_{r})\times S$. As a complex variety $\Mbold_{\Bet}^{\ir}(X_{0},\SL_{r})$ depends on $X_{0}$ only as a topological surface, and it carries a natural holomorphic symplectic 2-form called the Atiyah--Bott--Goldman form. The latter depends on the orientation of $X_{0}$, and a change of the orientation has the effect of changing the sign of the form. 

The following theorem summarizes some of the key properties of our construction. 

\begin{theorem}\label{thm:A}
There is a canonical way to  associate to $X \to S$ a holomorphic line bundle with holomorphic connection\footnote{By holomorphic connection we mean a connection which preserves holomorphicity of sections.} $\Lcal_{\CS}(X/S)$ on $\Mbold_{\Bet}^{\ir}(X/S,\SL_{r})$, such that:
\begin{itemize}
     \item[]\hspace{-0.7cm}\textbf{Functoriality:} the formation of $\Lcal_{\CS}(X/S)$ commutes with base changes $S' \to S$.
     
     \item[]\hspace{-0.7cm}\textbf{Crystalline:} if $S$ is simply connected and $0\in S$ is a base point, there is a unique isomorphism
     \begin{displaymath}
         \Lcal_{\CS}(X/S)\simeq p^{\ast}_{0}\Lcal_{\CS}(X_{0})
     \end{displaymath}
     which restricts to the identity over $0$.
     
    \item[]\hspace{-0.7cm}\textbf{Curvature:} for a single Riemann surface $X_{0}$, the curvature of $\Lcal_{\CS}(X_{0})$ is given by the Atiyah--Bott--Goldman form.
    
    \item[]\hspace{-0.7cm}\textbf{Complex metric:} for a couple of complex conjugate Riemann surfaces $(X_{0},\ov{X}_{0})$, there is a canonical flat trivialization of
    \begin{displaymath}
         \Lcal_{\CS}(X_0) \otimes \Lcal_{\CS}(\overline{X}_0).
    \end{displaymath}
    Along the locus of unitary representations, the line bundles $\Lcal_{\CS}(X_{0})$ and $\Lcal_{\CS}(\ov{X}_{0})$ are complex conjugate, and the trivialization is identified with a Hermitian metric.
\end{itemize}
\end{theorem}
\begin{definition-intro}
We call $\Lcal_{\CS}(X/S)$ the complex Chern--Simons bundle on $\Mbold_{\Bet}^{\ir}(X/S,\SL_{r})$. 
\end{definition-intro}

In the absolute case, the curvature feature guarantees that $\Lcal_{\CS}(X_{0})$ is isomorphic to the complex Chern--Simons line bundle constructed by either the method of Witten or Ramadas--Singer--Weitsman, extended to complex gauge groups and based on the Chern--Simons action. The properties stated in Theorem \ref{thm:A} are counterparts to those obtained in these classical approaches. In particular, the crystalline nature of $\Lcal_{\CS}(X/S)$ can be understood as the topological invariance in the absolute case.  However, we will see that the functorial intersection theory reveals a non-trivial geometric content of $\Lcal_{\CS}(X/S)$, unseen from a purely topological point of view. This aspect of $\Lcal_{\CS}(X/S)$, whose numerous consequences we summarize below, is one of the novel contributions of this article.


The results above, whose proofs occupy the main body of the text, can be found in Section \ref{section:CS-theory}, to which we refer the reader for details. In the rest of the introduction, we elaborate on the techniques that lead to Theorem \ref{thm:A}, followed by some applications. 



\subsection{Intersection bundles and connections}\label{intro:intersectionbundles}
\subsubsection*{Intersection bundles}\label{subsub:introintersectionbundles}
The geometrical underpinnings of the construction of $\Lcal_{\CS}(X/S)$ rely on a functorial approach to characteristic classes, initiated by Deligne in \cite{Deligne-determinant}, building upon \cite[Expos\'e XVIII, \textsection 1.3]{SGA4}, and later further developed by Elkik \cite{Elkikfib}. See also Franke \cite{FrankeChow, FrankeChern}, Mu\~noz--Garc\'ia \cite{Munoz} and Ducrot \cite{Ducrot}. This theory provides a natural way to represent some integrals of cohomology classes of vector bundles in terms of line bundles. These are thus referred to as intersection bundles. The central examples in our formalism will be refinements, for families of compact Riemann surfaces $X\to S$, of the cohomology classes $\int_{X/S} c_1(E') c_1(E'')$ and $\int_{X/S} c_2(E)$ into line bundles over $S$ denoted $\langle \det(E'), \det(E'') \rangle$, and called Deligne pairing, and  $IC_2(E)$. In particular, we have the following equality of Chern classes:
\begin{displaymath}
    c_1(IC_2(E)) = \int_{X/S} c_2(E).
\end{displaymath}
The construction of these bundles is done in such a way that standard manipulations with characteristic classes, \emph{e.g.} the Whitney sum formula, lift to functorial isomorphisms. Several properties are however not easily obtained with the existing technology, for example functorial isomorphisms corresponding to polynomial identities of Chern classes. 

For the purpose of systematically furnishing functorial isomorphisms, in Section \ref{sec:linefun} we develop a formalism of line functors, for general families of varieties $X\to S$, which is of independent interest. This is to be understood as a functorial way to associate line bundles $\Gcal(E)$ on $S$ to  vector bundles $E$ on $X$, capturing the main features of intersection bundles. We prove the following splitting principle, which is a fundamental tool that was missing in the work of Deligne and Elkik. We refer to Proposition \ref{prop:split} and Theorem \ref{thm:secondsplitting} for a precise formulation. 


\begin{theorem}\label{thm:B}
 Suppose that $\mathcal{G}$ and $\mathcal{G}^{\prime}$ are line functors. To construct a functorial isomorphism  $\mathcal{G}(E) \to \mathcal{G}^{\prime}(E)$ for a vector bundle $E$ of rank $r$,  it suffices to suppose that $E = L_1 \oplus \ldots \oplus L_r$ is a sum of line bundles. 
\end{theorem}

In the text, we further elaborate on line functors equipped with a multiplicativity structure with respect to short exact sequences, modeled on the Whitney sum formula and explained in \textsection \ref{subsec:appliadddatum}. For those, we can further reduce the conclusion of Theorem \ref{thm:B} to the case of line bundles, where constructions are usually much better understood. 

In Section \ref{sec:intersectionbund}, Theorem \ref{thm:B} and its variants are systematically applied in the study of intersection bundles in relative dimension one. To highlight an application, we provide a self-contained proof, by reduction to the case of line bundles, of Deligne's version of the Riemann--Roch theorem for families of compact Riemann surfaces $X \to S$, reviewed below in \textsection \ref{intro:applications}. 



We expect that the theory of line functors sets the foundations of future applications to analogous problems in higher dimensions. In this generality, the only available approach to intersection bundles is that of Elkik. This motivates us to rely on this framework, which also requires some further developments of her $IC_2$, including a hitherto unknown comparison with Deligne's construction. In the setting of Riemann surfaces, this allows us to provide a construction of $IC_2(E)$, whenever $E$ is of rank 2, in terms of generators and relations, mirroring the more standard setting of Deligne pairings $\langle L, M \rangle $. It also permits us to explicitly describe the complex metric in Theorem \ref{thm:A}, in the $\SL_{2}$ case, illustrating how our constructions conceptualize and improve heavy explicit computations by Fay in his studies of non-abelian theta functions \cite{Fay, Fay-2}.

\subsubsection*{Intersection connections} The work \cite{Freixas-Wentworth-1} of the last two authors in the case of Deligne pairings, motivated the interest in equipping intersection bundles, such as $IC_2(E)$, with natural connections, whenever $E$ admits one. When $E$ carries a hermitian metric,  $IC_2(E)$ is equipped with a natural intersection metric, and it is natural to require that a construction of the above  extends that of the associated Chern connections. The hermitian theory is intimately tied to the construction of Bott--Chern secondary classes as developed by Gillet--Soul\'e in arithmetic intersection theory \cite{GS:ACC1, GS:ACC2}. This is recalled and complemented in Section \ref{sec:met-int-bund}. Using a variant of Chern--Simons transgression classes, suggested by Gillet--Soul\'e in relation to Arakelov geometry and differential characters \cite{GS:diff-char}, we find such a theory for connections on $IC_2$. This is summarized in the following proposition, whose proof and properties occupy Section \ref{section:intersection-connections-generalities}.

\begin{prop-intro}
Let $(E, \nabla)$ be a holomorphic vector bundle with a $\Ccal^{\infty}$ connection, compatible with the holomorphic structure. Then, there is a canonically associated intersection connection $\nabla^{\ICmini}$, also compatible with the holomorphic structure on $IC_2(E)$. Furthermore, in terms of the Chern--Weil representatives, we have
\begin{displaymath}
    c_{1}(IC_{2}(E),\nabla^{\ICmini})=\int_{X/S} c_{2}(E,\nabla).
\end{displaymath}
\end{prop-intro}

For the purpose of constructing complex Chern--Simons bundles, we recall in Section \ref{section:universal-IC2} the moduli spaces of flat connections considered by Simpson \cite{Simpson:moduli-1,Simpson:moduli-2}. There, we are naturally faced with relative connections, rather than global ones. An application of the previous formalism thus requires an extension result of relative connections. While such extensions always exist, we stress that the associated intersection connection generally depends on this choice. However, in the flat relative setting we have the following theorem, which is a higher rank generalization of one of the main results of \cite{Freixas-Wentworth-1}. It is the object of study of Section \ref{section:canonical-extensions}. 

\begin{theorem}\label{thm:C}
Let $E$ be a holomorphic vector bundle over $X$ and $\nabla\colon E\to E\otimes\Acal_{X/S}^{1}$ a $\Ccal^{\infty}$ flat, relative connection, compatible with the holomorphic structure. Then:
\begin{enumerate}
    \item Given an extension $\widetilde{\nabla}$ of $\nabla$ to a global compatible connection, $\widetilde{\nabla}^{\ICmini}$ only depends on $\nabla$, and is therefore denoted by $\nabla^{\ICmini}$. We also refer to it as the intersection connection associated to $\nabla$.
    \item If $\nabla$ is holomorphic, then $\nabla^{\ICmini}$ is holomorphic as well.
    \item If $\nabla$ is fiberwise irreducible, then the curvature of $\nabla^{\ICmini}$ is locally described in terms of the Atiyah--Bott--Goldman form.
\end{enumerate}
\end{theorem}
 While the independence of the extension of the relative connection in the flat setting is straightforward, the other results require the non-trivial engineering of special extensions, that we call canonical extensions. This construction, of independent interest, draws inspiration from the theory of deformations of harmonic maps in non-abelian Hodge theory, as studied by Spinaci \cite{Spinaci}, and makes the description of the curvature of $\nabla^{\ICmini}$ rather transparent. The existence and uniqueness of canonical extensions generalize another key statement of \cite{Freixas-Wentworth-1}. Since an accurate presentation would be too technical for this introduction, we refer to \textsection \ref{subsec:harmonic-ext}--\textsection \ref{subsec:canonical-ext} below for a thorough discussion.

\subsubsection*{Complex Chern--Simons}\label{intro:Chern-Simons}

We are now in position to explain the construction of the relative complex Chern--Simons bundle. Recall that on $X\times_{S} \Mbold_{\Bet}^{\ir}(X/S, \SL_{r})$ there exists, locally with respect to $\Mbold_{\Bet}^{\ir}(X/S, \SL_{r})$, a universal holomorphic vector bundle with holomorphic, flat, relative connection $(\Fcal^{\un},\nabla^{\un})$, which is unique up to tensoring by line bundles coming from $\Mbold_{\Bet}^{\ir}(X/S, \SL_{r})$. 

\begin{proposition-definition-intro}
The complex Chern--Simons bundle $\Lcal_{\CS}(X/S)$ is a line bundle with a holomorphic connection on $\Mbold_{\Bet}^{\ir}(X/S,\SL_{r})$, locally defined as the dual of  $IC_2(\Fcal^{\un})$, together with the dual of the intersection connection associated to $\nabla^{\un}$. 
\end{proposition-definition-intro}

We notice that it is part of the statement that the locally defined $IC_{2}$ bundles with connections globalize. This follows from our extensions of the theory of intersection bundles and intersections connections, previously reviewed. The above definition is dealt with in Section \ref{section:CS-theory}, where we also establish Theorem \ref{thm:A} and other fundamental facts of complex Chern--Simons line bundles, based on developments such as Theorem \ref{thm:C}. 

\begin{remark-intro} 
Actually, one can prove that the holomorphic line bundle underlying $\Lcal_{\CS}(X/S)$ extends to all of $\Mbold_{\Bet}(X/S,\SL_{r})$. In a sense, we show that the connection extends too, cf. Corollary \ref{cor:extension-int-connection}. 
\end{remark-intro}

We notice that some of the statements in this work are differential geometric counterparts to algebraic results by Faltings in \cite[Section IV]{Faltings:stable}, in his approach to the Hitchin connection \cite{Hitchin:flat}. The complex analytic method allows us to ultimately treat problems that are not of a purely algebraic nature, such as the conjectures of Cappell--Miller and Bertola--Korotkin--Norton, presented below. Let us also mention the differential geometric re-development of Faltings' work sketched by Ramadas \cite{Ramadas}, of which our formalism provides rigorous foundations. More recently, other authors have investigated variants of Faltings' observations. See for instance Biswas--Hurtubise \cite{Biswas-Hurtubise}. In comparison with other constructions in the literature, intersection connections yield objects with a better behaviour, such as the functorial and the crystalline properties. We also refer to Takhtajan  \cite{Takhtajan:Goldman}, who studies the symplectic geometry of the moduli spaces of stable vector bundles with flat connections, by analogy with projective structures on Riemann surfaces. In forthcoming work \cite{EFW:applications}, we plan to explain how one can approach some of the results of Biswas--Hurtubise in the light of the work of Takhtajan, in a manner analogous to our theory of Chern--Simons transforms of projective structures presented below. 

\subsection{Applications}\label{intro:applications}

\subsubsection*{Projective structures and Chern--Simons transforms}
A projective structure on a compact Riemann surface of genus $g\geq 2$ is a system of holomorphic charts with changes of coordinates induced by complex M\"obius transformations. Prototypes are provided by the classical uniformization by Fuchsian or Schottky groups. Projective structures give rise to $\PSL_{2}(\CBbb)$-representations of the fundamental group, called holonomy representations. For a family of compact Riemann surfaces, one can likewise consider families of projective structures. There is a universal space of relative projective structures denoted by $\pi\colon\Pcal(X/S)\to S$, which has the structure of a holomorphic affine bundle with respect to the action of the bundle of relative holomorphic quadratic differentials. The holonomy representations of the projective structures on the fibers induce a holomorphic map
\begin{displaymath}
    \hol\colon\Pcal(X/S)\longrightarrow\Mbold_{\Bet}^{\ir}(X/S,\PSL_{2}),
\end{displaymath}
called the relative holonomy map. The $\PSL_{2}$ version of the complex Chern--Simons line bundle can be pulled back to $\Pcal(X/S)$, which is denoted by $\Kcal_{\CS}(X/S)$. This is a holomorphic line bundle endowed with a holomorphic connection. 

\begin{theorem}\label{thm:D}
There is a canonical functorial isomorphism of holomorphic line bundles
\begin{displaymath}
    \Kcal_{\CS}(X/S)\simeq \pi^{\ast}\langle\omega_{X/S},\omega_{X/S}\rangle.
\end{displaymath}
\end{theorem}
The proof is not a simple computation of characteristic classes, but rather is based upon the functoriality properties of the intersection bundles and the classical interpretation of projective structures in terms of theta characteristics. Notice that, even over a simply connected base, the theorem reveals a non-trivial geometric content of the complex Chern--Simons line bundle, unattainable by the classical topological approach. For details on the proof, we refer to Theorem \ref{prop:iso-Kcal-Deligne}.

An immediate consequence of Theorem \ref{thm:D} is that, given a smooth family of projective structures on $X\to S$, namely a $\Ccal^{\infty}$ section $\sigma\colon S\to\Pcal(X/S)$, by pulling back the built-in connection on $\Kcal_{\CS}(X/S)$ we obtain an induced connection on the Deligne pairing $\langle\omega_{X/S},\omega_{X/S}\rangle$, which is holomorphic if $\sigma$ is. This connection is called the Chern--Simons transform of $\sigma$, and we think of $\Kcal_{\CS}(X/S)$ as being the kernel of the transform. An essential property of the Chern--Simons transform is that it is compatible with the holomorphic structure of the Deligne pairing, and it is linear with respect to the affine structure of $\Pcal(X/S)$, cf. Proposition \ref{prop:int-conn-Deligne-compatible}. 

The first application of Chern--Simons transforms is a new, conceptual proof of the following observation due to S.-Y. Zhao \cite{Zhao:projective}, in turn implicitly related to Ivanov \cite{Ivanov}:
\begin{cor-intro}\label{cor-intro:Zhao}
Let $X\to S$ be a non-isotrivial family of compact Riemann surfaces of genus $g\geq 2$, over a compact Riemann surface $S$. Then $X\to S$ admits no holomorphic relative projective structures.
\end{cor-intro}
Indeed, if such a family existed, then its Chern--Simons transform would be a holomorphic connection on $\langle\omega_{X/S},\omega_{X/S}\rangle$, necessarily flat since $S$ is one-dimensional. Therefore, the corresponding Chern class would vanish. Since this class is given by the self-intersection number $(\omega_{X/S}^{2})$, it contradicts Arakelov's positivity theorem \cite{Arakelov}. Let us mention an alternative approach to the corollary within the theory of holomorphic foliatons, due to Deroin--Guillot \cite[Corollary 6.1]{Deroin-Guillot}. These statements seem to contradict \cite[Remark 3.15]{Biswas-Raina}, which claims that holomorphic relative projective structures always exist. 

We next place ourselves in the setting of the Teichm\"uller space $\Tcal=\Tcal(X_{0})$, where $X_{0}$ is a compact Riemann surface of genus $g\geq 2$. Let $\Ccal\to\Tcal$ be the universal Teichm\"uller curve of Bers. 
\begin{theorem}\label{thm:E}
The Chern--Simons transform establishes a canonical bijective correspondence between smooth families of projective structures parametrized by $\Tcal$, and connections on $\langle\omega_{\Ccal/\Tcal},\omega_{\Ccal/\Tcal}\rangle$, which are compatible with the holomorphic structure. Furthermore: 
\begin{enumerate}
    \item Holomorphic families exactly correspond to holomorphic connections.
    \item The Chern--Simons transform of the Fuchsian uniformization $\sigma_{\Fu}\colon\Tcal\to\Pcal(\Ccal/\Tcal)$ is the Chern intersection connection on the Deligne pairing, denoted $\nabla^{\chmini}$, associated to the hyperbolic metric.
\end{enumerate}
\end{theorem}


For a proof, we refer to Theorem \ref{thm:rel-proj-str-rel-conn} and Theorem \ref{thm:fuchsian-chern}. That the correspondence is bijective rests on  symplectic aspects of $\Pcal(\Ccal/\Tcal)$, going back to Kawai \cite{Kawai} and more recently systematically studied by Loustau \cite{Loustau}, together with Theorem \ref{thm:A} and the linearity of the Chern--Simons transform. The case of $\sigma_{\Fu}$  utilizes Hitchin's theory of uniformizing Higgs bundles, and indicates the potential of the methods of non-abelian Hodge theory in the study of intersection connections. A variant of Theorem \ref{thm:E} has been independently announced, employing a somewhat \emph{ad hoc} construction, by Biswas--Favale--Pirola--Torelli \cite{BFPT}. While the bulk of their proof focuses on the holomorphicity property of (1) of Theorem \ref{thm:E}, this is automatically satisfied in our setting, since the built-in connection of $\Kcal_{\CS}(\Ccal/\Tcal)$ is already holomorphic. As an asset, our theory brings out the origin of this correspondence. 

An immediate outcome of (2) of Theorem \ref{thm:E} is the following celebrated result of Wolpert \cite{Wolpert:Chern} (see Corollary \ref{cor:Wolpert-curvature}): 
\begin{cor-intro}[Wolpert]
In terms of the Weil--Petersson K\"ahler form on $\Tcal$, the curvature of $\nabla^{\chmini}$ is $\frac{1}{\pi^{2}}\omega_{\WP}$.
\end{cor-intro}
The key point is Goldman's interpretation of the Weil--Petersson form in terms of the Atiyah--Bott--Goldman form for $\PSL_{2}(\RBbb)$ \cite{Goldman}. Actually, in light of Theorem \ref{thm:E}, Goldman and Wolpert's results are equivalent.  We clarify that the aforementioned work of Biswas--Favale--Pirola--Torelli does not entail Wolpert's curvature formula, but rather builds on an equivalent formulation in terms of the Hodge bundle endowed with the Quillen metric. Thus, the implications of our methods are broader in scope.

A further application of Theorem \ref{thm:E} shows that the intersection metric on  $\langle\omega_{\Ccal/\Tcal},\omega_{\Ccal/\Tcal}\rangle$ associated to the hyperbolic metric, has a holomorphic extension to the quasi-Fuchsian space  $\Qcal=\Tcal\times\ov{\Tcal}$, parametrizing couples $(X,Y)$ of compact Riemann surfaces deforming $(X_{0},\ov{X}_{0})$. The Teichm\"uller space is diagonally embedded in $\Qcal$, as a totally real submanifold. Let us denote by $\Xcal^{+}\to\Qcal$ and $\Xcal^{-}\to\Qcal$ the universal curves, arising from Bers' simultaneous uniformizations by quasi-Fuchsian groups. The next statement is a rough form of Theorem \ref{thm:CS-Fuchsian}, to which we refer for the precise meaning. It plays a key role in the proof of the conjecture of Bertola--Korotkin--Norton. 
\begin{theorem}\label{thm:F}
The intersection metric on $\langle\omega_{\Ccal/\Tcal},\omega_{\Ccal/\Tcal}\rangle$ associated to the hyperbolic metric, uniquely extends to a holomorphic trivialization of $\langle\omega_{\Xcal^{+}/\Qcal},\omega_{\Xcal^{+}/\Qcal}\rangle\otimes \langle\omega_{\Xcal^{-}/\Qcal},\omega_{\Xcal^{-}/\Qcal}\rangle$.
\end{theorem}

We conclude the discussion on projective structures by a reference to the constructions of natural Weil--Petersson potentials on the Schottky and quasi-Fuchsian spaces, by Takhtajan--Zograf \cite{Zograf-Takhtajan} and Takhtajan--Teo \cite{Takhtajan-Teo}, respectively. These potentials are known as classical Liouville actions. Our methods allow for a conceptual, simpler approach, cf. \textsection\ref{subsec:Schottky}--\textsection\ref{subsec:QF}. For instance, let $\tau_{\QF}$ denote the holomorphic trivialization in Theorem \ref{thm:F}, and endow the Deligne pairings $\langle\omega_{\Xcal^{\pm}/\Qcal},\omega_{\Xcal^{\pm}/\Qcal}\rangle$ with the intersection metrics associated to the hyperbolic metric. Then, the Liouville action on the quasi-Fuchsian space is identified with $\log\|\tau_{\QF}\|$, up to normalization. We also mention related work of Guillarmou--Moroianu \cite{Guillarmou-Moroianu}. These authors construct Chern--Simons line bundles on Teichm\"uller spaces of convex cocompact hyperbolic 3-manifolds, derived from a regularized Chern--Simons action, and in turn connected to the Liouville action. In \textsection \ref{subsection:convex-hyp} we show that their results can also be subsumed in our formalism. From this perspective, the existence of the trivialization $\tau_{\QF}$ reflects that, by Bers' simultaneous uniformization, couples of Riemann surfaces $(X,Y)$ as above appear as the conformal boundaries of quasi-Fuchsian hyperbolic 3-manifolds. This picture is consistent with the axioms advocated by topological quantum field theories \cite{Atiyah}. 

\subsubsection*{Riemann--Roch theorems and the Cappell--Miller torsion}
We return to the setting of a family of compact Riemann surfaces $X\to S$ of genus $g\geq 3$. Let $E$ be a holomorphic vector bundle on $X$.  The determinant of the cohomology of $E$ is a holomorphic line bundle $\lambda(E)$ on $S$, whose fiber at $s \in S$ is given by 
\begin{displaymath}
    \lambda(E)_s = \det H^0(X_s, E) \otimes \det H^1(X_s, E)^{\vee}.
\end{displaymath}
Going back to Quillen \cite{Quillen-Cauchy-Riemann}, whenever $T_{X/S}$ and $E$ are equipped with hermitian metrics, there is a natural hermitian metric on $\lambda(E)$, called the Quillen metric. It is formed by a combination of the $L^{2}$ metric, provided by Hodge theory, and the regularized determinant of the Dolbeault Laplacian $\Delta^{0,1}_{\ov{\partial}_{E}}$ acting on $A^{0,1}(X_{s},E)$. The latter is defined as 
\begin{equation}\label{eq:determinantofLaplacian}
    \mathrm{det}^{\prime} \Delta^{0,1}_{\ov{\partial}_{E}} = \exp(-\zeta'_{\ov{\partial}_{E}}(0)),
\end{equation} 
where $\zeta_{\ov{\partial}_{E}}(t) = \sum \lambda_k^{-t}$ is the zeta function on the non-zero eigenvalues of $\Delta^{0,1}_{\ov{\partial}_{E}}$, analytically continued to a neighborhood of $t=0$.  

As a consequence of the work of Deligne \cite{Deligne-determinant}, heavily relying on the combined works of Bismut--Freed \cite{freed1, freed2}, Bismut--Gillet--Soul\'e \cite{BGS1, BGS2, BGS3}, Bismut--Lebeau \cite{Bismut-Lebeau} and Gillet--Soul\'e \cite{GS:ARR}, the Grothendieck--Riemann--Roch formula for $E$ lifts to an isometry, up to an explicit topological constant, of hermitian line bundles
\begin{equation}\label{eq:deligneiso-intro}
    \lambda(E)^{12} 
    \isorightarrow
    \langle \omega_{X/S}, \omega_{X/S} \rangle^{\rk E} \otimes \langle \det E,  \omega_{X/S}^{-1} \otimes \det E \rangle^{6} \otimes IC_{2}(E)^{-12}.
\end{equation}
Here $\lambda(E)$ is considered with the Quillen metric, and the intersection bundles with their intersection metrics. We usually refer to \eqref{eq:deligneiso-intro} as the Deligne--Riemann--Roch isomorphism, or self-explanatory variants of this terminology.

In \cite{Cappell-Miller}, Cappell--Miller introduced a natural extension of the Quillen metric to flat connections on holomorphic vector bundles, nowadays called the holomorphic Cappell--Miller torsion. For a vector bundle endowed with a flat metric, the Quillen metric and the Cappell--Miller torsion of the associated Chern connection can be identified. In general, in the introduction of \emph{op. cit.}, they conjectured that it satisfies similar properties as the Quillen metric. A compact formulation of this expectation is that the Cappell--Miller torsion should fit into a Riemann--Roch isomorphism. In the case of flat line bundles, this was accomplished, under the form of an extension of Deligne's isometry, by the last two authors in \cite{Freixas-Wentworth-2}.\footnote{As a matter of fact, in \emph{op. cit.} a preliminary form of the present article was announced under the title \emph{Complexified Chern--Simons and Deligne's intersection bundles}.} In Section \ref{section:CM-ARR}, we address the case of general rank (cf. Theorem \ref{theorem:DRR-isom-flat} and Theorem \ref{thm:variant-DRR-flat}):

\begin{theorem}\label{thm:G}
The Cappell-Miller torsion of flat irreducible vector bundles on compact Riemann surfaces satisfies a Riemann--Roch formula, extending Deligne's isometry in the flat unitary case.
\end{theorem}
The functorial setting is particularly adapted to the problem at hand, since the Cappell--Miller torsion is a section of a combination of determinant bundles, rather than a numerical invariant. The key point of the proof is that, for a flat vector bundle on a fixed compact Riemann surface, the Cappell--Miller torsion depends holomorphically on the holonomy representation of the flat connection. In turn, this is proven by means of Kato's perturbation theory for closed linear operators \cite{Kato}, applied to a suitable holomorphic family of non-self-adjoint Laplacians. This is then combined with the complex metrics in Theorem \ref{thm:A}, and the fact that \eqref{eq:deligneiso-intro} induces an isometry in the case of flat unitary vector bundles. 

We stress that Theorem \ref{thm:G} furnishes an interpretation of the complex metrics in Theorem \ref{thm:A} in terms of the Cappell--Miller torsion. This is a new property of complex Chern--Simons bundles, inaccessible to the classical topological approach. It is analogous to the realization of the $\SU(2)$ version of the Chern--Simons bundle by Ramadas--Singer--Weitsman as a determinant bundle with a Quillen metric on the space of unitary connections on a compact Riemann surface. Their observation is based on a curvature computation, which does not apply to the Cappell--Miller torsion. In our setting, the Deligne--Riemann--Roch approach provides an appropriate alternative to their curvature argument. 

\subsubsection*{A conjecture of Bertola--Korotkin--Norton}
In \cite{Bertola}, Bertola--Korotkin--Norton studied a symplectic equivalence relation between families of projective structures parametrized by the Teichm\"uller space. They were led to predict the relationship between the Bergman and Bers quasi-Fuchsian projective structures. The formulation they propose is a quasi-Fuchsian analogue of a result of Takhtajan--Zograf \cite{Zograf-Takhtajan:Bergman}. We here confirm their expectation as an application of our results.

Let $X_{0}$ be a compact Riemann surface of genus $g\geq 2$, endowed with a marking of the fundamental group, $\Tcal=\Tcal(X_{0})$ its Teichm\"uller space and $\Ccal\to\Tcal$ the universal curve. We consider two families of projective structures. The first one is given by the Bergman projective connection on the fibers of $\Ccal\to\Tcal$, which depends on the marking of $X_{0}$, and is related to the Bergman and Schiffer kernels.\footnote{The terminology seems to have been coined by Hawley--Schiffer \cite{Hawley-Schiffer} in connection with the Bergman kernel on planar domains. The Bergman kernel in \emph{op. cit.} is called the Schiffer kernel by some authors, \emph{e.g.} Fay \cite{Fay:Fourier}. See also Bergman--Schiffer \cite{Bergman-Schiffer} and references therein.} The second one is defined by the Bers' quasi-Fuchsian uniformization of the fibers of $\Ccal\to\Tcal$. We denote the corresponding sections $\Tcal\to\Pcal(\Ccal/\Tcal)$ by $\sigma_{\Be}$ and $\sigma_{\Bers}$. The determinant of the hyperbolic Laplacian, defined as in \eqref{eq:determinantofLaplacian}, determines a $\Ccal^{\infty}$ function $\mathrm{\det}^{\prime}\Delta_{\hyp}$ on $\Tcal$. After Kim \cite{Kim}, there is a holomorphic extension thereof, denoted by $\widetilde{\det}\ \Delta_{\hyp}$, from the totally real embedding of $\Tcal$ to the whole quasi-Fuchsian space $\Qcal=\Tcal\times\ov{\Tcal}$, and we henceforth restrict it to the Bers slice $\Tcal\simeq\Tcal\times\lbrace\ov{X}_{0}\rbrace\subset\Qcal$. Finally, for Riemann surfaces $X\in\Tcal$, denote by $\Omega$ their period matrices, and by $\Omega_{0}$ the period matrix of $X_{0}$. The following theorem confirms the conjecture of Bertola--Korotkin--Norton \cite[Conjecture 1.1]{Bertola} (cf. Theorem \ref{thm:Bertola-conjecture}):

\begin{theorem}\label{thm:H}
On the Bers slice, the projective structures of Bergman and Bers are related by
\begin{displaymath}
     \sigma_{\Be}-\sigma_{\Bers}=6\pi\partial\log\left(\frac{\widetilde{\det}\ \Delta_{\hyp}}{\det(\Omega-\ov{\Omega}_{0})} \right).
\end{displaymath}
\end{theorem}
The proof is based on the following considerations, which underscore the importance of the Deligne isometry. One can use the theory of Chern--Simons transforms to reinterpret the Bergman and Bers sections as connections on a Deligne pairing. The Bers connection has an extension to the quasi-Fuchsian space, and can be related to Kim's complex valued determinant thanks to Theorem \ref{thm:F} and Deligne's isometry. The Bergman connection, on the other hand, can be written in terms of the Fuchsian connection by the work of Takhtajan--Zograf \cite{Zograf-Takhtajan:Bergman}, and can be understood via Theorem \ref{thm:E} and Deligne's isometry again. We highlight that a byproduct of the method is a new, conceptual and self-contained proof of Kim's extension theorem. 

It is important to stress that Kim's determinant is not an instance of the Cappell--Miller torsion, and therefore Theorem \ref{thm:H} does not seem to follow from general properties of the latter. Nevertheless, in the conclusion of \textsection \ref{subsec:DRR-quasi-Fuchsian}, we show that a non-trivial relationship can be established. This is stated as a correspondence between Kim's determinant and a Cappell--Miller torsion via a natural isomorphism between two determinant bundles on the quasi-Fuchsian space. The isomorphism involves the crystalline property of complex Chern--Simons line bundles, whose abstract nature  makes it unsuitable for explicit purposes such as Theorem \ref{thm:H}. 

We notice that the proof of Theorem \ref{thm:H} relies on a closed and explicit expression of the difference between the Bergman and Fuchsian structures. The difference between the Bergman and Schottky structures was similarly computed by Zograf and McIntyre--Takhtajan in \cite{McIntyre-Takhtajan}. Combined with Theorem \ref{thm:H}, the difference of all four of the above mentioned projective structures can be computed. By Theorem \ref{thm:E}, the Chern--Simons transform of the Fuchsian section itself is calculated to be an intersection connection of the hyperbolic metric, and we derive the following corollary:

\begin{cor-intro}
The Chern--Simons transforms of the Fuchsian, Schottky, Bergman and Bers projective structures can all be described by closed explicit expressions. 
\end{cor-intro}

\begin{remark-intro}
The original question of Bertola--Korotkin--Norton was formulated in terms of a quasi-Fuchsian analogue of Selberg's zeta function. The necessary properties of this function were claimed, without proof, in Bowen's posthumous article \cite[Remark b, p. 272]{Bowen}. To the best of our knowledge, the validity has not been confirmed to date. However, Kim shows in \cite[Section 3]{Kim} that in a neighborhood of  the point $(X_{0},\ov{X}_{0})$,  $\widetilde{\det}\ \Delta_{\hyp}$ can be computed as the determinant of a non-self-adjoint Laplacian, which reveals an equally interesting property of $\sigma_{\Be}-\sigma_{\Bers}$.
\end{remark-intro}

\subsubsection*{Future applications}
There are other applications that we did not include in this paper for reasons of length, and that we plan to develop in future work. We briefly mention the theme of three of them.
\begin{itemize}
    \item Our treatment of families of projective structures can be generalized to the case of opers; families of opers give rise to connections on Deligne pairings. 
    \item The relative complex Chern--Simons bundles and the results on the Cappell--Miller torsion, lead to a new understanding of Hitchin's hyperholomorphic line bundle with meromorphic connection, on the moduli space of Deligne's $\lambda$-connections. This is a higher rank extension of \cite[\textsection 5.4]{Freixas-Wentworth-1}.
    \item An Arakelov theoretic formalism for vector bundles with flat connections on arithmetic surfaces, generalizing \cite{Freixas-Wentworth-2} to higher rank.
\end{itemize}

\subsubsection*{Acknowledgements} We warmly thank S. Boucksom and L. Takhtajan for several remarks on a preliminary version of this article, leading to an improvement of the presentation.

\subsection{Notation and conventions}

Depending on the context, we will interchangeably discuss either smooth projective curves or their analytifications compact Riemann surfaces. They are always assumed to be connected. 

In this article, ``functorial'' will mean that a formulation commutes with base change. For example, on a family of curves $X \to S$, the association $E\mapsto IC_2(E)$ is functorial. This concretely means that for any map $g: S' \to S$, there is a natural isomorphism $g^* IC_2(E) \overset{\sim}{\to} IC_2(g^{*} E)$, where we abuse notation and also denote by $g$ the induced map $X \times_S S' \to X$. Likewise, there are isomorphisms of functors associated to isomorphisms of $S$-families $X\to X'$.

When performing GIT quotients, we will write $\GL_{r/\CBbb}$ for the general linear group scheme over $\CBbb$, while the notation $\GL_{r}(\CBbb)$ will be reserved to its set of complex points, namely the usual group of invertible matrices with complex coefficients. Similar conventions apply to other algebraic groups.

We will adopt a sheaf theoretic point of view on connections. If $E$ is a $\Ccal^{\infty}$ complex vector bundle on a manifold $X$, identified with its sheaf of $\Ccal^{\infty}$ sections, then a connection on $E$ is a $\CBbb$-linear morphism of sheaves $\nabla\colon E\to E\otimes_{\Ccal^{\infty}_{X}}\Acal^{1}_{X}$, satisfying the Leibniz rule with respect to the $\Ccal^{\infty}_{X}$-module structure of $E$. Here, we write $\Acal^{1}_{X}$ for the sheaf of complex differential $1$-forms on $X$. If $X$ is moreover a complex manifold and $E$ a holomorphic vector bundle, then we say that $\nabla$ is compatible with the holomorphic structure of $E$, or more simply compatible, if its $(0,1)$ projection coincides with the Dolbeault operator of $E$. Identifying $E$ with its sheaf of germs of holomorphic sections, a compatible connection can equivalently be seen as a $\CBbb$-linear morphism of sheaves $\nabla\colon E\to E\otimes_{\Ocal_{X}}\Acal^{1,0}_{X}$, satisfying the Leibniz rule with respect to the $\Ocal_{X}$-module structure of $E$. These notions have relative counterparts for submersions of complex manifolds $f\colon X\to S$. In this article we will exclusively deal with compatible connections. After the previous discussion, we will use notation such as $\nabla\colon E\to E\otimes\Acal^{1,0}_{X/S}$, to mean a compatible, relative connection. Notice this is now $f^{-1}\Ocal_{S}$-linear. Compatible connections are called holomorphic when they preserve holomorphicity of sections. For a holomorphic relative connection we rather write $\nabla\colon E\to E\otimes\Omega^{1}_{X/S}$, where $\Omega^{1}_{X/S}$ is the sheaf of relative holomorphic differentials. Holomorphic relative connections can more generally be defined for smooth morphisms of complex analytic spaces.

If $L$ is a complex line bundle on a manifold $X$, it will be useful to interpret a $\Ccal^{\infty}$ hermitian metric on $L$ as an isomorphism $h\colon L\otimes\ov{L}\overset{\sim}{\to}\underline{\CBbb}_{X}$, where $\ov{L}$ is the complex conjugate line bundle and $\underline{\CBbb}_{X}$ is the trivial line bundle on $X$. Then $h$ defines a trivialization of $L\otimes\ov{L}$, given by $h^{-1}(1)$. We will sometimes confuse both points of view, \emph{i.e.} the metric and the trivialization, if there is no danger of misunderstanding. This interpretation of hermitian metrics will be systematically used in \textsection \ref{subsection:complex-metrics} and in Section \ref{section:CM-ARR}.

\endgroup
\setcounter{theorem}{\thetmp}
\section{Line functors and two splitting principles}\label{sec:linefun}

Suppose that we are given a flat family of algebraic varieties $X\to S$. In this section we will consider line functors, which are functors $\Gcal$ taking vector bundles $E$ on $X$ and associating line bundles on $S$, depending on $E$ and $X\to S$ in a functorial way. Our main examples, which appear in Section \ref{sec:intersectionbund}, are the functorial lifts of direct images of characteristic classes introduced by Deligne \cite{Deligne-determinant} and Elkik \cite{Elkikfib}. With this in mind, in this section we provide fundamental tools akin to the splitting principles in classical intersection theory. Particular care is taken regarding multiplicative properties of line functors. These functorial splitting principles were not developed previously, and constitute a complement to the work of Deligne and Elkik of independent interest. While this section is utilized in the rest of the text, it can be safely omitted in a first reading. 

In the context of Chow categories and the functorial intersection theory developed by Franke in \cite{FrankeChow, FrankeChern}, splitting principles were established under additional hypotheses. Unfortunately, the approach is Chow homological and, to produce line bundles, additional regularity assumptions on our varieties must be imposed. These are not satisfied in some of our applications.

To fix a setting and simplify language, we consider the category of algebraic varieties over a field. The results are valid in greater generality, for instance for \emph{fpqc} families between schemes, and also in the category of complex analytic spaces and for families which admit local sections, such as smooth morphisms. The verification is straightforward and is left to the reader. 

\subsection{Line functors}\label{subsec:linefunctors}
Suppose that we are given a class of families of algebraic varieties $\mathcal{P}$, always meaning a class of flat families $f \colon X\to S$, assumed to be surjective over $S$. These are hence always faithfully flat. This class is supposed to be closed under base changes along morphisms of varieties $S' \to S$, and closed under isomorphisms of families. The main relevant cases in this article will all be flat projective families of curves, or the class generated, via pullbacks and isomorphisms, by a single flat family of projective curves over a variety $S$. 

For any family $f$ in the class $\Pcal$, we will consider functors of the type
\begin{equation} \label{def:determinantfunctor} 
   \Gcal_f \colon (\hbox{Vect}_{/X}, \text{\emph{iso}})^{k} \to (\hbox{Line bundles}_{/S}, \text{\emph{iso}}),
\end{equation}
from the category of $k$-tuples of vector bundles on $X$, with $k$-tuples of isomorphisms as morphisms, to that of the category of line bundles on $S$, also with isomorphisms as morphisms. This means that for any $k$-tuple of vector bundles $\underline{E}=(E_{1},\ldots, E_{k})$ on $X$, there is an association 
\begin{displaymath}
    \underline{E} \mapsto \Gcal_f(\underline{E})\quad\text{line bundle on } S,
\end{displaymath}
and for isomorphisms $\varphi: \underline{E} \to \underline{E}'$ there are isomorphisms $[\varphi]: \Gcal(\underline{E}) \to \Gcal(\underline{E}')$ of line bundles, compatible with compositions of isomorphisms. In the following definition and beyond, it is understood that standard manipulations with vector bundles, such as pullback, are extended componentwise to $k$-tuples.
\begin{definition}\label{def:linefunctorkvar}
\begin{enumerate}
    \item A line functor of $k$-variables $\Gcal$ consists in giving, for any family $f$ in the class $\Pcal$, a functor $\Gcal_{f}$ as in \eqref{def:determinantfunctor}, fulfilling:
    \\
    \begin{enumerate}
         \item Whenever there is an isomorphism $i\colon X' \to X$ of $S$-varieties in  $\mathcal{P}$, there is a natural equivalence of functors $\Gcal_{f \circ i} \circ i^{\ast} $ and $\Gcal_{f}$, compatible with compositions of isomorphisms.
         \item The formation of $\Gcal$ commutes with base changes $g \colon S' \to S$ in  $\mathcal{P}$, i.e. there are natural isomorphisms of functors
          \begin{displaymath}
              g^\ast \Gcal_{f} \to \Gcal_{f'} \circ {g'}^\ast
          \end{displaymath}
         where $f': X \times_S S' \to S'$ and $g': X \times_S S' \to X$ are the natural maps. It should be compatible with compositions $S'' \to S' \to S$ in the obvious way. Since the formation of Cartesian products is only unique up to unique isomorphism, this employs (a).
         \\
\end{enumerate}
    \item A natural transformation of two line functors $\Gcal$ and $\Gcal'$ consists in natural transformations of functors $\Gcal_{f}\to\Gcal^{\prime}_{f}$, for $f$ in $\Pcal$, commuting with the above base changes. 
\end{enumerate}

\end{definition}
The class $\mathcal{P}$ will often be clear from the context, and will be omitted. When the family $f$ is clear too, we simplify $\Gcal_f$ as $\Gcal$. Finally, line functors of $1$-variable will simply be called line functors. Most of the time we will focus on the latter, but the generality of $k$-variables will be needed in some intermediate results.

\subsection{Transgression bundles and a first splitting principle}\label{sec:splitting2}

The first splitting principle of this section will employ a construction sometimes referred to as transgression sequences, originally introduced by Bismut--Gillet--Soul\'e in the context of Bott--Chern secondary classes, cf. \cite[\textsection 1.2.3]{GS:ACC1}. For this, suppose we are given an exact sequence
\begin{equation}\label{eq:varepsilonsequence}
    \varepsilon\colon 0 \to E' \to E \to E'' \to 0
\end{equation} 
on a variety $X$. On $X\times \PBbb^1$ we will construct an exact sequence 
\begin{equation}\label{transgression}
    \widetilde{\varepsilon}\colon 0 \to p^* E'(1) \to \widetilde{E} \to p^* E'' \to 0. 
\end{equation}
Here, $p$ denotes the natural projection $p: X \times \PBbb^1 \to X$ and $\widetilde{E} = \left(p^* E'(1) \oplus p^* E  \right)/p^* E'$, where the map $p^* E' \to  p^* E'(1)\oplus p^* E  $ is determined by combining the inclusion $E' \to E$ in \eqref{eq:varepsilonsequence} with $p^* E' \to p^* E^{\prime}(1)$ given by $f \mapsto f \otimes \sigma$. Here, we choose $\sigma$ to be the standard section of $\Ocal(1)$ vanishing only at $\infty.$ The map $p^* E'(1) \to \widetilde{E}$ is the natural one on the second coordinate. In practice, we will only need the following properties, which follow from the construction:
 \begin{enumerate}
    \item $\widetilde{\varepsilon}_{\mid X\times\lbrace t\rbrace}$ is isomorphic to $\varepsilon$ in $\eqref{eq:varepsilonsequence}$, for $t \neq \infty$;
    \item $\widetilde{\varepsilon}_{\mid X\times\lbrace\infty\rbrace}$ is isomorphic to the standard split exact sequence:
    \begin{equation}\label{eq:varepsilonsequenceprime}
        \varepsilon^{\prime}\colon 0\to E'\to E'\oplus E''\to E''\to 0;
    \end{equation}
    \item $\widetilde{\varepsilon}$ is functorial.
\end{enumerate}
In the constructions below, it will be necessary to fix the above isomorphisms, compatibly with base changes $X' \to X$ and isomorphisms of exact sequences. This is always possible, and we choose any of them. 

The following proposition, asserting the existence of a splitting isomorphism for line functors, is a useful application of transgression sequences, and provides a geometric interpretation of more classical versions of the  splitting principle. 

\begin{proposition}\label{prop:split}

Let $\Gcal$ be a line functor. Then, for any exact sequence of vector bundles $\varepsilon$ as in \eqref{eq:varepsilonsequence} there is a unique isomorphism
\begin{displaymath}
    \psi_\varepsilon\colon \Gcal(E) \to \Gcal(E' \oplus E'')
\end{displaymath}
which:
\begin{enumerate}
    \item is functorial with respect to pullback and isomorphisms of exact sequences;
    \item is the identity whenever $\varepsilon$ is the standard split exact sequence $\varepsilon'$ in \eqref{eq:varepsilonsequenceprime}.
\end{enumerate}

\end{proposition}

\begin{proof}
We begin by showing uniqueness. Suppose there are two different such isomorphisms, $\psi, \psi'$. The transgression exact sequence \eqref{transgression} induces isomorphisms \begin{displaymath}
    \psi_{\widetilde{\varepsilon}}, \psi_{\widetilde{\varepsilon}}' \colon \Gcal(\widetilde{E}) \to \Gcal(p^* E'(1) \oplus p^* E'')
\end{displaymath} 
for the family $X \times \PBbb^1 \to S \times \PBbb^1.$ We claim this is determined by the split case. The difference between $\psi_{\widetilde{\varepsilon}}$ and $\psi_{\widetilde{\varepsilon}}'$ is given by an element $g \in H^0(\Ocal_{S \times \PBbb^1}^{\times})$. Since the restriction map $H^0(\Ocal_{S \times \PBbb^1}^{\times}) \to H^0(\Ocal_S^{\times})$ along $S \times \{\infty\} \to S \times \PBbb^1 $ is an isomorphism, $g$ is determined by the restriction to the fiber at $\infty$. The restriction to $\infty$ of $g$ is the image in $H^0(\Ocal_S^{\times})$ of the difference between $\psi_{\varepsilon'}$ and $\psi_{\varepsilon'}'$. But since the isomorphism commutes with base change, this corresponds to the difference of isomorphisms in the split case $\varepsilon'$, which are both the identity themselves, so $g$ itself is the identity. Since the fiber at 0 of $\psi_{\widetilde{\varepsilon}} = \psi_{\widetilde{\varepsilon}}'$ is identified with both $\psi_\varepsilon$ and $\psi_{\varepsilon}'$, we see that they coincide. 

To prove existence, we first notice that the bundle  $\Gcal(\widetilde{E})$ is a line bundle on $S \times \PBbb^1$. By the classification of line bundles on projective bundles, described in Lemma \ref{lemma:linesonprojectivebundles} below, it is hence non-canonically isomorphic to a line bundle of the form $p^{\ast} L \otimes \Ocal(k)$, where we also denote by $p$ the natural projection $S \times \PBbb^1 \to S$. The restrictions of $\Gcal(\widetilde{E})$ to $0$ and $\infty$ are canonically isomorphic to $\Gcal(E)$ and $\Gcal(E' \oplus E'')$, as well as non-canonically isomorphic to $L.$ Choose an arbitrary such isomorphism in the case of the transgression exact sequence, $\Gcal(\widetilde{E}) \to \Gcal(p^* E'(1) \oplus p^* E'')$. It is unique up to multiplication by an element of $H^0(\Ocal_{S \times \PBbb^1}^{\times}) = H^0(\Ocal_S^{\times})$. We modify it with the unique invertible function on $S \times \PBbb^1$ so that the restriction to $\infty$ is the identity, and define $\psi_{\varepsilon}$ via the restriction of this isomorphism to $0$.

We still need to prove functoriality, both with respect to pullbacks and with respect to isomorphisms of exact sequences. Suppose first that we have an isomorphism of two exact sequences

$$
    \xymatrix{ \varepsilon \colon  0 \ar[r] & \ar[r] \ar[d]^{\varphi'} E' & E \ar[r] \ar[d]^{\varphi} & \ar[d]^{\varphi''} E'' \ar[r]  & 0   \\
\delta  \colon  0 \ar[r] & \ar[r] F' & F \ar[r]  &  F'' \ar[r] & 0 .} 
$$
They induce corresponding isomorphisms on the transgression exact sequences, which we denote by $\widetilde{\varphi}', \widetilde{\varphi}$ and $\widetilde{\varphi}''$. We have two isomorphisms $\Gcal(\widetilde{E}) \to \Gcal(p^* E'(1) \oplus p^* E'')$, namely one given by $\psi_{\widetilde{\varepsilon}}$, and another via the following isomorphisms 
\begin{displaymath}
    \Gcal(\widetilde{E}) \overset{[\widetilde{\varphi}]}{\longrightarrow} \Gcal(\widetilde{F}) \overset{\psi_{\widetilde{\delta}} }{\longrightarrow} \Gcal(p^* F'(1) \oplus p^* F'') \overset{[\widetilde{\varphi}' \oplus \widetilde{\varphi}'']}{\longleftarrow}  \Gcal(p^* E'(1) \oplus p^* E'').
\end{displaymath}
Since the latter is identified with the identity when restricting to $\infty$, it must coincide with $\psi_{\widetilde{\varepsilon}}$. It follows that their restrictions to $0$ are the same, which amounts to functoriality with respect to isomorphisms of exact sequences. 

For the functoriality with respect to base changes, one notes that the transgression constructions commute with base change, from which it follows that so does $\psi_{\varepsilon}$. Also, since the transgression exact sequence is canonically split when the original exact sequence is the standard split exact sequence $\varepsilon'$, one immediately sees that $\psi_{\varepsilon'}$  is the identity.



\end{proof}

\begin{definition} 
With the notation as in Proposition \ref{prop:split}, we say that $\psi_{\varepsilon}$ is a splitting isomorphism.
\end{definition}

\begin{remark}
Let $\Gcal$ be a line functor, and $V, W$ fixed vector bundles on $X\to S$. Then, we can define another line functor by $E\mapsto\Gcal(V\oplus E\oplus W)$. Applying Proposition \ref{prop:split} to the latter, we obtain splitting isomorphisms
\begin{displaymath}
    \psi_\varepsilon\colon \Gcal(V \oplus E \oplus W) \to \Gcal(V \oplus E' \oplus E'' \oplus W).
\end{displaymath}
 We will sometimes use this modified version of splitting isomorphisms. 
\end{remark}

The following lemma was used in the proof of Proposition \ref{prop:split}. We include it for lack of a proper reference in this generality. 
\begin{lemma}\label{lemma:linesonprojectivebundles}
Let $X$ be a variety, and $V$ a vector bundle of rank $r$ on $X$.  Consider the projectivization $p \colon \PBbb(V) \to X$ of $V$, with its tautological line bundle $\Ocal(1)$. Then any line bundle on $ \PBbb(V)$ is isomorphic to $p^{\ast} L \otimes \Ocal(k)$, for some (locally constant) integer $k$.
\end{lemma}
\begin{proof}
Denote by $K_0(X)$ the Grothendieck group of vector bundles of a variety $X$, defined as the free abelian group on vector bundles modulo exact sequences. Taking determinants of a vector bundle is multiplicative with respect to exact sequences, and determines a surjective homomorphism 
\begin{displaymath}
    \det \colon K_0(X) \to \Pic(X).
\end{displaymath}
This is well-known, but also described later in \textsection \ref{determinantfun}. By \cite[Expos\'e VI, Th\'eor\`eme 1.1]{SGA6} we know that $K_0(\PBbb(V))$ is freely generated as a $K_0(X)$-module by $\Ocal_{\PBbb(V)}, \Ocal(1), \ldots, \Ocal(r-1)$. The statement then follows from an application of \eqref{eq:isodetline} below, showing that for a virtual vector bundle $E$ of virtual  rank $e$,  $\det(p^\ast E \otimes \Ocal(k)) = p^\ast (\det E) \otimes \Ocal(k e)$.
\end{proof}

Recall that a filtration of vector bundles is admissible if all quotients that can be formed are locally free. 

\begin{proposition}\label{prop:splittingassociative}
Let $\Gcal$ be a line functor, and suppose that we are given an admissible filtration $E_2' \subseteq E_1' \subseteq E$ of vector bundles. Then, the natural diagram of splitting isomorphisms
\begin{displaymath}
\xymatrix{\Gcal(E) \ar[r] \ar[d] & \Gcal(E_1' \oplus E/E_1') \ar[d] \\  
\Gcal(E_2' \oplus E/E_2') \ar[r] & \Gcal(E_2' \oplus E_1'/E_2' \oplus E/E_1')
}
\end{displaymath}
commutes. 



\end{proposition}
\begin{proof}
Associated to the filtration there is a commutative diagram of vector bundles on $X \times \PBbb^1$
 
\begin{displaymath}
    \xymatrix{& 0 \ar[d] & 0 \ar[d] & 0 \ar[d] & \\ 
    0 \ar[r] & p^{\ast} E_2'(1) \ar[d] \ar[r] & \widetilde{E}_1' \ar[r] \ar[d] & p^{\ast} (E'_1/E_2') \ar[r] \ar[d] &  0 \\
    0 \ar[r] & p^{\ast} E_2'(1) \ar[d] \ar[r] & \widetilde{E} \ar[r] \ar[d] & p^{\ast} (E/E_2') \ar[r] \ar[d] &  0 \\
    0 \ar[r] & 0    \ar[d] \ar[r] & p^{\ast}({E/E_1'}) \ar[r] \ar[d] & p^{\ast} (E/E_1') \ar[r] \ar[d] &  0 \\
         & 0                  & 0                 & 0 
    }
\end{displaymath}
with exact rows and columns. The top two rows are transgression exact sequences for the sequences associated  to the inclusions $E_1' \to E$ and $E_2' \to E$. 

The diagram  of splitting isomorphisms
\begin{equation}\label{eq:diagramtransgression}
    \xymatrix{\Gcal(\widetilde{E}) \ar[r] \ar[d]  & \Gcal(p^\ast E_2'(1) \oplus p^\ast (E/E_2')) \ar[d] \\ \Gcal(\widetilde{E}_1'\oplus p^\ast (E/E_1')) \ar[r] & \Gcal(p^\ast E_2'(1) \oplus p^\ast(E_1'/E_2') \oplus p^\ast (E/E_1'))
        } 
\end{equation}
is identified with the diagram in the proposition at $0$, and is identified with the triple split case at $\infty$. The maps in the latter case are all identities, which is proven as in Proposition \ref{prop:split}. It follows that the diagram \eqref{eq:diagramtransgression} commutes, and hence so does also the diagram restricted to 0.

\end{proof}

\subsection{Flags and a second splitting principle}

A (complete) flag of a vector bundle $V$ on a variety $X$ is an increasing filtration $F^\bullet$ of $V$, whose graded quotients are line bundles. The variety of flags $D_V$ is a repeated tower of projective bundles over $X$, and admits a universal flag. 

A general theme of this section concerns reducing constructions involving line functors to the case of line bundles, via flag varieties as in the classical splitting principles. Since the flag varieties live above $X$ and our constructions are relative to $S$, the naive approaches cannot directly be used. In the context of Chow categories \cite[1.13.2]{FrankeChern} similar properties are established, under additional hypotheses. Since we are working with line bundles, automorphisms do not change while passing to flag varieties,  and part of the descent condition can be removed. 

\begin{definition}\label{def:flagiso}
Let $\Gcal$ and $\Gcal'$ be line functors of $k$-variables. Suppose given a $k$-tuple of vector bundles $\underline{E}=(E_1, \ldots, E_k)$ on $X\to S$. A flag isomorphism $\varphi \colon \Gcal(\underline{E}) \dashrightarrow \Gcal'(\underline{E})$ is the following collection of data:
\begin{enumerate}
    \item For each $g \colon S' \to S$, with induced morphisms $f' \colon X\times_S S' \to S', g' \colon X\times_S S' \to X$ and $k$-tuple of flags $\underline{F}^\bullet = (F_1^\bullet, \ldots, F_k^\bullet)$, where $F_i^\bullet$ is a flag on ${g'}^\ast E_i$, there is an isomorphism
        \begin{equation}
            \varphi(\underline{F}^\bullet) \colon \Gcal_{f'}({g'}^\ast \underline{E}) \to \Gcal'_{f'}({g'}^\ast \underline{E}).
        \end{equation}
    \item\label{item:flagiso-2} If $i\colon X\to X'$ is an isomorphism of families in the class $\mathcal{P}$, then under the natural equivalence of functors $\Gcal_{f \circ i} \circ i^\ast \to \Gcal_f$, $\varphi(\underline{F}^\bullet)$ is identified with $\varphi(i^\ast \underline{F}^\bullet)$.
    \item\label{item:flagiso-3} The formation of the isomorphism commutes with base change, i.e. if $g\colon S' \to S$ is a morphism, with induced morphism $g' \colon X\times_S S' \to X$ then, under the natural equivalence of functors $g^\ast \Gcal_f \to \Gcal_{f'} \circ {g'}^\ast$, $g^{\ast}\varphi(\underline{F}^\bullet)$ is identified with $\varphi({g'}^\ast \underline{F}^\bullet)$. 
\end{enumerate}
\end{definition}
The identifications \eqref{item:flagiso-2} and \eqref{item:flagiso-3} in the definition are summarized with the notation
\begin{displaymath}
    \varphi(\underline{F}^\bullet) = \varphi(i^\ast \underline{F}^\bullet)
\end{displaymath}
and
\begin{equation} \label{eq:flagisorelationship}
       g^\ast \varphi(\underline{F}^\bullet)= \varphi({g'}^\ast \underline{F}^\bullet).
   \end{equation}

The first remark about these isomorphisms is that they do not, in fact, depend on the choice of the flags:

\begin{lemma}\label{lemma:independenceoffiltration}
Let $\varphi$ be a flag isomorphism. Suppose that $\underline{F}^\bullet, {\underline{F}'}^\bullet$ are $k$-tuples of flags on $X\to S$. Then $\varphi(\underline{F}^\bullet) = \varphi({\underline{F}'}^\bullet)$. In particular, flag isomorphisms are independent of the choices of flags. 
\end{lemma} 
   
\begin{proof}
   
   We will consider the variety of flags on $E_1, \ldots, E_k$,  $D = D_{E_1} \times_X \ldots \times_X D_{E_k}$, and denote by $q\colon D \to S$ the composition of the projection $p\colon D \to X$ and $f\colon X\to S.$  We have the diagram
   \begin{displaymath} 
        \xymatrix{D \times_S X \ar@<0.5ex>[r]^-{i_1} \ar@<-0.5ex>[r]_-{i_2} \ar[dr]_{f_D} & D \times_X (X \times_S X) \ar[r]^-{p'} \ar@<-0.5ex>[d]_{q_1}\ar@<0.5ex>[d]^{q_2} & X \times_S X \ar@<-0.5ex>[d]_{f_1}\ar@<0.5ex>[d]^{f_2} \\
        &D \ar[r]^p  \ar[dr]_q  & X \ar[d]^f  \\
        & & S}
    \end{displaymath}
   where $f_1$ (resp. $f_2$) denote the projection on the first (resp. second) factor. The arrow $q_1$ (resp. $q_2$) is obtained through base change along $p$ and $f_1$ (resp. $f_2$). The leftmost upper arrows are the two natural isomorphisms. The diagram obtained by removing all the arrows with indices $1$ (resp. $2$) commutes.
   
   To simplify the notation, set $\varphi_S = \varphi(\underline{F}^\bullet)$. Since $f$ is faithfully flat, $\varphi_S$ is uniquely determined by $f^\ast \varphi_S$, which is identified with $\varphi(f_1^\ast \underline{F}^\bullet)$ under the base change isomorphism $f^\ast \Gcal_f \to \Gcal_{f_1} \circ f_1^\ast $. We show that this is independent of the filtration.
   
   The vector bundles $p^{\ast} E_{1},\ldots, p^{\ast}E_{k}$ admit universal flags $\Fcal_{1}^{\bullet},\ldots,\Fcal_{k}^{\bullet}$, and we denote the corresponding tuple of flags by $\underline{\mathcal{F}}^\bullet$. By pullback, all of the  $f_D^* p^* E_1, \ldots, f_D^* p^* E_k$ admit complete flags as well. There is thus a corresponding isomorphism
   \begin{equation}\label{eq:isovarhpiprime}
       \varphi_D = \varphi(f_D^\ast \underline{\mathcal{F}}^\bullet) \colon \Gcal_{f_D}(f_D^* p^* \underline{E}) \to \Gcal_{f_D}'(f_D^* p^* \underline{E}).
   \end{equation}
   The initial flags on $E_1, \ldots, E_k$ correspond to a section $i\colon X\to D$ of $p\colon D\to X$. Denote by ${i^\prime}$ the corresponding section of $p'$. Then the restrictions of the flags $f_D^\ast \underline{\mathcal{F}}^\bullet$ along ${i^\prime} i_1^{-1}$ equal $f_1^{\ast} \underline{F}^\bullet$ and hence, in the notation of \eqref{eq:flagisorelationship}, one has
   \begin{equation}\label{eq:independenceofflag1} i^\ast \varphi_D = f^\ast \varphi_S.
   \end{equation}
   Now, by functoriality there is a natural transformation of functors 
   \begin{equation} \label{eq:naturaltransflinefunctor}
        p^{\ast} \Gcal_{f_1}\circ f_1^\ast \to  \Gcal_{q_1} \circ {p'}^\ast f_1^\ast \to \Gcal_{f_D} \circ i_1^\ast {p'}^\ast f_1^\ast \to \Gcal_{f_D} \circ f_D^\ast p^\ast 
   \end{equation}
   and likewise for $\Gcal'$. Then, via \eqref{eq:naturaltransflinefunctor} the isomorphism $\varphi_D$ in \eqref{eq:isovarhpiprime} is identified with an isomorphism
   \begin{equation}\label{eq:isovarphiDsigma} p^{\ast} \Gcal_{f_1}(f_1^\ast \underline{E}) \simeq p^{\ast} \Gcal'_{f_1}(f_1^\ast \underline{E}). 
   \end{equation} 
   We claim $\varphi_D$, identified as \eqref{eq:isovarphiDsigma}, is the pullback along $p$ of an isomorphism 
   \begin{equation}\label{eq:defvarphiX} \varphi_X: \Gcal_{f_1}(f_1^\ast \underline{E}) \to  \Gcal_{f_1}'(f_1^\ast \underline{E}).
   \end{equation} Since $p$ is flat and $\Gcal, \Gcal'$ are line bundles, the natural morphism
   \begin{displaymath} p^* \SheafHom_{\Ocal_X} \left(\Gcal_{f_1}(f_1^\ast \underline{E}), \Gcal_{f_1}'(f_1^\ast \underline{E})\right)  \to \SheafHom_{\Ocal_D}\left(p^{\ast} \mathcal {G}_{f_1}(f_1^\ast \underline{E}), p^{\ast} \Gcal_{f_1}'(f_1^\ast \underline{E} )\right)
   \end{displaymath}
   is an isomorphism, cf. \cite[\href{https://stacks.math.columbia.edu/tag/0C6I}{0C6I}]{stacks-project}. The rightmost coherent sheaf is trivialized by $\varphi_D$ via \eqref{eq:isovarphiDsigma}. Taking direct images along $p$ and using that $p_\ast \Ocal_D = \Ocal_X$, we find that $ \varphi_D$ identifies uniquely with  an isomorphism $\varphi_X$ on $X$ via $p^\ast$ as claimed, which we summarize as
   \begin{equation}\label{eq:independenceofflag2}
       \varphi_D = p^\ast \varphi_X.
   \end{equation}
   By \eqref{eq:independenceofflag1} and \eqref{eq:independenceofflag2} we infer that $f^\ast \varphi_S = i^\ast \varphi_D = i^\ast p^\ast \varphi_X = \varphi_X$. Given yet another flag, giving rise to an isomorphism $\varphi'_S$, we find that $f^\ast \varphi_S = \varphi_X = f^\ast \varphi'_S$ and hence $\varphi_S = \varphi'_S$.  
    \end{proof}

Thanks to the lemma, the following definition is well-posed.
\begin{definition}
Let $\varphi\colon\Gcal(\underline{E})\dashrightarrow\Gcal^{\prime}(\underline{E})$ and $\varphi^{\prime}\colon\Gcal(\underline{E}^{\prime})\dashrightarrow\Gcal^{\prime}(\underline{E}^{\prime})$ be two flag isomorphisms. A morphism of flag isomorphisms $\varphi\to\varphi^{\prime}$ is an isomorphism of $k$-tuples $\underline{E}\simeq\underline{E}^{\prime}$ such that, possibly after a base change $S'\to S$, for any tuple of flags $\underline{F}^{\bullet}$ of $\underline{E}$ (resp. ${\underline{F}'}^{\bullet}$ of $\underline{E}'$), the isomorphism $\varphi(\underline{F}^\bullet)$ is identified with $\varphi^{\prime}({\underline{F}'}^\bullet)$. 
\end{definition}

The second splitting principle takes the following form:

\begin{theorem}\label{thm:secondsplitting}
 Let $\Gcal$ and $\Gcal^{\prime}$ be line functors of $k$-variables.
\begin{enumerate}
    \item A flag isomorphism $\varphi \colon  \Gcal(\underline{E}) \dashrightarrow \Gcal'(\underline{E})$ uniquely determines an isomorphism of line bundles
 \begin{displaymath} \varphi \colon  \Gcal(\underline{E}) \to \Gcal'(\underline{E}),
 \end{displaymath}
 coinciding with the flag isomorphism when $\underline{E}$ admits a flag, possibly after a base change $S' \to S$.
 
 \item A morphism of flag isomorphisms $\varphi \to \varphi'$ induces isomorphisms 
 \begin{displaymath}
     \xymatrix{  \Gcal(\underline{E}) \ar[r]^{\varphi} \ar[d] & \Gcal'(\underline{E})\ar[d]  \\
     \Gcal(\underline{E}') \ar[r]^{\varphi'} & \Gcal'(\underline{E}').}
 \end{displaymath}
 In particular, $\varphi \colon  \Gcal(\underline{E}) \to \Gcal'(\underline{E})$ is compatible with isomorphisms of vector bundles.
 \end{enumerate}
\end{theorem}
\begin{remark}
After the theorem, it is justified to denote flag isomorphisms simply by $\varphi \colon  \Gcal(\underline{E}) \to \Gcal'(\underline{E})$, instead of $\varphi \colon  \Gcal(\underline{E}) \dashrightarrow \Gcal'(\underline{E})$.
\end{remark}
\begin{proof}[Proof of Theorem \ref{thm:secondsplitting}]
    We start by noticing that the isomorphisms $\varphi_D$ and $\varphi_X$ exhibited in the proof of Lemma \ref{lemma:independenceoffiltration} exist even if we do not start with a filtration on our vector bundles. We want to verify that the isomorphism $\varphi_X$ from \eqref{eq:defvarphiX}, $\Gcal_{f_1}(f_1^\ast \underline{E})\to \Gcal_{f_1}'(f_1^\ast \underline{E})$, which we identify with an isomorphism $f^\ast \Gcal_{f}(\underline{E})\to f^\ast \Gcal_{f}'(\underline{E})$, also denoted by $\varphi_X$, is the pullback of an isomorphism $\varphi$ on $S$ of the same objects. 
    
    Any morphism of descent datum along a \emph{fpqc} morphism is effective \cite[Th\'eor\`eme 1.1, Expos\'e VIII]{SGA1}, and we must hence verify that the diagram 
    \begin{displaymath}
        \xymatrix{f_1^\ast f^\ast \Gcal_f(\underline{E}) \ar[r] \ar[d]^{f_1^\ast \varphi_X}  & ( f \times f)^\ast \Gcal_f(\underline{E}) \ar[r] & f_2^\ast f^\ast \Gcal_f(\underline{E}) \ar[d]^{f_2^\ast \varphi_X} \\ 
        f_1^\ast f^\ast \Gcal_f'(\underline{E}) \ar[r] & ( f \times f)^\ast \Gcal_f'(\underline{E}) \ar[r] & f_2^\ast f^\ast \Gcal_f'(\underline{E})} 
    \end{displaymath} commutes, where the upper and lower horizontal morphisms are the canonical descent datum for line bundles of the form $f^\ast L$ on $X$.  Since ${p'}^\ast$ is fully faithful, it is enough to verify the analogous statement by pulling back the diagram along $p'$. Set $\widetilde{q} = q q_1 = q q_2$ and denote by $\widetilde{f}_D$ the base change of $f$ along $\widetilde{q}$, with projection $\widetilde{p}= \widetilde{f}_D q_1$ onto $D$. To rewrite the descent condition, we consider the diagram 
    \begin{equation}\label{diagram-descent}
        \xymatrix{{p'}^\ast f_1^\ast f^\ast \Gcal_f(\underline{E}) \ar[d] \ar[rr] & & {p'}^\ast (f \times f)^\ast \Gcal_f(\underline{E}) \ar[d] \\
        {q_1}^\ast p^\ast f^\ast \Gcal_f(\underline{E}) \ar[r] \ar[d] & {q_1}^\ast q^\ast  \Gcal_f(\underline{E}) \ar[r] \ar[d] & \widetilde{q}^\ast \Gcal_f(\underline{E}) \ar[d] \\
         {q_1}^\ast p^\ast  \Gcal_{f_1}(f_1^\ast \underline{E}) \ar[r]  & {q_1}^\ast   \Gcal_{f_D}(f_D^\ast p^\ast \underline{E}) \ar[r]  & \Gcal_{\widetilde{f}_D}( \widetilde{p}^\ast p^\ast \underline{E}). }
    \end{equation}
    The upper square consists of natural transformations of pullbacks, and hence commutes. The other squares commute due to functoriality of the line functors and the natural identifications, as in \eqref{eq:naturaltransflinefunctor}, of the various families. Taking into  account the analogous diagram for $\Gcal'_f(\underline{E})$, one finds that the isomorphism ${p'}^\ast f_1^\ast \varphi_X \colon {p'}^\ast f_1^\ast f^\ast \Gcal_f(\underline{E}) \to {p'}^\ast f_1^\ast f^\ast \Gcal_f'(\underline{E})$ is identified with the isomorphism $q_1^\ast \varphi_D \colon q_1^\ast \Gcal_{f_D}(f_D^\ast p^\ast \underline{E})\to q_1^\ast \Gcal_{f_D}'(f_D^\ast p^\ast \underline{E})$, under the natural isomorphism ${p'}^\ast f_1^\ast f^\ast \Gcal_f(\underline{E}) \to q_1^\ast \Gcal_{f_D}(f_D^\ast p^\ast \underline{E})$ in \eqref{diagram-descent}. Also, performing the same construction for the second projection, a diagram chase reduces the descent statement to the commutativity of the diagram
        \begin{displaymath}
            \xymatrix{
            q_1^{\ast} \Gcal_{f_D} ( f_D^\ast p^\ast \underline{E}) \ar[r] \ar[d]^{q_1^\ast \varphi_D}  & \Gcal_{\widetilde{f}_D} (\widetilde{p}^\ast p^\ast \underline{E}) \ar[r] \ar[d]^{\widetilde{\varphi}_{D}} & q_2^{\ast} \Gcal_{f_D} ( f_D^\ast p^\ast \underline{E} ) \ar[d]^{q_2^\ast \varphi_D} \\            q_1^{\ast} \Gcal_{f_D}'( f_D^\ast p^\ast \underline{E}) \ar[r] & \Gcal_{\widetilde{f}_D}'( \widetilde{p}^\ast p^\ast \underline{E}) \ar[r] & q_2^{\ast} \Gcal_{f_D}'( f_D^\ast   p^\ast \underline{E}), }
        \end{displaymath}
     where $\widetilde{\varphi}_D$ is the flag isomorphism associated to the flag $\widetilde{p}^{\ast}\underline{\mathcal{F}}^{\bullet}$. But this diagram commutes as $q_1^\ast f_D^\ast \underline{\mathcal{F}}^{\bullet}  = \widetilde{p}^\ast \underline{\mathcal{F}}^{\bullet} $ and the fact that flag isomorphisms commute with base change.  
     
     To verify that the descended isomorphism coincides with the flag isomorphism in the presence of a flag $\underline{F}^\bullet$, it suffices to base change to $D$ and apply  Lemma \ref{lemma:independenceoffiltration}.

     The last property concerning isomorphisms of flag isomorphisms is formal to verify, using the unicity of Lemma \ref{lemma:independenceoffiltration}, since all the spaces commute with base change in an obvious way.  
     
     
     \end{proof}
     
     \begin{remark}
     The two splitting principles in Proposition \ref{prop:split} and  Theorem \ref{thm:secondsplitting} can be combined to the effect that it can be assumed that the flags split the vector bundles, so they are of the form $E_j = L_1^{(j)} \oplus \ldots \oplus L_{r_j}^{(j)}$. The flags $F_j^\bullet$ are then the flags $L_1^{(j)} \subseteq L_1^{(j)} \oplus L_2^{(j)} \subseteq \ldots \subseteq E_j.$
     \end{remark}
\subsection{Applications to multiplicativity datum}\label{subsec:appliadddatum}

In this subsection we will develop, for the purposes of this article, line functors equipped with a multiplicativity datum. We develop a minimal amount of formalism to be able to treat cases where the product is not the tensor product, with the perspective of treating Whitney-type isomorphisms in Section \ref{sec:intersectionbund}.

Recall that a Picard category $P$ is a monoidal groupoid, where all the objects are invertible. The facts that we use can be found in \cite[Appendix A]{Knudsen}, \cite{MacLane} and \cite[Section 4]{Deligne-determinant}. More precisely, all morphisms are isomorphisms, and there is a functor $\utimes \colon P \times P \to P$, such that for any object $Y$, both functors given by $X \mapsto Y \utimes X$ and $X \mapsto Y \utimes X$ are equivalences of categories. One assumes that the product $\utimes$ admits an associativity isomorphism 
\begin{displaymath}
    (X \utimes Y) \utimes Z \simeq X \utimes (Y \utimes Z)
\end{displaymath}
satisfying the pentagonal axiom explained in \cite[p. 33]{MacLane}. Any diagram  involving only the associativity isomorphisms commutes by the coherence theorem proven in \cite{MacLane}, allowing one to consider finite ordered products $\bugtimes_{i \in I} X_i$. 

The Picard category is moreover said to be commutative, if there is an isomorphism \linebreak $c_{X,Y} : X \utimes Y \to Y \utimes X$, compatible with associativity, usually referred to as the hexagonal axiom for symmetric monoidal categories, explained in \cite[p. 38]{MacLane}, and such that  $c_{X,Y} c_{Y,X} = \id$. Any diagrams involving only the associativity isomorphism and the commutativity isomorphisms commute, by another coherence theorem also proven in \cite{MacLane}. In particular, it makes sense to consider finite products, not necessarily ordered, $\bugtimes_{i \in I} X_i$. 

Denote by $Y = X^{-1}$ a left inverse object to $X$. By applying the commutativity isomorphism it can be equipped with the structure of a right inverse. Depending on if one contracts on the right or on the left, there are hence two ways to construct an isomorphism $X \utimes Y \utimes X = X$. The two choices differ by an automorphism determined by $c_{X,X}$. While this is not always the identity in our contexts, we will not have to deal with this particular issue.  

\begin{definition}\label{def:Picardcategory}
For a class of objects $\mathcal{P}$ as in \textsection \ref{subsec:linefunctors}, a family of line categories with products, $\Lcal$, consists in the following data : 
\begin{enumerate}
    \item For any family $X\to S$ in the class, there is a Picard category $(\Lcal(X/S), \utimes)$, whose objects and morphisms are those of the products of two Picard categories $\Ccal(X/S) \times (\hbox{Line bundles}_{/S}, \text{\emph{iso}})$. 
    \item We say it is commutative if all the Picard categories involved are commutative. 
    \item The formation is functorial with respect to isomorphisms $X\to X^{\prime}$ and base changes $S' \to S$, in a sense directly analogous to that of Definition \ref{def:linefunctorkvar}, replacing the natural transformations with functors. These are supposed to be (symmetric) monoidal, i.e. preserve the (commutative) Picard category structure.
    \item The category $\Ccal(X/S)$ is referred to as the product category. We assume that the product structure in $\Ccal(X/S) $ is the projection of that of $\utimes$.
    \item  We denote by $\uOs$ the unit object of $\Lcal(X/S)$, determined uniquely up to unique isomorphism. We assume the line bundle part of $\uOs$ is isomorphic to $\Ocal_S$.

\end{enumerate}

\end{definition}
In this article, we only consider two types of families of line categories with products, namely: \emph{i)} graded line bundles with the usual tensor product of line bundles and the Koszul rule of signs applied to the commutativity, and \emph{ii)} Deligne's category in \textsection \ref{subsubsec:Delignecategory} below, where the product captures the Whitney isomorphism for the line functor $IC_2$ for families of curves. It allows us to treat products which are multiplicative in nature in a multiplicative fashion. 

\begin{definition}\label{def:determinantfunctordefinition}
Let $\Gcal$ be a line functor. 
\begin{enumerate}
    \item We will say that $\Gcal$ is multiplicative (on exact sequences) if it is equipped with a multiplicativity, or multiplicative, datum, meaning that it fulfils: 
    \\
  \begin{enumerate}
    \item (Product data) There is a functor $\uG$, with values in a family of line categories with products $\Lcal$, extending that of $\Gcal$. 
    \item (Multiplicativity) For any exact sequence $\varepsilon$ as in \eqref{eq:varepsilonsequence},  there is an isomorphism 
\begin{equation}\label{eq:splittingisoF} 
    [\varepsilon]\colon \uG(E) \to \uG(E') \utimes \uG(E''),
\end{equation} 
compatible with isomorphisms of exact sequences and base changes. 
    \item (Zeros) There is a (fixed) isomorphism with a unit a object $\uOs$ of $P(X \to S)$, 
    \begin{displaymath}
    \upsilon\colon \uG(0) \to \uOs,
    \end{displaymath} compatible with base changes. 
    \item (Normalization) In the case $E''$ is of rank 0, so that $\psi: E' \to E$ is an isomorphism, we require the composition
\begin{equation}\label{eq:compatibilityadditive}
    \uG(E) \overset{[\varepsilon]}{\longrightarrow} \uG(E') \utimes \uG(0) \overset{\id \utimes \upsilon}{\longrightarrow} \uG(E')
\end{equation}
to be $[\psi]^{-1}$. We require the analogous composition to be $[\psi]$ for the case when $E'$ is of rank 0. 
    \item (Associativity) \label{eq:associativity}
  Given an admissible filtration $E_2' \subseteq E_1' \subseteq E$, the following diagram of natural isomorphisms commutes:
\begin{equation} \label{eq:associativityG}\xymatrix{ \uG(E) \ar[r] \ar[d] & \uG(E'_1) \utimes \uG(E/E'_1) \ar[d] \\
            \uG(E_2') \utimes \uG(E/E_2') \ar[r] & \uG(E_2') \utimes \uG(E'_1/E'_2) \utimes \uG(E/E_1'). }
\end{equation} 
\end{enumerate}
\item A multiplicativity datum as above is said to be commutative if the diagram
 \begin{equation}\label{eq:commutativityadditivedatum}
        \xymatrix{\uG(E' \oplus E'') \ar[r]^{[c]} \ar[d] & \uG(E'' \oplus E') \ar[d] \\ 
        \uG(E') \utimes \uG(E'')  \ar[r]^c & \uG(E'') \utimes \uG(E') }
    \end{equation}
commutes. Here $c$ denotes the commutativity isomorphism interchanging the two factors of line bundles or vector bundles. 
    \item We say that $\Gcal$ is equipped with a product structure if $\Gcal$ extends to a functor $\uG$ as in  (1a), and if the projection of the functor $\uG$ onto the product category $\Ccal$ satisfies the analogues of (1b)--(1e) above.
    \item A morphism of multiplicative line functors $\Gcal$ and $\Gcal^{\prime}$ is a natural transformation of functors $\uG \to \uG^{\prime}$, which is compatible with the multiplicativity $[\varepsilon]$ of the multiplicativity datum. A morphism of product data is defined analogously, but for the projection of $\uG$ and $\uG'$ onto $\Ccal.$ 
\end{enumerate}
\end{definition}  


To emphasize the dependency on the extension, we will sometimes refer to $\uG$ as a multiplicative line functor. 

 \begin{remark} \label{remark:adddatum} 
 \begin{enumerate}
    \item A multiplicativity datum for $\Gcal$ is a product structure together with isomorphisms of line bundles
    \begin{displaymath}
        [\varepsilon] : \Gcal(E) \to \Gcal(E') \utimes \Gcal(E'')
    \end{displaymath}
    satisfying of (1)(a)--(1)(e), where $\utimes$ refers to the product in $\Lcal$ projected to the line bundle part.
    \item Given product data, which we want to equip with a multiplicativity datum, the normalization property can be rephrased as a definition for the multiplicativity datum when either $E'$ or $E''$ are the trivial vector bundles. 
    \item The isomorphism $\upsilon \colon \uG(0) \to \uOs$ is likewise equivalent with the multiplicativity datum \eqref{eq:splittingisoF} and the natural isomorphism $\uOs  \to \uOs \utimes \uOs$ when all the vector bundles are of rank 0. 
    \item The property of being commutative is not a consequence of the properties (1)(a)--(1)(e). 
    \item For a fixed family, the above definition is, in the language of \cite[\textsection 4.3]{Deligne-determinant}, a determinant functor, from the exact category of vector bundles to the Picard category $\Lcal(X/S)$.

 \end{enumerate}
 \end{remark}



For the following statement, recall the splitting isomorphisms of Proposition \ref{prop:split}. 
\begin{proposition}\label{prop:splitassociativity}
Suppose that there is a multiplicativity datum for a line functor $\Gcal$. Then, for any exact sequence $\varepsilon$ as in \eqref{eq:varepsilonsequence}, the isomorphism $[\varepsilon]$ factors via an isomorphism $\psi_\varepsilon$, which is the splitting isomorphism when projecting to the line bundle category, and the split multiplicativity datum $[\varepsilon']$:
\begin{displaymath}
    \xymatrix{\uG(E) \ar[r]^{\psi_{\varepsilon}} \ar[dr]_{[\varepsilon]} &  \ar[d]^{[\varepsilon']}  \uG(E' \oplus E'')  \\ 
    & \uG(E') \utimes \uG(E'').} 
\end{displaymath}
Conversely, suppose  
\begin{enumerate}
    \item $\Gcal$ is equipped with product data;
    \item $\Gcal$ is equipped with a multiplicativity datum for split exact sequences $[\varepsilon']$;
\end{enumerate}
Then, there is a uniquely defined multiplicativity datum in the general case. 
\end{proposition}

\begin{proof} 
 The composition $[\varepsilon']^{-1} [\varepsilon]$ commutes with base change and reduces to the identity for split diagrams, and the projection onto the line bundle category must hence be the splitting isomorphism by Proposition \ref{prop:split}. When the product data is non-trivial, this is taken to be the definition of $\psi_\varepsilon$. 
The second point follows from definition, taking into account Remark \ref{remark:adddatum}, and the corresponding properties of the splitting isomorphisms, in particular the unicity.  
\end{proof}
In preparation for the formulation of the main theorem of this subsection, we anticipate that for an exact sequence $\varepsilon$ as in \eqref{eq:varepsilonsequence}, we can reduce the construction of a multiplicativity isomorphism $[\varepsilon]$ to the case when there are given flags ${F'}^\bullet$ and ${F''}^\bullet$ on $E'$ and $E''$, with line bundle quotients. These filtrations canonically induce a flag  $F^{\bullet}$ of $E$ with the same graded quotients. We can then repeatedly apply the isomorphism of \eqref{eq:splittingisoF} to the line bundles $(F')^i/(F')^{i+1}$ in $E'/(F')^{i+1}$ and similarly to $F^{\bullet}$ and ${F''}^\bullet$. One obtains an isomorphism
   \begin{equation}\label{def:addisolinebundles}
   \begin{split}
        \varphi({F'}^\bullet, {F''}^\bullet) \colon \uG(E) \to \bugtimes_i &\uG(\Gr_F^i E)\\
        &= \bugtimes_{j} \uG(\Gr_{F'}^j E')\utimes\bugtimes_k \uG(\Gr_{F''}^k E'') \to \uG(E') \utimes \uG(E'').
    \end{split}
   \end{equation}
By an application of \eqref{eq:associativityG}, one sees that \eqref{def:addisolinebundles} coincides with $[\varepsilon]$ for a line functor equipped with a multiplicativity datum. Conversely, we have the following proposition : 
\begin{proposition}\label{prop:associativityfromlines}
Let $\Gcal$ be a line functor. Suppose that:
\begin{enumerate}
    \item $\Gcal$ has a product structure;
    \item we are given a multiplicativity datum for exact sequences $\varepsilon$ as in \eqref{eq:varepsilonsequence}, whenever $E'$ is a line bundle.
\end{enumerate}
Then, the isomorphism \eqref{def:addisolinebundles} automatically satisfies the associativity in \eqref{eq:associativityG}. Moreover, this diagram is compatible with the splitting diagram of Proposition \ref{prop:splittingassociative}, via the isomorphism in Proposition \ref{prop:splitassociativity}.

The conclusion is equally valid if we instead suppose that $E''$ is a line bundle. 
\end{proposition}
\begin{proof}
     We show the case when we assume a multiplicativity datum in the case of $E'$ being a line bundle. By assumption we only need to show that the line bundle part satisfies the properties in Definition \ref{def:determinantfunctordefinition} (1)(a)--(1)(e).
     By construction, \eqref{def:addisolinebundles} is a flag isomorphism and in particular the involved isomorphisms are independent of the flags, by Proposition \ref{lemma:independenceoffiltration}. In the argument to follow, this allows us to choose our filtrations in a compatible way. 
     
     Let $D = D_{E_2'} \times_X D_{E_1'/E_2'} \times_X D_{E/E'_1}$ be the variety of flags of $E_2', E_1'/E_2', E/E_1'$  and consider the morphism $q \colon D_{E_2'} \times_X D_{E_1'/E_2'} \times_X D_{E/E'_1} \to X \to S$. By faithfullness of $q^\ast$, it is enough to verify that the diagram \eqref{eq:associativityG} commutes after base changing to $D$. We can thus suppose $E$ admits the universal filtration $F^\bullet$, which is compatible with filtrations on $E_2', E_1', E_1'/E_2', E/E_1', E/E_2'$ and whose graded quotients identify with each other. Then, \eqref{eq:associativityG} takes the form
\begin{displaymath}
    \xymatrix{ & \uG(E) \ar[d] \ar[dl] \ar[dr] & \\ 
    \uG(E'_2) \utimes \uG(E/E_2') \ar[r]\ar[dr] & \bugtimes_{i} \uG(\Gr_F^i E)  &  \uG(E_1') \utimes \uG(E/E_1') \ar[l] \ar[dl] . \\ 
    & \uG(E_2') \utimes \uG(E'_1/E_2') \utimes \uG(E/E_1'). \ar[u] }
\end{displaymath}
To prove that all the triangles commute one relies on the formal identity $(\id \utimes \varphi_2)  \utimes (\varphi_1 \utimes \id) = \varphi_1 \utimes \varphi_2 = (\varphi_1 \utimes \id) \utimes (\id \otimes \varphi_2) $, so that certain isomorphisms can be performed in an arbitrary order. It follows that the big diagram also commutes. 

The compatibility with the splitting isomorphism is now a formal consequence, which we leave to the interested reader. 
\end{proof}

One of the main applications of our splitting principles is the following: 

\begin{theorem}\label{thm:flagfiltration}
Let $\Gcal$ be a line functor. Suppose:
\begin{enumerate}
    \item $\Gcal$ has a product structure;
    \item there is a functorial isomorphism  $\upsilon: \uG (0)\to \uOs$; 
    \item for every exact sequence $\varepsilon$ as in \eqref{eq:varepsilonsequence}, with $E'$ a line bundle, there is a functorial isomorphism $[\varepsilon]$, which is multiplicative as in Definition \ref{def:determinantfunctordefinition}.
    \end{enumerate}
Then, $\Gcal$ extends uniquely to a multiplicative line functor, compatible with the above data. The last property can be replaced by assuming there is a multiplicativity datum when $E = E' \oplus E''$ and $E'$ is a line bundle.

The same conclusion holds if we throughout change the assumption that $E'$ is a line bundle, with $E''$ being one. 
\end{theorem}


\begin{proof} The last remark in the theorem is analogously proven, and we just focus on $E'$ being a line bundle. 

As discussed in Remark \ref{remark:adddatum}, the first assumption of the theorem fixes uniquely the multiplicativity datum whenever either $E'$ or $E''$ are trivial vector bundles. 
    Any multiplicative line functor is uniquely determined as in \eqref{def:addisolinebundles}, and we use it as a definition for $[\varepsilon]$. Because $\Gcal$ has a product structure, we only need to prove statements for the line bundle part. By construction, on the line bundle part this is a flag isomorphism as in Definition \ref{def:flagiso} and hence, by Theorem \ref{thm:secondsplitting}, defines a multiplicativity datum in general. From Proposition \ref{prop:associativityfromlines}, we see that it is automatically associative. That we can reduce to the case of split sequences follows from Proposition \ref{prop:splitassociativity}. 
     \end{proof}
The following is a useful complement to the above theorem. 
\begin{proposition}\label{prop:useful-version-splitting}
Suppose that:
\begin{enumerate}
    \item $\uG$ and $\uG^{\prime}$ are two multiplicative line functors into the same line categories with products;
    \item the underlying product structures are equivalent;
    \item for any line bundle $L$ on $X$, there is an isomorphism $\uG(L) \to \uG^{\prime}(L)$, compatible with isomorphisms of line bundles and commuting with base changes $S' \to S$;
\end{enumerate} Then, there is a unique extension to an isomorphism of  multiplicative line functors  $\uG \to \uG^{\prime}$ for general vector bundles, compatible with the above data.
\end{proposition}
\begin{proof}
The proof is analogous to Theorem \ref{thm:flagfiltration}. In the same vein, one can define a flag isomorphism, decomposing $\uG(E)$ and $\uG'(E)$ according to \eqref{def:addisolinebundles} and applying the second assumption. The details are left to the reader.
\end{proof}

Below, we record a criteria for when a multiplicative line functor is commutative.

\begin{corollary}\label{cor:criteriacommutative}
Suppose that $\Gcal$ is a multiplicative line functor whose product structure is commutative in general, and such that \eqref{eq:commutativityadditivedatum} commutes whenever both $E'$ and $E''$ are line bundles. Then $\uG$ is commutative. 

In this case, the construction of $\Gcal$ extends to the bounded derived category of vector bundles on $X$ with the exact sequences replaced by exact triples, and isomorphisms replaced with quasi-isomorphisms of complexes.
\end{corollary}

\begin{proof}
The diagram \eqref{eq:commutativityadditivedatum} when either $E'$ or $E''$ are of rank 0 automatically commutes. This follows from the the description of \eqref{eq:compatibilityadditive} and the fact that commutativity is compatible with units. 

For the general case, notice that by assumption we are reduced to the statement for the line bundle part of the functor, and that the commutativity datum is compatible with base change by assumption. We pass to the variety of flags of $E'$ and $E''$, $D$ via the projection $q: D \to S$. By faithfulness of $q^\ast$, it is enough to verify that \eqref{eq:commutativityadditivedatum} commutes over $D$. There the splitting isomorphism and an iterated version of the associativity in Definition \ref{def:determinantfunctordefinition} allows us to reduce to the case when $E'$ and $E''$ are sums of line bundles. 

Write the line bundle decomposition as $E' = \bigoplus_{j=1}^{e'} L_j'$ and $E'' = \bigoplus_{j=1}^{e''} L_j''$.  Here, the direct sums are ordered by the natural order of the integers. Denote by $c_i$ the transposition isomorphism exchanging the $i$-th and $(i+1)$-th factor of either the vector bundles or the tensor products. Then if we consider the vector bundle $\oplus_{j=1}^e L_j$, where the set of line bundles $L_j$ are a permutation of the bundles $L_j'$ and $L_j''$, by assumption the rightmost square in the diagram  
\begin{equation}\label{diagram:commutativethroughlines1}
    \xymatrix{\uG(\oplus_{j=1}^e L_j) \ar[r] \ar[d]^{[c_i]} & \uG(\oplus_{j=1}^{i-1} L_j) \utimes \uG(L_i \oplus L_{i+1}) \utimes \uG(\oplus_{j=i+2}^e L_j) \ar@{-}[r] \ar[d]^{[c_i]} &   \cdots \\ 
    \uG(\oplus_{j=1}^e L_j) \ar[r] & \uG(\oplus_{j=1}^{i-1} L_j) \utimes \uG(L_{i+1} \oplus L_{i}) \utimes \uG(\oplus_{j=i+2}^e L_j) \ar@{-}[r] &  \cdots  }
\end{equation}
\begin{displaymath}
    \xymatrix{
     \hspace{3.5cm}   &\cdots\ar[r]  &\bugtimes_{j=1}^{i-1} \uG(L_j) \utimes \uG(L_i) \utimes \uG(L_{i+1}) \utimes \bugtimes_{j=i+1}^{e} \uG(L_j) \ar[d]^{c_i}\ar[d]^{c_i}       \\
     \hspace{3.5cm}    &\cdots\ar[r]  &\bugtimes_{j=1}^{i-1} \uG(L_j) \utimes \uG(L_{i+1}) \utimes \uG(L_{i}) \utimes \bugtimes_{j=i+1}^{e} \uG(L_j)
    }
\end{displaymath}
commutes. Since the leftmost square commutes for formal reasons, the whole diagram also commutes. Furthermore, the leftmost and rightmost vertical arrows in the diagram 
\begin{equation}\label{diagram:commutativethroughlines2}
    \xymatrix{\uG(E' \oplus E'') \ar[r] \ar[d]^{[c]} & \uG(E') \utimes \uG(E'') \ar[r] \ar[d]^{c} & \bugtimes_{j=1}^{e'} \uG(L_j') \utimes \bugtimes_{j=1}^{e''} \uG(L_j'')\ar[d]^c \\
    \uG(E'' \oplus E') \ar[r] & \uG(E'') \utimes \uG(E')   \ar[r] & \bugtimes_{j=1}^{e''} \uG(L_j'') \utimes \bugtimes_{j=1}^{e'} \uG(L_j')}
\end{equation}
can be written as a combination of isomorphisms of the form $c_i$ and $[c_i]$. The upper and lower horizontal isomorphisms of \eqref{diagram:commutativethroughlines2} equal the upper and lower horizontal isomorphisms in \eqref{diagram:commutativethroughlines1}, for two choices of permutations $L_j$. This follows similarly as before, from an application of the associativity in Definition \ref{def:determinantfunctordefinition}. By nesting this diagram by compositions of diagrams of the form \eqref{diagram:commutativethroughlines1}, it follows that the outer contour of the diagram \eqref{diagram:commutativethroughlines2} commutes. Since the rightmost square of \eqref{diagram:commutativethroughlines2} commutes for formal reasons, so does also the left square, which proves the first part of the proposition. 

The last remark is proven by Knudsen, see \cite[Theorem 2.3, Corollary 2.12]{Knudsen, Knudsen-err}. It is also observed in \cite[\textsection 4.10]{Deligne-determinant}. Both refer to constructions by Knudsen--Mumford \cite{KnudsenMumford}. 
\end{proof}


\section{Intersection bundles}\label{sec:intersectionbund}


We will rely on the language of line functors introduced in Section \ref{sec:linefun}. The main examples of line functors developed in this article will be that of line bundles that represent direct images of Chern classes, namely Deligne pairings and the refined integral of the second Chern class recalled below. Together they form the geometrical underpinnings of our later construction, and we refer to these constructions as intersection bundles. 

As in the previous section, the results below are stated for algebraic varieties, but can all be formulated over a locally noetherian base scheme or even complex analytic spaces. For the latter, see \textsection \ref{subsec:Analytification}. Also, while the rest of the article focuses on families of non-singular varieties, in this section we work in greater generality for completeness and future reference.

\subsection{Determinants}
We recall results concerning graded lines, determinants and Deligne pairings, which are used in subsequent subsections. 

\subsubsection{Graded line bundles}\label{subsub:gradedlines}

The category of $\ZBbb$-graded line bundles on a variety $X$ consists of objects which are a line bundle together with a locally constant function $\alpha: X \to \ZBbb$. An isomorphism of two graded line bundles $(\alpha, L) \simeq (\beta, M)$ is an isomorphism of line bundles $\phi:L \simeq M$, with the requirement that $\alpha = \beta.$
Given two $\ZBbb$-graded line bundles, there is a natural product $(\alpha, L) \otimes (\beta, M) = (\alpha + \beta, L \otimes M)$. Commutativity 
\begin{equation}\label{eq:commutativitygradedlines} (\alpha, L) \otimes (\beta, M) \simeq (\beta, M) \otimes (\alpha, L)
\end{equation}
is defined using the Koszul rule of signs. On the level of line bundle sections, the isomorphism is given by  
\begin{equation} \label{koszulrule} 
    \ell \otimes m \mapsto (-1)^{ \alpha + \beta } m \otimes \ell.
\end{equation}
There is a natural left inverse $( -\alpha, L^\vee)$ to $(\alpha, L)$, such that 
\begin{displaymath}
    (-\alpha, L^\vee) \otimes (\alpha, L) \simeq (0, \Ocal_X).
\end{displaymath}
The right inverse is defined using the commutativity isomorphism \eqref{eq:commutativitygradedlines}. Graded line bundles constitute a commutative line category with products, in the sense of Definition \ref{def:Picardcategory}, where the product is the usual tensor product.

\subsubsection{Determinants}\label{subsubsec:determinants}

The reference for details on this material can be found in Knudsen-- \linebreak Mumford's work \cite{KnudsenMumford}. The maximal exterior power $\Lambda^{r} E$ of a vector bundle $E$ is a line functor. The determinant of a vector bundle $E$ is the $\ZBbb$-graded line bundle
\begin{equation}\label{determinantfun} \det E = (r, \Lambda^{r} E),
\end{equation}
where $r = \op{rk} E$. Sometimes we implicitly identify $\det E$ with the line functor $E \mapsto \Lambda^r E$.  Given an exact sequence of vector bundles
\begin{equation}\label{eq:multdet}
    \varepsilon \colon 0 \to E' \to E \to E'' \to 0,
\end{equation}
there is a natural isomorphism 
\begin{equation}\label{eq:isodet}
   [\varepsilon] \colon  \det E\simeq \det E' \otimes \det E''
\end{equation} of graded line bundles. It is defined, in the split case when $E = E' \oplus E''$, by the rule $$e'_1 \wedge \ldots \wedge e_{r'}' \wedge e_1''\wedge \ldots \wedge e_{r''}'' \mapsto e'_1 \wedge \ldots \wedge e_{r'}' \otimes  e_1''\wedge \ldots \wedge e_{r''}''.$$ 
Arguing by locally splitting the sequence provides the general isomorphism by gluing. This also follows from Theorem \ref{thm:flagfiltration} and the surrounding splitting principles, proving that the determinant is a line functor with multiplicativity datum.  

If we denote by $c$ the commutativity isomorphism in \eqref{eq:commutativitygradedlines}, we see that with the Koszul rule of signs, the diagram 
$$
    \xymatrix{
        \det(E' \oplus E'') \ar[r]^c \ar[d] & \det(E'' \oplus E') \ar[d] \\ 
        \det E' \otimes \det E'' \ar[r]^c & \det E'' \otimes \det E'
    }
$$
in fact commutes, so $\det$ is commutative. This provides a framework to deal with a plethora of problems related to signs. 
\begin{lemma}\label{lemma:detprodiso}
If $E$ is a vector bundle of rank $r$, and $L$ is a line bundle, there is an isomorphism of multiplicative and commutative line functors,
\begin{equation}\label{eq:isodetline}
    \det(E \otimes L) \simeq (\det E) \otimes L^{\otimes r}.
\end{equation}
It is uniquely determined by the case when $E$ is a line bundle. More generally, there is a natural isomorphism of line functors of $2$-variables, 
\begin{equation}\label{eq:isodetline2}
    \det(E \otimes E') \simeq (\det E)^{\otimes r'} \otimes (\det E')^{\otimes r},
\end{equation}
multiplicative in either $E$ or $E'$.  It is characterized as reducing to \eqref{eq:isodetline} when either $E$ or $E'$ are line bundles. 
\end{lemma}
\begin{proof}
The two line functors are multiplicative in the obvious sense, with both functors commutativity being immediate. The lemma then follows from Proposition \ref{prop:useful-version-splitting}.
\end{proof}
In fact, it is not difficult to see that the above isomorphism is given by
\begin{displaymath}
    (e_1 \otimes e_1') \wedge (e_1 \otimes e_2') \wedge \ldots    \wedge (e_1 \otimes e_{r'}') \wedge (e_2 \otimes e_1') \wedge \ldots \wedge (e_r \otimes e_{r'}')\mapsto (e_1 \wedge \ldots \wedge e_r)^{\otimes r'} \otimes (e'_1 \wedge \ldots \wedge e_{r'}^{\prime})^{\otimes r}.
\end{displaymath}
Indeed, this is readily verified to be an isomorphism of multiplicative line functors, and it is the identity when both $E$ and $E'$ are line bundles. 


For a vector bundle $E$, taking determinants is compatible with taking duals on the underlying line bundles. The isomorphism 
\begin{equation}\label{eq:detdual}
    \det(E^\vee) \simeq (\det E)^\vee
\end{equation}
we choose is the one which is induced by the pairing $\det(E) \times \det(E^\vee) \to \Ocal_X$, given locally by $\langle e_1 \wedge \ldots \wedge e_r, f_1 \wedge \ldots \wedge f_r \rangle = \det(f_j(e_i)).$  Nevertheless, note that \eqref{eq:detdual}  is not an isomorphism of graded line bundles, since the gradings are $r$ and $-r$, which do not agree. In forthcoming arguments, the grading in principle only appears in the context of \eqref{eq:commutativitygradedlines}, in which case $r$ and $-r$ produce the same sign and the difference is hence irrelevant. 

\subsubsection{Determinant of the cohomology}\label{subsec:determinantofcoh}

For a bounded complex $E^\bullet$ of vector bundles, the notion of determinant thereof is 
\begin{equation}\label{def:detcomplex1}
    \det E^\bullet = \bigotimes_{i} (\det E^i)^{(-1)^i}.
\end{equation}
If the cohomology sheaves $\Hcal^i(E^\bullet)$ are also vector bundles, one has an isomorphism
\begin{equation}\label{def:detcomplex2}
    \det E^\bullet \simeq \det \Hcal^\bullet(E^\bullet) := \bigotimes_{i} (\det \Hcal^i(E^\bullet))^{(-1)^i}.
\end{equation}
Recall that a perfect complex on a variety $X$ is a complex $E^\bullet$ of $\Ocal_X$-modules, such that each point has an open neighborhood $U$ for which $E^\bullet_{\mid U}$ is quasi-isomorphic to a bounded complex of vector bundles. In \cite{KnudsenMumford} a procedure is described to extend determinants to the category of perfect complexes, defined as a subcategory of  $\Dbold^{\bmini}_{\scriptscriptstyle{\mathsf{coh}}}(X)$, with morphisms defined as the quasi-isomorphisms. Also see Proposition \ref{prop:useful-version-splitting}.

If $f: X\to S$ is a flat projective morphism, and $E$ a vector bundle on $X$, the derived pushforward $Rf_* E$ is a perfect complex \cite[Proposition 4.8, Expos\'e III]{SGA6}. The determinant of the cohomology will be the graded line bundle 
\begin{equation}\label{def:determinantcoh} \lambda(E) = \det Rf_* E
\end{equation}
on $S.$ For a point $s \in S$, the fiber of $\lambda(E)$ at $s$ naturally is identified with the alternating product of determinants of coherent sheaf cohomology: 
\begin{displaymath}
    \lambda(E)_s \simeq \bigotimes_i (\det H^i(X_s, E_{\mid X_s}))^{(-1)^i}.
\end{displaymath}
The grading of $\lambda(E)$ is thus given by the locally constant function $s\mapsto\chi(X_{s},E_{\mid X_{s}})$. For an exact sequence as in \eqref{eq:multdet}, there is a natural multiplicative isomorphism 
\begin{displaymath}
    \lambda(E) \to \lambda(E') \otimes \lambda(E'')
\end{displaymath}
making $\lambda$ into a multiplicative and commutative line functor. The construction and properties of the determinant of the cohomology extends to flat, locally projective morphisms with respect to $S$.

\subsection{Deligne pairings}\label{subsec:delignepairings}

\subsubsection{The norm functor}\label{subsubsec:norm-functor}

Suppose that $f: X \to S$ is a finite flat morphism, locally of degree $d$ say, and that $h$ is a function on $X$. On the level of structure sheaves, multiplication by $h$ induces an endomorphism $[h]: \Ocal_X \to \Ocal_X.$ Since locally on $S$ we have $\Ocal_X \simeq \Ocal_S^{\oplus d}$, we can consider the determinant of the endomorphism $[h]$, providing a function $N_{X/S}(h)$ of $\Ocal_S$. If $h$ is invertible, so is $N_{X/S}(h)$. 

If $L$ is a line bundle on $X$, it is possible to describe it by 1-cocycles with values in $\Ocal_X^{\times}$, on open covers of the form $f^{-1}(U_i)$ for an open cover $U_i$ of $S$. Taking norms of these 1-cocycles provides a 1-cocycle with values in $\Ocal_S^{\times}$. The corresponding line bundle is denoted by $N_{X/S}(L).$

\subsubsection{The general case}\label{sec:generalcasedeligneproduct}
Now suppose that $f\colon X\to S$ is a flat, locally projective morphism, whose fibers are Cohen--Macaulay of pure dimension $n\geq 0$. Given line bundles $L_{0}, \ldots, L_n$ on $X$, the associated Deligne pairing
\begin{displaymath}
\langle L_{0}, \ldots, L_n \rangle = \langle L_{0}, \ldots, L_n \rangle_{X/S}
\end{displaymath}
is a line bundle on $S$. This was constructed by Elkik in \cite{Elkikfib}. Below we provide a description in terms of local generators and relations, along the lines of \cite[\textsection II.3.3 \& \textsection III.2]{Elkikfib}:

\begin{enumerate}
    \item[\emph{(G)}] \emph{Generators.} Given rational sections, also called meromorphic (cf. \cite[\textsection 20.1.8 \& \textsection 21.1.4]{EGAIV4}) $\ell_{0}, \ldots, \ell_n$ in general position, there is trivialization 
    \begin{displaymath}
    \langle \ell_{0}, \ldots, \ell_n \rangle.
    \end{displaymath} Here general position signifies:
    \begin{enumerate}
        \item the Cartier divisor corresponding to the rational section $\ell_i$ is of the form $D_{i}=D_i^0 - D_i^1,$ where $D_i^0 $ and $D_i^1$ are  effective and relatively ample, flat over $S$; 
        \item for any function $\epsilon: [0,n] \to \{0,1 \}$ and any index $i\in\lbrace 0,\ldots,n\rbrace$, the scheme theoretic intersection
        \begin{equation}\label{eq:intersection-divisors}
             \bigcap_{j\neq i} D_{j}^{\epsilon(j)}
        \end{equation} 
        is finite flat over $S$, and 
        \begin{displaymath}
            \bigcap_{j}D_{j}^{\epsilon(j)}=\emptyset.
        \end{displaymath}
    \end{enumerate} 
    
    \item[\emph{(R)}] \emph{Relations.} If $\ell_i = h \ell_i^{\prime}$ for a rational function $h$, then 
    \begin{equation}\label{eq:relationElkikDucrot}
    \langle \ell_{0}, \ldots, \ell_i ,\ldots   \ell_n \rangle = N_{D/S}(h_{\mid D}) \langle \ell_{0}, \ldots, \ell_i', \ldots, \ell_n \rangle
    \end{equation}
    where $D = \cap_{j \neq i} D_j $ and $N_{D/S}(h_{\mid D})$ denotes the norm of the regular function $h_{\mid D}$ on $D$. Here we understand $D$ as a cycle, consisting of several components finite flat over $S$, possibly affected by a negative sign, and we extend $N_{D/S}$ multiplicatively with respect to this decomposition.
\end{enumerate}
\begin{remark}\label{rmk:construction-Deligne-pairing}
Some comments on the construction are in order:
\begin{enumerate}
    \item\label{comment:existence-sections} The symbols might be defined only locally with respect to $S$, and exist by the projective and flatness assumptions. The proof uses the graded prime avoidance lemma. See \cite[Corollaire 4.5.4]{EGAI} or \cite[\href{https://stacks.math.columbia.edu/tag/00JS}{00JS}]{stacks-project}. For completeness, in subsection \ref{subsec:Analytification} we provide the argument in the complex analytic setting.  
   \item By definition, if a symbol $\langle\ell_{0},\ldots,\ell_{n}\rangle$ is defined, then so is $\langle\ell_{\sigma(0)},\ldots,\ell_{\sigma(n)}\rangle$ for any permutation $\sigma$.
    \item In the case of relative dimension $0$, the construction specializes to the norm functor defined previously \textsection\ref{subsubsec:norm-functor}. 
    \item\label{comment:Weil} That the generators and relations produce a well-defined line bundle is part of the content of \cite{Elkikfib}. This can be recast as a form of the Weil reciprocity, which here amounts to the compatibility between the various relations derived from \emph{(R)} and the simultaneous change of two sections. In the complex analytic setting, this is precisely formulated in the proof of Lemma \ref{lemma:Weil-analytic}.
   
\end{enumerate}
\end{remark}
The below statements summarize the main properties of Deligne pairings that we will need. For a proof, we refer to \cite{Elkikfib}.

\begin{proposition}\label{Prop:generalpropertiesDeligneproduct}
The Deligne pairings, constructed as above, enjoy the following properties.
\begin{itemize}
    \item[(1)] The Deligne pairing is well-defined and commutes with arbitrary base change.
    \item[(2)] The Deligne pairing is additive in every entry: there are natural isomorphisms
    \begin{displaymath}
        \langle L_{0}, \ldots, L_{i}\otimes L_{i}^{\prime},\ldots L_n \rangle_{X/S}\simeq
       \langle L_{0}, \ldots, L_{i},\ldots L_n \rangle_{X/S}\otimes \langle L_{0}, \ldots, L_{i}^{\prime},\ldots L_n \rangle_{X/S}.
    \end{displaymath}
    \item[(3)]For any permutation $\sigma \in S_{n+1}$, there is a natural isomorphism 
    \begin{displaymath}
        [\sigma] \colon \langle L_{0}, \ldots, L_n \rangle_{X/S} \simeq \langle L_{\sigma(0)}, \ldots, L_{\sigma(n)} \rangle_{X/S}.
    \end{displaymath}
    If $L_i = L_j$, and $\sigma = (i,j)$, then  $[\sigma] = (-1)^{\delta},$ where $\delta = \int_{X/S} \prod_{k\neq i} c_1(L_k)$.
    \item[(4)] If $\lambda \in H^{0}(\Ocal_S^\times)$, the multiplication by $\lambda$ isomorphism $L_i \to L_i$ induces an isomorphism $$[\lambda]: \langle L_{0}, \ldots,  L_n \rangle_{X/S} \to \langle L_{0}, \ldots,  L_n \rangle_{X/S}.$$ 
    Then $[\lambda] = \lambda^{\delta}$, with $\delta = \int_{X/S} \prod_{k\neq i} c_1(L_k)$.
    \item[(4$^{\ \prime}$)]Let $q: X' \to X$ be a flat morphism, such that $X^{\prime}\to S$ is locally projective, with Cohen--Macaulay fibers of pure dimension $n+n^{\prime}$. If $L_{0}, \ldots, L_{\ell}$ are line bundles on $X$, and $M_{\ell+1}, \ldots, M_{n+n'} $ are line bundles on $X'$, then:
        \\
    \begin{enumerate}
        \item[(a)] if $\ell=n-1$, there is a natural isomorphism
        \begin{displaymath}
            \langle q^* L_{0}, \ldots, q^* L_{n-1}, M_{n}, \ldots, M_{n'+n} \rangle_{X'/S} \simeq \langle L_{0}, \ldots, L_{n-1}, \langle M_{n}, \ldots, M_{n'+n}\rangle_{X'/X} \rangle_{X/S};
        \end{displaymath}
        \item[(b)] if $\ell=n$, there is a natural isomorphism
            \begin{displaymath} \langle q^* L_{0}, \ldots, q^* L_{n}, M_{n+1}, \ldots, M_{n'+n} \rangle_{X'/S} \simeq \langle L_{0}, \ldots, L_{n}\rangle_{X/S}^{\delta}
            \end{displaymath}
            where $\delta = \int_{X'/X} \prod c_1(M_i)$ is assumed to be constant;
        \item[(c)] if $\ell \geq n+1$, the line bundle
        \begin{displaymath}
            \langle q^* L_{0}, \ldots, q^* L_{\ell}, M_{\ell+1}, \ldots, M_{n'+n} \rangle_{X'/S}
        \end{displaymath}
    is canonically trivial.
    \\
    \end{enumerate}
    \item [(5)] If $D \to S$ is an effective Cartier divisor of $X \to S$, flat over $S$, there is a natural isomorphism
        \begin{equation}\label{eq:divisorrestriction}
            \langle L_{0}, \ldots, L_{n-1}, \Ocal(D) \rangle_{X/S} \simeq \langle L_{0\mid D}, \ldots, L_{n-1 \mid D} \rangle_{D/S}.
        \end{equation}
    \item[(6)] If $E$ is of rank $r$, there is a natural isomorphism \begin{displaymath}\langle \Ocal(1) \{r\}\rangle_{\PBbb(E)/S} \simeq \det E, \end{displaymath} where $\langle \Ocal(1) \{r\}\rangle_{\PBbb(E)/S}$ denotes the Deligne pairing where all $r$ factors are given by $\Ocal(1)$ . A section $e$ of $E$ induces a section $\widetilde{e}$ of $\Ocal(1)$, and the isomorphism is given by 
\begin{displaymath}
    \langle \widetilde{e_1}, \ldots, \widetilde{e_r}\rangle \mapsto e_1 \wedge \ldots \wedge e_r.
\end{displaymath}
\end{itemize}
The isomorphisms above are the natural ones at the level of symbols, and commute with base change.
\end{proposition}
\qed

In the statement, we observe that the assumptions of $(4')$ guarantee that the morphism $X^{\prime}\to X$ is projective locally with respect to $S$ \cite[\href{https://stacks.math.columbia.edu/tag/0C4P}{0C4P}]{stacks-project}, with Cohen--Macaulay fibers \cite[\href{https://stacks.math.columbia.edu/tag/0C0W}{0C0W}]{stacks-project} of pure dimension $n^{\prime}$ \cite[\href{https://stacks.math.columbia.edu/tag/02NM}{02NM}, \href{https://stacks.math.columbia.edu/tag/02NL}{02NL}]{stacks-project}. In $(5)$, we notice that $D\to S$ has Cohen--Macaulay fibers of pure dimension $n-1$ \cite[Proposition 16.5.5]{EGAIV1}.

To conclude this subsection, let us mention that the theory of Deligne pairings can be extended to a more general class of morphisms. In \cite{Munoz}, Mu\~noz-Garc\'ia develops the formalism for projective morphisms of finite Tor-dimension, with fibers of constant and pure dimension. In the particular case of flat morphisms, an equivalent construction is due to Ducrot \cite{Ducrot}. His approach provides a formulation in terms of the determinant of the cohomology, which is useful even for Cohen--Macaulay morphisms. For families of curves, we will encounter particular instances of this below.

\subsection{An intersection second Chern class}\label{subsubsec:Delignecategory}
In the sequel, by a flat family of projective curves, we mean a faithfully flat projective morphism $f:X \to S$, with Cohen--Macaulay fibers of pure dimension 1. In this setting, in \cite[\textsection 9.6--9.7]{Deligne-determinant}, for a vector bundle $E$ on $X$, a line bundle $IC_2(E)$ on $S$ is constructed. It is naturally a line functor, and enjoys properties similar to that of the cohomology class $f_* c_2(E).$ Deligne provides a construction using filtrations, as well as a method utilizing the determinant of the cohomology of special line bundles. The latter, which is inspired by the Grothendieck--Riemann--Roch theorem, is recalled in \textsection \ref{subsec:Delignecomparison} below. Due to its higher dimensional generalisations, we will focus on Elkik's method for defining the intersection second Chern class. It is also defined using the analogues of Segre classes, relying on the Deligne pairings in the previous section, which is one of the reasons for their flexibility. 

It will be convenient to consider a commutative line category with products (see Definition \ref{def:Picardcategory}), to deal with certain product properties related to exact sequences, initially considered in \cite[\textsection 9.1]{Deligne-determinant}. It extends the graded line bundles in \textsection \ref{subsub:gradedlines} in this setting.  We refer the reader to \emph{loc. cit.} for further details.

\begin{definition} \label{def:tripplecategory}
For a flat family of projective curves $X\to S$, we consider the following commutative line category with products:
\begin{enumerate} 
\item the objects are triples $(r, L, M)$, where $r$ is a locally constant function $X\to \ZBbb$, $L$ is a line bundle on $X$ and $M$ is a line bundle on $S$;
\item a morphism $(r, L, M) \to (r', L', M')$ consists of isomorphisms $L \to L', M \to M'$, with the requirement that $r = r'$;
\item the product is provided by the rule
\begin{displaymath} 
    (r, L, M) \utimes  (r', L', M') = (r + r', L \otimes L', M \otimes M' \otimes \langle L, L' \rangle);
\end{displaymath}
    \item the commutativity isomorphism is \eqref{koszulrule} on the second factor, and the naive isomorphism on the third factor times $(-1)^N$, where $N = r r'  \left(\deg L + \deg L' \right)$.
\end{enumerate}

\end{definition}


The following amounts to Elkik's definition of the functorial direct image of the second Chern class.  
\begin{definition}\label{def:algcurv}
Let $X\to S$ be a flat family of projective curves. If $E$ is a vector bundle on $X$, we define 
\begin{equation} \label{eq:IC2def} IC_2(E) = \langle \det E, \det E \rangle_{X/S} \otimes \langle \Ocal(1) \{r+1\}\rangle_{\PBbb(E)/S}^{-1},
\end{equation}
where $\langle \Ocal(1) \{r+1\}\rangle_{\PBbb(E)/S}$ refers to the Deligne pairing with $\Ocal(1)$ iterated $r+1$ times.
\end{definition}

The following theorem is a special case, or versions of results found in \cite{Elkikfib}. We provide a self-contained proof relying on the splitting principles in Section \ref{sec:linefun}. 

\begin{theorem}\label{thm:ic2properties}
The rule $E \mapsto IC_2(E)$ is a line functor. It satisfies the following properties.
\begin{enumerate}
\item[(IC1)] The functor $E \mapsto (r, \det E, IC_2(E))$ equips $IC_2$ with the structure of a multiplicative and commutative line functor. In particular:
\begin{enumerate}
    \item for every exact sequence as in \eqref{eq:multdet}, there is a Whitney isomorphism
\begin{equation}\label{eq:Whitney} 
    IC_2(E) \simeq IC_2(E') \otimes IC_2(E'') \otimes \langle \det E', \det E'' \rangle,
\end{equation}
which is compatible with isomorphisms of exact sequences (cf. \cite[\textsection V.4.8]{Elkikfib}) ;
    \item suppose we are given an admissible filtration $ E_2' \subseteq E_1' \subseteq E$. Then associativity holds, in the sense that the diagram of Whitney isomorphisms below commutes:

\begin{displaymath}
    \xymatrix{ 
    &IC_2(E) \ar[rd] \ar[ld]            & \\
    IC_2(E_{1}^{\prime}) \otimes IC_2(E/E_1')\ar@{}[d]|-{\bigotimes}     &      
    &  IC_2(E_{2}^{\prime}) \otimes IC_2(E/E_2')\ar@{}[d]|-{\bigotimes} \\ 
     \langle \det(E_1'), \det(E/E_1') \rangle\ar[rd]   &   &\langle \det(E_2'), \det(E/E_2') \rangle\ar[ld]\\
    & \hspace{1.5cm}& }
\end{displaymath}
\vspace{-0.5cm} 
\begin{displaymath}
   \xymatrix{
        & IC_2(E_2') \otimes IC_2(E_1'/E_2') \otimes IC_2(E/E_1') \otimes \langle \det(E_1'/E_2'), \det(E/E_1')\rangle  \ar@{}[d]|-{\bigotimes} & \\
       &   \langle  \det(E_2'), \det(E_1'/E_2') \rangle \otimes \langle  \det(E_{1}'), \det(E/E_1') \rangle.    &
    }
\end{displaymath} 
\end{enumerate}
\item[(IC2)] For a line bundle $L$, there is a natural trivialization $IC_2(L) \simeq \Ocal_S$. If $\varphi \colon L \to L'$ is an isomorphism of line bundles, the induced isomorphism $[\varphi] \colon IC_2(L) \to IC_2(L')$ is compatible with these trivializations (cf. \cite[\textsection V.4.9]{Elkikfib}) .  


\end{enumerate}
\end{theorem}
\begin{proof}
The line functor property is clear from the definition of $IC_{2}$ and the properties of Deligne pairings. The discussion in \textsection\ref{subsubsec:determinants} shows that the functor in \emph{(IC1)} equips $IC_2$ with a product structure. By the second splitting principle in Theorem \ref{thm:flagfiltration}, it is enough to construct the multiplicativity datum \eqref{eq:Whitney} in the  case when $E' = L$ is a line bundle. By definition of the product structure, this amounts to an isomorphism as in (a) in  \emph{(IC1)}, the associativity in (b) being automatic by Proposition \ref{prop:associativityfromlines}. To this end, let $p: \PBbb(E) \to X$ be the natural projection. Then the inclusion $L \to E$ induces a closed immersion $i: D \to \PBbb(E)$ of the divisor $D = \PBbb(E/L)$. This closed immersion is cut out by the section of $\Ocal(1) \otimes (p^* L)^\vee$ determined by combining the morphism $p^* L \to p^* E $ with the tautological morphism $p^* E \to \Ocal(1).$  The Whitney isomorphism is then obtained by rewriting $\det E = L \otimes \det (E/L)$ via \eqref{eq:isodet}, and $\Ocal(1) = \Ocal(1) \otimes (p^* L)^\vee \otimes p^* L = \Ocal(D) \otimes p^* L$, whilst applying \eqref{eq:divisorrestriction} and multilinearity of the Deligne pairing to \eqref{eq:IC2def}. More precisely, we have using the properties described in Proposition \ref{Prop:generalpropertiesDeligneproduct}, isomorphisms

\begin{eqnarray*}
    \langle \det E, \det E \rangle & \to  &  \langle L \otimes \det(E/L), L \otimes \det(E/L) \rangle \\
     &\to  & \langle L, L \rangle \otimes \langle L, \det(E/L) \rangle \otimes \langle \det(E/L), L \rangle \otimes \langle \det(E/L), \det(E/L) \rangle \\
     & \to & \langle L, L \rangle \otimes \langle L, \det(E/L) \rangle^2 \otimes \langle \det(E/L), \det(E/L) \rangle
\end{eqnarray*}
and
\begin{eqnarray*}
    \langle \Ocal(1)\{r+1 \}\rangle_{\PBbb(E)/S} & \to  &  \langle \Ocal(D) \otimes p^* L, \Ocal(1)\{r\}\rangle_{\PBbb(E)/S} \\
     &\to  & \langle \Ocal(1)\{r\}\rangle_{\PBbb(E/L)/S} \otimes\langle  p^* L, \Ocal(1)\{r\}\rangle_{\PBbb(E)/S} \\
     & \to & \langle \Ocal(1)\{r\}\rangle_{\PBbb(E/L)/S} \otimes \langle p^\ast L, p^\ast L, \Ocal(1)\{r-1\} \rangle_{\PBbb(E)/S}  \otimes \langle  p^\ast L, \Ocal(D), \Ocal(1)\{r-1\}\rangle_{\PBbb(E)/S}    \\
     & \to &  \langle \Ocal(1)\{r\}\rangle_{\PBbb(E/L)/S} \otimes \langle L, L \rangle \otimes \langle  L, \det(E/L) \rangle.
\end{eqnarray*}
The difference of these two isomorphisms provides a Whitney-type isomorphism. We multiply it by $(-1)^{\deg(E/L) + r \deg L}$ . 

For the commutativity property, which is not formal, we postpone the proof to Proposition \ref{prop:ic2-commutativity}.

\end{proof}

\begin{remark}
\begin{enumerate}
    \item The line functor property of $IC_{2}$ includes in particular an obvious compatibility with isomorphisms $X^{\prime}\to X$ of curves over $S$. See Definition \ref{def:linefunctorkvar} (a).
    \item The rewriting of $\Ocal(1)\{r+1 \}$ in the proof involves the isomorphism 
\begin{displaymath}\langle p^\ast L, \Ocal(1)\{r \} \rangle \to  \langle L, L \rangle \otimes \langle  L, \det(E/L) \rangle. 
\end{displaymath}
By Proposition \ref{Prop:generalpropertiesDeligneproduct}, there is also an isomorphism $\langle p^\ast L, \Ocal(1)\{r \}\rangle \to \langle L, \det E \rangle$, which combined with $\det E \to L \otimes \det(E/L)$ provides another isomorphism as above. These two coincide, and we have chosen the above in the proof because it is the one used by Elkik in \cite{Elkikfib}.
    \item With the notation as in the proof of Theorem \ref{thm:ic2properties}, let $\ell_i, i = 0, \ldots, r$ be sections of $E$, with induced sections $\widetilde{\ell}_{i}$ of $\Ocal(1)$. Suppose that $\widetilde{\ell}_0 = \sigma_L \otimes p^\ast \ell$, where $\sigma_{L}$ is the canonical section of $\Ocal(D)$. Let also $u$ (resp. $v$) be sections of $\det E$, such that $u$ corresponds to $\ell \otimes u'$ (resp. $v$ corresponds to $\ell' \otimes v'$) under the isomorphism, $\det E \to L \otimes \det(E/L)$. Likewise, under the same isomorphism write $\ell_1 \wedge \ldots \wedge \ell_r = \ell'' \otimes w$. The Whitney isomorphism on the level of symbols is then given by 

\begin{equation}\label{Whitneyonsymbols}
   \begin{split} \langle  u, v \rangle\otimes \langle \widetilde{\ell}_0, \widetilde{\ell}_1, \ldots,  \rangle^{-1} \mapsto (-1)^{\deg(E/L) + r \deg L} & \langle \ell, \ell' \rangle \otimes \langle \ell, v' \rangle \otimes \langle \ell', u'  \rangle \otimes\langle u', v' \rangle\\
  & \otimes\langle {\widetilde{\ell}_{1| D}},{\widetilde{\ell}_{2|D}}, \ldots, {\widetilde{\ell}_{r|D}}  \rangle^{-1} \otimes\langle \ell_0, \ell'' \rangle^{-1} \otimes\langle \ell_0, w \rangle^{-1}.
    \end{split}
\end{equation}
\end{enumerate}
\end{remark}


\subsection{Sections of $IC_2$}\label{subsec:sectionsIC2}

In \textsection \ref{sec:generalcasedeligneproduct} the Deligne pairings were described in terms of symbols and relations. Such a concrete description is generally lacking for $IC_2$. Below, we provide general constructions of generators of $IC_2$ in favourable situations, and describe the case of rank 2 in more detail. Not only this enlightens on the geometric meaning of $IC_{2}$, but it will also be needed to compare Deligne and Elkik's approaches in \textsection\ref{subsec:Delignecomparison}. More concrete applications are explicit formulas for intersection metrics and complex metrics in  Proposition \ref{prop:explicit-metric} and \textsection \ref{subsub:explicit-formulas-rk-2}.

\subsubsection{Two constructions of local trivializations} Let $E$ be a vector bundle of rank $r \geq 1$ on $X\to S$, and $s= (s_1, \ldots, s_{r-1})$ an ordered $(r-1)$-tuple of sections. Suppose that all $s_i$ are non-vanishing. 

For the first construction, write $\det s = s_1 \wedge \ldots \wedge s_{r-1}.$ Suppose there are auxiliary sections $u,v$ such that $\det s \wedge u$ and $\det s \wedge v$ determine sections of $\det E$ with disjoint divisors, both flat over $S$. A section $t$ of $E$ induces a section $\widetilde{t}$ of $\Ocal(1)$ on $p: \PBbb(E) \to X$, by composing $p^* t: \Ocal_X \to p^* E$ with $p^* E \to \Ocal(1)$. Denote by $\widetilde{s}, \widetilde{u}, \widetilde{v}$ the induced sections or tuples of sections of $\Ocal(1)$. We can then define a trivialization of $IC_2(E)$ by using the trivialization 
\begin{equation}\label{eq:generalsection} \langle \det s \wedge u, \det s \wedge v \rangle_{X/S}  \otimes \langle 
\widetilde{s}, \widetilde{u}, \widetilde{v} \rangle_{\PBbb(E)/S}^{-1}. 
\end{equation}

For the second construction, notice that the $(r-1)$-tuple $s$ naturally yields a short exact sequence 
\begin{equation}\label{eq:generalsectionss} 0 \to \Ocal_X^{r-1} \to E \to L \to 0,
\end{equation}
with $L \simeq \det E.$ By filtering by the first factor $\Ocal_X \subseteq \Ocal_X^{r-1}$, and using the standard isomorphism $\det \Ocal_X^k \simeq \Ocal_X$, one readily obtains a natural trivialization of $IC_2(\Ocal_X^{r-1})$. We will later see in Proposition \ref{prop:relationic2symbols} that this trivialization is independent of the ordering of the filtration. 

The Whitney isomorphism of \eqref{eq:Whitney} applied to \eqref{eq:generalsectionss}  then also provides an isomorphism
\begin{equation}\label{eq:trivializationgeneralcase}
    \T{s}: IC_2(E) \simeq IC_2(\Ocal_X^{r-1})\otimes IC_2(\det E) \otimes \langle \det \Ocal_X^{r-1}, \det E \rangle \simeq \Ocal_S ,
\end{equation}
and hence a trivialization of the bundle. 

\begin{lemma}\label{lemma:equivalent-sections}
Consider a section as in \eqref{eq:generalsection}  of $IC_2(E)$ associated to $s = (s_1, \ldots, s_{r-1})$, and also the section $\T{s}$ of $IC_2(E)$ determined by \eqref{eq:trivializationgeneralcase}. Then the two sections differ by the sign $(-1)^{(r-1)\deg E}$. In particular, the former is independent of the choices of $u,v$. 

\end{lemma}
\begin{proof}
If $r=1$ there is nothing to prove. If $r \geq 2$, denote by $s'$ the tuple obtained by removing $s_1$. Also, denote by $0 \to \Ocal_X^{r-2} \to E/s_1 \Ocal_X \to L \to 0$ the exact sequence obtained by taking the quotient  by $s_1$. Applying the Whitney isomorphism to the sequence $0 \to \Ocal_X \overset{s_1}{\to}  E \to E/s_1 \Ocal_X \to 0$, and using natural trivialities, we obtain an isomorphism 
\begin{displaymath}
    IC_2(E) \simeq  IC_2(s_1 \Ocal_X) \otimes \langle \det(E/s_1 \Ocal_X), s_1 \Ocal_X \rangle \otimes  IC_2(E/s_1 \Ocal_X)\simeq IC_2(E/s_1 \Ocal_X).
\end{displaymath}
A diagram chase, involving the associativity in \emph{(IC1)} of Theorem \ref{thm:ic2properties}, shows that the section $\T{s}$ is sent to $\T{s'}$ under this isomorphism. 

Hence, by induction it is enough to prove that the section  \eqref{eq:generalsection} is sent to the analogous section induced by $s'$ of $IC_2(E/s_1 \Ocal_X)$. This follows from an immediate application of the description on the level of symbols in \eqref{Whitneyonsymbols}. 

\end{proof}



\subsubsection{Sections and the Whitney isomorphism}
For the formulation of the Proposition \ref{prop:ic2-commutativity} below, consider first global sections $\ell$ and $m$ of line bundles $L$ and $M$, without common zeros.  Then, the global section $(\ell, m)$ of $L \oplus M$ is a nowhere vanishing section, and the quotient of $L \oplus M$ by the subbundle generated by $(\ell, m)$ is isomorphic to $\det (L \oplus M) = L \otimes M$. We hence have a short exact sequence
\begin{equation}\label{Whitneylm}
    0 \to \Ocal_X \to L \oplus M \to L \otimes M \to 0 
\end{equation}
as in \eqref{eq:generalsectionss}, and we denote by $\T{(\ell, m)}: IC_2(L \oplus M) \simeq \Ocal_S$ the corresponding trivialization of $IC_2(L \oplus M)$. A computation using \eqref{eq:isodet} shows that the second map of \eqref{Whitneylm} is 
\begin{equation}\label{isodirectsumdet}(\ell', m') \mapsto - \ell' \otimes m + \ell \otimes m'.
\end{equation}
In general, if $\ell$ and $m$ are rational sections with $\Div \ell = D' - D$ and $\Div m = E' - E$ in general relative position, using the expression in \eqref{isodirectsumdet} gives a short exact sequence
\begin{equation}\label{Whitneylmrational}
    0 \to \Ocal_X(-D-E) \to L \oplus M \to L \otimes M \otimes \Ocal_{X} (D+E)\to 0
\end{equation}
which identifies with \eqref{Whitneylm} outside of $D$ and $E$, and hence an isomorphism $IC_2(L \oplus M) \simeq \langle \Ocal_X(-D-E), L \otimes M \otimes \Ocal_X(D+E) \rangle.$ Denoting by $\textbf{1}_{D+E}$ the canonical section of $\Ocal_{X} (-D-E)$, we see that $\langle \Ocal_X(-D-E), L \otimes M \otimes \Ocal_X(D+E) \rangle$ admits the trivializing section $\langle \textbf{1}_{D+E}, \ell \otimes m \otimes \textbf{1}_{D+E}^{-1} \rangle $.  We also denote the corresponding section of $IC_2(L\oplus M)$ by $\T{(\ell, m)}$. It coincides with the trivialization in \eqref{Whitneylm} if $D$ and $E$ are empty.

Also applying \eqref{eq:Whitney} to the standard split exact sequence 
\begin{equation}\label{exactseqLM} 0 \to L \to L \oplus M \to M \to 0 
\end{equation}
we find that $IC_2(L \oplus M) \simeq \langle L, M \rangle$, used in the formulation of the next proposition. 


\begin{proposition}\label{prop:ic2-commutativity}
Suppose that $\ell$ and $m$ are rational sections of $L$ and $M$, with disjoint divisors, individually finite and flat over the base. The trivialization  $\T{(\ell, m)}$ of $IC_2(L\oplus M)$ in $\langle L, M\rangle$ under the isomorphism induced by \eqref{exactseqLM} corresponds to the section 
\begin{displaymath}
    (-1)^{\deg M} \langle \ell, m \rangle, 
\end{displaymath}
where $\deg M$ is the fiberwise degree of $M$. It follows in particular that $IC_2$ is a multiplicative and commutative line functor, i.e. the diagram  
$$\xymatrix{
IC_2(E' \oplus E'')  \ar[r] \ar[d] & \ar[d] IC_2(E') \otimes IC_2(E'') \otimes \langle \det E', \det E'' \rangle \\
IC_2(E'' \oplus E')  \ar[r] &  IC_2(E'') \otimes IC_2(E') \otimes \langle \det E'', \det E' \rangle 
}$$
commutes up to the sign $(-1)^{N}$ where $N= r_1 r_2 \left(\deg E' + \deg E'' \right).$ 
\end{proposition}
\begin{proof}
A straightforward computation shows that the first property implies the commutativity, relying on Corollary \ref{cor:criteriacommutative} and the formalism of Deligne's category in \textsection \ref{subsubsec:Delignecategory}. See Definition \ref{def:tripplecategory}.

We will need to describe the two different Whitney isomorphisms associated to \eqref{Whitneylm} and \eqref{exactseqLM} recalled above, and first set up the problem. We simplify the argument by assuming that $\ell$ and $m$ are global regular sections, the general case is analogous. From the previous Lemma \ref{lemma:equivalent-sections} we know that the trivialization corresponding to $\langle (\ell, m) \rangle$ equals
\begin{displaymath}
    (-1)^{\deg L + \deg M}\langle (\ell,m) \wedge \ell', (\ell, m) \wedge m' \rangle \otimes \langle \widetilde{(\ell,m)}, \widetilde{\ell}', \widetilde{m}'  \rangle^{-1}.
\end{displaymath}
Throughout, denote by $p: \PBbb(L \oplus M) \to X$ the natural projection. Under the Whitney isomorphism described in \eqref{Whitneyonsymbols} we see that the first section of the above product is sent to 
\begin{equation}\label{Whitneysigndeterneverending}
    \langle -\ell'\otimes m, \ell \otimes m' \rangle \mapsto (-1)^{\deg L + \deg M}\langle \ell', \ell \rangle\otimes  \langle \ell', m' \rangle \otimes\langle \ell, m \rangle \otimes\langle m, m' \rangle.
\end{equation}
The other section $\langle \widetilde{(\ell,m)}, \widetilde{\ell'}, \widetilde{m}' \rangle$ can be written, after applying an even permutation of the sections (see (3) of Proposition \ref{Prop:generalpropertiesDeligneproduct})) as $\langle \widetilde{\ell'}, \widetilde{m}', \widetilde{(\ell, m)}\rangle$. 

For the rest, recall that there is an embedding $i: X=\PBbb(M)=\PBbb((L \oplus M)/L)  \to \PBbb(L \oplus M)$ determined by $\sigma_L: p^\ast L \to p^\ast\left( L \oplus M \right) \to \Ocal(1)$.
The Whitney isomorphism described in Theorem \ref{thm:ic2properties} sends this section to 
\begin{displaymath}
    \langle \sigma_L,  \widetilde{m}',\widetilde{(\ell, m)} \rangle \otimes \langle p^* \ell',  \sigma_L, \widetilde{(\ell, m)}  \rangle \otimes \langle p^* \ell', \widetilde{m}' \otimes \sigma_L^{-1},\widetilde{(\ell, m)} \rangle
\end{displaymath}
and then further to
\begin{displaymath}
    \langle i^\ast \widetilde{m}',i^\ast \widetilde{(\ell, m)} \rangle \otimes \langle i^\ast p^* \ell',  i^\ast \widetilde{(\ell, m)}  \rangle \otimes \langle i^\ast p^* \ell', i^\ast \left(\widetilde{m}' \otimes \sigma_L^{-1}\right) \rangle.
\end{displaymath}
A direct computation shows that $i^\ast \widetilde{m}' = m', i^\ast \widetilde{(\ell, m)} = m$. The computation of the section $i^\ast \widetilde{m}' \otimes \sigma_L^{-1}$ is more delicate, and we start to notice that it is necessarily of the form $p^\ast \ell''$ for a section $\ell''$ of $L$. It can be computed by restricting along $i': X \to \PBbb(L \oplus M)$ determined by the sequence \eqref{Whitneylm}. Here ${i'}^\ast \Ocal(1) = L \otimes M$. Expanding as in \eqref{Whitneyonsymbols} we find that this section is sent to 
\begin{displaymath}
    (-1)^{\deg L} \langle m', m \rangle \otimes\langle \ell', \ell \rangle \otimes\langle \ell', m' \rangle
\end{displaymath}
Then, since $\widetilde{\ell}' = p^* \ell' \otimes \sigma_L$, we can expand 
\begin{equation}\label{eq:Whitneysignneverend}
    \langle  \widetilde{\ell}', \widetilde{m}', \widetilde{(\ell, m)} \rangle \mapsto \langle \sigma_L,  \widetilde{m}',\widetilde{(\ell, m)} \rangle \otimes \langle p^* \ell',  \sigma_L, \widetilde{(\ell, m)}  \rangle \otimes \langle p^* \ell', \widetilde{m}' \otimes \sigma_L^{-1},\widetilde{(\ell, m)} \rangle.
\end{equation}

 The section $\widetilde{(\ell, m)} $ of $\Ocal(1)$ is obtained from the section $(\ell, m)$ of $L \oplus M$ via the composition of $p^*(\ell, m): p^* \Ocal_X\to p^* (L \oplus M)$ with the natural quotient map $p^* (L \oplus M) \to \Ocal(1)$. Its zero locus is identified with a section $i: X \to \PBbb(L \oplus M)$  of $p: \PBbb(L \oplus M) \to X$. Under the identifications of \eqref{Whitneylm}, we have $i^* \Ocal(1) = L \otimes M$.

The computations in \eqref{isodirectsumdet} show that ${i'}^* \sigma_L = -m$ and ${i'}^* \sigma_M = \ell$, so that we have $\ell'' = {i'}^* \left( \widetilde{m}' \otimes \sigma_L^{-1} \right) = {i'}^{\ast}\left(p^\ast m' \otimes \sigma_M \otimes \sigma_L^{-1} \right) = - m'/m \otimes \ell$. It follows that also $i^\ast \left( \widetilde{m}' \otimes \sigma_L^{-1} \right) = - m'/m \otimes \ell$. Hence, restricting to the divisor $i: X \to \PBbb(L \oplus M)$, we finally find that \eqref{eq:Whitneysignneverend} takes the form 
\begin{equation}\label{eq-with-no-name}
    \langle m', m \rangle\otimes \langle \ell', m \rangle \otimes \langle \ell', -m'/m \otimes \ell \rangle  =(-1)^{\deg L} \langle m', m \rangle \otimes \langle \ell', m' \rangle\otimes \langle \ell', \ell\rangle. 
\end{equation}
Using that $\langle m', m \rangle = (-1)^{\deg M} \langle m, m' \rangle$ and comparing \eqref{eq-with-no-name} with \eqref{Whitneysigndeterneverending}, we find that the two differ by the expression $\langle \ell, m \rangle$. Since our Whitney isomorphism multiplies this construction with $(-1)^{\deg M}$, we find the statement. 

\end{proof}

\begin{remark}\label{rem:extensiontocomplexes}
As recalled in Corollary \ref{cor:criteriacommutative}, the above proposition implies that $IC_2$ extends to a functor on complexes of vector bundles and quasi-isomorphisms between them. 
\end{remark}

\subsubsection{Relations between symbols}\label{subsubsection:ic2symbols}

Suppose we are given two $(r-1)$-tuples of everywhere non-vanishing sections,  $s$ and $t$, of a vector bundle $E$ of rank $r$. We say they are in general position if they generate a subvector bundle of rank $r-1$. Then $(\det s) \wedge (\det t)$ is a section of $\Lambda^2 \left(\Lambda^{r-1} E\right)$, whose zero locus cuts out a closed subspace $Z := Z(s,t)$. On $E_{\mid Z}$ we can write $\det s = \alpha \det t$, for $\alpha$ an invertible function on $Z.$ In the following propositions, we will describe what amounts to relations of the trivializations $\T{s}$ and $\T{t}.$

\begin{proposition}\label{prop:relationic2symbols}
Let $s$ and $t$ be two $(r-1)$-tuples of everywhere non-vanishing sections of a vector bundle $E$ of rank $r$, in general position, and let $Z$ be the closed  subspace where $s$ and $t$ generate the same subbundle of $E$. If $ Z= X$, then:
\begin{enumerate}
\item if $r = 1$, $IC_2(E)$ is canonically trivial;
\item if $r \geq 2$,  $\T{s} = \alpha^{\deg E} \T{t}$. Here $\det s = \alpha \det t$ and $\alpha = \det A$ for the matrix transforming the basis $s$ to $t$. In particular, for a permutation $\sigma$ of $\{1, \ldots, r\}$ and $\sigma (s) = (s_{\sigma(1)}, \ldots, s_{\sigma(r)})$, we find that $\T{\sigma(s)} = (-1)^{\op{sgn}(\sigma) \cdot \deg E}\T{s}$;
\item  if $\lambda$ is an invertible function on $S$, the multiplication $\lambda: E \to E$ induces an isomorphism $[\lambda]: IC_2(E) \to IC_2(E)$; then, $[\lambda] = \lambda^{(r-1) \cdot \deg E}$, where $\deg E$ is the fiberwise degree of $E$ (cf. \cite[\textsection V.4.11]{Elkikfib}).
\end{enumerate}

\end{proposition}
\begin{proof}
The first case corresponds to the fact that there is no input data for the trivialization $\T{s}$. For the second case, the hypothesis states that $s$ and $t$ generate the same subbundle of $E$, and hence there is a natural diagram  
\begin{displaymath}
    \xymatrix{
        \Ocal_X^{r-1} \ar[r]^t \ar[d]^{A} & E \ar[r] \ar[d]^{\id} & \det E \ar[d]^{\id} \\ 
\Ocal_X^{r-1} \ar[r]^s  & E \ar[r]  & \det E 
}
\end{displaymath}
for a $(r-1) \times (r-1)$-matrix with entries in $H^{0}(\Ocal_X) = H^{0}(\Ocal_S)$. It induces an isomorphism of the Whitney isomorphisms \eqref{eq:trivializationgeneralcase}, which allows us to compare the trivializations. The only non-trivial map induced from these two sequences is the map $\langle \Ocal_X, \det E \rangle \to \langle \Ocal_X, \det E \rangle$. The section $1$ in $\Ocal_X$ corresponds to $\det s$ or $\det t$ under the isomorphism $\det \Ocal_X^{r-1} \simeq \Ocal_X$, so that  $\det s= \alpha \det t$, where $\alpha = \det A$. Hence the map is described by $\langle 1, \omega \rangle \mapsto \langle \alpha, \omega \rangle $ for an arbitrary section $\omega$ of $\det E$. From Proposition \ref{Prop:generalpropertiesDeligneproduct}, we find that $ \langle \alpha, \omega \rangle  = \alpha^{\deg E} \langle 1, \omega \rangle $. 

The last point clearly holds if we have a section $\langle s \rangle$ by the previous part of the proposition. It holds more generally by induction and a filtration argument. 

\end{proof}

In the case of rank 2 vector bundles $E$, we can say more:
\begin{proposition}\label{prop:reciprocity-rank-2}
Let the assumptions and notation be as in Proposition \ref{prop:relationic2symbols}. Suppose furthermore that $r=2$ and $ Z \to S$ is finite and flat. Then $\T{s} = N_{Z/S}(\alpha) \T{t}$.
\end{proposition}
\begin{proof}
For simplicity we identify $s = s_1, t = t_1$, and note that $Z = \operatorname{div}(s \wedge t)$ is necessarily a Cartier divisor, flat over $S$. The sections $s,t$ both provide exact sequences of the form $0 \to \Ocal_X \to E \to \det E \to 0$, and we need to compare the trivializations of $IC_2$ they induce via the Whitney isomorphism. By definition, the Whitney isomorphism sends $\T{s}$ to $\langle s, \omega \rangle$ in $\langle s\Ocal_X, \det E \rangle$, for an arbitrary section $\omega$ of $\det E$ and where $s$ is considered as an element of $E$. Likewise $\T{t}$ is sent to $\langle t, \omega \rangle$. These are not sections of the same bundle, hence cannot be compared directly. We choose $\omega = s \wedge t$, so that $\det E \simeq \Ocal(Z)$, and apply Proposition \ref{Prop:generalpropertiesDeligneproduct} to obtain $\langle s \Ocal_X , \det E \rangle \simeq N_{Z/S}(s \Ocal_Z)$ and idem with $\langle t \Ocal_X, \det E \rangle$. Since $s$ and $t$ generate the same subbundle of $E$ on $Z$, we can compare the sections there, and the relation between them is that of the proposition. 

To prove the proposition, we need to prove that the  considered isomorphism $IC_2(E) \simeq IC_2(\Ocal_X) \otimes IC_2(\det E) \otimes \langle s \Ocal_X, \det E \rangle \simeq N_{Z/S}(s \Ocal_X) \simeq \Ocal_S$ coincide for $s$ and $t$, or equivalently that they are independent of $s$ and $t$. In this case, the comparison of the sections $\T{s}$  and $\T{t}$ in $IC_2(E)$ will coincide with the comparison of their images in $\Ocal_S$.

The two sequences involving $s$ and $t$ can be completed into a diagram of the form: 
\begin{equation}\label{eq:diagramfornormproof}
    \xymatrix{ 
        & 0 \ar[r] \ar[d] & \Ocal_X \ar[r]^{\id} \ar[d]^s & \Ocal_X  \ar[d]^{\overline{s}}    \\
        & \Ocal_X \ar[r]^t \ar[d]^{\id} & E \ar[r] \ar[d] & \det E   \ar[d] \\
        & \Ocal_X \ar[r]^{\overline{t}} & \det E  \ar[r]  & \det E_{\mid Z}.
    }
\end{equation}
This will allow us to compare the two sequences and thus also the Whitney isomorphisms. 
The Whitney isomorphism involving $s$ can be decomposed using that of $\overline{t}$ as 
\begin{eqnarray*}
    IC_2(E) & \simeq & IC_2(s \Ocal_X) \otimes IC_2(\det E) \otimes \langle s \Ocal_X, \det E  \rangle \\ &  \simeq & IC_2(s \Ocal_X) \otimes IC_2(\overline{t} \Ocal_X) \otimes IC_2(\det E_{\mid Z}) \otimes \langle \overline{t} \Ocal_X, \det E_{\mid Z} \rangle  \otimes \langle s \Ocal_X, \overline{t} \Ocal_X \rangle \otimes \langle s \Ocal_X, \det E_{\mid Z}  \rangle.
\end{eqnarray*}
Here, we use the extension of $IC_2$ to coherent sheaves quasi-isomorphic to a complex of vector bundles, see  Remark \ref{rem:extensiontocomplexes}. The same discussion applied to $t$ and and $\overline{s}$ provides a sequence with identical final product, providing an isomorphism 
\begin{displaymath}
    W: IC_2(s\Ocal_X)\otimes IC_2(\det E) \otimes \langle s \Ocal_X , \det E \rangle  \simeq IC_2(t\Ocal_X)\otimes IC_2(\det E) \otimes \langle t \Ocal_X , \det E \rangle.
\end{displaymath}
It follows from \cite[Lemme 4.8]{Deligne-determinant} that the various Whitney isomorphisms applied to \eqref{eq:diagramfornormproof}  commute, and that the diagram of isomorphisms
\begin{displaymath}
    \xymatrix{
        IC_2(E) \ar[r] \ar[d]^{\id} & IC_2(\Ocal_X) \otimes IC_2(\det E) \otimes \langle s \Ocal_X, \det E  \rangle  \ar[d]^W \ar[d]  \\
        IC_2(E) \ar[r] & IC_2(\Ocal_X) \otimes IC_2(\det E) \otimes \langle t \Ocal_X, \det E  \rangle
    }
\end{displaymath}
in fact commutes. Noticing that $\det(E_{\mid Z}) \simeq \Ocal_X(Z)$, this is a rewritten version of the sought diagram, except we have not verified that the isomorphism $\langle s \Ocal_X, \det E \rangle \simeq \langle s \Ocal_X, \Ocal_X(Z)\rangle$ is the one induced by choosing $\omega = s\wedge t$ as in the beginning of the proof.

To further investigate $W$, a computation similar to that of \eqref{isodirectsumdet} shows that the map $\overline{s}$ sends $1$ to $t \wedge s$ and $\overline{t}$ sends $1$ to $s \wedge t = - t \wedge s $. This shows that the lower horizontal and right vertical sequences are in fact isomorphic up to sign. From Proposition \ref{Prop:generalpropertiesDeligneproduct} we find  $\langle s, - s \wedge t \rangle = \langle s,  s \wedge t \rangle$ and $\langle t, - s \wedge t \rangle = \langle t,  s \wedge t \rangle $. This implies that, while the two sequences are not the same, since the only difference is a sign, this is not seen on these Deligne pairings. These facts taken together prove the proposition. 
\end{proof}

\begin{remark}
Proposition \ref{prop:reciprocity-rank-2} generalizes the relations defining the Deligne pairing of two line bundles, at least for $E$ admitting enough sections in general position, locally with respect to $S$. This can always be achieved after possibly tensoring $E$ by a suitable relatively ample line bundle. In Proposition \ref{prop:tensorlinebundleiso} below we will describe how $IC_{2}$ behaves under such modifications. One can conclude that in rank 2, it is possible to describe $IC_{2}$ directly in terms of generators and relations. We refer to Proposition \ref{prop:explicit-metric} and \textsection \ref{subsub:explicit-formulas-rk-2} below for an example of application along these lines.
\end{remark}

\subsection{Further properties of $IC_2$}

We will now study $IC_2$ of various constructions, in particular products of vector bundles or duals. The splitting principles established in Section \ref{sec:linefun} allow us to reduce otherwise complicated computations to simpler ones. 

\subsubsection{Products of vector bundles} Let $L$ be a line bundle and $E$ a vector bundle of rank $r$. The functor 
\begin{equation}\label{eq:Chernwithline}
    E \mapsto \left(r, \det(E \otimes L) , IC_2(E \otimes L)\right)
\end{equation}
is a commutative and multiplicative line functor into the line category with products in Definition \ref{subsubsec:Delignecategory}, with the first two factors equipping $IC_2(E \otimes L)$ with a product structure.

Consider also  the Chern polynomial-type line functor
\begin{equation}\label{eq:Chernpolynomials}
     E\mapsto IC_2(E) \otimes  \langle \det E, L \rangle^{r-1} \otimes \langle L, L \rangle^{{r \choose 2}}. 
\end{equation}
On the level of isomorphism classes, \eqref{eq:Chernpolynomials} corresponds to the direct image of a well-known expression for the class $c_2(E \otimes L)$, see \cite[\textsection 3.2.2]{Fulton}. Using Theorem \ref{thm:flagfiltration} and the Whitney isomorphism for $IC_2$, one sees that the functor 
\begin{equation}\label{eq:Chernwithlines2}
    E \mapsto (r, (\det E) \otimes L^{r}, IC_2(E) \otimes  \langle \det E, L \rangle^{r-1} \otimes \langle L, L \rangle^{{r \choose 2}})
\end{equation}
is likewise a multiplicative and commutative line functor. 
\begin{proposition}\label{prop:tensorlinebundleiso}
\begin{enumerate}
    \item Let $E$ be a vector bundle of rank $r$. There is a unique  isomorphism of multiplicative line functors
\begin{equation}\label{eq:IC2-E-otimes-L}
    IC_2(E \otimes L) \simeq IC_2(E) \otimes  \langle \det E, L \rangle^{r-1} \otimes \langle L, L \rangle^{{r \choose 2}}, 
\end{equation}
compatible with Lemma \ref{lemma:detprodiso} and such that the isomorphism is compatible with the natural trivializations of both sides when $E$ is a line bundle. 
    \item If $L$ and $M$ are line bundles, the isomorphism \eqref{eq:IC2-E-otimes-L} associated to the line bundle $L \otimes M$ can be identified with the isomorphism obtained by first applying $M$ and then $L$:
    \begin{eqnarray*} IC_2(E \otimes L \otimes M) & \simeq &   IC_2(E \otimes L) \otimes  \langle M, \det (E \otimes L) \rangle^{r-1}  \otimes \langle M, M \rangle^{{r \choose 2}}  \\
    & \simeq & \left(IC_2(E) \otimes  \langle \det E, L \rangle^{r-1} \otimes \langle L, L \rangle^{{r \choose 2}}\right) \otimes \langle M,  L^{r} \otimes \det E \rangle \otimes  \langle M, M \rangle^{{r \choose 2}} \\ & \simeq & IC_2(E) \otimes  \langle \det E, L \otimes M \rangle^{r-1} \otimes \langle L \otimes M, L \otimes M \rangle^{{r \choose 2}} .
    \end{eqnarray*}
\end{enumerate}
\end{proposition}
\begin{proof}
By Proposition \ref{prop:useful-version-splitting} one is reduced to proving that there is a natural isomorphism of the two functors in \eqref{eq:Chernwithline} and \eqref{eq:Chernwithlines2} whenever $E$ is a line bundle. In this case the two determinants appearing are identical.  We define the isomorphism of the last factor to be the one compatible with the trivializations $IC_2(E \otimes L) \to \Ocal_S$ and $IC_2(E) \to \Ocal_S$ whenever $E$ is a line bundle. 
The last part follows from the unicity and the fact that there is a natural isomorphism in the case of line bundles. 
\end{proof}

\begin{remark}
Using Elkik's definition of $IC_2$, we find that
\begin{equation}\label{eq:tensorwithL}
    IC_2(E \otimes L) = \langle \det(E \otimes L), \det (E \otimes L) \rangle_{X/S} \otimes \langle \Ocal_{\PBbb(E \otimes L)} (1)\{r+1\} \rangle_{\PBbb(E \otimes L)/S}^{-1}.
\end{equation}
If $p \colon \PBbb(E) \to X$ denotes the natural projection, $\Ocal_{\PBbb(E \otimes L)}(1)$ is identified with $\Ocal_{\PBbb(E)}(1) \otimes p^* L$ under the natural isomorphism $\PBbb(E \otimes L) \simeq \PBbb(E)$. Using this, the natural isomorphism $\det (E\otimes L) \simeq (\det E) \otimes L^{\otimes r}$, the isomorphism $\langle \Ocal(1) \{r\}\rangle_{\PBbb(E)/S} \simeq \det E$ and general properties of Deligne pairings in  Proposition \ref{Prop:generalpropertiesDeligneproduct}, there is a canonical way to write down an isomorphism as in Proposition \ref{prop:tensorlinebundleiso} explicitly. It is possible, but cumbersome, to show that such an isomorphism using symbols coincides up to sign with the abstractly constructed one. 

\end{remark}
\begin{proposition}\label{prop:whitneyproductwithbundle}
 Let $E$ and $F$ be vector bundles of rank $e$ and $f$ on the family of curves $X \to S$. There are unique isomorphisms of line functors, multiplicative in $E$ and $F$,
\begin{equation}\label{eq:polyiso}
IC_2(E \otimes F) \simeq IC_2(E)^{f} \otimes IC_2(F)^{e} \otimes \langle \det E, \det F \rangle^{e\cdot f-1} \otimes \langle \det E, \det E \rangle^{f \choose 2} \otimes \langle \det F, \det F \rangle^{e \choose 2},
\end{equation}
compatible with Lemma \ref{lemma:detprodiso} and such that if $F$ is a line bundle,
 the isomorphism is the isomorphism described in Proposition  \ref{prop:tensorlinebundleiso}.

\end{proposition}
\begin{proof}
Here the multiplicativity is provided by the product in 
\begin{displaymath}
    E \otimes F \mapsto (ef, (\det E)^f \otimes (\det F)^e, \text{right hand side of } \eqref{eq:polyiso}),
\end{displaymath}
and uses the isomorphism in \eqref{eq:isodetline2}. 





\end{proof}
We note that there are a priori two ways to define such an isomorphism as in the proposition. First, one can start by filtering $F$ with a flag, followed by filtering $E$ with a flag. Second, one might apply the isomorphism $E \otimes F \simeq F \otimes E$ and then instead first filter $E$ by a flag, followed by filtering $F$ by a flag. The two constructions clearly coincide when $E$ and $F$ are line bundles. It is  possible to show they coincide in general, reasoning by induction on $e f$, combining the multiplicativity of the construction and a flag argument.
\subsubsection{Duals of vector bundles}

To formulate Proposition \ref{prop:dual} below, notice there is a natural isomorphism, multiplicative and functorial in vector bundles $E_1$ and $E_2$,
    \begin{equation}\label{eq:isodualdelignepairing}
        \langle \det(E_1^\vee), \det(E_2^\vee)\rangle \simeq \langle \det E_1, \det E_2 \rangle.
    \end{equation}
Indeed, under the isomorphism \eqref{eq:detdual} it reduces to an isomorphism of the form $\langle L^\vee, M^\vee \rangle \to \langle L, M \rangle$, which can be defined on the level of symbols by $\langle \ell^\vee, m^\vee \rangle \mapsto \langle \ell, m \rangle$. This also corresponds to the natural isomorphism $\langle L^\vee, M^\vee \rangle \to \langle L^\vee, M \rangle^\vee \to \langle L, M \rangle^{\vee \vee } \to \langle L, M \rangle $. Given a short exact sequence $\varepsilon$ as in \eqref{eq:multdet}, we can thus write down a multiplicativity datum, by exploiting the Whitney isomorphism of the dual exact sequence  $\varepsilon^\vee$, 
\begin{equation}\label{eq:ic2dualadddatum}
    IC_2(E^\vee) \to IC_2({E'}^\vee) \otimes \langle \det E', \det E'' \rangle \otimes IC_2({E''}^\vee).
\end{equation}
More precisely, the functor 
\begin{displaymath}
    E \mapsto (r, \det E, IC_2(E^\vee))
\end{displaymath}
is a multiplicative line functor. We hence have the following proposition, whose proof is analogous to that of Proposition \ref{prop:tensorlinebundleiso}: 

\begin{proposition}\label{prop:dual}
Let $E$ be a vector bundle on a flat family of projective curves $X \to S$. Then, there is a unique multiplicative isomorphism of line functors
    \begin{displaymath}
        IC_2(E^\vee) \simeq IC_2(E),
    \end{displaymath}
which is the identity on the product structures and compatible with the trivializations of both sides when $E$ is a line bundle. 
\end{proposition}
\qed
    
An immediate corollary is the following consequence, which will be utilized in \textsection \ref{subsection:complements-CS-PSL}.
\begin{corollary}\label{cor:IC2-End}
    Let $E$ be a vector bundle. Then there is a natural isomorphism of line functors
    \begin{displaymath}
        IC_2(\End E) \to IC_2(E \otimes E^\vee) \to IC_2(E)^{2e} \otimes \langle \det E, \det E \rangle^{-e}.
    \end{displaymath}
\end{corollary}
\qed 
%

\subsection{Comparison with Deligne's constructions}\label{subsec:Delignecomparison}

The previous subsections developed and studied Elkik's definition of the functor $IC_2$. The original construction due to Deligne is a priori different. In \cite[\textsection 9.7]{Deligne-determinant}, the following line bundle products of determinants of cohomology plays the role of $IC_2$ (see \textsection \ref{subsec:determinantofcoh} and equation \eqref{def:determinantcoh}): 
\begin{equation}\label{def:DelIC2}
   I_{X/S}C^2(E) = \lambda(-\left(E - r - \left(\det(E) - 1 \right)\right)) \simeq \lambda(E)^{-1} \otimes \lambda(\Ocal_X)^{r-1} \otimes \lambda(\det E).
\end{equation}
Since the determinant of the cohomology is a graded line bundle, the expression on the right hand side can only be used as a definition up to sign. This is related to the remarks in the beginning of \textsection \ref{subsec:appliadddatum}. To make it precise, we refer the reader to the discussion surrounding \cite[(9.7.4), p. 168]{Deligne-determinant}. This only plays a minor role in the text, but it will be nevertheless applied in Proposition \ref{prop:Deligne-iso-descends} and Theorem \ref{thm:variant-DRR-flat}.

Properties analogous to those of Theorem \ref{thm:ic2properties} are established in \cite[\textsection 9.1]{Deligne-determinant}. In \cite[Proposition 9.4]{Deligne-determinant} a unicity statement for $IC_2$ is provided. The statement itself relies on a formulation in terms of virtual categories of vector bundles, but reduces to the statement that the line functor $IC_2$ is uniquely determined by the properties:
\begin{enumerate}
    \item \label{ic2deligne1} $IC_2(L)$ is canonically trivial \label{ic2deligne2} for line bundles $L$;
    \item \label{ic2deligne3} $IC_2$ satisfies a Whitney isomorphism for short exact sequences $\varepsilon$;
    \item \label{ic2deligne4} $IC_2$ is commutative up to a controlled sign; 
    \item a normalization criteria.
\end{enumerate}

\begin{theorem}\label{thm:comparisonElkDel}
Elkik's $IC_2$ enjoys the above properties and is hence uniquely isomorphic to that of Deligne.
\end{theorem}
\begin{proof}
Properties \eqref{ic2deligne1}, \eqref{ic2deligne3} are already proven in Theorem \ref{thm:ic2properties}. The property \eqref{ic2deligne4} is part of the conclusion of Proposition \ref{prop:ic2-commutativity} above. The last normalization criteria is proven in the context of Franke's $IC_2$ in \cite[\textsection 3.4]{FrankeChern}. A further inspection of the proof shows that it reduces to analogues of the three first properties, together with Proposition \ref{prop:tensorlinebundleiso} and the main statement of Proposition \ref{prop:ic2-commutativity}, which describes explicitly the isomorphism $IC_2(L \oplus M) \simeq \langle L, M \rangle$. 

\end{proof}

\begin{remark} The version of the splitting principle in Proposition \ref{prop:useful-version-splitting} already provides a canonical isomorphism, by virtue of  \eqref{ic2deligne1} above. However, this isomorphism does not automatically come with a characterization.  
\end{remark}
If $L$ and $M$ are line bundles, the Whitney isomorphism provides an isomorphism $IC_2(L \oplus M) \simeq \langle L, M\rangle$. Combining it with $\det(L \oplus M) \simeq L \otimes M$ and  \eqref{def:DelIC2}, we find the following corollary:
\begin{corollary}\label{cor:isodeligneic2pairing}
Let $L$ and $M$ be line bundles on $X$. Then there is a canonical identification 
\begin{displaymath}
    \langle L, M \rangle =\lambda(\Ocal_X) \otimes  \lambda(L \oplus M)^{-1} \otimes \lambda(L \otimes M).
\end{displaymath}
\end{corollary}\qed

We remark that in the two approaches of Deligne \cite[\textsection 7]{Deligne-determinant} and Ducrot \cite{Ducrot}, this is essentially the definition of the left hand side. We include this explicit corollary because we want to be able to reference it later.

Suppose now that $E$ is a vector bundle on $X$. One of the main results of \cite{Deligne-determinant} asserts the existence of an isomorphism of line bundles representing a Grothendieck--Riemann--Roch equality for $E$. We give a  proof based on our splitting principles for line functors, reducing to the simpler case of line bundles.  

\begin{theorem}[Deligne--Riemann--Roch isomorphism]\label{DRR:iso-general}
Suppose that $X \to S$ is a family of smooth projective curves with relative canonical sheaf $\omega_{X/S}$. Then there is a canonical, up to sign, multiplicative isomorphism of line functors,  
\begin{equation}\label{DRR:iso-general-first}
    \mathsf{DRR}(X/S,E)\colon\lambda(E)^{12}\simeq\langle\omega_{X/S},\omega_{X/S}\rangle^{\rk E}\otimes \langle\det E,\det E \otimes \omega_{X/S}^{-1}\rangle^{6}\otimes IC_{2}(E)^{-12}.
\end{equation}
\end{theorem}

\begin{proof}
The left hand side is a multiplicative line functor by the discussion in \textsection \ref{subsec:determinantofcoh}. For the existence part of the theorem, by the splitting principle in Proposition  \ref{prop:useful-version-splitting}, it is enough to equip the right hand side with the structure of a multiplicative line functor, and prove the isomorphism whenever $E$ is a line bundle. 

For the multiplicative structure on the right hand side, by Theorem \ref{thm:flagfiltration} it is uniquely determined by the multiplicative structure for bundles of the form $L \oplus E''$, where $L$ is a line bundle. In this case, it is straightforward to equip the line functor
\begin{displaymath}
    \langle\omega_{X/S},\omega_{X/S}\rangle^{\rk E}\otimes \langle\det E,\det E \otimes \omega_{X/S}^{-1}\rangle^{6}\otimes IC_{2}(E)^{-12}
\end{displaymath}
with a multiplicative structure with respect to the usual tensor product of line bundles, using the multiplicative structures on the Deligne pairings, $\det$ and the Whitney isomorphism. 

For line bundles, we rely on the constructions of \cite[Construction 7.5]{Deligne-determinant}, which asserts that for a line bundle $L$, there is an isomorphism of line bundles, determined up to sign,
\begin{displaymath}
    \lambda(L)^{2} \otimes \lambda(\Ocal_X)^{-2}\to \langle L, L \otimes \omega^{\vee}_{X/S} \rangle.
\end{displaymath}
The sixth power is thus a canonical isomorphism. For convenience, we recall that the proof uses Corollary \ref{cor:isodeligneic2pairing} and rewrites one of the two copies $\lambda(L) \otimes \lambda(\Ocal_X)^{-1}$ using the  isomorphism induced by Grothendieck--Serre duality $\lambda(L) \simeq \lambda(L \otimes \omega^{\vee}_{X/S})$. This is finally combined with Mumford's isomorphism $\lambda(\Ocal_{X})^{12}\simeq \langle \omega_{X/S}, \omega_{X/S} \rangle$, which is canonical up to sign, cf. \cite[Section 5]{Mumford:stability}. We get an isomorphism 
\begin{displaymath}
    \lambda(L)^{ 12} \to \langle \omega_{X/S}, \omega_{X/S} \rangle \otimes \langle L, L^{\vee} \otimes \omega_{X/S} \rangle^{ 6}.
\end{displaymath}
This is \eqref{DRR:iso-general-first} for line bundles, and concludes the proof.
\end{proof}

\begin{remark}
For the purposes of this article, the sign ambiguity in the Deligne--Riemann--Roch isomorphism is actually irrelevant, and we may thus  ignore it. Nonetheless, let us mention that in \cite[Appendice 3]{Deligne:letter-Quillen}, Deligne describes a procedure to remove this ambiguity.
\end{remark}

\subsection{Complements in the complex analytic setting}\label{subsec:Analytification}
The constructions in the previous subsections can be extended to the complex analytic category. For lack of reference on the theory of intersection bundles in this setting, and for future reference, we provide the argument for this. For the general facts in complex analytic geometry utilized below, we follow \cite{Banica, Banica-2, Fischer} and the lectures by Grothendieck and Houzel in the S\'eminaire Cartan \cite{Cartan-13}.

\subsubsection{Deligne pairings} Let $f\colon X\to S$ be a flat, locally projective morphism of complex analytic spaces, with Cohen--Macaulay fibers of  pure relative dimension $n\geq 0$. Given line bundles $L_{0},\ldots, L_{n}$ on $X$, we wish to construct the Deligne pairing $\langle L_{0},\ldots,L_{n}\rangle$ by mimicking the procedure described in \textsection \ref{subsec:delignepairings}. That this is possible relies upon the following lemma.

\begin{lemma}\label{lemma:general-sections}
Let $L$ be a line bundle on $X$. Then, locally with respect to $S$, there exist relatively ample effective Cartier divisors $D$ and $E$ in $X$,  flat over $S$, such that $L\simeq\Ocal(D-E)$. Furthermore, if $Y_{1},\ldots, Y_{r}$ is a finite collection of closed analytic subspaces of $X$, flat over $S$, then we can also suppose that $D\cap Y_{j}$ and $E\cap Y_{j}$ are effective Cartier divisors in $Y_{j}$, flat over $S$, for all $j$.
\end{lemma}
\begin{proof}
We begin with the first part of the statement, and then indicate how to extend it to cover the second one. Let $s\in S$ be any point, and $\Ical$ the coherent sheaf of ideals defining the fiber $X_{s}$ in $X$. Restricting to a suitable neighborhood of $s$, we may suppose that there exists a relatively very ample line bundle $A$ such that $L\otimes A$ is relatively very ample too, with $R^{1}f_{\ast}(\Ical\otimes A)=0$ and similarly for $L\otimes A$. We reduce to show that $A\simeq\Ocal(D)$, with $D$ as in the statement. The case of $L\otimes A$ is analogous.

The fiber $X_{s}$ is algebraizable, since it is projective. The line bundle $A_{\mid X_{s}}$ is then algebraizable too, by the GAGA theorems. We abusively use the same notation for the corresponding algebraic objects. We choose a global section $\ell$ of $A_{\mid X_{s}}$, whose zero locus avoids the associated points of $X_{s}$, as a scheme. This is possible, since the set of associated points is finite and $A_{\mid X_{s}}$ is very ample. Here we are implicitly using that the base field $\CBbb$ is infinite.\footnote{In the general algebraic setting of \textsection\ref{sec:generalcasedeligneproduct}, $A$ needs to be replaced by a large positive power. See comment \eqref{comment:existence-sections} in Remark \ref{rmk:construction-Deligne-pairing}.} Thus, the section $\ell$ induces an injective morphism $\Ocal_{X_{s}}\hookrightarrow A_{\mid X_{s}}$, in the algebraic category. It remains injective after analytification, because for a scheme $T$ of finite type over $\CBbb$ and $t\in T(\CBbb)$, the morphism of local rings $\Ocal_{T,t}\to\Ocal_{T^{\an},t}$ is flat. By the vanishing $R^{1}f_{\ast}(\Ical\otimes A)=0$, after possibly restricting to a Stein neighborhood of $s$, we can lift the section $\ell$ to a global section $\widetilde{\ell}$ of $A$. By \cite[Theorem 22.5]{Matsumura}, we see that the morphism $\widetilde{\ell}\colon\Ocal_{X}\to A$ is injective at the points of the fiber $X_{s}$, and  $A/\widetilde{\ell}\Ocal_{X}$ is flat over $S$ at the points of the fiber $X_{s}$. The kernel of $\widetilde{\ell}$ is a coherent sheaf, and its support is a closed analytic subset of $X$ disjoint from $X_{s}$. Because $f$ is proper, we infer that after possibly restricting $S$, we can suppose that multiplication by $\widetilde{\ell}$ is an injective map. Besides, by Frisch's theorem \cite{Frisch}, the locus of points in $X$ where $A/\widetilde{\ell}\Ocal_{X}$ is flat over $S$ is open. It contains the whole fiber $X_{s}$. Therefore, shrinking $S$ if necessary, we can suppose that $A/\widetilde{\ell}\Ocal_{X}$ is flat over $S$. We have thus proven the existence of a regular section $\widetilde{\ell}$, which defines a Cartier divisor $D$ with the desired property.

Suppose now that $Y_{1},\ldots, Y_{r}$ is a finite collection of closed subspaces of $X$, which are flat over $S$. In the previous paragraph, we can further impose that the zero locus of $\ell$ avoids the associated points of $Y_{j,s}$, as schemes. Then, the very same reasoning above shows that, after shrinking $S$ if need be, the restrictions $\widetilde{\ell}_{\mid Y_{j}}$ define $D\cap Y_{j}$ with the expected property.
\end{proof}

The construction of the Deligne pairing $\langle L_{0},\ldots,L_{n}\rangle$ then proceeds as in \textsection\ref{sec:generalcasedeligneproduct}, by prescribing generators and relations, locally with respect to the base. The lemma ensures the existence of generators, given by symbols $\langle\ell_{0},\ldots,\ell_{n}\rangle$, with $\ell_{0},\ldots\ell_{n}$ in general position. For the relations to make sense, we first need to observe that if $D\to S$ is a finite flat morphism of complex analytic spaces, then $\Ocal_{D}$ is a locally free $\Ocal_{S}$-algebra of finite rank, as in the algebraic case, and therefore the norm functor $N_{D/S}$ can be defined. Finally, we need to check that the construction is well-defined. By this we mean that the compatibility condition between various defining relations, analogous to the Weil reciprocity law, is fulfilled (see comment \eqref{comment:Weil} in Remark \ref{rmk:construction-Deligne-pairing}). The next lemma addresses this point, by reduction to the algebraic case.

\begin{lemma}\label{lemma:Weil-analytic}
The construction of the Deligne pairing $\langle L_{0},\ldots,L_{n}\rangle$ by generators and relations in the analytic setting, is well-defined.
\end{lemma}
\begin{proof}
Suppose that we are given a defining symbol $\langle\ell_{0},\ldots,\ell_{n}\rangle$ as above, thus satisfying the analogue of \emph{(G)} (b) in \textsection\ref{sec:generalcasedeligneproduct}. Write the Cartier divisor associated to $\ell_{i}$ as $D_{i}$. Suppose that $f$ and $g$ are meromorphic functions, such that for some $j\neq k$, changing $\ell_{j}$ into $f\ell_{j}$ and $\ell_{k}$ into $g\ell_{k}$ provides a new symbol. Let $Z$ be the finite flat cycle over $S$, defined by the intersection of the $D_{i}$, excepting  $D_{j}$ and $D_{k}$. We are led to show the reciprocity law
\begin{equation}\label{eq:general-Weil}
    N_{Z\cap\Div(g)/S}(f)=N_{Z\cap\Div(f)/S}(g)\quad\text{in}\quad\Ocal_{S}.
\end{equation}
It is enough to prove this relationship after localizing at an arbitrary $s\in S$. Because $\Ocal_{S,s}$ is noetherian, we have an injection into the completion $\Ocal_{S,s}\hookrightarrow\widehat{\Ocal}_{S,s}$. Therefore, we just need to establish the reciprocity equality after completing at $s$. 

For $m\geq 0$, let $S_{m}$ be the $m$-th infinitesimal neighborhood of $s$ in $S$. It is a finite analytic space, which can be algebraized into a finite scheme over $\CBbb$. We perform the base change of $X\to S$ by $S_{m}\to S$, and we obtain a flat projective morphism $X_{m}\to S_{m}$. Since $S_{m}$ is finite (hence projective) and the morphism $X_{m}\to S_{m}$ is projective, we can invoke Chow's theorem and conclude that $X_{m}$ and the morphism can be algebraized. By the GAGA theorems, the base change of all the objects involved in the Deligne pairing and in the symbols above, can be algebraized as well. The base change of \eqref{eq:general-Weil} by $S_{m}$ thus becomes an analogous relationship in an algebraic situation, already covered by the work of Elkik \cite{Elkikfib}. Since the compatibility condition is satisfied in this context, we infer that \eqref{eq:general-Weil} holds after pullback to $S_{m}$. We conclude by passing to the limit.
\end{proof}

We have thus justified that the method of \textsection\ref{sec:generalcasedeligneproduct} carries over to the analytic setting. 

\begin{proposition}\label{prop:properties-Deligne-analytic}
\begin{enumerate}
    \item The construction of Deligne pairings extends to flat, locally projective families of complex analytic spaces, with Cohen--Macaulay fibers of pure dimension $n$, in a way which is compatible with base change and the analytification functor. Precisely, if $X \to S$ is such an algebraic family, with line bundles $L_{0}, \ldots, L_n$, then there is a canonical isomorphism
\begin{displaymath}
    \langle L_{0}, \ldots, L_n \rangle^{\an} \simeq \langle L_{0}^{\an}, \ldots, L_n^{\an} \rangle.
\end{displaymath}
    \item The complex analytic counterpart of Proposition \ref{Prop:generalpropertiesDeligneproduct} holds, in a way compatible with the analytification functor.
\end{enumerate}
\end{proposition}
\begin{proof}
The first point has essentially been addressed. The compatiblity with base change and the analytification functor can be checked at the level of symbols. For the second point, all the items in Proposition \ref{Prop:generalpropertiesDeligneproduct} but $(4^{\prime})$ can be addressed by reasoning with symbols. For $(4^{\prime})$, below we perform a reduction to the algebraic case. We rely on the theory of Hilbert and Hom schemes, for which we refer to \cite{Grothendieck-Hilbert}. We use that the analytifications of those represent the appropriate functors in the complex analytic category. This is proven for Quot schemes by Simpson in \cite[Proposition 5.3]{Simpson:moduli-1}. The claim follows for Hilbert and Hom schemes, by the very construction from Quot schemes.

We suppose first that we are given relatively very ample line bundles $L_{0},\ldots,L_{n}$ on $X$, inducing projective factorizations $X\hookrightarrow\PBbb^{N_{i}}_{S}\to S$. To every $L_{i}$, we can associate a projective Hilbert scheme $\Hcal_{i}=\mathrm{Hilb }_{\PBbb^{N_{i}}}^{P_{i}}$, where $P_{i}$ is the Hilbert polynomial of $L_{i}$. Hence, for each $(X\to S, L_{i})$ we have a classifying map $ S\to\Hcal_{i}^{\an}$ such that, if $\Xcal_{i}\to\Hcal_{i}$ is the universal family, with relatively very ample line bundle $\Ocal_{\Xcal_{i}}(1)$, then $X\simeq S\times_{\Hcal_{i}^{\an}}\Xcal_{i}^{\an}$ and $\Ocal_{\Xcal_{i}}(1)^{\an}$ pulls back to $L_{i}$. Now form $\Hcal=\Hcal_{0}\times\ldots\times\Hcal_{n}$, and base change the universal families and polarizations to $\Hcal$. We maintain the notation for those. The universal families are not necessarily isomorphic to each other. To fix this, we introduce the schemes of isomorphisms $\mathrm{Isom}_{\Hcal}^{Q_{i}}(\Xcal_{0},\Xcal_{i})$ for every $i\geq 1$, where the $Q_{i}$ are the Hilbert polynomials of the $L_{0}\otimes L_{i}$. See \cite[\textsection 4.c]{Grothendieck-Hilbert}. These are quasi-projective schemes over $\Hcal$. In the sequel, we need to keep in mind that associated to the identity map $X\to X$, where the source (resp. target) is endowed with the polarization $L_{0}$ (resp. $L_{i}$), there is a classifying map $S\to\mathrm{Isom}_{\Hcal}^{Q_{i}}(\Xcal_{0},\Xcal_{i})^{\an}$. Taking the fiber product over $\Hcal$ of all these $\mathrm{Isom}$ schemes, we obtain a quasi-projective scheme $\Scal\to\Hcal$, where the base changes of the families $\Xcal_{i}$ can all be identified to the base change of $\Xcal_{0}$, now denoted $\Xcal$. Accordingly, we pull back the polarizations $\Ocal_{\Xcal_{i}}(1)$ to $\Xcal$, and we denote the resulting bundles by $\Lcal_{i}$. Now we have a classifying map $\varphi\colon S\to\Scal^{\an}$, such that the base change of $\Xcal$ to $S$ is isomorphic to $X$, and the $\Lcal_{i}^{\an}$ pull back to the $L_{i}$. Notice that the universal family $\Xcal$ over $\Scal$ is not necessarily Cohen--Macaulay, with fibers of pure dimension $n$. However, for flat proper morphisms, being Cohen--Macaulay is an open condition on the base \cite[Th\'eor\`eme 12.2.1]{EGAIV3}. Hence, replacing $\Scal$ by a suitable open subscheme $\Ucal$, we obtain a restricted family $\Xcal\to\Ucal$ with Cohen--Macaulay fibers. These might not be of pure dimension $n$, but by \cite[\href{https://stacks.math.columbia.edu/tag/02NM}{02NM}]{stacks-project}, there exists a maximal closed and open subscheme $\Xcal^{(n)}\subseteq\Xcal$, whose fibers over $\Ucal$ are of pure dimension $n$. Thus, possibly replacing $\Xcal$ by $\Xcal^{(n)}$, we may suppose that the fibers of the universal family have pure dimension $n$. The classifying morphism $\varphi$ has image in $\Ucal^{\an}$, and the base changes of $\Xcal^{\an}$ and $\Lcal_{i}^{\an}$ are isomorphic to $X$ and $L_{i}$. The Deligne pairing $\langle \Lcal_{0}, \ldots, \Lcal_n \rangle_{\Xcal/\Ucal}$ is now defined. By the first point of the proposition, we see that there is a canonical isomorphism
\begin{equation}\label{eq:isom-Deligne-analyt-Hilb}
    \varphi^{\ast}\left(\langle \Lcal_{0}, \ldots, \Lcal_n \rangle^{\an}_{\Xcal/\Ucal}\right) \simeq \langle L_{0}, \ldots, L_n \rangle_{X/S}.
\end{equation}

Secondly, we suppose that $S$ is a general base and we are given line bundles $L_{0},\ldots,L_{n}$. Fix $s\in S$. Restricting to a neighborhood of $s$, we can suppose that $L_{i}=A_{i}\otimes B_{i}^{-1}$, for relatively very ample $A_{i}$ and $B_{i}$ inducing projective factorizations. Proceeding as in the previous case with all the $A_{i}$ and $B_{i}$, we construct an algebraic model situation $\Xcal\to\Ucal$, $\Acal_{i}$, $\Bcal_{i}$, from which $X\to S$, $A_{i}$, $B_{i}$ are deduced by base change by some classifying map $\varphi\colon S\to\Ucal^{\an}$. By the additivity of Deligne pairings, if we set $\Lcal_{i}=\Acal_{i}\otimes\Bcal_{i}^{-1}$, we obtain the counterpart of \eqref{eq:isom-Deligne-analyt-Hilb}, valid in a neighborhood of $s$.

Finally, let us apply the previous method to prove the analogue of $(4^{\prime})$ in Proposition \ref{Prop:generalpropertiesDeligneproduct}. We focus on the subcase (a), and adopt the notation therein. The other subcases are dealt with similarly. Now we have two flat, locally projective families $X\to S$ and $X^{\prime}\to S$, whose fibers are Cohen--Macaulay of pure dimensions $n$ and $n+n^{\prime}$, respectively, together with a flat morphism $X^{\prime}\to X$ and line bundles $L_{0},\ldots L_{n-1}$ on $X$ and $M_{n},\ldots, M_{n+n^{\prime}}$ on $X^{\prime}$. We reason locally on $S$ and proceed as in the previous paragraphs, in order to construct respective algebraic model situations $\Xcal\to\Ucal$, $\Lcal_{0},\ldots,\Lcal_{n-1}$, and $\Xcal^{\prime}\to\Ucal^{\prime}$, $\Mcal_{n},\ldots,\Mcal_{n+n^{\prime}}$. Implicitly, this involves auxiliary decompositions of the line bundles as differences of relatively very ample line bundles. We form the product $\Tcal=\Ucal\times\Ucal^{\prime}$, and base change to $\Tcal$ the families $\Xcal$, $\Xcal^{\prime}$ and the line bundles. We maintain the notation for the base changed objects. Then we introduce a scheme $\mathrm{Hom}_{\Tcal}^{P}(\Xcal^{\prime},\Xcal)$, where $P$ is the Hilbert polynomial of a suitable line bundle of the form $q^{\ast}A\otimes B$, with $A$ (resp. $B$) relatively very ample on $X$ (resp. $X^{\prime}$). Let $\Vcal$ be the open subscheme of the Hom scheme where the universal morphism $\Xcal^{\prime}\to\Xcal$ is flat. This is indeed an open subscheme, because the universal families are flat and proper, and by the criterion \cite[Th\'eor\`eme 11.3.10]{EGAIV3}. We have a classifying morphism $\varphi\colon S\to\Vcal^{\an}$, from which our initial geometric datum is deduced by base change. By Proposition \ref{Prop:generalpropertiesDeligneproduct}, $(4^{\prime})$ (a), applied to the algebraic model situation, plus analytification and base change by $\varphi$, we find an isomorphism
\begin{equation}\label{eq:projection-formula-analytic}
        \langle q^* L_{0}, \ldots, q^* L_{n-1}, M_{n}, \ldots, M_{n+n^{\prime}} \rangle_{X'/S} \simeq \langle L_{0}, \ldots, L_{n-1}, \langle M_{n}, \ldots, M_{n+n^{\prime}}\rangle_{X'/X} \rangle_{X/S}.
\end{equation}
This isomorphism is valid in a neighborhood of $s$. It could a priori depend on the auxiliary choices of polarizations. Two such isomorphisms differ by an invertible function $h\in\Ocal_{S,s}^{\times}$. We need to show $h=1$. It is enough to check this in the completion $\widehat{\Ocal}_{S,s}$. We then proceed as in the proof of Lemma \ref{lemma:Weil-analytic}. With the same notation as in \emph{loc. cit.}, we reduce to an algebraic setting by restricting to the infinitesimal neighborhoods $S_{m}$ of $s$. This is realized by pulling back to $S_{m}$ the model situations, associated to the  choices of polarizations. For instance, for the choice giving rise to $\varphi\colon S\to\Vcal^{\an}$ as above, the restriction $\varphi_{\mid {S_{m}}}$ can be algebraized into  $\varphi_{m}\colon S_{m}\to\Vcal$, by a similar argument as in Lemma \ref{lemma:Weil-analytic}. In the algebraic case, the independence of the choices is known, by the very construction of the isomorphism \cite[Proposition IV.2.2 (b)]{Elkikfib}. Hence, $h_{\mid {S_{m}}}=1$ for all $m\geq 1$, as was to be shown. By the same token, the locally constructed isomorphisms \eqref{eq:projection-formula-analytic} patch together into a global one. We leave as an exercise to verify that the resulting isomorphism is compatible with the analytification functor and base change.
\end{proof}

Similarly, one can transpose to the analytic category Ducrot's approach to Deligne pairings \cite{Ducrot}, under our running assumptions on the morphisms, in particular Cohen--Macaulay. We need two ingredients. First of all, the necessary theory on the complex analytic Knudsen--Mumford determinant of the cohomology is summarized in Bismut--Bost \cite[Section 4 (a)]{Bismut-Bost}. We notice that \emph{loc. cit.} restricts to non-singular parameter spaces. However, this constraint can be removed for the determinant of the cohomology of a relatively flat coherent sheaf, defined on the total space of a proper flat morphism of complex analytic spaces. The key point is the observation of Kiehl--Verdier \cite[Bemerkung 4.4.1]{Kiehl-Verdier} on the local existence of finite complexes of vector bundles which universally (\emph{i.e.} after any base change) compute the derived direct image of a coherent sheaf. This holds over any complex analytic space. Secondly, in the analytic category, the construction of Ducrot's cube structure can be reduced to the algebraic category as in the proof of Proposition \ref{prop:properties-Deligne-analytic}. We notice that this method of reduction to the algebraic setting fails for flat proper morphisms with fibers of pure dimension $n$, which is the generality allowed in \cite{Ducrot}. The reason is that for a flat proper morphism, such as the universal family over a Hilbert scheme, the locus where the morphism has fibers of pure dimension $n$ is in general just constructible \cite[Th\'eor\`eme 13.1.3]{EGAIV3}, but otherwise has no natural scheme structure. This is why we restrict to Cohen--Macaulay morphisms.

\subsubsection{Intersection bundles in relative dimension one}
In the setting of families of compact Riemann surfaces, the condition of being locally projective is automatic. Outside of genus one, the relative cotangent bundle can be used to provide a relatively ample line bundle. For genus one families we can always find sections of $X\to S$, locally on $S$. The line bundles associated to the Cartier divisors who are in the image of these sections produce relatively ample line bundles. The following can be derived:

\begin{proposition}\label{prop:propersubreldimone}
Elkik's intersection bundles and their properties are applicable in the setting of proper holomorphic submersions of complex analytic spaces, with fibers of  dimension one. Moreover, the construction and properties of these bundles are compatible with the analytification functor. 
\end{proposition}
\qed

Finally, it is formal to extend Deligne's approach to the $IC_{2}$ bundle in terms of the determinant of the cohomology, as described in \textsection\ref{subsec:Delignecomparison}. The same arguments as in \textsection\ref{subsec:Delignecomparison} entail the compatibility with Elkik's approach, and the analytic version of Deligne's isomorphism:

\begin{proposition}
The Deligne--Riemann--Roch isomorphism applies in the setting of proper holomorphic submersions of complex analytic spaces, with fibers of dimension one. It is compatible with the analytification functor.
\end{proposition}
\qed

\section{Intersection metrics}\label{sec:met-int-bund}

In this section we recall and elaborate on the construction of hermitian metrics on the Deligne pairings and $IC_{2}$, after Deligne \cite{Deligne-determinant}, Elkik \cite{Elkik2} and Gillet--Soul\'e \cite{GS:ACC1}. We refer to any of those as \emph{intersection metrics}. We discuss in further detail the case of flat hermitian vector bundles.  

\subsection{Bott--Chern theory} 

\subsubsection{Chern connections and Chern--Weil theory} Let $X$ be a complex manifold, and $\ov{E}=(E,h)$ a hermitian holomorphic vector bundle\footnote{Indicating the choice of a hermitian metric by a bar is customary notation in Arakelov geometry. This notation will only be used in this section.} on $X$, of rank $r$. Recall that the Chern connection of $\ov{E}$ is the unique hermitian connection which is compatible with the holomorphic structure of $E$. We will usually refer to Chern connections with the notation $\nabla^{\hmini}$, or $\nabla^{\chmini}$ if the metric is clear from the context.

Let $F\in A^{1,1}(X,\End E)$ be the curvature of $\nabla^{\hmini}$. Chern--Weil theory produces de Rham representatives of the characteristic classes of $E$, extracted from the symmetric polynomials in $F$. Precisely, the total Chern class of $E$, $c(E)$, can be lifted to the closed differential form 
\begin{displaymath}
    c(\ov{E})=\det\left(1+\frac{i}{2\pi}F\right).
\end{displaymath}
This is a sum of real forms of type $(p,p)$, for $p=0,\ldots, r$. The component of type $(p,p)$ represents the $p$-th Chern class $c_{p}(E)$ and is denoted by $c_{p}(\ov{E})$. More generally, any complex formal power series in the characteristic classes can be lifted to the level of differential forms, by formally replacing $c_{p}(E)$ by $c_{p}(\ov{E})$. These are called Chern--Weil representatives, or forms associated to $\nabla^{\hmini}$ or $\ov{E}$. 

\subsubsection{Examples}
In this article we will mostly need first and second Chern forms. They are given by
\begin{displaymath}
    c_{1}(\ov{E})=c_{1}(\det \ov{E})=\frac{i}{2\pi}\tr F
\end{displaymath}
and 
\begin{equation}\label{eq:formula-c2}
        c_{2}(\ov{E})
        =\frac{1}{8\pi^{2}}\left(\tr\left(F^{2}\right)-(\tr F)^{2}\right).
\end{equation}


Chern--Weil theory applies more generally to $\Ccal^{\infty}$ connections on vector bundles. For a vector bundle with connection $(E,\nabla)$ we will denote the associated Chern--Weil forms by $c(E,\nabla)$, $c_{p}(E,\nabla)$, etc. They are defined by the same expressions as above, in terms of the curvature of $\nabla$. In this article we will deal with compatible connections, non-necessarily hermitian. In this generality, the curvature has components of types $(1,1)$ and $(2,0)$. 

\subsubsection{Bott--Chern classes: axioms} \label{subsec:Bott-Chern}
For Chern--Weil forms of holomorphic hermitian vector bundles, Bott--Chern theory measures the dependence on choices. We adopt the axiomatic approach of Gillet--Soul\'e \cite[Section 1]{GS:ACC1}, Bismut--Gillet--Soul\'e \cite[Section f)]{BGS1} and Burgos--Litcanu \cite[Section 2]{Burgos-Litcanu}.

Let $\varphi$ be a formal power series in the Chern classes. For every complex manifold $X$ and holomorphic vector bundle $E$ on $X$, there is an associated de Rham class $\varphi(E)\in H^{\bullet}(X,\CBbb)$. This correspondence is functorial under pullback: given a morphism $f\colon Y\to X$, we have $\varphi(f^{\ast}E)=f^{\ast}\varphi(E)$. Finally, for every exact sequence of holomorphic vector bundles on $X$
\begin{equation}\label{eq:gg-1}
    \varepsilon\colon 0\to E' \to E \to E''\to 0,
\end{equation}
there is a relationship in $H^{\bullet}(X,\CBbb)$
\begin{displaymath}
    \varphi(E)=\varphi(E'\oplus E'').
\end{displaymath}

 Given smooth hermitian metrics $h, h', h''$ on $E, E', E''$, denote  $\ov{E}' = (E', h'), \ov{E} = (E, h), \ov{E}'' = (E'', h'')$. Given an exact sequence \eqref{eq:gg-1} , the combination of  Chern--Weil representatives
\begin{equation}\label{eq:gg-2}
    \varphi\left(\ov{E}'\overset{\perp}{\oplus} \ov{E}'' \right)-\varphi(\ov{E})
\end{equation}
is exact. Bott--Chern theory actually provides $dd^{c}=\frac{i}{2\pi}\partial\ov{\partial}$ primitives. For an exact sequence $\varepsilon$ as in \eqref{eq:gg-1}, we  indicate the choice of arbitrary hermitian metrics on its constituents by writing $\ov{\varepsilon}$. We refer to it as a metrized exact sequence. There exists a unique assignment
\begin{equation}\label{eq:gg-3}
    \ov{\varepsilon}\text{ on }X \mapsto \widetilde{\varphi}(\ov{\varepsilon})\in \bigoplus_{p}A^{p,p}(X)\Big /(\Imag\partial+\Imag\ov{\partial})
\end{equation}
satisfying the following properties:
\begin{enumerate}
    \item[\emph{(BC1)}] (\emph{Differential equation}) $dd^{c}\widetilde{\varphi}(\ov{\varepsilon})=\varphi(\ov{E}'\overset{\perp}{\oplus} \ov{E}'')-\varphi(\ov{E})$.
    \item[\emph{(BC2)}] (\emph{Functoriality}) The formation of $\widetilde{\varphi}$ commutes with pullback.
    \item[\emph{(BC3)}] (\emph{Normalization}) If $\ov{\varepsilon}$ is both holomorphically and metrically split, then $\widetilde{\varphi}(\ov{\varepsilon})=0$.
\end{enumerate}
We call $\widetilde{\varphi}(\ov{\varepsilon})$ the \emph{Bott--Chern class} associated to $\varphi$ and $\ov{\varepsilon}$. In the literature it is sometimes referred to as the Bott--Chern secondary form.

\subsubsection{Bott--Chern classes: construction} 
Recall  the construction of transgression exact sequences from \textsection\ref{sec:splitting2}. For an exact sequence $\varepsilon$ as in \eqref{eq:gg-1}, denote by $\widetilde{\varepsilon}$ the associated transgression exact sequence. See equation \eqref{transgression}. Indicate by $\ov{\widetilde{\varepsilon}}$ a choice of smooth hermitian metrics on the constituents of $\widetilde{\varepsilon}$, in such a way that:
\begin{enumerate}
    \item $\ov{\widetilde{\varepsilon}}_{\mid X\times\lbrace 0\rbrace}$ is isometric to $\ov{\varepsilon}$.
    \item $\ov{\widetilde{\varepsilon}}_{\mid X\times\lbrace\infty\rbrace}$ is the standard metrically split exact sequence
    \begin{displaymath}
        0\to \ov{E}'\to \ov{E}'\overset{\perp}{\oplus} \ov{E}'' \to \ov{E}''\to 0.
    \end{displaymath}
\end{enumerate}
We set
\begin{equation}\label{eq:gg-4}
    \widetilde{\varphi}(\ov{{\varepsilon}})=\int_{\PBbb^{1}}\log|t|^{-2}\varphi (\ov{\widetilde{E}})\quad\mod\quad\Imag\partial+\Imag\ov{\partial},
\end{equation}
where $t$ is the usual coordinate on $\CBbb\subset\PBbb^{1}$. Here, $\varphi(\ov{\widetilde{E}})$ is the Chern--Weil representative of the middle term of the short exact sequence $\ov{\widetilde{\varepsilon}}$. It can be checked that $\widetilde{\varphi}$ thus defined fulfills \emph{(BC1)}--\emph{(BC3)}. For later use, we remark that \emph{(BC1)} follows from the equation of currents $dd^{c}[\log|t|^{-2}]=\delta_{\infty}-\delta_{0}$. We also notice that the freedom in the choice of metrics on $\ov{\widetilde{\varepsilon}}$ is reflected in the ambiguity $\Imag\partial+\Imag\ov{\partial}$ in \eqref{eq:gg-3}. 

All the above can be specialized to describe the dependence of Chern--Weil forms on the metrics. Let $E$ be a holomorphic vector bundle on $X$ and $h$, $h^{\prime}$ two smooth hermitian metrics on $E$. Consider the short exact sequence
\begin{displaymath}
    \ov{\varepsilon}\colon 0\to (E,h)\to (E,h^{\prime})\to 0 \to 0.
\end{displaymath}
The associated Bott--Chern  class $\widetilde{\varphi}(\ov{\varepsilon})$ will be denoted simply by $\widetilde{\varphi}(E,h,h^{\prime})$, or $\widetilde{\varphi}(\ov{E},\ov{E}^{\prime})$. It thus satisfies $dd^{c}\widetilde{\varphi}(E,h,h^{\prime})=\varphi(E,h)-\varphi(E,h^{\prime})$. It follows from the construction that 
\begin{equation}\label{eq:gg-7}
        \widetilde{\varphi}(E,h,h^{\prime})=\int_{\PBbb^{1}}\log|t|^{-2}\varphi(p^{\ast}E,\widetilde{h}),
\end{equation}
where $p\colon X\times\PBbb^{1}\to X$ is the projection map, and $\widetilde{h}$ is any smooth hermitian metric on $p^{\ast}E$ interpolating between $h$ and $h^{\prime}$, meaning $\widetilde{h}_{\mid X\times\lbrace\infty\rbrace}=h$ and $\widetilde{h}_{\mid X\times\lbrace 0\rbrace}=h^{\prime}$.

We record a few lemmas related to the Bott--Chern classes $\widetilde{c}_{p}$. 
\begin{lemma}\label{lemma:Bottcherndual} Suppose  we are given a metrized short exact sequence $\ov{\varepsilon}$, and denote by $\ov{\varepsilon}^\vee$  the dual exact sequence, equipped with the dual metrics. Then   $\widetilde{c}_p(\ov{\varepsilon}^\vee) = (-1)^p \widetilde{c}_p(\ov{\varepsilon})$ .
\end{lemma}
\begin{proof}
On the level of Chern forms, for a hermitian vector bundle $(E,h)$, we have the relationship $c_p({E}^\vee, h^\vee) = (-1)^k c_p({E}, h).$ It follows that the associated Bott--Chern secondary forms both satisfy the same axioms and hence must be equal.
\end{proof}

\begin{lemma}\label{lemma:c2tilde}
Let $\ov{\varepsilon}$ be a metrized short exact sequence, and $\ov{F}$ a hermitian vector bundle of rank $f$. Then 
\begin{displaymath} \widetilde{c}_2(\ov{\varepsilon}\otimes \ov{F}) =f  \widetilde{c}_2(\ov{\varepsilon}) + (e\cdot f-1) \widetilde{c}_1(\ov{\varepsilon}) + {f \choose 2} \widetilde{c}_1(\ov{\varepsilon}) \left( c_1(\ov{E}) + c_1(\ov{E'}) + c_1(\ov{E''})\right)
\end{displaymath}
\end{lemma}

\begin{proof}
For this statement when $f=1$, we refer the reader to  \cite[Proposition 1.3.3]{GS:ACC1} or \cite[Proposition 2]{Tamvakis}. The proofs are essentially all the same, but we include one for the convenience of the reader. 

Suppose $\phi(E)$ is an additive (or multiplicative) class in the Chern classes of $E$, and suppose there is a decomposition 
\begin{equation}\label{eq:dd-2} \phi(E \otimes F) = \sum_{\beta} \phi_\beta(E) \psi_\beta(F)
\end{equation}
for polynomials in Chern classes $\phi_\beta, \psi_\beta$. In \cite[Proposition 1.3.3]{GS:ACC1} a decomposition 
\begin{displaymath}
    \widetilde{\phi}({\ov{\varepsilon}} \otimes \overline{F})=   \sum_\beta \widetilde{\phi}_\beta({\ov{\varepsilon}}) \psi_\beta(\overline{F})
\end{displaymath}
is then proven. The proof is based on the fact that the formula \eqref{eq:dd-2} lifts to the level of forms, and then employs the explicit formula in \eqref{eq:gg-7}. 

In our setting, a standard computation using the splitting principle shows that
\begin{displaymath}
    c_2(E \otimes F) = f \cdot c_2(E) + e  \cdot c_2(F) + (e\cdot f -1 ) c_1(E) c_1(F) + {f \choose 2}  c_1(E)^2 + {e \choose 2}   c_1(F)^2,
\end{displaymath} 
where $e$ is the rank of $E$. An inspection of this equality applied to the above discussion shows that this accounts for all terms in the lemma, except for the last one. The last one corresponds to the Bott--Chern class associated to $c_1^2$, and we find from  \cite[Proposition 1.3.1]{GS:ACC1} that $\widetilde{c_1 \cdot c_1}(\ov{\varepsilon}) = \widetilde{c}_1(\ov{\varepsilon}) c_1(\overline{E}) + \left(c_1(\overline{E}')+ c_1(\overline{E}'') \right) \widetilde{c}_1(\ov{\varepsilon})$. This proves the lemma. 

\end{proof}

\subsection{Metrics on intersection bundles}

\subsubsection{Metrics on Deligne pairings}\label{subsubsec:metrics-on-deligne-pairings}
Let $f\colon X\to S$ be a proper submersion between complex manifolds, of relative pure dimension one. Let $\ov{L}$ and $\ov{M}$ be hermitian line bundles on $X$, whose norms we simultaneously denote by $\|\cdot\|$. We review the induced intersection metric on the Deligne pairing $\langle L,M\rangle$. For references see Deligne \cite[Section 6]{Deligne-determinant} and Gillet--Soul\'e \cite[\textsection 4.10]{GS:ACC1}. 

The metric on $\langle L,M\rangle$ is determined pointwise on $S$, and we first assume that $S$ is a point and $X$ a compact Riemann surface. Let $\ell$ and $m$ be rational sections of $L$ and $M$, respectively, with disjoint divisors. The hermitian norm of the symbol $\langle\ell,m\rangle$ is determined by the rule
\begin{equation}\label{eq:gg-8}
    \log\|\langle\ell,m\rangle\|^{2}=\int_{X}\left(\log\|\ell\|^{2}c_{1}(\ov{M})+\log\|m\|^{2}\delta_{\Div \ell}\right). 
\end{equation}
By Stokes' theorem it is checked that this respects the relations between symbols, and that $\|\langle\ell,m\rangle\|=\|\langle m,\ell\rangle\|$. Over a general base $S$, the construction defines a smooth hermitian metric on $\langle L,M\rangle$. We write $\langle\ov{L},\ov{M}\rangle$ for the resulting hermitian line bundle. It commutes with base change and is compatible with the bi-multiplicativity property of Deligne pairings.

The curvature form of $\langle\ov{L},\ov{M}\rangle$ is easily computed from \eqref{eq:gg-8}, and is summarized in the relationship
\begin{displaymath}
        c_{1}(\langle\ov{L},\ov{M}\rangle)=\int_{X/S} c_{1}(\ov{L})\wedge c_{1}(\ov{M}).
\end{displaymath}
Finally, directly from the definition we find the formula for a change of metrics:
\begin{equation}\label{eq:changemetricsdeligne}
    \langle\ov{L}^{\prime},\ov{M}\rangle=\langle\ov{L},\ov{M}\rangle\otimes(\Ocal_{S},q),
\end{equation}
where $q$ is the hermitian metric on the trivial line bundle $\Ocal_{S}$ with value at $s\in S$ given by
\begin{displaymath}
    q_{s}(1,1)=\exp\left(\int_{X_{s}}\widetilde{c}_{1}(\ov{L},\ov{L}^{\prime})c_{1}(\ov{M})\right).
\end{displaymath}
Here, we recall that the convention for $\widetilde{c}_{1}$ is such that $dd^{c}\widetilde{c}_{1}(\ov{L},\ov{L}^{\prime})=c_{1}(\ov{L})-c_{1}(\ov{L}^{\prime})$.

\subsubsection{Metrics on $IC_{2}$}\label{subsubsec:ic2metraxiom}
Continuing with the previous setting, let now $\ov{E}$ be a smooth hermitian vector bundle on $X$. Following Deligne \cite[Section 10]{Deligne-determinant}, Elkik \cite[Section III]{Elkik2} and Gillet--Soul\'e \cite{GS:ACC1}, the intersection bundle $IC_{2}(E)$ can  naturally be equipped with a smooth hermitian metric, denoted by $IC_{2}(\ov{E})$, which we recall now. 

The intersection metric on $IC_{2}(\ov{E})$ is characterized by the following three properties:
\begin{enumerate}
    \item[\emph{(MIC1)}] \emph{(Functoriality)} Formation of $IC_{2}(\ov{E})$ commutes with base change. 
    \item[\emph{(MIC2)}] \emph{(Rank one normalisation)} If $E$ is a line bundle, the canonical trivialization $IC_{2}(E)\simeq\Ocal_{S}$ induces an isometry between $IC_{2}(\ov{E})$ and the trivial hermitian line bundle.
    \item[\emph{(MIC3)}] \emph{(Whitney isometry)} Let $\ov{\varepsilon}$ be a metrized exact sequence, with underlying sequence \eqref{eq:gg-1}. Then the Whitney isomorphism
    \begin{displaymath}
        IC_{2}(E)\simeq IC_{2}(E')\otimes IC_{2}(E'')\otimes\langle\det E',\det E''\rangle
    \end{displaymath}
    induces an isometry
    \begin{displaymath}
        IC_{2}(\ov{E})\simeq IC_{2}(\ov{E}')\otimes IC_{2}(\ov{E}'')\otimes\langle\det \ov{E}',\det \ov{E}''\rangle\otimes (\Ocal_{S},q),
    \end{displaymath}
    where $q$ is the hermitian metric on the trivial line bundle with value at $s\in S$ given by
    \begin{equation}\label{eq:gg-5}
        q_{s}(1,1)=\exp\left(\int_{X_{s}}\widetilde{c}_{2}(\ov{\varepsilon})\right).
    \end{equation}
\end{enumerate}
It follows from \emph{(MIC1)} that the metric on $IC_{2}(\ov{E})$ can be computed pointwise on $S$. 

\begin{proposition}\label{prop:isometry_IC2-tensor}
Let $\ov{E}, \ov{F}$ be a hermitian vector bundles of ranks $e$ and $f$ on $X$. Then the isomorphism in Proposition \ref{prop:whitneyproductwithbundle} induces an isometry
\begin{displaymath}
    IC_2(\ov{E} \otimes \ov{F}) \simeq IC_2(\ov{E})^{f} \otimes IC_2(\ov{F})^{e} \otimes \langle \det \ov{E}, \det \ov{F} \rangle^{e\cdot f-1} \otimes \langle \det \ov{E}, \det \ov{E} \rangle^{f \choose 2} \otimes \langle \det \ov{F}, \det \ov{F} \rangle^{e \choose 2}.
\end{displaymath}
\end{proposition}
\begin{proof}
If $E, F$ are both of rank one, the isomorphism is compatible with the trivializations of $IC_2$ of a line bundle. Since the latter are isometries, this case is covered.

We proceed by induction on the ranks, supposing that $F$ is a line bundle first. We can suppose that there is an exact sequence $\varepsilon : 0 \to M \to E \to E'' \to 0$, with $M$ a line bundle. Then we know that the diagram  
\begin{displaymath}
    \xymatrix{
    IC_2(E \otimes F) \ar[r] \ar[d] &    IC_2(M \otimes F) \otimes IC_2(E'' \otimes F) \otimes  \langle M \otimes F, \det (E'' \otimes F) \rangle   \ar[d]  \\
        IC_2(E) \otimes \langle \det E, F \rangle^{e-1} \otimes \langle F, F \rangle^{ e \choose 2} \ar[r] & IC_2(M) \otimes IC_2(E'') \otimes \langle  M, \det E''\rangle \otimes \langle M, F \rangle^{e-1}\ar@{}[d]|-{\bigotimes}\\
   & \langle  \det E'' , F \rangle^{e-1} \otimes \langle F, F \rangle^{ {e \choose 2} }
    }
\end{displaymath}
commutes. We denote by $\ov{\varepsilon}$ the metrized sequence with the metrics on $M, E''$ induced by $E$. By the induction hypothesis, and the fact that the isomorphism \eqref{eq:isodetline} is an isometry, we see that the rightmost arrow is an isometry.  To metrically describe the first row, we notice that there is an isometry
\begin{displaymath}
    IC_2(\ov{E} \otimes \ov{F}) \to   IC_2(\ov{M} \otimes \ov{F}) \otimes IC_2(\ov{E}'' \otimes \ov{F}) \otimes  \langle \ov{M} \otimes \ov{F}, \det (\ov{E}'' \otimes \ov{F})\rangle \otimes (\Ocal_S, q)
\end{displaymath}
where $q_s(1,1) = \exp\left(\int_{X_s} \widetilde{c}_2(\ov{\varepsilon} \otimes \ov{F}) \right).$
For the lower row, there is an isometry
\begin{displaymath} 
    \xymatrix{
        IC_2(\ov{E}) \otimes \langle \det \ov{E}, \ov{F} \rangle^{e-1} \otimes \langle \ov{F}, \ov{F} \rangle^{ e \choose 2}\ar[r] & IC_2(\ov{M}) \otimes IC_2(\ov{E}'') \otimes \langle  \ov{M}, \det \ov{E}'' \rangle \otimes \langle \ov{M}, \ov{F} \rangle^{e-1}\ar@{}[d]|-{\bigotimes} \\ 
    &  \langle \det\ov{E}'' ,  \ov{F} \rangle^{e-1} \otimes \langle \ov{F}, \ov{F} \rangle^{ {e \choose 2} } \otimes (\Ocal_S, q')
    }
\end{displaymath}
where $q_s'(1,1) = \exp\left(\int_{X_s} c_2(\widetilde{\varepsilon}) + (e-1) c_1(\ov{F}) \widetilde{c}_1(\ov{\varepsilon}) \right)$. The computation of the last term follows from the description of the change of metrics on Deligne pairings in \eqref{eq:changemetricsdeligne}. From Lemma \ref{lemma:c2tilde} we see that $q_s = q_s'$, so the isomorphism is an isometry whenever $F$ is a line bundle and $E$ is of arbitrary rank, independently of the chosen metrics. 

The analogous induction argument on the rank of $F$ together with the general form of Lemma \ref{lemma:c2tilde} proves the statement in general. 
\end{proof}

\begin{proposition}\label{lemma:IC2-dual-metric}
The natural isomorphism exhibited in Proposition  \ref{prop:dual} induces an isometry $IC_2(\ov{E}) \simeq IC_2(\ov{E}^\vee)$. 
\end{proposition}

\begin{proof}
We can argue by the splitting principle as in the proof of the proposition, and reduce to the statement that $\widetilde{c}_2(\ov{\varepsilon})=\widetilde{c}_2(\ov{\varepsilon}^\vee)$. This follows from Lemma \ref{lemma:Bottcherndual}. 
\end{proof}

\begin{corollary}\label{cor:IC2-End-metric}
The natural isomorphism in Corollary \ref{cor:IC2-End} induces an isometry
 \begin{displaymath}
        IC_2(\End \ov{E}) \to IC_2(\ov{E})^{2e} \otimes \langle \det \ov{E}, \det \ov{E} \rangle^{-e}.
    \end{displaymath}
\end{corollary}
\qed

Recall from Definition \ref{def:linefunctorkvar} (a) that the line functor propery of $IC_{2}$ includes a compatibility with isomorphisms $X^{\prime}\to X$ over $S$. The following lemma is obvious, but we record it for later reference.
\begin{lemma}\label{lemma:stupid-lemma}
Let $g\colon X^{\prime}\to X$ be an isomorphism of relative curves over $S$. Let $\ov{E}$ be a hermitian vector bundle on $X$. Then, the canonical isomorphism $IC_{2}(g^{\ast}E)\simeq IC_{2}(E)$ induces an isometry $IC_{2}(g^{\ast}\ov{E})\simeq IC_{2}(\ov{E})$.
\end{lemma}
\qed

\subsubsection{Example: flat unitary vector bundles of rank 2} \label{subsec:IC2-metric-flat-unitary}
For concreteness, we now assume that $\ov{E}$ is a hermitian vector bundle of rank $2$ on $X$, whose metric is flat on the fibers of $f\colon X\to S$. We will describe the metric on $IC_{2}(\ov{E})$ at the level of trivializations. In Section \ref{section:CS-theory} this will be applied to exhibit concrete expressions for complex metrics.

Introduce an auxiliary relatively ample line bundle $L$, such that $E\otimes L$ is globally generated on fibers and $R^{1}f_{\ast}(E\otimes L)=0$. Endow $L$ with a smooth hermitian metric, positive on fibers. Since metrics on Deligne pairings are already understood by \textsection \ref{subsubsec:metrics-on-deligne-pairings}, computing the metric on $IC_{2}(\ov{E})$ is tantamount to computing the metric on $IC_{2}(\ov{E}\otimes \ov{L})$ by Proposition \ref{prop:isometry_IC2-tensor}. Thus, we just need to treat the latter.

According to \textsection \ref{subsec:sectionsIC2}, trivializations of $IC_{2}(E\otimes L)$ are produced as follows. By \cite[Lemme 7.1]{Drezet-Narasimhan} and \cite[Proof of Lemma 4.10.4]{GS:ACC1}, locally with respect to $S$, we can find a nowhere vanishing section $s$ of $E\otimes L$ and an exact sequence of vector bundles 
\begin{equation}\label{eq:ex-seq-def-s}
     \varepsilon \colon 0\rightarrow\Ocal_{X}\overset{s}{\rightarrow} E\otimes L\rightarrow Q\rightarrow 0.
\end{equation}
By the Whitney isomorphism we obtain a section $\T{s}$ of $IC_2(E \otimes L)$. See equation \eqref{eq:trivializationgeneralcase} and Lemma \ref{lemma:equivalent-sections}. We endow $\Ocal_{X}$ and $Q$ with the hermitian metrics induced from $\ov{E}\otimes\ov{L}$. In particular, the norm of  $1\in\Ocal_{X}$ is given by the norm of the section $s$. We can then compute the norm of $\T{s}$ by applying \emph{(MIC3)}. For this, we begin with a particular instance of \cite[Theorem 5]{Tamvakis}, whose proof we leave as an exercise.

\begin{lemma}\label{lemma:tamvakis}
The Bott--Chern secondary form of \eqref{eq:ex-seq-def-s} satisfies
\begin{displaymath}
    \int_{X/S}\widetilde{c}_{2}(\ov{\varepsilon})=-\deg L. 
\end{displaymath}
\end{lemma}\qed

Similarly, for the norm of the canonical trivialization $e^{\prime}$ of $\langle\Ocal_{X},Q\rangle$ we find
\begin{displaymath}
    \log\|e^{\prime}\|^{2}=\int_{X/S}\log \|s\|^{2} c_{1}(\ov{Q}).
\end{displaymath}
Using that for an exact sequence with induced metrics $\widetilde{c}_{1}$ vanishes, we find  $c_{1}(\ov{Q})=c_{1}(\ov{E}\otimes\ov{L})-c_{1}(\ov{\Ocal}_{X}) = c_{1}(\ov{E}) + c_1(\ov{L})-c_{1}(\ov{\Ocal}_{X}).$ Besides, $c_{1}(\ov{E})$ vanishes on fibers. Thus, we can equivalently write
\begin{displaymath}
    \log\|e^{\prime}\|^{2}= \int_{X/S}\log\|s\|^{2} c_{1}(\ov{L})+\int_{X/S} \log\|s\|^{2} dd^{c}\log\|s\|^{2}.
\end{displaymath}
The result of applying \emph{(MIC3)} is then recorded in the following:
\begin{proposition}\label{prop:explicit-metric}
With the notation above, the norm of the section $\T{s}$ is given by 
\begin{equation}\label{eq:metric-on-Ts}
         \log\|\T{s}\|^{2}= -\deg L+ \int_{X/S}\log\|s\|^{2} c_{1}(\ov{L})-\frac{i}{2\pi}\int_{X/S}\partial\log\|s\|^{2} \wedge \ov{\partial}\log\|s\|^{2}.
\end{equation}
\end{proposition}\qed

Notice that the second integral, without the minus sign, is a relative version of the Dirichlet norm of $\log\|s\|^{2}$, denoted below $\left\lbrace\log\|s\|^{2}\right\rbrace_{\scriptscriptstyle{\mathsf{Dir}}}$. 

Suppose next that $t$ is another nowhere vanishing section of $E\otimes L$, such that the zero locus $Z=\Div(s\wedge t)$ is finite \'etale over $S$. On $Z$, we can write $s_{\mid Z}=\alpha t_{\mid Z}$, for some $\alpha\in\Gamma(Z,\Ocal_{Z}^{\times})$. After Proposition \ref{prop:reciprocity-rank-2}, the relationship between the trivializations $\T{s}$ and $\T{t}$ is $\T{s}=N_{Z/S}(\alpha)\T{t}$. Combining with Proposition \ref{prop:explicit-metric}, we derive an analytic expression for the holomorphic function $N_{Z/S}(\alpha)$, which determines it up to a constant of modulus one.
\begin{corollary}\label{cor:explicit-metric}
With the notation above, we have
\begin{displaymath}
    \log|N_{Z/S}(\alpha)|^{2}=\int_{X/S}\log\frac{\|s\|^{2}}{\|t\|^{2}}\ c_{1}(\ov{L})+\left\lbrace\log\|t\|^{2}\right\rbrace_{\scriptscriptstyle{\mathsf{Dir}}}-\left\lbrace\log\|s\|^{2}\right\rbrace_{\scriptscriptstyle{\mathsf{Dir}}}.
\end{displaymath}
\end{corollary}
\qed

As a side remark, we observe that Proposition \ref{prop:explicit-metric} and Corollary \ref{cor:explicit-metric} provide an alternative intrinsic  construction of the metric on $IC_{2}(E\otimes L)$. 


\section{Intersection connections}\label{section:intersection-connections-generalities}
In this section, $f\colon X\to S$ is a proper submersion of complex manifolds with fibers of pure dimension $n$. We develop a refinement of classical Chern--Simons trangression forms \cite{Chern-Simons}, for holomorphic vector bundles with compatible connections. The formalism originates from the construction of characteristic classes of holomorphic vector bundles with compatible connections sketched by Gillet--Soul\'e \cite[Section 4]{GS:diff-char}. When $X\to S$ is a family of Riemann surfaces, we apply this theory to construct connections on Deligne pairings and $IC_{2}$ bundles, naturally induced by compatible connections. These we call \emph{intersection connections}. We introduce the theory in  greater generality than what is needed, so that it might be easily applied to the case of higher relative dimensions in a future work.

\subsection{Transgression classes} 

Suppose we have an exact sequence \eqref{eq:gg-1} of holomorphic vector bundles on $X$.
If the individual vector bundles are equipped with compatible connections $\nabla', \nabla, \nabla''$, we denote the corresponding sequence with connections by $\varepsilon_{\nabla}$. If $\varphi$ is a power series as in \textsection \ref{subsec:Bott-Chern}, we know as in the discussion surrounding \eqref{eq:gg-2} that 
\begin{equation}\label{eq:simpletrans}
    \varphi(E' \oplus E'', \nabla' \oplus \nabla'')  - \varphi(E, \nabla)
\end{equation}
is exact. We will consider in detail a specific construction of trangression class, $T \varphi$.

\subsubsection{Construction}
We form the product $p\colon X\times\PBbb^{1}\to X$, and recall the notation from transgression bundles \textsection \ref{sec:splitting2}. Let $\widetilde{\nabla}$ be any compatible connection on $\widetilde{E}$ such that:
\begin{displaymath}
    \widetilde{\nabla}_{\mid X\times\lbrace\infty\rbrace}=\nabla,\quad \widetilde{\nabla}_{\mid X\times\lbrace 0\rbrace}= \nabla' \oplus \nabla''.
\end{displaymath}
We say that $\widetilde{\nabla}$ interpolates between $\nabla$ and $\nabla' \oplus \nabla''$. Such a connection always exists by a partition of unity argument. When $E'$ (or $E''$) is of  rank zero, the above can be seen as an interpolation between two connections $\nabla_1, \nabla_2$ on $E$. A relevant example, called the naive interpolation, is
\begin{equation}\label{eq:naiveint}
    \nabla^{\na}=\frac{|x|^{2}}{|x|^{2}+|y|^{2}}p^{\ast}\nabla_{2}+\frac{|y|^{2}}{|x|^{2}+|y|^{2}}p^{\ast}\nabla_{1},
\end{equation}
where $(x : y)$ are the standard homogeneous coordinates on $\PBbb^{1}$.

\begin{definition}
Given an interpolating connection $\widetilde{\nabla}$ as above, we define the transgression class of $\varphi$ as
\begin{equation} \label{def:transforms}
    T\varphi(\varepsilon_{\nabla},  \widetilde{\nabla}) = \int_{\PBbb^1} \frac{-dt}{t} \varphi(\widetilde{E}, \widetilde{\nabla}),\quad t = x/y.
\end{equation}
 \end{definition}
The integral is absolutely convergent. By the equation of currents,
\begin{displaymath}
    -d\left[\frac{dt}{t}\right]=2\pi i(\delta_{\infty}-\delta_{0})
\end{displaymath}
one finds that $T\varphi(\varepsilon_{\nabla}, \widetilde{\nabla})$ indeed provides a simple transgression of $\eqref{eq:simpletrans}$. We will say more in Proposition \ref{prop:the-classes-in-64-satisfy}.

\begin{lemma}\label{lemma:interpolatingindependence}
For two interpolating connections $\widetilde{\nabla}_{\amini}, \widetilde{\nabla}_{\bmini}$ we have   \begin{displaymath} 
T\varphi(\varepsilon_{\nabla}, \widetilde{\nabla}_{\amini})= T\varphi(\varepsilon_{\nabla}, \widetilde{\nabla}_{\bmini}) + d \eta
\end{displaymath}
for a form $\eta$. If $\varphi$ is of pure degree $n+1$, the form $\eta$ can be chosen so that it has only types of bidegrees $(p,q)$, with $q\leq n-1$. 
\end{lemma}
\begin{proof}
We carry out a double transgression. Form the product $q\colon (X\times\PBbb^{1})\times\PBbb^{1}\to X\times\PBbb^{1}$. Henceforth, to avoid ambiguity we denote the first (resp. second) copy of $\PBbb^{1}$ by $P$ (resp. $P^{\prime})$, with inhomogenous coordinate $t$ (resp. $t^{\prime}$). On $q^{\ast}\widetilde{E}$, we consider $\nabla^{\na}$ the naive interpolation \eqref{eq:naiveint} between $q^{\ast}\widetilde{\nabla}_{\amini}$ and $q^{\ast}\widetilde{\nabla}_{\bmini}$. Then,
\begin{displaymath}
    T\varphi(\varepsilon_{\nabla},\widetilde{\nabla}_{\amini})=T\varphi(\varepsilon_{\nabla},\widetilde{\nabla}_{\bmini})-\frac{1}{2\pi i}d\int_{P^{\prime}}\frac{dt^{\prime}}{t^{\prime}}\varphi(q^{\ast}\widetilde{E},\nabla^{\na}).
\end{displaymath}
Plugging this relationship into the definition of $T\varphi(E,\widetilde{\nabla}_{\amini})$, we find
\begin{displaymath}
    T\varphi(\varepsilon_{\nabla},\widetilde{\nabla}_{\amini})=T\varphi(\varepsilon_{\nabla},\widetilde{\nabla}_{\bmini})+\frac{1}{2\pi i}\int_{P}\frac{dt}{t}d\int_{P^{\prime}}\frac{dt^{\prime}}{t^{\prime}}\varphi (q^{\ast}\widetilde{E},\nabla^{\na}).
\end{displaymath}
For this, we try to move the differential $d$ out of the first integral. We compute 
\begin{displaymath}
   \begin{split}
         \frac{1}{2\pi i}d\int_{P}\frac{dt}{t}\left(\int_{P^{\prime}}\frac{dt^{\prime}}{t^{\prime}}\varphi(q^{\ast}\widetilde{E},\nabla^{\na})\right)=&\int_{P}(\delta_{t=0}-\delta_{t=\infty})\left(\int_{P^{\prime}}\frac{dt^{\prime}}{t^{\prime}}\varphi(q^{\ast}\widetilde{E},\nabla^{\na})\right)\\
            &-\frac{1}{2\pi i}\int_{P}\frac{dt}{t}d\left(\int_{P^{\prime}}\frac{dt^{\prime}}{t^{\prime}}\varphi(q^{\ast}\widetilde{E},\nabla^{\na})\right).
    \end{split}
\end{displaymath}
We claim that the first term on the right hand side vanishes. Indeed, if $p^{\prime}\colon X\times P^{\prime}\to X$ is the projection map, then the very construction of the naive interpolation $\nabla^{\na}$ is such that
\begin{displaymath}
    \varphi(q^{\ast}\widetilde{E},\nabla^{\na})_{\mid t=0}=p^{\prime\ast}\varphi(E,\nabla),\quad
    \varphi(q^{\ast}\widetilde{E},\nabla^{\na})_{\mid t=\infty}=p^{\prime\ast}\varphi(E' \oplus E'',\nabla'\oplus \nabla'').
\end{displaymath}
In other words, at $t=0$, $\nabla^{\na}$ is the naive interpolation in variable $t^{\prime}$ between $\nabla$ and $\nabla$ itself, hence constant in $t^{\prime}$, and similarly at $t=\infty$. We infer
\begin{displaymath}
    \int_{P}(\delta_{t=0}-\delta_{t=\infty})\left(\int_{P^{\prime}}\frac{dt^{\prime}}{t^{\prime}}\varphi(q^{\ast}\widetilde{E},\nabla^{\na})\right)=\left(\int_{P^{\prime}}\frac{dt^{\prime}}{t^{\prime}}\right)(\varphi(E,\nabla)- \varphi(E' \oplus E'',\nabla' \oplus \nabla''))=0,
\end{displaymath}
as was to be shown. Hence, 
\begin{displaymath}
    \frac{1}{2\pi i}\int_{P}\frac{dt}{t}d\left(\int_{P^{\prime}}\frac{dt^{\prime}}{t^{\prime}}\varphi(q^{\ast}\widetilde{E},\nabla^{\na})\right)=
    -\frac{1}{2\pi i}d\int_{P\times P^{\prime}}\frac{dt}{t}\wedge\frac{dt^{\prime}}{t^{\prime}}\varphi(q^{\ast}\widetilde{E},\nabla^{\na}).
\end{displaymath}
Now, since $\nabla^{\na}$ is compatible, $\varphi(q^\ast \widetilde{E}, \nabla^{\na})$ has only components of type $(p,q)$ with $q \leq n+1$ if $\varphi$ is of pure degree $n+1$. Integrating over the surface $P\times P^{\prime}$ decreases types by $(2,2)$. The double integral hence has only components of type $(p,q)$ with $q \leq n-1$. This concludes the proof.
\end{proof}

The following proposition will be applied in \textsection \ref{subsec:intersectionconnections} to develop a sensible theory of compatible connections on $IC_2(E)$, provided a compatible connection on $E$. It can in principle be developed to a more general theory of connections on Elkik's intersection bundles \cite{Elkikfib}, and we include for possible future reference the general formulation.

\begin{proposition}\label{prop:Chern-Simons}
Let $\varphi$ be a polynomial in Chern classes of pure degree  $n+1$. Then the form  $\int_{X/S} T\varphi(\varepsilon_{\nabla}, \widetilde{\nabla})$ only depends on $\varepsilon_{\nabla}$, and not the interpolating connection $\widetilde{\nabla}$, and hence provides a well-defined form on $S$. It is of type $(1,0)$.
\end{proposition}

\begin{proof}
The difference for two choices of interpolating connection is of the form $\int_{X/S} d \eta = d\int_{X/S} \eta$. Since the fiber integral reduces the bidegree by $(n,n)$, this contribution disappears when integrating, by Lemma \ref{lemma:interpolatingindependence} and for type reasons. The final integral is also of type $(1,0)$ for similar reasons. 
\end{proof}

The following is an immediate consequence of the proposition and the construction \eqref{eq:gg-4} of Bott--Chern classes.
\begin{corollary}\label{lemma:chern-conn-1}
Let the assumptions and notation be as in Proposition \ref{prop:Chern-Simons}. Let $\ov{\varepsilon}$ be a metrized exact sequence, and $\varepsilon_\nabla$ the corresponding short exact sequence equipped with the corresponding Chern connections. Then there is an equality of $(1,0)$-forms
\begin{displaymath}
    \int_{X/S}T\varphi(\varepsilon_\nabla, \widetilde{\nabla})=\partial \int_{X/S}\widetilde{\varphi}(\ov{\varepsilon}).
\end{displaymath}
\end{corollary}\qed

For $\varphi$ of pure degree $n+1$ as above, let us denote by 
\begin{equation}\label{def:Tvarphi}
    T\varphi(\varepsilon_{\nabla})= \text{ class of }\ T\varphi(\varepsilon_{\nabla},\widetilde{\nabla}) \ \text{ modulo }  \bigoplus_{\substack{p\\ q\leq n-1}} dA^{p,q}(X).
\end{equation}

We leave to the reader to check the following formal properties.
\begin{proposition}\label{prop:the-classes-in-64-satisfy}
The class $T\varphi(\varepsilon_{\nabla})$ in \eqref{def:Tvarphi} satisfies:
\begin{enumerate}
    \item[(TC1)] (Differential equation) $dT{\varphi}({\varepsilon}_\nabla)=2 \pi i \left(\varphi(E' \oplus E'', \nabla' \oplus \nabla'')  - \varphi(E, \nabla) \right)$.
    \item[(TC2)] (Functoriality) The formation of $T{\varphi}(\varepsilon_{\nabla})$ commutes with pull-back.
    \item[(TC3)] (Normalization) If ${\varepsilon}_{\nabla}$ is holomorphically split in a way which respects the connections, then $T{\varphi}({\varepsilon}_{\nabla})=0$.
\end{enumerate}
\end{proposition}\qed

\subsubsection{Chern--Simons integrals}\label{subsec:CS-integrals}
Until the end of this section, we suppose that $f\colon X\to S$ is a family of compact Riemann surfaces. For the purpose of this article, the most important case of transgression class is for $\varphi = c_2$. In that case, we define 
\begin{equation}\label{def:ChernSimons}
    Tc_2(\varepsilon_{\nabla})=-\int_{\PBbb^{1}}\frac{dt}{t}c_{2}(\widetilde{E},\widetilde{\nabla}),\quad t=x/y,\quad (x : y)\in\PBbb^{1},
\end{equation}
so that 
\begin{equation}\label{eq:ICS2exactseq}
    d Tc_{2}(\varepsilon_\nabla)=2\pi i(c_{2}(E' \oplus E'',\nabla' \oplus \nabla'')- c_{2}(E,\nabla)), 
\end{equation}
where the first term on the right hand side can be expanded according to the Whitney formula. In the special case of an exact sequence 
\begin{displaymath}
    \varepsilon_\nabla\colon  0 \to (E, \nabla_1) \to (E, \nabla_2) \to 0 \to 0 
\end{displaymath}
we write $Tc_{2}(E, \nabla_1, \nabla_2) = Tc_{2}(\varepsilon_{\nabla})$. Then we have 
\begin{equation}\label{eq:gg-6}
    d Tc_{2}(E, \nabla_1, \nabla_2)=2\pi i(c_{2}(E,\nabla_{1})-c_{2}(E,\nabla_{2})).
\end{equation}

\begin{definition}
With the previous assumptions and notation, we define the Chern--Simons integral of an exact sequence with connections $\varepsilon_\nabla$ as
\begin{displaymath}
    IT(\varepsilon_\nabla) = \int_{X/S} Tc_{2}(\varepsilon_\nabla).
\end{displaymath}
In a similar vein, for a vector bundle with two compatible connections $(E,\nabla_{1},\nabla_{2})$ define
\begin{displaymath}
    IT(E,\nabla_{1},\nabla_{2})=\int_{X/S}Tc_{2}(\varepsilon_{\nabla})\in A^{1,0}(S).
\end{displaymath}
\end{definition}

\begin{proposition}\label{prop:ICS-explicit-formula}
Let $\nabla\colon E\to E\otimes\Acal_{X}^{1,0}$ be a connection and $\theta$ a $(1,0)$-form with values in $\End(E)$. Denote by $F$ the curvature of $\nabla$. Then 
\begin{displaymath}
    \begin{split}
         IT(E,\nabla,\nabla+\theta)=&\frac{1}{2\pi i}\int_{X/S}\left(\tr(F\wedge\theta)+\frac{1}{2}\tr(\theta\wedge\ov{\partial}\theta)\right)\\
         &-\frac{1}{2\pi i}\int_{X/S}\left(\tr(F)\wedge\tr(\theta)+\frac{1}{2}\tr(\theta)\wedge\tr(\ov{\partial}\theta)\right).
    \end{split}
\end{displaymath}
In particular, if $F=0$, then 
\begin{displaymath}
    IT(E,\nabla,\nabla+\theta)=\frac{1}{4\pi i}\int_{X/S}\left(\tr(\theta\wedge\ov{\partial}\theta)-\tr(\theta)\wedge\tr(\ov{\partial}\theta)\right).
\end{displaymath}
\end{proposition}
\begin{proof}
We compute the naive interpolation, cf. equation \eqref{eq:naiveint}: 

\begin{displaymath}
    \nabla^{\na}=\nabla+\frac{1}{1+|t|^{2}}\theta.
\end{displaymath} 
The curvature of $\nabla^{\na}$ is
\begin{displaymath}
    F^{\na}=F+d\left(\frac{1}{1+|t|^{2}}\right)\wedge\theta+\frac{1}{1+|t|^{2}}[\nabla,\theta]+\left(\frac{1}{1+|t|^{2}}\right)^{2}\theta\wedge\theta.
\end{displaymath}
For the computation of $ IT(E,\nabla,\nabla+\theta)$, we gather the terms in $\tr({F^{\na}}^{\ 2})$ and $(\tr F^{\na})^{2}$ which contain exactly one factor $d\left(1/(1+|t|^{2})\right)$. Later, when we take the fiber integral, we discard all the terms of type $(3,0)$. To carry out this step, we notice that $[\nabla,\theta]^{\scriptscriptstyle{(0,1)}}=\ov{\partial}\theta$. With these observations at hand, the lemma reduces to a routine computation.
\end{proof}

\begin{remark}
We notice that the expression provided by the proposition slightly differs from the well-known Chern--Simons form for the change of connection on a $\Ccal^{\infty}$ vector bundle:
\begin{displaymath}
    \mathsf{CS}(\nabla+\alpha,\nabla)=-\frac{1}{8\pi^{2}}\left(\tr(\alpha\wedge\nabla\alpha)+2\tr(\alpha\wedge F_{\nabla}) +\frac{2}{3}\tr(\alpha\wedge\alpha\wedge\alpha)\right) .
\end{displaymath}
This transgression form is attached to the piece of degree 2 of the Chern character, while we rather work with the second Chern class to avoid the denominator $2$ in $ch_{2}$. Furthermore, some simplifications occur in the holomorphic setting and in the fiber integrals above, for type reasons. 
\end{remark}

A routine computation using Proposition \ref{prop:ICS-explicit-formula} proves the following lemma, recorded for reference: 
\begin{lemma}\label{lemma:ics2whitney}
Let $E', E''$ be holomorphic vector bundles with connections $\nabla', \nabla''$. Suppose that $\theta', \theta''$ are $(1,0)$-forms with values in the endomorphims of $E'$ and $E''.$ Then the Chern--Simons integral in the direct sum case 
\begin{displaymath}
    IT(E' \oplus E'', \nabla' \oplus \nabla'', (\nabla' + \theta') \oplus (\nabla''+ \theta'')) 
\end{displaymath}
is given by the formula
\begin{displaymath} IT(E', \nabla', \nabla' +\theta') + IT(E'', \nabla'', \nabla'' + \theta'') + \int_{X/S} \tr(\theta'')  c_1(E', \nabla') + \tr(\theta')  c_1(E'', \nabla'' + \theta'').
\end{displaymath}
\end{lemma}\qed

The difference between two Chern--Simons integrals for variations of connections on the same underlying exact sequence can likewise be computed via a Whitney type formula:
\begin{proposition}\label{prop:cocycle}
Let $\varepsilon_{\nabla_a}$ and $\varepsilon_{\nabla_b}$ be the same short exact sequences equipped with different compatible connections. Then
\begin{eqnarray*}
        IT(\varepsilon_{\nabla_{\amini}}) - IT(\varepsilon_{\nabla_{\bmini}}) & = & 
        IT(E', \nabla_{\amini}', \nabla_{\bmini}') + IT(E', \nabla_{\amini}'', \nabla_{\bmini}'') \\ 
        & &  +\int_{X/S}\left( \tr(\nabla''_{\bmini}-\nabla''_{\amini})  c_1(E', \nabla'_{\amini}) + \tr(\nabla_{\bmini}'-\nabla_{\amini}')  c_1(E'', \nabla''_{\bmini})\right) - IT(E, \nabla_{\amini}, \nabla_{\bmini}).
\end{eqnarray*}
As a particular case, one can deduce that the Chern--Simons integral satisfies the cocycle relation
\begin{displaymath}
    IT(E,\nabla_{1},\nabla_{3})=IT(E,\nabla_{1},\nabla_{2})+IT(E,\nabla_{2},\nabla_{3}).
\end{displaymath}
\end{proposition}
\begin{proof}
The second statement follows from the first one by taking $E''$ to be of rank 0, and we focus  on the first statement. 

By construction, the difference of the Chern--Simons integrals of the proposition is
\begin{displaymath}
    \int_{X \times \PBbb^1/S} - \frac{dt}{t} \left(c_2(\widetilde{E}, \widetilde{\nabla}_{\amini}) - c_2(\widetilde{E}, \widetilde{\nabla}_{\bmini})\right)
\end{displaymath}
for connections $\widetilde{\nabla}_{\amini}, \widetilde{\nabla}_{\bmini}$ interpolating between the split and non-split cases. On the other hand, the difference $c_2(\widetilde{E}, \widetilde{\nabla}_{\amini}) - c_2(\widetilde{E}, \widetilde{\nabla}_{\bmini})$ can be written as the differential of \eqref{def:ChernSimons}, and has the form 
\begin{displaymath}
    d \frac{1}{2\pi i} \int_{\PBbb^1} -\frac{dt'}{t'} c_2(q^* \widetilde{E}, \nabla^{\na}),
\end{displaymath}
where $q: X \times \PBbb^1 \times \PBbb^1 \to X \times \PBbb^1$ is the projection on the first two factors, and $\nabla^{\na}$ is the naive interpolation between $\widetilde{\nabla}_{\amini}, \widetilde{\nabla}_{\bmini}$, see \eqref{eq:naiveint}. This is the same expression as the one appearing in Proposition \ref{prop:Chern-Simons}, and expanding the integral in the same way as in the proof shows that the left hand side equals
\begin{displaymath}
    IT(E^{\prime} \oplus E^{\prime\prime}, \nabla_{\amini}' \oplus \nabla_{\amini}'', \nabla_{\bmini}' \oplus \nabla_{\bmini}'') - IT(E, \nabla_{\amini}, \nabla_{\bmini}).
\end{displaymath}
The first term is computed according to Lemma \ref{lemma:ics2whitney}, and this concludes the proof.
\end{proof}
\subsubsection{Properties}

Similar to Bott--Chern classes, Chern--Simons integrals are characterized by a list of axioms:
\begin{proposition}\label{prop:ics2axioms}
The Chern--Simons integral satisfies, and is uniquely determined by, the following properties.
\begin{enumerate}
    \item[(CS1)] (Type) $IT(\varepsilon_\nabla) \in A^{1,0}(S)$.
    \item[(CS2)] (Functoriality) For any base change $g: S' \to S$ between complex manifolds, we have 
    \begin{displaymath}
        IT(g^* \varepsilon_\nabla) = g^* IT(\varepsilon_\nabla),
    \end{displaymath}
    where we abusively write $g$ for the base changed morphism $X \times_S S' \to X$.
    \item[(CS3)] (Normalization) If $\varepsilon_\nabla$ is a holomorphically split exact sequence, compatibly with the connections, then $IT(\varepsilon_\nabla)=0.$
    \item[(CS4)] (Differential equation) 
    \begin{displaymath}
        dIT(\varepsilon_\nabla) = 2\pi i \int_{X/S} \left( c_2(E', \nabla') + c_2(E'', \nabla'') + c_1(E', \nabla') c_1(E'', \nabla'') - c_2(E, \nabla)\right).
    \end{displaymath}
    
\end{enumerate}
\end{proposition}

\begin{proof}
Properties \emph{(CS1)} and \emph{(CS3)}--\emph{(CS4)} are already known, and the functoriality is clear. We assume that there is another assignment $IT^{\prime}$ satisfying the above. We first base change our family $X\to S$  by $p\colon S\times\PBbb^{1}\to S$ and consider the transgression exact sequence $\widetilde{\varepsilon}_{\nabla}$, equipped with connections $\widetilde{\nabla}', \widetilde{\nabla}, \widetilde{\nabla}''$ interpolating between $\varepsilon_\nabla$ at $0$ and the split case $\varepsilon'_\nabla$ at $\infty$. 

The form $D(\widetilde{\varepsilon}_\nabla) =IT^{\prime}(\widetilde{\varepsilon}_\nabla) - IT(\widetilde{\varepsilon}_\nabla) $ is $d$-closed by \emph{(CS4)} and in $A^{(1,0)}(S \times \PBbb^1)$ by \emph{(CS1)}, and hence holomorphic for type reasons. Since $\PBbb^1$ admits no holomorphic 1-forms, we conclude that $D(E,\nabla_{1},\nabla_{2})=p^* \theta$ for $\theta \in A^{(1,0)}(S)$. By \emph{(CS2)}, the specialization to $S \times \{0\}$ equals $D(\varepsilon_\nabla) = \theta$, and likewise the specialization to $S \times \{\infty\}$ equals $D(\varepsilon'_\nabla) = \theta$. By the normalization in \emph{(CS3)} the latter is zero, and hence so is $D(\varepsilon_\nabla),$ thus proving the claim.
\end{proof}

\begin{remark}
The analogous properties are enjoyed by, and characterize, $IT(E, \nabla_1, \nabla_2)$. 
\end{remark}

\subsection{Intersection connections}\label{subsec:intersectionconnections}

\subsubsection{On Deligne pairings}

Let $L$ and $M$ be  holomorphic line bundles on $X$ and $\nabla^{\Lmini}, \nabla^{\Mmini}$ compatible connections, with curvatures $F_{\Lmini}$ and $F_{\Mmini}$. Suppose we are given meromorphic sections $\ell$ of $L$ and $m$ of $M$, whose divisors are in general position defining a symbol $\langle \ell, m \rangle$. We will define the intersection connection $\nabla^{\scriptscriptstyle{\langle L, M \rangle}}$, following the definition of Freixas--Wentworth of the intersection connection of two relative flat line bundles \cite{Freixas-Wentworth-1}. The precise relation is explored in \textsection \ref{subsec:intersection-conn-flat}.
\begin{definition}\label{def:intconnectiondelignepairing}
With the above notation and assumptions, the intersection connection on  $\langle L, M \rangle$ is given by the formula
\begin{equation}\label{eq:gg-10}
    \frac{\nabla^{\scriptscriptstyle{\langle L, M \rangle}} \langle\ell,m\rangle}{\langle\ell,m\rangle}=\frac{i}{2\pi}\int_{X/S}\frac{\nabla^{\Mmini}m}{m}\wedge F_{\Lmini}+\tr_{\scriptscriptstyle{\Div m/S}}\left(\frac{\nabla^{\Lmini}\ell}{\ell}\right).
 \end{equation} 
\end{definition}
In the definition, the trace of a differential form defined in a neighborhood of $\Div m$ makes sense, provided the latter is finite \'etale over an open subset of $S$. The Deligne pairing $\langle L,M\rangle$ can be locally generated by symbols $\langle\ell,m\rangle$ where $\ell$ and $m$ satisfy this property, which can be subsumed in the general position assumption.

The next proposition provides an alternative description of the  intersection connection on the Deligne pairings, and relates it to Chern--Simons integrals. As a byproduct, it also proves that the above is a well-defined connection.



\begin{proposition}\label{prop:int-conn-Deligne}
Let $(L,\nabla^{\Lmini})$, $(M,\nabla^{\Mmini})$ be holomorphic line bundles with compatible connections on $X$. The intersection connection on $\langle L, M \rangle$ is the unique connection which satisfies the following properties: \begin{enumerate}
    \item it coincides with the Chern connection whenever $\nabla^{\Lmini}$ and $\nabla^{\Mmini}$ are Chern connections;
    \item if $\,\nabla^{\prime\ \Lmini} =  \nabla^{\Lmini} + \theta_{\Lmini}$ and $\,\nabla^{\prime\ \Mmini} = \nabla^{\Mmini} + \theta_{\Mmini}$, then 
    \begin{displaymath}
        \nabla^{\prime\ \scriptscriptstyle{\langle L, M \rangle}}=\nabla^{\scriptscriptstyle\langle L, M \rangle }+IT(L \oplus M,\nabla^{\Lmini} \oplus \nabla^{\Mmini},\nabla^{\prime\ \Lmini}  \oplus \nabla^{\prime\ \Mmini}).
    \end{displaymath}
\end{enumerate}
\end{proposition}
\begin{proof}
The unicity statement is clear, since we can always compare with the situation of Chern connections. 

For the first statement, suppose we are given hermitian metrics on $L$ and $M$. The Chern connection on $\langle L,M  \rangle$ in the frame $\langle \ell, m \rangle$ is $\partial \log\|\langle \ell, m \rangle \|^2 $, which is computed from the definition in \eqref{eq:gg-8}. It is readily seen to coincide with the expression in \eqref{eq:gg-10}. Notice that the cited formula is actually symmetric in $\ell,m$.

For the second statement, notice that by Lemma \ref{lemma:ics2whitney} 
\begin{equation}\label{eq:ics2sumlinebundles}
    IT(L \oplus M,\nabla^{\Lmini} \oplus \nabla^{\Mmini},\nabla^{\Lmini} + \theta_{\Lmini} \oplus \nabla^{\Mmini}+\theta_{\Mmini})  = \int_{X/S}\left( \theta_{\Mmini} \wedge c_1(L,\nabla^{\Lmini}) +\theta_{\Lmini} \wedge  c_1(M,\nabla^{\Mmini} + \theta_{\Mmini}) \right).
\end{equation}
By transitivity, it is enough to prove the statement when either $\theta_{\Lmini}$ or $\theta_{\Mmini}$ is zero. We carry out the case when $\theta_{\Mmini}=0$; the other case being similar. Then we can write 
\begin{displaymath}
    IT(L \oplus M,\nabla^{\Lmini} \oplus \nabla^{\Mmini},(\nabla^{\Lmini} + \theta_{\Lmini})\oplus \nabla^{\Mmini}) =  \frac{i}{2\pi}\int_{X/S} F_{\Mmini} \wedge \theta_{\Lmini} .
\end{displaymath}
On the other hand, the corresponding difference between the right hand side of \eqref{eq:gg-10} for the two inputs $\nabla^{\Lmini}$ and $\nabla^{\Lmini}+ \theta_{\Lmini}$ is
\begin{displaymath}
    \frac{i}{2\pi} \int_{X/S} \frac{\nabla^{\Mmini} m}{m} \wedge d\theta_{\Lmini} + \tr_{\scriptscriptstyle{\Div m/S}}\left(\theta_{\Lmini} \right).
\end{displaymath}
For type reasons, we have 
\begin{displaymath}
    \frac{i}{2\pi} \int_{X/S} \frac{\nabla^{\Mmini} m}{m} \wedge d\theta =\frac{i}{2\pi} \int_{X/S} d\left[\frac{\nabla^{\Mmini} m}{m}\right] \wedge \theta_{\Lmini}.
\end{displaymath}
We claim that the  equation of currents: $\displaystyle\frac{i}{2\pi}F_{\Mmini} = \delta_{\Div m} + \frac{i}{2\pi} d [\nabla^{\Mmini} m / m]$ holds, from which the required identification of the forms follow. In case the connection on $M$ is a Chern connection, this is the classical Poincar\'e--Lelong formula. Both sides transform in the same way upon modifying the connection $\nabla^{\Mmini}$ and hence hold generally, from which we conclude. 
\end{proof}
\subsubsection{On $IC_2$ bundles}\label{subsec:CS-integrals-IC2}
We introduce the intersection connection on $IC_2(E)$, given a compatible connection on $E$. We will use the logic of Proposition \ref{prop:int-conn-Deligne} to actually define the intersection connection on $IC_2$.
Let $E$ be a holomorphic vector bundle on $X$ and $\nabla\colon E\to E\otimes\Acal^{1,0}_{X}$ a connection. Fix  an auxiliary hermitian metric $h$ on $E$, with Chern connection $\nabla^{\hmini}$. Denote by $\nabla^{\hmini, \ICmini}$ the Chern connection of $IC_{2}(E,h)$. 

\begin{definition}\label{def:intersection-connection-general}
With the above notation and assumptions, the intersection connection on  $IC_2(E)$ is given by the formula
\begin{equation}\label{eq:gg-9}
    \nabla^{\ICmini}=\nabla^{\hmini, \ICmini}+IT(E,\nabla^{\hmini},\nabla).
\end{equation}

\end{definition}
Because $IT(E,\nabla^{\hmini},\nabla)$ is a differential form of type $(1,0)$, the connection $\nabla^{\ICmini}$ is compatible with the holomorphic structure. If $h^{\prime}$ is another metric on $E$, then by the relationship between Bott--Chern theory and transgression classes exhibited in Corollary \ref{lemma:chern-conn-1}, we have 
\begin{displaymath}
    \nabla^{\hmini^{\prime}, \ICmini}=\nabla^{\hmini,\ICmini}+IT(E,\nabla^{\hmini},\nabla^{\hmini^{\prime}}).
\end{displaymath}
It follows from the the cocycle property of Proposition \ref{prop:cocycle} that $\nabla^{\ICmini}$ in \eqref{eq:gg-9} is independent of the auxiliary metric. By construction, the curvature is computed as:
     \begin{equation}\label{eq:curv-IC2-jenesaisplus}
        c_{1}(IC_{2}(E),\nabla^{\ICmini})=\int_{X/S}c_{2}(E,\nabla) = \frac{1}{8\pi^{2}}\int_{X/S}\left(\tr\left(F_{\nabla}^{2}\right)-(\tr F_{\nabla})^{2}\right).
    \end{equation}

The next proposition is an alternative characterization of the intersection connection: 
\begin{proposition}\label{prop:basic-properties-int-conn}
The assignment $(E,\nabla)\mapsto (IC_{2}(E),\nabla^{\ICmini})$ is the unique one that satisfies the following properties.
\begin{enumerate}
   \item\label{item:nabla-IC2-1:a} (Normalization) If $E$ is a line bundle, the natural trivialization $IC_2(E) \simeq \Ocal_S$ identifies the intersection connection with the trivial connection on $\Ocal_S$.
   \item\label{item:nabla-IC2-1:b} (Whitney isomorphism) If $\varepsilon_\nabla$ is a short exact sequence of vector bundles with connections, the Whitney isomorphism 
   \begin{displaymath}
       IC_2(E) \to IC_2(E') \otimes IC_2(E'') \otimes \langle \det E', \det E'' \rangle \otimes (\Ocal_S, IT(\varepsilon_\nabla))
   \end{displaymath}
   is parallel. Here, the Deligne pairing is equipped with the intersection connection \eqref{eq:gg-10} and $(\Ocal_S, \theta)$  denotes the trivial line bundle equipped with the connection $d+\theta$.
    \item\label{item:nabla-IC2-1:c} (Functoriality) If $g\colon S^{\prime}\to S$ is a morphism of complex manifolds and $\nabla^{\prime}$ is the pull-back of $\nabla$ to $X\times_{S}S^{\prime}$, then $\nabla^{\prime\ \ICmini}=g^{\ast}\nabla^{\ICmini}$.
\end{enumerate}
\end{proposition}
\begin{remark}
The proposition includes as a special case that for two line bundles $L$ and $M$ with connections, the Whitney isomorphism $IC_2(L \oplus M) \to \langle L, M \rangle $ is parallel for the intersection connections. This can be proven independently using Proposition \ref{prop:int-conn-Deligne} and by reducing to the case of Chern connections. Hence, we could  have used this as a definition of the intersection connections on the Deligne pairings.
\end{remark}
\begin{proof}[Proof of Proposition \ref{prop:basic-properties-int-conn}]
For the unicity statement, remark that the connections are determined locally, and by \eqref{item:nabla-IC2-1:c} we can restrict ourselves to small open sets. The statement then follows by induction on the rank of the vector bundle, by a filtering argument applied to $E$. This filtering exists for sufficiently small open sets of $S$. We leave the details to the reader.

For the properties, we first notice that they are satisfied when $\nabla$ is a Chern connection. It is trivially true when $E$ is of rank one. For higher ranks, we again use the filtering argument and proceed by induction on the rank of $E$. To establish the Whitney isomorphism in the Chern connection setting, it is then enough to notice two facts: the intersection connection on the Deligne pairings is the Chern connection when the connections on $E'$ and $E''$ are Chern connections by Proposition \ref{prop:int-conn-Deligne}, and by  \emph{(MIC3)} in \textsection \ref{subsubsec:ic2metraxiom} and Corollary \ref{lemma:chern-conn-1} the Chern--Simons integral is the $\partial$ derivative of the integral of a Bott--Chern class.  

All the properties then follow from formal properties of the Chern--Simons integrals and by comparisons with the Chern connection case. The case of the Whitney isomorphism is less direct and we furnish details. By Proposition \ref{prop:ICS-explicit-formula} and Proposition \ref{prop:cocycle}, both sides transform the same way when modifying the connections. Since it holds in the case of Chern connections, it holds in general.  
\end{proof}







\begin{proposition}\label{prop:properties-nabla-IC2}
Let $(E,\nabla^{\Emini})$ be a holomorphic vector bundle with compatible connection on $X$. Then:
\begin{enumerate}
    \item\label{prop:properties-nabla-IC2-3}  If $(F, \nabla^{\Fmini})$ is also a holomorphic vector bundle with compatible connection, the natural isomorphism 
\begin{displaymath} IC_2(E \otimes F) \simeq IC_2(E)^{f} \otimes IC_2(F)^{e} \otimes \langle \det E, \det F \rangle^{e\cdot f-1} \otimes \langle \det E, \det E \rangle^{{f \choose 2}} \otimes \langle \det F, \det F \rangle^{{e \choose 2}}
\end{displaymath} 
from Proposition \ref{prop:whitneyproductwithbundle} is parallel for the connections induced from the intersection connections.
    \item The isomorphism $IC_2(E) \simeq IC_2(E^\vee)$ from Proposition \ref{prop:dual} is parallel for the intersection connections.
    \item The isomorphism $IC_2(\End E) \simeq IC_2(E)^{2e} \otimes \langle \det E, \det E \rangle^{-e}$ from Corollary \ref{cor:IC2-End} is parallel for the interesection connections;
\end{enumerate}
\end{proposition}
\begin{proof}
From Proposition \ref{prop:isometry_IC2-tensor} and Proposition \ref{lemma:IC2-dual-metric}, this follows from the case of Chern connections associated to metrics. Direct proofs by computations of the necessary Chern--Simons integrals to reduce to the case of Chern connections are tedious but straightforward. We remark rather that the proofs of the cited propositions provide the same type of formulas for transgression classes and Chern--Simons integrals, since they are all based on integrals of the form $\int_{\PBbb^1}  \log|t|^{-2} c_2( \widetilde{E}, \widetilde{\nabla})$, or their derivatives, as exploited in the discussion surrounding Corollary \ref{lemma:chern-conn-1}. Indeed, the referenced formulas do not require $ \widetilde{\nabla}$ to be a Chern connection but simply an interpolating connection.

The third item follows from the first two.



\end{proof}

The following obvious lemma is the counterpart of Lemma \ref{lemma:stupid-lemma} for intersection connections, and we state it without proof.

\begin{lemma}\label{lemma:stupid-lemma-2}
Let $g\colon X^{\prime}\to X$ be an isomorphism of relative curves over $S$. Let $(E,\nabla)$ be a vector bundle with compatible connection on $X$. Endow $g^{\ast}E$ with the pullback connection. Then, the canonical isomorphism $IC_{2}(g^{\ast}E)\simeq IC_{2}(E)$ is parallel for the induced intersection connections.
\end{lemma}
\qed

\section{Moduli spaces and intersection bundles}\label{section:universal-IC2}
In this section we review and elaborate on several moduli spaces of vector bundles in the relative setting. We address descent properties of intersection bundles and Deligne's isomorphism from representation spaces of rigidified objects to their GIT quotients.

Our exposition on moduli spaces is based on Seshadri \cite{Seshadri}, Drezet--Narasimhan \cite{Drezet-Narasimhan} and Simpson \cite{Simpson:moduli-1, Simpson:moduli-2}. The reader is referred to these sources for proofs of general facts. We will reserve detailed citations to justify those features which are less well-documented. To facilitate the comparison with Simpson's work, we adopt notation similar to his.  

Most of the time we work in the algebraic category, in order to be consistent with the literature. The constructions below have complex analytic counterparts that we will need. We shall indicate the necessary modifications at the end of this section \textsection \ref{subsec:complements-moduli-betti}.

\subsection{Moduli spaces of semistable and flat vector bundles}
Let $f\colon X\to S$ be a smooth projective morphism of schemes of finite type over $\CBbb$, with connected fibers of dimension one and genus $g\geq 2$.  

\subsubsection{Representation spaces}\label{subsub:repspacesdeRham}
Assume that $f$ has a fixed section $\sigma\colon S\to X$. Let $r\geq 1$ be an integer. We consider Simpson's representation spaces of the following types:
\begin{enumerate}
    \item[\textbullet] $\Rbold(X/S,\sigma,r) = $ \emph{moduli scheme of rigidified (slope) semistable vector bundles of rank $r$ and degree 0, up to isomorphism.} It represents the functor which, to a scheme $S^{\prime}\to S$, associates the set of isomorphism classes of $(E, \tau)$, where: 
    \begin{enumerate}
        \item $E$ is a vector bundle of rank $r$ on $X^{\prime}=X\times_{S}S^{\prime}$, whose restriction to fibers is slope semistable of degree 0;
        \item $\tau\colon\sigma^{\ast}E\overset{\sim}{\to}\Ocal_{S^{\prime}}^{\oplus r}$ is a rigidification of $E$ along $\sigma$; 
        \item morphisms of triples are isomorphisms of vector bundles compatible with the trivializations.
    \end{enumerate}
    The structure morphism $\Rbold(X/S,\sigma,r)\to S$ is quasi-projective. The fibers over the closed points of $S$ are irreducible, smooth and of a fixed dimension $e\geq 1$. There is a non-empty open subscheme $\Rbold^{\sst}(X/S,\sigma, r)$ parametrizing slope stable vector bundles.
   
   On $X\times_S\Rbold(X/S,\sigma,r)$ there is universal object denoted by $\Ecal^{\un}$, which is unique up to unique isomorphism. For simplicity, the rigidification is implicit in the notation.\bigskip
   
    \item[\textbullet] $\Rbold_{\dR}(X/S,\sigma,r)=$ \emph{moduli scheme of rigidified flat vector bundles of rank $r$, up to isomorphism.} With similar notation as in the previous item, it classifies triples $(E, \tau, \nabla)$, where $(E,\tau)$ is a rigidified vector bundle of rank $r$ on $X^{\prime}$ and $\nabla\colon E\to E\otimes\Omega^{1}_{X^{\prime}/S^{\prime}}$ is a relative connection. We say that $(E,\nabla)$ is a relative flat vector bundle. The structure morphism  $\Rbold_{\dR}(X/S,\sigma,r)\to S$ is quasi-projective, and the fibers over closed points are irreducible and normal. We will also consider:
        \begin{enumerate}
            \item $\Rbold^{\ir}_{\dR}(X/S,\sigma,r)$ the non-empty open subscheme parametrizing irreducible relative flat vector bundles. The morphism $\Rbold^{\ir}_{\dR}(X/S,\sigma,r)\to S$ has smooth fibers over closed points.
            \item  $\Rbold_{\dR}^{\sst}(X/S, \sigma, r)$ is the non-empty open subscheme of $\Rbold^{\ir}_{\dR}(X/S,\sigma,r)$ parametrizing slope stable vector bundles. Recall that stability is an open condition \cite[Lemma 3.7]{Simpson:moduli-1}. 
        \end{enumerate}
     On $X\times_S\Rbold_{\dR}(X/S,\sigma,r)$ there is a universal object, denoted by $(\Ecal^{\un}_{\dR},\nabla^{\un})$. The rigidification is implicit in the notation. Thus, $\Ecal^{\un}_{\dR}$ is a vector bundle on $X\times_{S}\Rbold_{\dR}(X/S,\sigma,r)$, and $\nabla^{\un}$ is a flat relative holomorphic connection on it.
\end{enumerate}
Forgetting the connection provides "a forgetful map" 
\begin{displaymath}
    \Rbold_{\dR}^{\sst}(X/S, \sigma, r)\to \Rbold^{\sst}(X/S, \sigma, r).
\end{displaymath}
The pullback of $\Ecal^{\un}$ by the forgetful map is isomorphic to $\Ecal^{\un}_{\dR}$, through a unique isomorphism preserving the rigidifications.

\subsubsection{Coarse moduli}\label{subsub:coarsemodulidR} 
The discussion in this section will be common to the the cases $\Rbold_{\dR}$ and $\Rbold$. We will only discuss the case of $\Rbold_{\dR}$. To translate to the setting of $\Rbold$, we need to replace $\Rbold_{\dR}^{\ir}$ by $\Rbold^{\sst}$.

There is an action of $\GL_{r/\CBbb}$ on $\Rbold_{\dR}(X/S,\sigma,r)$ corresponding to changing the rigidifications. The action preserves the fibers. All the points are GIT semistable with respect to a suitable linearization of the action. Therefore, we can take the GIT quotient $\Mbold_{\dR}(X/S,r)=\Rbold_{\dR}(X/S,\sigma, r)\doubleslash\GL_{r/\CBbb}$. This is a good quotient in the sense of Gieseker, and in particular a universal categorical quotient. It can be shown that $\Mbold_{\dR}(X/S,r)$ does not depend on the choice of section. Since sections exist locally for the \'etale topology, by a descent argument this actually allows to get rid of the assumption  of a section $\sigma$. The structure map $\Mbold_{\dR}(X/S,r)\to S$ is quasi-projective. The properly stable points in $\Rbold_{\dR}(X/S,\sigma,r)$ for the action of $\SL_{r/\CBbb}\subset\GL_{r/\CBbb}$ constitute exactly  $\Rbold_{\dR}^{\ir}(X/S,\sigma,r)$. The geometric quotient of the latter is a non-empty open subscheme $\Mbold_{\dR}^{\ir}(X/S,r)$ of $\Mbold_{\dR}(X/S,r)$. 

The scheme $\Mbold_{\dR}(X/S, r)$ is a coarse moduli scheme of polystable objects on the geometric fibers of $X\to S$, meaning completely reducible flat vector bundles in the case of $\Mbold_{\dR}(X/S, r)$. The universal vector bundle $\Ecal^{\un}_{\dR}$ does not descend to the coarse moduli scheme. However, \'etale locally with respect to $\Mbold_{\dR}^{\ir}(X/S,r)$, there is a vector bundle with connection $\Fcal^{\un}_{\dR}$ on $X\times_{S} \Mbold_{\dR}^{\ir}(X/S,r)$, which is universal up to twisting by a line bundle coming from $\Mbold_{\dR}^{\ir}(X/S,r)$ (also \'etale locally). 

In the case of de Rham representation spaces and stable vector bundles, we summarize the above discussion in the following diagram of natural morphisms, where the vertical arrows denote GIT quotients:

\begin{displaymath}
    \xymatrix{ \Rbold^{\sst}(X/S, \sigma, r) \ar@{->>}[d] & \ar@{->}[l] \Rbold_{\dR}^{\sst} (X/S, \sigma, r)\ar@{^{(}->}[r] \ar@{->>}[d] &  \Rbold^{\ir}_{\dR} (X/S, \sigma, r)\ar@{^{(}->}[r] \ar@{->>}[d]  & \Rbold_{\dR}(X/S, \sigma, r)   \ar@{->>}[d] \\
 \Mbold^{\sst}(X/S,  r)  & \ar@{->}[l] \Mbold_{\dR}^{\sst} (X/S, r) \ar@{^{(}->}[r] & \Mbold_{\dR}^{\ir} (X/S, r) \ar@{^{(}->}[r]  & \Mbold_{\dR}(X/S, r). }
\end{displaymath}
 
\subsubsection{The trivial determinant case} There are analogues of all the above where we further impose that vector bundles and connections have trivial determinant. In this case, we will write $\Rbold_{\dR}(X/S,\sigma,\SL_{r})$ and $\Rbold(X/S,\sigma,\SL_{r})$. 

For the sake of clarity, we indicate the necessary modifications for flat vector bundles. In this situation, $\Rbold_{\dR}(X/S,\sigma, \SL_{r})$ parametrizes quadruples $(\Ecal,\tau,\nabla,\rho)$, where $(\Ecal,\tau, \nabla)$ is an object of $\Rbold_{\dR}(X/S, \sigma, r)$ over a base scheme $S^{\prime}\to S$, and $\rho\colon\det\Ecal\overset{\sim}{\to}\Ocal_{X^{\prime}}$ is a trivialization. We require that the determinant of $\nabla$ corresponds to the trivial relative connection on $\Ocal_{X^{\prime}}$ via $\rho$. Furthermore, the rigidification $\tau$ is required to have determinant one with respect to the rgidification along $\sigma$. The coarse moduli $\Mbold_{\dR}(X/S,\SL_{r})$ is the good quotient of $\Rbold_{\dR}(X/S,\sigma, \SL_{r})$ by the action of $\SL_{r/\CBbb}$ on rigidifications. On $\Mbold_{\dR}^{\ir}(X/S,\SL_{r})$, a universal bundle exists \'etale locally, and it is unique up to twisting by an $r$-torsion line bundle coming from $S$. Forgetting $\rho$ induces a closed embedding $\Rbold_{\dR}(X/S,\sigma, \SL_{r})\hookrightarrow\Rbold_{\dR}(X/S, \sigma, r)$ and a natural morphism $\Mbold_{\dR}(X/S,\SL_{r})\to\Mbold_{\dR}(X/S, r)$. Universal objects on representation spaces pull-back to universal objects, and therefore no distinction in the notation will be made for those.  

\begin{remark}
The formation of the previous moduli spaces is obviously functorial with respect to isomorphisms of relative curves $X^{\prime}\to X$ over $S$. For instance, such an isomorphism induces a natural isomorphism $\Rbold_{\dR}(X^{\prime}/S,\sigma^{\prime},r)\simeq\Rbold_{\dR}(X/S,\sigma,r)$ over $S$, and a compatible natural isomorphism of the corresponding universal objects.
\end{remark}

\subsubsection{Schematic properties}
We state schematic features of relative moduli spaces. For lack of adequate reference, we provide complete proofs in the Appendix.\footnote{It is best to read the Appendix after \textsection \ref{subsec:Betti-moduli}.}  

\begin{proposition}\label{prop:flatness}
\begin{enumerate}
    \item\label{item:flatness-0} The formation of $\Rbold_{\dR}(X/S,\sigma,r)\to S$ and $\Mbold_{\dR}(X/S,r)\to S$ commutes with base change.
    \item\label{item:flatness-1} The structure morphisms $\Rbold_{\dR}(X/S,\sigma,r)\to S$ and $\Mbold_{\dR}(X/S,r)\to S$ are flat. The loci of irreducible connections are smooth over $S$.  
    \item\label{item:flatness-2} If $S$ is irreducible (resp. reduced), then the schemes $\Rbold_{\dR}(X/S,\sigma,r)$ and $\Mbold_{\dR}(X/S,r)$ are irreducible (resp. reduced).
\end{enumerate}
\noindent The analogues of \eqref{item:flatness-0}--\eqref{item:flatness-2} hold in the determinant one case.
\end{proposition}
\qed

\begin{proposition}\label{prop:flatness-2}
\begin{enumerate}
    \item\label{item:flatness-2-0} The formation of $\Rbold(X/S,\sigma,r)\to S$ and $\Mbold(X/S,r)\to S$ commutes with base change. 
    \item\label{item:flatness-2-3} The morphism $\Rbold(X/S,\sigma,r)\to S$ is smooth. 
    \item\label{item:flatness-2-4} The morphism $\Mbold(X/S, r)\to S$ is flat and the morphism $\Mbold^{\sst}(X/S,r)\to S$ is smooth.
    \item\label{item:flatness-2-2} If $S$ is irreducible (resp. reduced), the schemes $\Rbold(X/S,\sigma,r)$ and $\Mbold(X/S,r)$ are irreducible (resp. reduced).
\end{enumerate}
The analogues of \eqref{item:flatness-2-0}--\eqref{item:flatness-2-2} hold in the determinant one case.
\end{proposition}
\qed

\begin{remark}
The base change property stated in the propositions is mentioned in passing in \cite[p. 50]{Simpson:moduli-2}, and is an elementary consequence of the fact that $\Mbold(X/S,r)$ and $\Mbold_{\dR}(X/S,r)$ are universal categorical quotients of $\Rbold(X/S,\sigma,r)$ and $\Rbold_{\dR}(X/S,\sigma,r)$, respectively. 
\end{remark}

\subsection{Moduli spaces of representations of the fundamental group}\label{subsec:Betti-moduli}
Moduli spaces of flat vector bundles are related to moduli spaces of representations of the fundamental group, via the Riemann--Hilbert correspondence. The construction and properties of the latter partly hinge on topological considerations. The spaces of representations are easier to construct in the complex analytic category, which will be enough for our purposes. Hence, in this subsection we allow $f\colon X\to S$ to be a smooth proper morphism of complex analytic spaces with one-dimensional connected fibers of genus $g\geq 2$.

\subsubsection{Representation spaces}\label{subsection:Betti-Representation}
Let $\Gamma$ be a finitely generated group and $\mathsf{G}$ a complex linear reductive group. We  consider the representation space
\begin{displaymath}
    \Rbold(\Gamma, \mathsf{G}) = \Hom(\Gamma, \mathsf{G}(\CBbb)).
\end{displaymath}
It has a natural structure of complex, affine variety. We write 
\begin{displaymath}
     \Rbold(\Gamma, r)=\Rbold(\Gamma,\GL_{r}(\CBbb)).
\end{displaymath}

Now let $X_{0}$ be a compact Riemann surface and $p\in X_{0}$. The most important examples we shall consider are of the form 
$\Rbold_{\Bet}(X_{0}, p, \mathsf{G}) = \Rbold(\pi_1(X_{0},p), \mathsf{G})$. These are referred to as \emph{Betti spaces of representations of the fundamental group in} $\mathsf{G}$. Notice that the structure of complex analytic space does not depend on the complex structure of $X_{0}$. When $\mathsf{G}=\GL_{r}$, we shall simply write $\Rbold_{\Bet}(X_{0},p,r)$. The latter is irreducible and normal, and the locus classifying irreducible representations $\Rbold_{\Bet}^{\ir}(X_{0},p,r) $ is smooth. Similarly for $\SL_{r}$.

More generally, these definitions extend to the relative setting, provided $f\colon X\to S$ admits a section $\sigma$:
\begin{displaymath}
    \begin{split}
    \Rbold_{\Bet}(X/S, \sigma, \mathsf{G}) = &\emph{relative Betti representation space}\\
    &\emph{of representations of the fundamental group in } \mathsf{G}.
    \end{split}
\end{displaymath}
We will only need that the formation of $\Rbold_{\Bet}(X/S,\sigma,\mathsf{G})$ commutes with base change, and it has the structure of a local system of analytic spaces over $S$. If $S$ is algebraic, $\Rbold_{\Bet}(X/S,\sigma,\mathsf{G})$ is even a local system of schemes. Since the reader may not be familiar with these notions of local systems, we refer to Simpson \cite[pp. 12--14]{Simpson:moduli-2} for the precise definition and properties. When $\mathsf{G}=\GL_{r}$, the spaces of irreducible representations $\Rbold_{\Bet}^{\ir}(X/S,\sigma,r)$ inherit the structure of local systems of analytic spaces or schemes over $S$. Similary for $\SL_{r}$.

The above constructions are functorial for morphisms of linear reductive groups $\mathsf{G}\to\mathsf{H}$. For instance, we have induced natural maps $\Rbold_{\Bet}(X/S,\sigma, \mathsf{G})\to\Rbold_{\Bet}(X/S,\sigma,\mathsf{H})$.

\subsubsection{Coarse moduli}\label{subsub:coarseBetti} The group $\mathsf{G}$ acts by conjugation on the affine space $\Rbold(\Gamma, \mathsf{G})$, and we can consider the GIT quotient
\begin{displaymath}
    \Mbold(\Gamma, \mathsf{G}) = \Rbold(\Gamma, \mathsf{G})\doubleslash\mathsf{G}.
\end{displaymath}
Analogously, for a fixed Riemann surface $X_{0}$ and $p\in X_{0}$, consider the GIT quotient
\begin{displaymath}
    \Mbold_{\Bet}(X_{0}, \mathsf{G})  = \Rbold_{\Bet}(X_{0}, p, \mathsf{G})\doubleslash\mathsf{G}.
\end{displaymath}
As the notation suggests, the construction does not depend on the base point. When $\mathsf{G}=\GL_{r}$ we simply write $\Mbold_{\Bet}(X_{0},r)$. It classifies isomorphism classes of semisimple representations of the fundamental group of $X_{0}$. It is an irreducible and normal space, and the locus $\Mbold^{\ir}_{\Bet}(X_{0},r)$ classifying irreducible representations is smooth. Similarly for $\SL_{r}$.

Likewise, in the relative setting $X\to S$, when there exists a section $\sigma$ we define
\begin{displaymath}
    \Mbold_{\Bet}(X/S, \mathsf{G}) = \Rbold_{\Bet}(X/S, \sigma, \mathsf{G})\doubleslash\mathsf{G},
\end{displaymath}
where $\mathsf{G}$ acts fiberwise by conjugation on $\mathsf{G}$.\footnote{Even though the theory of GIT quotients can only be performed in the algebraic category, this definition works in the analytic category by reducing to the product situation as in \cite[\textsection 6]{Simpson:moduli-2}.} The resulting quotient is independent of the section, and by a descent argument the existence of a section is no longer required. 

The morphism $\Mbold_{\Bet}(X/S, \mathsf{G}) \to S$ has the structure of a local system of analytic spaces, and its formation commutes with base change. In particular, if $S$ is simply connected and $0\in S$ is a base point, there is a natural isomorphism of analytic spaces over $S$
\begin{displaymath}
    \Mbold_{\Bet}(X/S,\mathsf{G})\simeq\Mbold_{\Bet}(X_{0},\mathsf{G})\times S,
\end{displaymath}
whence a retraction
\begin{displaymath}
    p_{0}\colon\Mbold_{\Bet}(X/S,\mathsf{G})\longrightarrow\Mbold_{\Bet}(X_{0},\mathsf{G});
\end{displaymath}
see \cite[Lemma 6.2]{Simpson:moduli-2} and its proof. For $\mathsf{G}=\GL_{r}$ or $\SL_{r}$, the subspaces classifying irreducible representations inherit the structure of local systems.

Finally, if $\mathsf{G}\to\mathsf{H}$ is a morphism of reductive groups, there is an induced morphism $\Mbold_{\Bet}(X/S,\mathsf{G})\to\Mbold_{\Bet}(X/S,\mathsf{H})$ of local systems of analytic spaces.




We summarize our notation, for the case of representations $\Gamma=\pi_{1}(X_{0},p)\to\GL_{r}(\CBbb)$, in the following diagram:
\begin{displaymath}
    \xymatrix{
        \Rbold_{\Bet}^{\ir} (X_{0}, p, r) \ar@{^{(}->}[r] \ar@{->>}[d]  & \Rbold_{\Bet}(X_{0}, p, r)  = \Rbold(\Gamma, r) \ar@{->>}[d] \\
        \Mbold_{\Bet}^{\ir} (X_{0}, r) \ar@{^{(}->}[r]  & \Mbold_{\Bet}(X_{0}, r) =   \Mbold (\Gamma, r). 
    }
\end{displaymath}
Here the downwards arrows denote the GIT quotients. 

\subsubsection{Deformation theory of representations} \label{subsub:atiyahbottgoldman}
\hspace*{\fill} 
\bigskip

\noindent\emph{Tangent spaces and Atiyah--Bott--Goldman forms.} Consider the case of a compact Riemann surface $X_{0}$ of genus $g\geq 2$ with a base point $p$, and set as before $\Gamma=\pi_{1}(X_{0},p)$. Let $x \in \Mbold_{\Bet}^{\ir}(X_{0}, r)$, corresponding to the conjugacy class of an irreducible representation $\rho \colon\Gamma\to\GL_{r}(\CBbb)$. The fiber $T_{x}\Mbold_{\Bet}^{\ir}(X_{0},r)$ of the $\Ccal^{\infty}$ vector bundle $T_{\RBbb}\Mbold_{\Bet}^{\ir}(X_{0},r) \simeq T^{\scriptscriptstyle{(1,0)}}\Mbold_{\Bet}^{\ir}(X_{0},r)$ is canonically isomorphic to $H^{1}(\Gamma,\Ad(\rho))$. This can be described concretely. For a smooth 1-parameter family of irreducible representations  $\rho_t: \Gamma \to \GL_{r}(\CBbb)$ in $\Rbold_{\Bet}^{\ir}(X_{0}, p, r)$, with $\rho_{0}=\rho$, then $\kappa_t = d\rho_t \rho_t^{-1} \in Z^1(\Gamma, \Ad(\rho_t))$ is a $1$-cocycle of $\Gamma$ with values in $\Ad(\rho_t)$.  The cohomology class of $\kappa_{0}$ in $H^1(\Gamma, \Ad(\rho))$ is the tangent vector of $T_{x}\Mbold_{\Bet}^{\ir}(X_{0},r)$ corresponding to the 1-parameter family. See \cite[Chapter 2]{Lubotzky-Magid} and \cite[Theorem 10.4 \& Lemma 11.2]{Simpson:moduli-2} for further details.

The space $\Mbold_{\Bet}^{\ir}(X_{0},r)$ carries a canonical holomorphic 2-form, known as the Atiyah--Bott--Goldman symplectic form (cf. Goldman \cite[Section 1.7]{Goldman}). At a point corresponding to an irreducible representation $\rho$, it is given in terms of the skew-symmetric pairing $H^1(\Gamma, \Ad(\rho)) \times H^1(\Gamma, \Ad(\rho)) \to \CBbb$
\begin{equation}\label{eq:int-X0-cup-product}
    (\alpha, \beta) \mapsto \int_{X_{0}} \tr(\alpha \cup \beta),
\end{equation} 
where $\tr\colon H^{2}(\Gamma,\Ad(\rho))\to H^{2}(\Gamma,\CBbb)\simeq H^{2}(X_{0},\CBbb)$ is induced by the trace functional on matrices. Notice that while the complex structure of $\Mbold_{\Bet}^{\ir}(X_{0},r)$ does not depend on the complex structure of $X_{0}$, the integration map above involves the orientation of $X_{0}$. For instance, replacing $X_{0}$ by the conjugate Riemann surface would change the sign of the form. We will use the notation $\omega_{\scriptscriptstyle{\GL_{r}}}$ for the Atiyah--Bott--Goldman form, where the orientation will be implicit from the context. In the case of $\SL_{r}$ representations, we similarly have a holomorphic 2-form $\omega_{\scriptscriptstyle{\SL_{r}}}$.
\hspace*{\fill} 
\bigskip

\noindent\emph{Universal property of the Atiyah--Bott--Goldman form.}
In the theory of intersection connections and complex Chern--Simons line bundles to be developed later (Section \ref{section:canonical-extensions} and Section \ref{section:CS-theory}), the following universal property of the Atiyah--Bott--Goldman forms will be fundamental. For concreteness, we consider the case of $\omega_{\scriptscriptstyle{\GL_{r}}}$.

Let be given a smooth family of irreducible representations $\rho\colon\Gamma \to \GL_{r}(\Ccal^{\infty} (S))$, where $S$ is a complex manifold. There is an associated $\Ccal^{\infty}$ classifying map \begin{displaymath}
   \nu\colon S\to\Mbold_{\Bet}^{\ir}(X_{0},r).
\end{displaymath}
We wish to describe the differential $d\nu$ as a section of $\nu^\ast T^{\scriptscriptstyle{(1,0)}} \Mbold_{\Bet}(X_{0}, r)\otimes \Acal^{1}_{S}$. Introducing local holomorphic coordinates $s_i$ on $S$, $d\nu$ is obtained as the class of the cocycle $\kappa\in Z^{1}(\Gamma,\Ad(\rho))\otimes A^{1}(S)$, where 
\begin{equation}\label{eq:cocycle-kappa}
    \kappa(\gamma)=d\rho(\gamma)\ \rho(\gamma)^{-1}=\sum_{j}\frac{\partial\rho(\gamma)}{\partial s_{j}}\rho(\gamma)^{-1}ds_{j}+\sum_{j}\frac{\partial\rho(\gamma)}{\partial \ov{s}_{j}}\rho(\gamma)^{-1}d\ov{s}_{j}. 
\end{equation}
We define
\begin{equation}\label{eq:int-tr-dnu-cup-dnu}
    \int_{X_{0}}\tr(d\nu\cup d\nu):=\int_{X_{0}}\kappa\cup\kappa\in A^{2}(S),
\end{equation}
where the latter integral is evaluated by the following rule. If $\alpha = \sum_j \lambda_j \otimes \theta_j$, $\beta = \sum \lambda'_\ell \otimes \theta'_\ell$,  with $\lambda_j, \lambda_\ell' \in  H^{1}(\Gamma,\Ad(\rho))$ and $\theta_{j}, \theta_{\ell}' \in A^{1}(S)$, then
\begin{equation}\label{eq:sign-conv-cup}
    \int_{X_{0}} \tr(\alpha \cup \beta):=\left(\int_{X_{0}} \tr(\lambda_j \cup \lambda'_\ell) \right) \theta_j \wedge \theta'_\ell.
\end{equation}
See equation \eqref{eq:int-X0-cup-product} for the integral involving the $\lambda$'s. With this understood, the universal property is summarized in the following lemma.
\begin{lemma}\label{lemma:functorial-ABG}
With the notation as above, we have the equality of differential forms
\begin{displaymath}
    \frac{1}{2}\int_{X_{0}}\tr(d\nu\cup d\nu)=\nu^{\ast}\omega_{\scriptscriptstyle{\GL_{r}}}.
\end{displaymath}
A similar relationship holds in the $\SL_{r}$ case.
\end{lemma}
\begin{proof}
The proof is an easy exercise left to the reader. We only bring the reader's attention to the factor $2$, which arises due to the following basic fact. Let $\omega\colon V\times V\to\CBbb$ be an alternating 2-form on a finite dimensional $\CBbb$-vector space $V$. Let $\lbrace e_{j}\rbrace_{j}$ be a basis of $V$, with dual basis $\lbrace e_{j}^{\vee}\rbrace_{j}$. Then, $\omega$ is identified with
\begin{displaymath}
    \frac{1}{2}\sum_{j,k}\omega(e_{j},e_{k})\ e_{j}^{\vee}\wedge e_{k}^{\vee}\in \bigwedge\nolimits^{2} V^{\vee}.
\end{displaymath}
\end{proof}



\subsubsection{The Riemann--Hilbert correspondence}\label{subsubsec:RH} Assume again that $f\colon X\to S$ is algebraic. The Riemann--Hilbert correspondence, which associates to a flat holomorphic vector bundle its holonomy representation, yields an isomorphism of complex analytic spaces, over $S^{\an}$, $\Rbold_{\dR}(X/S,\sigma,r)^{\an}\simeq\Rbold_{\Bet}(X^{\an}/S^{\an},\sigma^{\an},r)$. The locus of irreducible connections and representations correspond via this isomorphism. The actions of $\GL_{r}$ are compatible as well, and there is an induced isomorphism of the GIT quotients $\Mbold_{\dR}(X/S,r)^{\an}\simeq\Mbold_{\Bet}(X^{\an}/S^{\an},r)$. The Riemann--Hilbert isomorphisms are compatible with base change. It is enough to justify this fact for the representations spaces. In this case, the claim is a consequence of the proof of the isomorphism \cite[Theorem 7.1]{Simpson:moduli-2}, which consists in showing that $\Rbold_{\dR}(X/S,\sigma,r)^{\an}$ and $\Rbold_{\Bet}(X^{\an}/S^{\an},\sigma^{\an},r)$ represent the same functor in the category of complex analytic spaces over $S$. Similar facts hold in the determinant one case. 

\begin{remark}
\begin{enumerate}
    \item Through the Riemann--Hilbert correspondence, the structure of local system of the Betti spaces corresponds to Simpson's Gauss--Manin connection of de Rham spaces \cite[Section 8]{Simpson:moduli-2}. The latter is formalized by saying that the de Rham spaces are crystals of schemes on $S$. While we will not need this description, this motivates the terminology \emph{crystalline} for Chern--Simons line bundles, introduced in Theorem \ref{thm:A} and elaborated in detail in Section \ref{section:CS-theory} below.
    \item Contrary to the irreducible locus, it is not clear that the stable locus of $\Rbold_{\dR}^{\sst}(X/S,\sigma,r)$ inherits the crystalline structure, since stability is not a topological condition. This seems to be inadvertently used in the approach to Hitchin's connection in \cite[p. 139]{Ramadas}.
\end{enumerate}
\end{remark}

\subsection{Universal $IC_{2}$ and Deligne's isomorphism}
We establish descent properties of $IC_{2}$ and Deligne's isomorphism for the universal objects on relative representation spaces. The approach is classical and is based on Kempf's descent lemma. 


\begin{proposition}\label{prop:descent-IC2}
\begin{enumerate}
    \item\label{item:descent-IC2-1} The line bundle $IC_{2}(\Ecal^{\un}_{\dR})$ on $\Rbold_{\dR}(X/S,\sigma,r)$ descends to $\Mbold_{\dR}(X/S,r)$. The formation of the descended object commutes with base change.
    \item\label{item:descent-IC2-2} The line bundle $IC_{2}(\Ecal^{\un})$ on $\Rbold(X/S,\sigma,r)$ descends to $\Mbold(X/S,r)$. The formation of the descended object commutes with base change.
\end{enumerate}
\noindent The analogues of \eqref{item:descent-IC2-1} and \eqref{item:descent-IC2-2} hold in the determinant one case.
\end{proposition}

\begin{proof}
 
 We will prove the first point, since the second one is addressed analogously. We first treat the case of an integral base scheme. The scheme $\Mbold_{\dR}(X/S, r)$ is the good quotient of $\Rbold_{\dR}(X/S,\sigma, r)$ by the reductive group $\GL_{r/\CBbb}$, and $\Rbold_{\dR}(X/S, \sigma, r)$ is an integral complex algebraic variety by Proposition \ref{prop:flatness}. Hence we can apply Kempf's descent lemma \cite[Th\'eor\`eme 2.3]{Drezet-Narasimhan}: it is enough to show that for every closed point $x\in\Rbold_{\dR}(X/S, \sigma, r)$ with closed orbit under $\GL_{r/\CBbb}$, the stabilizer $G$ of $x$ acts trivially on $IC_{2}(\Ecal^{\un}_{\dR})$. Let $s\in S$ be the closed point lying below $x$. Notice that both $x$ and $s$ are $\CBbb$-rational. Thus, $x$ corresponds to a rigidified flat vector bundle $(E,\nabla)$ on the complex projective curve $X_{s}$. That the orbit of $x$ is closed means that $(E,\nabla)$ is completely reducible. We decompose
\begin{displaymath}
    E=E_{1}^{\oplus k_{1}}\oplus\ldots \oplus E_{n}^{\oplus k_{n}},
\end{displaymath}
where each $E_{i}$ is a vector bundle of rank $r_{i}$, preserved by $\nabla$, and such that $(E_{i},\nabla)$ is irreducible. Furthermore, we suppose that the $(E_{i},\nabla_{i})$ are pairwise non-isomorphic. Then, for distinct $i$ and $j$ we have $\Hom((E_{i},\nabla),(E_{j},\nabla))=0$, while $\Aut((E_{i},\nabla)^{\oplus k_{i}})=\GL_{k_{i}}(\CBbb)$. Accordingly, we have an identification 
\begin{displaymath}
    G\simeq \GL_{k_{1}}(\CBbb)\times\ldots\times\GL_{k_{n}}(\CBbb).
\end{displaymath}
We need to show that the latter acts trivially on
\begin{displaymath}
    IC_{2}(E)\simeq\bigotimes_{i} IC_{2}(E_{i}^{\oplus k_{i}})\otimes\bigotimes_{i<j}\langle\det E_{i}^{\oplus k_{i}},\det E_{j}^{\oplus k_{j}}\rangle. 
\end{displaymath}
The action of $G$ on the Deligne pairings is trivial. Indeed, an automorphism of $\det E_{i}^{\oplus k_{i}}$ given by multiplication by $u\in\CBbb^{\times}$ acts on $\langle\det E_{i}^{\oplus k_{i}},\det E_{j}^{\oplus k_{j}}\rangle$ as multiplication by $u^{k_{j}\deg E_{j}}$, by Proposition \ref{Prop:generalpropertiesDeligneproduct}. This equals 1 because $\deg E_{j}=0$. For $IC_{2}$, we notice that $G$ acts by a character, hence a power of the determinant. The power is determined as in the case Deligne pairings, and it is zero by Proposition \ref{prop:relationic2symbols}. This shows that  $IC_{2}(\Ecal^{\un}_{\dR})$ descends. 

Before treating the case of a general base scheme, we make some preliminary remarks. Suppose that for some base scheme $S$, non-necessarily integral, the descent property of $IC_{2}(\Ecal_{\dR}^{\un})$ has been shown. Denote the descended line bundle by $\Lcal$. The first observation is that the descent property means that $\VBbb(IC_{2}(\Ecal^{\un}_{\dR}))$, the bundle scheme associated to $IC_{2}(\Ecal^{\un}_{\dR})$, admits a good GIT quotient, given by $\VBbb(\Lcal)$. See \cite[Lemme 2.2 \& Remarque in p. 66]{Drezet-Narasimhan}. The second observation concerns base change. Let $S^{\prime}\to S$ be any morphism of schemes of finite type over $\CBbb$, and decorate with primes the base changed objetcs to $S^{\prime}$. We claim that $IC_{2}(\Ecal^{\un\ \prime}_{\dR})$ on $\Rbold_{\dR}(X^{\prime}/S^{\prime},\sigma^{\prime},r)$ descends to $\Mbold_{\dR}(X^{\prime}/S^{\prime},r)$, and the descended line bundle is naturally isomorphic to $\Lcal^{\prime}$. But this is inferred from the fact that $\VBbb(\Lcal)$ is a universal categorical quotient of $\VBbb(IC_{2}(\Ecal^{\un}_{\dR}))$, so that its formation commutes with base change, together with the compatibility of $IC_{2}(\Ecal^{\un}_{\dR})$ with base change. See the proof of Proposition \ref{prop:flatness} \eqref{item:flatness-0} in the Appendix.

We can now treat the case of a general base scheme. We proceed as in the proof of Proposition \ref{prop:flatness-2} in the Appendix, and argue through Hilbert schemes. First, suppose that $f\colon X\to S$, with its section $\sigma$, is a tri-canonically embedded one-pointed smooth curve. Hence, $f_{\ast}(\omega_{X/S}(\sigma)^{\otimes 3})$ is trivial and $f$ factors through a closed embedding $X\hookrightarrow\PBbb(f_{\ast}(\omega_{X/S}(\sigma)^{\otimes 3}))\simeq\PBbb_{S}^{5g-2}$. We have a corresponding classifying morphism $S\to H_{g,1}$, where $H_{g,1}$ is the Hilbert scheme of tri-canonically embedded one-pointed smooth curves of genus $g$ over $\CBbb$ \cite[p. 210]{Knudsen:proj-III}. The scheme $H_{g,1}$ is integral by an immediate extension of \cite[proof of Proposition 5.3]{GIT}. We can apply the previous descent and base change results to $\Rbold_{\dR}(\Xcal/H_{g,1},\xi,r)$ and $\Mbold_{\dR}(\Xcal/H_{g,1},r)$, where $\Xcal\to H_{g,1}$ is the universal curve and $\xi$ the universal section. By the previous paragraph, we deduce by base change the desired descent property for $IC_{2}(\Ecal^{\un}_{\dR})$ on $\Rbold_{\dR}(X/S,\sigma,r)$. In general, $f\colon X\to S$, with its section, is only tri-canonically one-pointed embedded, locally with respect to $S$. Hence, the desired descent property holds locally with respect to $S$. This globalizes, by the base change property established in the previous paragraph. This completes the proof.

\end{proof}

\begin{definition}
 We call the descended bundles of Proposition \ref{prop:descent-IC2} the universal intersection bundles, or universal $IC_{2}$ bundles. We still denote them $IC_{2}(\Ecal^{\un})$ and $IC_{2}(\Ecal^{\un}_{\dR})$.
 \end{definition}

 \begin{remark}
Via the natural maps $\Mbold_{\dR}(X/S,\SL_{r})\to\Mbold_{\dR}(X/S, r)$, the universal intersection bundles pull-back to the universal intersection bundles.
\end{remark}

\begin{corollary}\label{cor:descend-IC2-EF}
 Let $F$ be a fixed vector bundle on $X$. In the determinant one case, the intersection bundle $IC_{2}(\Ecal^{\un}_{\dR}\otimes F)$ (resp. $IC_{2}(\Ecal^{\un}\otimes F)$) descends to $\Mbold_{\dR}(X/S,\SL_{r})$ (resp. $\Mbold(X/S,\SL_{r})$). The construction commutes with base change.
\end{corollary}

\begin{proof}
We treat the case of $\Ecal^{\un}_{\dR}$. Let $f$ be the rank of $F$. By Proposition \ref{prop:whitneyproductwithbundle} and the determinant one condition, we have a natural isomorphism
\begin{displaymath}
    IC_{2}(\Ecal^{\un}_{\dR}\otimes F)\simeq IC_{2}(\Ecal^{\un}_{\dR})^{f}\otimes IC_{2}(F)^{r}\otimes\langle\det F,\det F\rangle^{\binom{r}{2}}.
\end{displaymath}
It is compatible with the $\SL_{r}$ actions by the functoriality of the isomorphism. By Proposition \ref{prop:descent-IC2}, we know that $IC_{2}(\Ecal^{\un}_{\dR})$ descends. For the intersection bundles involving $F$ there is nothing to say, since they are actually defined on $S$. The compatibility with base change follows from that of $IC_{2}(\Ecal^{\un}_{\dR})$.
\end{proof}
 
 For the next proposition, let $F$ be a fixed vector bundle of rank $f$ on $X$, and recall from \eqref{def:DelIC2} and Theorem \ref{thm:comparisonElkDel} the functorial isomorphism relating the $IC_{2}$ bundle and the determinant of the cohomology. Specializing for instance to $\Ecal^{\un}_{\dR}\otimes F$ on $\Rbold_{\dR}(X/S,\sigma,\SL_{r})$, we obtain a natural isomorphism, compatible with the $\SL_{r}$ actions 
 \begin{equation}\label{eq:iso-IC2-E-F-lambda}
    IC_{2}(\Ecal^{\un}_{\dR}\otimes F)\simeq\lambda(\Ecal^{\un}_{\dR}\otimes F)^{-1}\otimes\lambda((\det F)^{r})\otimes\lambda(\Ocal_{X})^{r\cdot f-1}.
 \end{equation}
Recall also the Deligne--Riemann--Roch isomorphism of Theorem \ref{DRR:iso-general}.
\begin{proposition}\label{prop:Deligne-iso-descends}
\begin{enumerate}
    \item\label{item:descent-deligne-1} The universal determinant line bundle $\lambda(\Ecal_{\dR}^{\un}\otimes F)$ on $\Rbold_{\dR}(X/S, \sigma, \SL_{r})$ descends to $\Mbold_{\dR}(X/S,\SL_{r})$.
    \item\label{item:descent-deligne-2} The isomorphism \eqref{eq:iso-IC2-E-F-lambda} descends to $\Mbold_{\dR}(X/S,\SL_{r})$.
    \item Deligne's isomorphism for $\Ecal_{\dR}^{\un}\otimes F$ descends to $\Mbold_{\dR}(X/S,\SL_{r})$. 
\end{enumerate}
\noindent The analogues hold for $\Ecal^{\un}\otimes F$ on $\Rbold(X/S, \sigma, \SL_{r})$.
\end{proposition}
\begin{proof}
The first two items are proven at once, similarly to Corollary \ref{cor:descend-IC2-EF}. For the third item, we recall that Deligne's isomorphism is obtained as a combination of $\lambda(\Ocal_{X})^{12}\simeq\langle\omega_{X/S},\omega_{X/S}\rangle$ and \eqref{eq:iso-IC2-E-F-lambda}, which both descend. Finally, the case of $\Ecal^{\un}$ is analogous.
\end{proof}

\begin{remark}
In the proposition, the descended isomorphisms necessarily commute with base change. Since we will not need this fact, we leave this as an exercise to the interested reader.
\end{remark}

\subsubsection{Descent on the stable and irreducible loci}\label{subsubsec:descent-stable-locus}

The descent statements of Proposition \ref{prop:descent-IC2} and Proposition \ref{prop:Deligne-iso-descends} admit an alternative treatment if we restrict to the stable or irreducible loci, based on the local existence of universal vector bundles and the functorial properties of intersection bundles. This approach will be important in geometric differential considerations in Section \ref{section:CS-theory}. We provide the main lines of the argument for $IC_{2}(\Ecal^{\un}_{\dR})$ restricted to $\Rbold^{\ir}_{\dR}(X/S,\sigma, r)$, and leave details and other instances to the reader.

Locally with respect to the \'etale topology on $\Mbold^{\ir}_{\dR}(X/S, r)$, there is a universal vector bundle $\Fcal^{\un}_{\dR}$, unique modulo twising by a line bundle coming from the base. We form $IC_{2}(\Fcal^{\un}_{\dR})$. For any line bundle $L$ on a suitable \'etale open of $\Mbold^{\ir}_{\dR}(X/S, r)$, it follows from Proposition \ref{prop:tensorlinebundleiso} that we have a canonical functorial isomorphism
\begin{displaymath}
    IC_{2}(\Fcal^{\un}_{\dR}\otimes\pi^{\ast}L)\simeq IC_{2}(\Fcal^{\un}_{\dR}) \otimes \langle \det \Fcal^{\un}_{\dR}, \pi^* L  \rangle^{r-1} \otimes \langle \pi^* L , \pi^* L\rangle^{r \choose 2} ,
\end{displaymath}
where $\pi\colon X\times_{S}\Mbold^{\ir}_{\dR}(X/S, r)\to \Mbold^{\ir}_{\dR}(X/S, r)$ is the natural morphism, or a localization thereof in the \'etale topology of the base. Since for any line bundle $M$ we have $\langle M, \pi^* L \rangle \simeq L^{\deg M}$, it follows that the latter two line bundles are canonically trivial, hence reducing to a natural isomorphism 
\begin{equation}\label{eq:IC2-Fun-otimes-L-descent}
  IC_{2}(\Fcal^{\un}_{\dR}\otimes\pi^{\ast}L)\simeq IC_{2}(\Fcal^{\un}_{\dR}).
\end{equation}
By Proposition \ref{prop:tensorlinebundleiso} and the functoriality of $IC_{2}$ under base change, it can be checked that the collection of such $IC_{2}(\Fcal^{\un}_{\dR})$ satisfies the \'etale descent (\emph{i.e.} gluing) conditions \cite[Th\'eor\`eme 1.1]{SGA1}. Let us momentarily denote  by $IC_{2}(\Fcal^{\un}_{\dR})$ the resulting line bundle on $\Mbold_{\dR}^{\ir}(X/S, r)$. Consider the quotient map $p\colon\Rbold^{\ir}_{\dR}(X/S, \sigma, r)\to\Mbold^{\ir}_{\dR}(X/S, r)$. We still have to construct a natural isomorphism $p^{\ast}IC_{2}(\Fcal^{\un}_{\dR})\simeq IC_{2}(\Ecal^{\un}_{\dR})$. This exists \'etale locally, because $\Fcal^{\un}_{\dR}$ pulls-back to $\Ecal^{\un}_{\dR}$ up to twisting by a line bundle coming from $\Rbold^{\ir}_{\dR}(X/S, \sigma, r)$, and by equation \eqref{eq:IC2-Fun-otimes-L-descent}. One obtains a collection of isomorphisms defined \'etale locally over $\Rbold^{\ir}_{\dR}(X/S,\sigma, r)$, and again the functoriality properties of $IC_{2}$ ensure descent.

It is similarly checked that the descended $IC_{2}(\Ecal^{\un}_{\dR})$ on $\Mbold^{\ir}_{\dR}(X/S, r)$ obtained by this method  coincides with the one provided by Proposition \ref{prop:descent-IC2}. For this, we use that $\Rbold^{\ir}_{\dR}(X/S, \sigma, r)\to\Mbold^{\ir}_{\dR}(X/S, r)$ is a principal $\PSL_{r/\CBbb}$-fibration. Therefore, if $U\to\Mbold^{\ir}_{\dR}(X/S, r)$ is a surjective \'etale cover and $U^{\prime}$ is the pull-back to $\Rbold^{\ir}_{\dR}(X/S, \sigma, r)$, then the composition $U^{\prime}\to\Mbold^{\ir}_{\dR}(X/S, r)$ is faitfully flat and quasi-compact, hence a morphism of effective descent for quasi-coherent sheaves \cite[Expos\'e VIII, Th\'eor\`eme 1.1]{SGA1}.

The base change functoriality is addressed in an analogous manner. Let $q\colon S^{\prime}\to S$ be a morphism of schemes of finite type over $\CBbb$. We still denote by $q$ the natural induced morphism $\Mbold_{\dR}^{\ir}(X^{\prime}/S^{\prime}, r)\to \Mbold_{\dR}^{\ir}(X/S, r)$ and similarly for the universal curves. These morphisms exist by the base change property of the moduli schemes. If $\Fcal^{\un}$ is a local universal object on $X\times_{S}\Mbold_{\dR}^{\ir}(X/S, r)$, then $q^{\ast}\Fcal^{\un}_{\dR}\simeq\Fcal^{\un\ \prime}_{\dR}\otimes \pi^{\prime\ \ast}L$, for some local universal object $\Fcal^{\un\ \prime}_{\dR}$ on $X^{\prime}\times_{S^{\prime}}\Mbold_{\dR}^{\ir}(X^{\prime}/S^{\prime}, r)$ and some locally \'etale defined line bundle $L$ on $\Mbold_{\dR}^{\ir}(X^{\prime}/S^{\prime}, r)$. Reasoning as above and taking into account the base change functoriality of the intersection bundles, we find a natural isomorphism $q^{\ast}IC_{2}(\Fcal^{\un}_{\dR})\simeq IC_{2}(\Fcal^{\un\ \prime}_{\dR})$. These isomorphisms glue together again by  functoriality.
\bigskip

\subsection{Complements in the complex analytic setting}\label{subsec:complements-moduli-betti}
Simpson's construction of the relative moduli schemes of semistable and flat bundles spaces is only well-documented in the algebraic category.\footnote{To that extent, see the first phrase in the proof of \cite[Proposition 5.3]{Simpson:moduli-1}.} However, in the theory of complex Chern--Simons line bundles, we will need an extension of the results of this section to the complex analytic setting. We now outline how to carry out this extension. This might also be useful for future reference. We will conclude with some complements on Betti moduli spaces for the group $\PSL_{r}$.

\subsubsection{Relative moduli of semistable and flat vector bundles}
We construct relative moduli of semistable and flat vector bundles in the complex analytic category. The method proceeds by reduction to the algebraic case, similar to that in the proof of Proposition \ref{prop:properties-Deligne-analytic} in \textsection \ref{subsec:Analytification}, for Deligne pairings. In particular, as in \emph{loc. cit.} we systematically use that the analytifications of Hilbert and $\Hom$ schemes represent the appropriate functors in the complex analytic category. We recall this is a consequence of the corresponding fact for Quot schemes, proven by Simpson in \cite[Proposition 5.3]{Simpson:moduli-1}. Although we focus on families of compact Riemann surfaces, the discussion we deliver is valid for more general smooth families of polarized, projective varieties. See Remark \ref{rmk:extension-moduli-general-dimension}. We will treat the case of the relative moduli of semistable bundles $\Rbold(X/S,\sigma,r)$ and $\Mbold(X/S,r)$. The de Rham spaces, and the $\SL_{r}$ variants, are dealt with analogously.

Let $f\colon X\to S$ be a smooth projective morphism of complex analytic spaces, whose fibers are connected compact Riemann surfaces of genus $g\geq 2$.\footnote{The condition on the genus is imposed for consistency with the rest of this section, but actually not necessary.} For simplicity, we assume that $f\colon X\to S$ admits a section $\sigma$. We wish to construct $\Rbold(X/S,\sigma,r)$ and $\Mbold(X/S,r)$, in the complex analytic category, which enjoy of analogous properties as their algebraic geometric counterparts. We seek a construction which is compatible with the analytification functor. 

We begin by supposing that $L$ is a relatively very ample line bundle for $X\to S$, inducing a factorization through a closed immersion $X\hookrightarrow\PBbb^{N}_{S}$. We denote by $P$ the Hilbert polynomial of the fibers $X_{s}$. We thus have a classifying map $S\to\Hcal^{\an}$, where $\Hcal$ is the open subscheme of $\mathrm{Hilb}_{\PBbb^{N}_{\CBbb}}^{P}$ where the universal family $\Xcal\to\Hcal$ is smooth with irreducible fibers \cite[Th\'eor\`eme 12.2.4 (iii) \& (viii)]{EGAIV3}. Let $\Scal\to\Hcal$ be the scheme of sections of $\Xcal\to\Hcal$. We will still denote by $\Xcal$ the universal family over $\Scal$, and $\xi$ for its universal section. Associated to $X\to S$ and the section $\sigma\colon S\to X$, there is a lift $\varphi\colon S\to\Scal^{\an}$ of the classifying map $S\to\Hcal^{\an}$. We may form $\Rbold(\Xcal/\Scal,\xi,r)$ and $\Mbold(\Xcal/\Scal, r)$. After Simpson \cite[Proposition 5.5]{Simpson:moduli-2}, the associated morphism of complex analytic spaces $\Rbold(\Xcal/\Scal,\xi,r)^{\an}\to\Mbold(\Xcal/\Scal,r)^{\an}$ is a universal categorical quotient of complex analytic spaces, for the natural action of $\GL_{r/\CBbb}$. It is therefore natural to define $\Rbold_{L}(X/S,\sigma,r)$ and $\Mbold_{L}(X/S,r)$ by base changing $\Rbold(\Xcal/\Scal,\xi,r)^{\an}$  and $\Mbold(\Xcal/\Scal,r)^{\an}$ to $S$ via $\varphi\colon S\to\Scal^{\an}$. We recorded the a priori dependence on $L$ in the notation. By construction, the natural morphism $\Rbold_{L}(X/S,\sigma,r)\to\Mbold_{L}(X/S,r)$ is a universal categorical quotient for the action of $\GL_{r/\CBbb}$, and in particular its formation commutes with base change over $S$. See the proof of Proposition \ref{prop:flatness} \eqref{item:flatness-0}.

Suppose now that we are given relatively very ample line bundles $L_{1}$ and $L_{2}$ as above. We claim that there is a canonical isomorphism $\psi_{21}\colon \Rbold_{L_{1}}(X/S,\sigma,r)\to\Rbold_{L_{2}}(X/S,\sigma,r)$, and similarly with the $\Mbold$ versions. For each $L_{i}$ we repeat the constructions with Hilbert schemes as above, and decorate them with the corresponding index. We base change the universal families and sections to $\Scal_{1}\times\Scal_{2}$, and maintain the notation for the base changed objects. We next introduce the scheme $\mathrm{Isom}_{\Scal_{1}\times\Scal_{2}}^{Q}(\Xcal_{1},\Xcal_{2})$, where the universal families become isomorphic. Here $Q$ is the Hilbert polynomial of $L_{1}\otimes L_{2}$. It contains a closed subscheme $\Scal_{12}$ defined by the condition that the universal sections correspond through the universal isomorphism $h_{21}\colon\Xcal_{1}\to\Xcal_{2}$, that is $\xi_{2}=h_{21}\circ\xi_{1}$. Associated to the identity isomorphism $X\to X$, the section $\sigma$ and the polarizations $L_{1}, L_{2}$, we have a classifying map $\varphi_{12}\colon S\to\Scal^{\an}_{12}$, such that composing with the natural projections $p_{i}\colon \Scal^{\an}_{12}\to\Scal_{i}^{\an}$ recovers the classifying map $\varphi_{i}$ associated to $L_{i}$, as defined in the previous paragraph. The desired isomorphism $\psi_{12}$ is defined from the sequence of natural isomorphisms
\begin{displaymath}
    \begin{split}
         \Rbold_{L_{1}}(X/S,\sigma,r)=&\varphi_{1}^{\ast}\Rbold(\Xcal_{1}/\Scal_{1},\xi_{1},r)^{\an}\\
    &\simeq\varphi_{12}^{\ast}p_{1}^{\ast}\Rbold(\Xcal_{1}/\Scal_{1},\xi_{1},r)^{\an}\\
    &\hspace{0.5cm} \simeq \varphi_{12}^{\ast}\Rbold(\Xcal_{1}/\Scal_{12},\xi_{1},r)^{\an}\\
    &\hspace{1cm}\simeq \varphi_{12}^{\ast}\Rbold(\Xcal_{2}/\Scal_{12},\xi_{2},r)^{\an}\\
    &\hspace{1.5cm}\simeq\varphi_{12}^{\ast}p_{2}^{\ast}\Rbold(\Xcal_{2}/\Scal_{2},\xi_{2},r)^{\an}\\
    &\hspace{2cm}\simeq\varphi_{2}^{\ast}\Rbold(\Xcal_{2}/\Scal_{2},\xi_{2},r)^{\an}\\
    &\hspace{2.5cm}=\Rbold_{L_{2}}(X/S,\sigma,r).
    \end{split}
\end{displaymath}
If $L_{3}$ is yet a third relatively very ample line bundle, we have morphisms $\psi_{12}$, $\psi_{13}$ and $\psi_{23}$. We need to check that $\psi_{31}=\psi_{32}\circ\psi_{21}$. We explain the key point. With the notation as above, we base change the schemes $\Scal_{ij}$ to $\Scal_{1}\times\Scal_{2}\times\Scal_{3}$, and then perform their fiber product over $\Scal_{1}\times\Scal_{2}\times\Scal_{3}$. We denote this by $\Wcal$. We maintain the notation for the base change of the universal curves to $\Wcal$. We have universal isomorphisms $h_{ij}\colon\Xcal_{j}\to\Xcal_{i}$ over $\Wcal$. It is not clear from the outset that $h_{31}=h_{32}\circ h_{21}$. Nonetheless, this relationship determines a closed subscheme $\Scal_{123}$ of $\Wcal$. We notice that we still have a classifying map $S\to\Scal^{\an}_{123}$ associated to $X\to S$, $\sigma$, the identity map $X\to X$, etc., such that composing with the projections $\Scal^{\an}_{123}\to\Scal_{ij}^{\an}$ (resp. $\Scal^{\an}_{123}\to\Scal_{i}^{\an}$) we recover the classifying maps $\varphi_{ij}$ (resp. $\varphi_{i}$). From this, it is an exercise to conclude.

Over a general base $S$, we can find an open covering $U_{i}$ of $S$, such that $f_{\mid {U_{i}}}$ admits a relatively very ample line bundle $L_{i}$ as above. We can form $\Rbold_{L_{i}}(X_{U_{i}}/U_{i},\sigma,r)$ over $U_{i}$. By the previous paragraph, these glue together into a complex space $\Rbold(X/S,\sigma,r)$. From \cite[Lemma 5.7]{Simpson:moduli-1}, one infers that $\Rbold(X/S,\sigma,r)$ represents the appropriate moduli functor in the complex analytic category. In particular, it does not depend on choices, up to unique isomorphism. We can reason similarly to define $\Mbold(X/S,r)$. By construction, there is a natural morphism $\Rbold(X/S,\sigma,r)\to\Mbold(X/S,r)$, which defines a universal categorical quotient for the action of $\GL_{r/\CBbb}$. Consequently, $\Mbold(X/S,r)$ is also independent of choices, up to unique isomorphism, and commutes with base change. Finally, the construction is such that the formation of $\Rbold$ and $\Mbold$ is compatible with the analytification functor. 

Reasoning through Hilbert schemes, we see that all the properties and constructions discussed in this section carry over to the complex analytic setting. The key point is the compatibility with base change in the statements. In particular, we have Riemann--Hilbert isomorphisms. These allow to transfer to the Betti spaces all the results regarding de Rham spaces, their universal objects and the associated intersection bundles. In the complex analytic category, we can thus indistinctly work either with de Rham spaces or with Betti spaces.

\begin{remark}\label{rmk:extension-moduli-general-dimension}
In \cite{Simpson:moduli-1} Simpson constructs other relative moduli schemes in the algebraic category. Under quite general assumptions, the previous method adapts verbatim to extent his constructions to the complex analytic category. Assume that $f\colon X\to S$ is a flat projective morphism of complex analytic spaces, with irreducible fibers.\footnote{Irreducibility is assumed in Simpson's construction of representation spaces \cite[pp. 104--105]{Simpson:moduli-1}.} Suppose for simplicity that we are given a section $\sigma\colon S\to X$. Let $L$ be a fixed relatively very ample line bundle, and $Q$ a Hilbert polynomial, needed to define the stability conditions with respect to $L$.  Then, with the notation as in \cite{Simpson:moduli-1}:
\begin{enumerate}
    \item If moreover $f\colon X\to S$ has reduced fibers, we can construct the relative moduli spaces $\Rbold(X/S,\sigma,Q)$ and $\Mbold(X/S,Q)$, with analogous properties as in the algebraic setting. For this, we proceed as in the above discussion for families of Riemann surfaces, and we work with the Hilbert scheme associated to $X\to S$ and $L$, denoted by $\mathrm{Hilb}_{\PBbb^{N}_{\CBbb}}^{P}$. A key point now is that the locus where the universal family has geometrically integral fibers is open in $\mathrm{Hilb}_{\PBbb^{N}_{\CBbb}}^{P}$, by \cite[Th\'eor\`eme 12.2.4 (viii)]{EGAIV3}. This reasoning no longer works if we only assume that $f\colon X\to S$ has irreducible fibers, since the locus in $\mathrm{Hilb}_{\PBbb^{N}_{\CBbb}}^{P}$ where the universal family has irreducible fibers is in general just locally constructible \cite[Th\'eor\`eme 9.7.7 (i)]{EGAIV3}, and hence has no natural schematic structure.
    \item If moreover $f\colon X\to S$ has smooth fibers, we can construct the relative moduli spaces $\Rbold_{\dR}(X/S,\sigma,Q)$ and $\Mbold_{\dR}(X/S,Q)$. Similarly for other relative moduli spaces, such as those of Deligne's $\lambda$-connections.
\end{enumerate}
\end{remark}

\subsubsection{Betti spaces for $\PSL_{r}$}\label{subsubsec:Betti-spaces} In Section \ref{section:CS-applications}, we will have need for $\Mbold_{\Bet}(X/S,\PSL_{r})$ and the distinguished component of locally liftable representations. This is defined as the image of the natural map $\Mbold_{\Bet}(X/S,\SL_{r})\to\Mbold_{\Bet}(X/S,\PSL_{r})$, and denoted by $\Mbold_{\Bet}(X/S,\PSL_{r})_{\ell}$. This is a finite \'etale quotient of $\Mbold_{\Bet}(X/S,\SL_{r})$. Indeed, the $r$-torsion part of the relative Jacobian, $\Jac(X/S)[r]$ acts properly and freely on $\Mbold_{\Bet}(X/S,\SL_{r})$, and the natural morphism $\Mbold_{\Bet}(X/S,\SL_{r})\to \Mbold_{\Bet}(X/S,\PSL_{r})$ consists in taking the quotient by $\Jac(X/S)[r]$. Notice that the latter can also be interpreted as $\Mbold_{\Bet}(X/S,\mu_{r})$, where $\mu_{r}$ is the group of $r$-th roots of unity. All this can be restricted to the loci of irreducible representations.

In the case of a single Riemann surface $X_{0}$, the space $\Mbold_{\Bet}^{\ir}(X_{0},\PSL_{r})$ also carries a symplectic holomorphic $2$-form, denoted by $\omega_{\scriptscriptstyle{\PSL_{r}}}$. On the subspace of liftable representations, it can be described as the pushforward of $\omega_{\scriptscriptstyle{\SL_{r}}}$ via the finite \'etale map $\Mbold_{\Bet}^{\ir}(X,\SL_{r})\to\Mbold_{\Bet}^{\ir}(X,\PSL_{r})_{\ell}$. 


\subsubsection{Universal endomorphism bundles}\label{subsubsec:uni-end} Let $(\Ecal^{\un}_{\Bet},\nabla^{\un})$ be the universal vector bundle with relative connection on $X\times_{S}\Rbold_{\Bet}(X/S,\sigma,r)$, deduced from $(\Ecal_{\dR}^{\un},\nabla^{\un})$ via the Riemann--Hilbert isomorphism. The endomorphism bundle $\End(\Ecal^{\un}_{\Bet})$ and its connection descend to $X\times_{S}\Mbold_{\Bet}^{\ir}(X/S,r)$, because the center of $\GL_{r}(\CBbb)$ acts trivially on them. We still denote by  $\End(\Ecal^{\un}_{\Bet})$  the descended endomorphism bundle. If $\Fcal^{\un}_{\Bet}$ is a universal bundle over $X\times_{S}\Mbold_{\Bet}^{\ir}(X/S,r)$, which we know exists locally with respect to $\Mbold_{\Bet}^{\ir}(X/S,r)$, then there is a natural isomorphism $\End(\Ecal^{\un}_{\Bet})\simeq \End(\Fcal^{\un}_{\Bet})$; similarly for their flat relative connections. The analogous statement holds for $\SL_{r}$ monodromies. For a similar reason, $\Mbold_{\Bet}^{\ir}(X/S,\PSL_{r})_{\ell}$ carries a universal endomorphism bundle with connection, which can be obtained by descent from $\Mbold_{\Bet}^{\ir}(X/S,\SL_{r})$. We will use the notation $(\Ucal,\nabla^{\un})$ for universal endomorphism bundles. We will next encounter those in \textsection \ref{subsection:complements-CS-PSL}.



\section{Canonical extensions of flat relative connections}\label{section:canonical-extensions}
In Section \ref{section:intersection-connections-generalities}, we developed the general formalism of intersection connections. The setting requires a family of compact Riemann surfaces $f\colon X\to S$ and a holomorphic vector bundle $E$ on $X$, endowed with a compatible connection $\nabla\colon E\to E\otimes\Acal_{X}^{1,0}$. However, in the study of moduli spaces of flat vector bundles in Section \ref{section:universal-IC2}, the natural datum we encountered is a \emph{relative} flat connection. This poses an extension problem. In rank one this was addressed in a canonical manner by Freixas--Wentworth \cite{Freixas-Wentworth-1} and was applied to the construction of intersection connections on Deligne pairings and reciprocity laws. In this section, we tackle the general case and discuss consequences for intersection connections on $IC_{2}$ bundles.

\subsection{Intersection connections induced by flat relative connections}\label{subsec:intersection-conn-flat}
Let $f\colon X\to S$ be a proper submersion of complex manifolds, whose fibers have pure dimension one. We fix a holomorphic vector bundle $E\to X$, and a flat relative connection $\nabla\colon E\to E\otimes\Acal^{1,0}_{X/S}$. We will show that $IC_{2}(E)$ inherits a natural compatible connection, and describe some of its properties. 

A connection $\widetilde{\nabla} \colon E\to E\otimes\Acal^{1,0}_{X}$ is called a compatible extension of $\nabla$ if the vertical projection of $\widetilde{\nabla}$ onto $E\otimes\Acal^{1,0}_{X/S}$ recovers $\nabla$. Such an extension always exists. To see this, by a partition of unity argument, it is enough to construct an extension locally, and by the holomorphic implicit function theorem we can even suppose that $X\to S$ is of the form  $U \times V \to V$ for small complex open sets $U, V$, in which case the statement is clear. 

\begin{lemma}\label{lemma:canonical-connection-IC2}
Let $\nabla_{1},\nabla_{2}\colon E\to E\otimes\Acal^{1,0}_{X}$ be compatible extensions of $\nabla$. Then, the attached intersection connections on $IC_{2}(E)$ coincide: $\nabla_{1}^{\ICmini}=\nabla_{2}^{\ICmini}$. 
\end{lemma}
\begin{proof}
We must check that $IT(\nabla_{1},\nabla_{2})=0$. Because $\nabla_{1},\nabla_{2}$ are both compatible extensions of $\nabla$, we can write $\nabla_{2}=\nabla_{1}+\theta$, where $\theta$ is a $\Ccal^{\infty}$ section of $\End E\otimes f^{\ast}\Acal^{1,0}_{S}$. We examine the several contributions in the expression of $IT(\nabla_{1},\nabla_{1}+\theta)$ provided by Proposition \ref{prop:ICS-explicit-formula}. The various terms all vanish for similar reasons, and we only explain $\int_{X/S}\tr(F\wedge\theta)$. Decomposing the curvature $F$ of $\nabla_{1}$ into its $(2,0)$ and $(1,1)$ components, we find
\begin{displaymath}
    \begin{split}
         \int_{X/S}\tr(F\wedge\theta)=&\int_{X/S}\tr(F^{\scriptscriptstyle{(2,0)}}\wedge \theta)+\int_{X/S}\tr(F^{\scriptscriptstyle{(1,1)}}\wedge\theta)\\
         =&\int_{X/S}\tr(F^{\scriptscriptstyle{(1,1)}}\wedge\theta),
    \end{split}
\end{displaymath}
for type reasons. For the latter integral, we work locally on $S$ and introduce holomorphic coordinates $s_{j}$. By linearity we reduce to the case $\theta=\varphi\otimes ds_{j}$, where $\varphi$ is a smooth section of $\End E$. Then
\begin{displaymath}
    \int_{X/S}\tr(F^{\scriptscriptstyle{(1,1)}}\wedge\theta)=\left(\int_{X/S}\tr(F^{\scriptscriptstyle{(1,1)}}\varphi)\right) ds_{j}.
\end{displaymath}
The integral on the right is a function on $S$, and thus it is computed fiber by fiber. On fibers $F^{\scriptscriptstyle{(1,1)}}=0$, since $\nabla$ is flat. Therefore, the integral vanishes. 
\end{proof}
The fiberwise flat assumption on $\nabla$ cannot be removed from the above lemma. This can be seen, for example, by a computation using Proposition \ref{prop:int-conn-Deligne} and \eqref{eq:ics2sumlinebundles}.

\begin{definition}\label{def:intersection-connection}
The compatible connection $\nabla^{\ICmini}$ on $IC_{2}(E)$ constructed above is called the intersection connection attached to, or induced by, $\nabla$. 

Likewise, for two relatively flat connections on two line bundles $L$ and $M$, we define the intersection connection on $\langle L, M \rangle$, as the one obtained by the intersection connection of any compatible extensions.
\end{definition}

Notice that the formation of $\nabla^{\ICmini}$ satisfies the analogous of \textsection \ref{subsec:intersectionconnections}, by Lemma \ref{lemma:canonical-connection-IC2}. In particular it commutes with base change. For future reference, we write down the following, which is now immediate.

\begin{proposition}\label{prop:int-conn-chern}
Suppose that $\nabla$ is the vertical projection of the Chern connection $\nabla^{\hmini}$ attached to a smooth hermitian metric $h$ on $E$. Then $\nabla^{\ICmini}=\nabla^{\hmini,\ICmini}$.\qed 
\end{proposition}

For completeness, we address the compatibility with the  connection on Deligne pairings found in \cite{Freixas-Wentworth-1}. 

\begin{proposition}
Let $(L,\nabla^{\Lmini})$ and $(M,\nabla^{\Mmini})$ be line bundles with compatible flat relative connections. Then, the induced intersection connection on $IC_{2}(L\oplus M)\simeq\langle L,M\rangle$ agrees with the Freixas--Wentworth intersection connection on Deligne pairings.
\end{proposition}
\begin{proof}
We compute the intersection connection on  $IC_2(L \oplus M)$ using the direct sum of any compatible extensions of $\nabla^{\Lmini}$ and $\nabla^{\Mmini}$. It is independent of this choice by Lemma \ref{lemma:canonical-connection-IC2}. The Whitney isomorphism being parallel for direct sums, we need to compare the expression in Definition \ref{def:intconnectiondelignepairing} with that of \cite[Theorem 3.12 \& Definition 3.13]{Freixas-Wentworth-1}. The expressions are identical, modulo that the latter is computed for a specific choice of extensions, the canonical extensions in \emph{op. cit.} Since the intersection connection is actually independent of this choice, we conclude. 
\end{proof}

\subsection{Harmonic and normalized extensions}\label{subsec:harmonic-ext}
Assume now that $f\colon X\to S$ admits a section $\sigma\colon S\to X$, and that we are given the following data:
\begin{itemize}
    \item a holomorphic vector bundle $E$ on $X$ of rank $r\geq 1$, rigidified along $\sigma$;
    \item a flat relative connection $\nabla\colon E\to E\otimes\Acal_{X/S}^{1,0}$, irreducible on fibers;
    \item a hermitian metric $h$ on $E$.
\end{itemize}
Let $\widetilde{\nabla}\colon E\to E\otimes\Acal_{X}^{1,0}$ be an extension of $\nabla$. We will define notions of harmonicity and normalization for $\widetilde{\nabla}$. We derive some inspiration from Spinaci's \cite[Section 4]{Spinaci}, on deformations of harmonic maps, although we actually do not rely on non-abelian Hodge theory. 

\subsubsection{Local structure of $\widetilde{\nabla}$}\label{subsub:local-structure-nabla-tilde}
We begin our considerations by supposing that $S$ is contractible, and for concreteness we assume it admits complex coordinates.  In this case, we can form the following diagram:
\begin{equation}\label{eq:diagram-Ehresmann}
    \xymatrix{
        &  \hspace{-3cm}\Gamma=\pi_{1}(X_{0},p)\curvearrowright\ \widetilde{\Xcal}\simeq\underset{\overset{\rotatebox[origin=c]{90}{$\in$}}{\widetilde{p}}}{\widetilde{X}_{0}}\times S\ar[r]\ar[drr]_-{\widetilde{\pi}}   &\Xcal=\underset{\overset{\rotatebox[origin=c]{90}{$\in$}}{p}}{X_{0}}\times S\ar[r]^-{\phi}\ar[dr]_-{\pi}     &X\ar[d]^{f}\\
        &   &     &\underset{\overset{\rotatebox[origin=c]{90}{$\in$}}{0}}{S}\ar@/_2pc/[u]_{\sigma}
    }
\end{equation}
Let us explain the various items.
\begin{itemize}
    \item $0\in S$ and $p=\sigma(0)\in X_{0}$ are fixed base points. 
    \item Tildes indicate universal covers with respect to the base points. We fix a lift $\widetilde{p}\in\widetilde{X}_{0}$ of $p$.
    \item $\phi\colon \Xcal:=X_{0}\times S\overset{\sim}{\rightarrow} X$ is a Ehresmann trivialization, such that $\sigma$ corresponds to the constant section $p$ of $\pi$. This exists by \cite[Theorem 5.8]{topological-stability}.
    \item  Through the canonical identification $\widetilde{\Xcal}\simeq \widetilde{X}_{0}\times S$, the action of $\pi_{1}(\Xcal,(p,0))$ on $\widetilde{\Xcal}$ corresponds to the action of $\pi_{1}(X_{0},p)$ on the first factor of $\widetilde{X}_{0}\times S$.
    \item To lighten the presentation, we make no distinction of notation between objects on $X$ and their pullbacks to $\Xcal$ or $\widetilde{\Xcal}$. 
\end{itemize}

The complex structure on $X$ is transported to a complex structure on $\Xcal$ via $\phi$, so that the latter is tautologically a biholomorphism and $\pi$ is holomorphic. Also, the universal cover $\widetilde{\Xcal}$ is considered with the complex structure induced from $\Xcal$. We proceed in the same way for $E$ and its holomorphic structure. 

Denote by $\vbsymb = \lbrace v_{j}\rbrace_{j=1}^{r}$ the rigidification of $E$ along $\sigma$, organized in a column vector. It gives rise to a rigidification of $E$ on $\widetilde{\Xcal}$, along the constant section $\widetilde{p}$. By means of parallel transport along fibers, there exists a unique $\Ccal^{\infty}$ trivialization of $E$ on $\widetilde{\Xcal}$, which is fiberwise flat and coincides with $\vbsymb$ along the constant section $\widetilde{p}$. We also organize this trivialization in a column vector $\ubsymb$. We notice that $\ubsymb$ is holomorphic on fibers. 

On $\widetilde{\Xcal}$, the connection $\widetilde{\nabla}$ is expressed in the $\Ccal^{\infty}$-basis $\ubsymb$ as follows:
\begin{equation}\label{eq:dec-nabla-tilde}
    \widetilde{\nabla} \leftrightsquigarrow d+\Xi,\quad \Xi\in A^{1}(\widetilde{\Xcal},\glfrak_{r}).
\end{equation}
This means that for any column vector $\gbsymb\in\Ccal^{\infty}(\widetilde{\Xcal})^{\oplus r}$, we have
\begin{equation}\label{eq:dec-nabla-tilde-bis}
    \widetilde{\nabla}(\gbsymb^{t}\ubsymb)=d\gbsymb^{t}\ \ubsymb+\gbsymb^{t}\ \Xi\ \ubsymb.
\end{equation}
Because $\widetilde{\nabla}$ extends $\nabla$, the matrix of forms $\Xi$ vanishes on fibers. Therefore, it takes the form
\begin{equation}\label{eq:Xi-comes-S}
    \Xi\in\glfrak_{r}(\Ccal^{\infty}(\widetilde{\Xcal}))\otimes\widetilde{\pi}^{\ast}A^{1}(S).
\end{equation}
The matrix of forms $\Xi$ decomposes as $\Xi^{\scriptscriptstyle{(1,0)}}+\Xi^{\scriptscriptstyle{(0,1)}}$ according to $A^{1}(S)=A^{1,0}(S)\oplus A^{0,1}(S)$. 

The condition that $\widetilde{\nabla}$ be compatible is that $\widetilde{\nabla}^{\scriptscriptstyle{(0,1)}} = \ov{\partial}_{E}$, and  hence this completely determines $\Xi^{\scriptscriptstyle{(0,1)}}$. Precisely, 
\begin{equation}\label{eq:def-Xi-10}
    \ov{\partial}_{E}\ubsymb=\Xi^{\scriptscriptstyle{(0,1)}} \ubsymb,\quad \Xi^{\scriptscriptstyle{(0,1)}}\in \glfrak_{r}(\Ccal^{\infty}(\widetilde{\Xcal}))\otimes\widetilde{\pi}^{\ast}A^{0,1}(S).
\end{equation}

\subsubsection{Harmonicity and normalization} 
Let $\Phi$ be an endomorphism of $E$ on $\widetilde{\Xcal}$. If we choose a hermitian metric on $T_{X/S}$, then we can take the formal adjoint $\nabla^{\ast}$ with respect to $h$, fiberwise. We say that $\Phi$ is \emph{harmonic with respect to} $h$ if we have $\nabla^{\ast}\nabla\Phi=0$ on fibers. Since we are in relative dimension one, this notion is independent of the choice of hermitian metric on $T_{X/S}$, but depends on the complex structure of the fibers. If $\Phi$ is already defined on $X$, then harmonicity entails $\nabla\Phi=0$, by compactness of the fibers. This means that $\Phi$ is a fiberwise flat endomorphism. Finally, if we express $\Phi$ in the basis $\ubsymb$ as a matrix $\Phi\in\glfrak_{r}(\Ccal^{\infty}(\widetilde{\Xcal}))$, then the harmonicity condition becomes $d^{\ast}d\Phi$=0. This formulation compares to \cite[Definition 4.9]{Spinaci}; see Proposition 2.5 (4) \emph{op. cit.} for the notation. 

Now in the setting of \textsection \ref{subsub:local-structure-nabla-tilde}, introduce complex coordinates $\lbrace s_{j}\rbrace_{j=1}^{m}$ on $S$, and expand
\begin{displaymath}
     \Xi^{\scriptscriptstyle{(1,0)}}=\sum_{j} \Xi^{\scriptscriptstyle{(1,0)}}_{j}\otimes \widetilde{\pi}^{\ast} ds_{j}, \quad \Xi^{\scriptscriptstyle{(1,0)}}_{j}\in \glfrak_{r}(\Ccal^{\infty}(\widetilde{\Xcal})).
\end{displaymath}
Requiring that all the $\Xi^{\scriptscriptstyle{(1,0)}}_{j}$ are harmonic with respect to $h$ is  independent of the choice of coordinates. In this situation, we may just say that $\Xi^{\scriptscriptstyle{(1,0)}}$ is harmonic with respect to $h$. It is formal to check that this is intrinsic to $\widetilde{\nabla}$, namely it does not depend on the choice of Ehresmann trivialization and lift $\widetilde{p}$ of $p$. Over a general base $S$, we say that $\widetilde{\nabla}$ is harmonic with respect to $h$ if, locally over contractible coordinate open subsets of $S$, the forms $\Xi^{\scriptscriptstyle{(1,0)}}$ defined above are harmonic with respect to $h$. 

Another condition we introduce is that of \emph{normalization}. Restricting $\widetilde{\nabla}$ along the section $\sigma$, we obtain a compatible connection on $\sigma^{\ast}E$. Using the rigidification, this is seen as a compatible connection on the trivial bundle $\Ocal_{S}^{\oplus r}$. This is of the form $d+\varphi$, with $\varphi\in A^{1,0}(S,\glfrak_{r})$. We say that $\widetilde{\nabla}$ is normalized if  $\tr(\varphi)=0$, that is $\sigma^{\ast}\widetilde{\nabla}$ has trivial determinant. More concretely, in terms of the connection forms above, the normalization condition amounts to $\Xi(\widetilde{p},s)=0$. 

It is immediate that harmonicity and normalization are preserved under base change.

\begin{proposition}\label{prop:uniqueness-can-ext}
There is at most one normalized, compatible extension of $\,\nabla$ which is harmonic with respect to $h$.
\end{proposition}
\begin{proof}
 We can suppose that $S$ is contractible and admits holomorphic coordinates $s_{j}$.
 
 Suppose that we are given two harmonic and normalized extensions $\widetilde{\nabla}_{1}$ and $\widetilde{\nabla}_{2}$. We take the difference $\varphi:=\widetilde{\nabla}_{1}-\widetilde{\nabla}_{2}\in A^{0}(X,\End E)\otimes f^{\ast}A^{1,0}(S)$. Therefore, if $\varphi=\sum_{j}\varphi_{j}\otimes f^{\ast} ds_{j}$, then $\varphi_{j}\in A^{0}(X,\End E)$. By the normalization condition, we have $\tr(\sigma^{\ast}\varphi_{j})=0$. We will show that the $\varphi_{j}$ vanish.

 We perform the construction \eqref{eq:diagram-Ehresmann}. We express the connections in the flat relative basis $\ubsymb$, as $d+\Xi_{1}$ and $d+\Xi_{2}$. By \eqref{eq:def-Xi-10}, we already know that $\Xi_{1}^{\scriptscriptstyle{(0,1)}}=\Xi_{2}^{\scriptscriptstyle{(0,1)}}$. The difference $\Xi_{1}^{\scriptscriptstyle{(1,0)}}-\Xi_{2}^{\scriptscriptstyle{(1,0)}}$ is the expression of $\varphi$ in the basis $\ubsymb$. The harmonicity condition entails that the $\varphi_{j}$ are fiberwise flat sections of $\End E$. From the irreducibility assumption on $\nabla$ it then follows that $\varphi_{j}=f_{j}\id_{E}$, for some $f_{j}\in\Ccal^{\infty}(S)$. Finally, $\tr(\sigma^{\ast}\varphi_{j})=r f_{j}$ vanishes, so that $\varphi_{j}=0$ as stated.
\end{proof}

\subsection{Canonical extensions}\label{subsec:canonical-ext}
The goal of this subsection is a higher rank generalization of \cite[Theorem 1.1 (i)]{Freixas-Wentworth-1}: 

\begin{theorem}\label{theorem:can-ext}
Let $f\colon X\to S$ be a proper submersion of complex manifolds, with one-dimensional connected fibers, and $\sigma\colon S\to X$ a given section. Let $(E,h)$ be a hermitian holomorphic vector bundle on $X$, endowed with a rigidification along $\sigma$ and a flat relative connection $\nabla\colon E\to E\otimes\Acal_{X/S}^{1,0}$. Assume that $\nabla$ is fiberwise irreducible. Then, there exists a unique normalized extension $\widetilde{\nabla}\colon E\to E\otimes\Acal^{1,0}_{X}$ of $\,\nabla$, which is harmonic with respect to $h$.
\end{theorem}

\begin{proof}
Uniqueness has been settled in Proposition \ref{prop:uniqueness-can-ext}. Incidentally, this allows us to reduce the existence part to the case when $S$ is contractible and admits holomorphic coordinates $\lbrace s_{j}\rbrace_{j=1}^{m}$. In general, we can first reason locally over $S$, and then the uniqueness property ensures the necessary gluing. 

We perform the construction \eqref{eq:diagram-Ehresmann}, and adopt the notation and conventions in \textsection\ref{subsec:harmonic-ext}. Recall in particular the vertically flat trivialization $\ubsymb$ of $E$ on $\widetilde{\Xcal}$. The fundamental group $\Gamma$ acts by pullback functoriality on $\ubsymb$, through a representation $\rho\colon\Gamma\to\GL_{r}(\Ccal^{\infty}(S))$: for $\gamma\in\Gamma$, we have 
\begin{equation}\label{eq:gamma-action-u}
    \gamma^{\ast}\ubsymb(x,s)=\rho(\gamma,s)\ubsymb(x,s),\quad \gamma\in\Gamma,\quad (x,s)\in\Xcal.
\end{equation}
By assumption, $\rho$ is pointwise irreducible.

Let $e$ be a smooth local section of $E$ on $X$, and pull it back to a section on a $\Gamma$-invariant open subset $\Ucal\subseteq\widetilde{\Xcal}$. In the basis $\ubsymb$, we can write 
\begin{displaymath}
    e=\gbsymb^{t}\ \ubsymb,\quad\text{with}\quad \gbsymb\in\Ccal^{\infty}(\Ucal)^{\oplus r}\quad \text{column vector}.
\end{displaymath}
Because $e$ is $\Gamma$-invariant, the action of $\Gamma$ on $\gbsymb$ is given by 
\begin{equation}\label{eq:gamma-action-g}
    \gamma^{\ast}\gbsymb=\rho^{\vee}(\gamma)\gbsymb,
\end{equation}
 where $\rho^{\vee}$ is the contragradient representation: $\rho^{\vee}(\gamma)=\rho(\gamma^{-1})^{t}$. 
 
 As in \eqref{eq:dec-nabla-tilde}--\eqref{eq:Xi-comes-S}, the extension we seek should have an expression $\widetilde{\nabla} \leftrightsquigarrow d+\Xi$ in the basis $\ubsymb$, with $\Xi\in\glfrak_{r}(\Ccal^{\infty}(\widetilde{\Xcal}))\otimes\widetilde{\pi}^{\ast}A^{1}(S)$. Taking into account \eqref{eq:gamma-action-u}--\eqref{eq:gamma-action-g}, we find that for $\widetilde{\nabla}$ to descend to $X$, the connection form $\Xi$ needs to obey the transformation law
 \begin{equation}\label{eq:gamma-ast-Xi}
    \gamma^{\ast}\Xi=\Ad(\rho)(\gamma)\ \Xi+d\rho(\gamma)\ \rho(\gamma)^{-1},
 \end{equation}
where we write $d\rho(\gamma)\ \rho(\gamma)^{-1}$ instead of $\widetilde{\pi}^{\ast}(d\rho(\gamma)\ \rho(\gamma)^{-1})$, in order to simplify the notation.  Notice that $d\rho\ \rho^{-1}\in Z^{1}(\Gamma,\Ad(\rho))\otimes A^{1}(S)$ is the cocycle $\kappa$ in \eqref{eq:cocycle-kappa}. We will separately define the $(0,1)$ and $(1,0)$ parts of $\Xi$.
 
 As we saw in \eqref{eq:def-Xi-10}, the $(0,1)$ part $\Xi^{\scriptscriptstyle{(0,1)}}$ is already dictated by $\ov{\partial}_{E}\ubsymb=\Xi^{\scriptscriptstyle{(0,1)}}\ubsymb$. Applying $\gamma^{\ast}$ to this equation, we obtain 
 \begin{equation}\label{eq:gamma-ast-Xi-01}
    \gamma^{\ast}\Xi^{\scriptscriptstyle{(0,1)}}=\Ad(\rho)(\gamma)\ \Xi^{\scriptscriptstyle{(0,1)}}+\ov{\partial}\rho(\gamma)\ \rho(\gamma)^{-1}.
 \end{equation}
 Notice that $\Xi^{\scriptscriptstyle{(0,1)}}(\widetilde{p},s)=0$, since $\ubsymb(\widetilde{p},s)$ is identified with the rigidification, which is holomorphic.
 
Next, we decompose
 \begin{equation}\label{eq:kappa-prime}
    \partial\rho(\gamma)\ \rho(\gamma)^{-1}=\sum_{j}\kappa^{\prime}_{j}(\gamma)ds_{j}.
 \end{equation}
Then $\kappa^{\prime}_{j}\in Z^{1}(\Gamma,\Ad(\rho))$ is a smooth family of 1-cocyles parametrized by $S$. For $s\in S$, we take the cohomology class $[\kappa_{j}^{\prime}(s)]\in H^{1}(\Gamma,\Ad(\rho(s)))$. Consider the period isomorphism of the latter with $H^{1}_{\dR}\big(X_{s},(\End(E_{s}),\nabla)\big)$, the de Rham cohomology of $\End(E_{s})$ with respect to the flat connection induced by $\nabla$. We denote by $\omega_{j}(s)\in A^{1}(X_{s},\End(E_{s}))$ the harmonic representative of the class corresponding to $[\kappa_{j}^{\prime}(s)]$, taken with respect to $\nabla$ and $h$. We express $\omega_{j}(s)$ in the basis $\ubsymb$:
\begin{displaymath}
    \Lambda(\omega_{j}):=\text{matrix expression of }\omega_{j}\text{ in the basis }\ubsymb.
\end{displaymath}
Because of the irreducibility assumption on the connection, $\Lambda(\omega_{j})$ is a smooth family of matrices of closed 1-forms on $\widetilde{X}_{0}$, parametrized by $S$. The action of the fundamental group on $\Lambda(\omega_{j})$ is given by
\begin{equation}\label{eq:period-bis}
    \gamma^{\ast}\Lambda(\omega_{j})=\Ad(\rho)(\gamma)\ \Lambda(\omega_{j}).
\end{equation}
The cohomological relationship between $\kappa_{j}^{\prime}(s)$ and $\omega_{j}(s)$ is expressed as the period equation
\begin{equation}\label{eq:period}
    \kappa_{j}^{\prime}(\gamma)=\int_{\widetilde{p}}^{\gamma\widetilde{p}}\Lambda(\omega_{j})+\left(\Ad(\rho)(\gamma)-1\right)\beta_{j}(s),
\end{equation}
for $\beta_{j}(s)\in\slfrak_{r}(\CBbb)$ providing a coboundary term $\left(\Ad(\rho)-1\right)\beta_{j}(s)\in B^{1}(\Gamma, \Ad(\rho(s)))$. By the irreducibility assumption $H^{0}(\Gamma,\Ad_{0}(\rho(s)))=0$, so $\beta_{j}(s)$ is in fact unique. The index $0$ in $\Ad_{0}$ means the restriction to traceless elements. By Lemma \ref{lemma:Cell-cocycle} below, $\beta_{j}(s)$ is actually $\Ccal^{\infty}$ in $s$. We define

\begin{equation}\label{eq:Theta}
        \Xi^{\scriptscriptstyle{(1,0)}}(x,s)=
        \sum_{j}\left(\int_{\widetilde{p}}^{x}\Lambda(\omega_{j})\right)ds_{j}-\sum_{j}\beta_{j}(s)ds_{j}\quad \text{in}\quad \glfrak_{r}(\Ccal^{\infty}(\widetilde{\Xcal}))\otimes\widetilde{\pi}^{\ast}A^{1,0}(S),
\end{equation}
where the integral takes place on $\widetilde{X}_{0}$ and is well-defined since the $\Lambda(\omega_{j})$ are closed. By equations \eqref{eq:kappa-prime}--\eqref{eq:period}, $\Xi^{\scriptscriptstyle{(1,0)}}$ satisfies the transformation law
\begin{equation}\label{eq:Theta-law}
    \gamma^{\ast}\Xi^{\scriptscriptstyle{(1,0)}}=\Ad(\rho)(\gamma)\ \Xi^{\scriptscriptstyle{(1,0)}}+\partial\rho(\gamma)\ \rho(\gamma)^{-1}.
\end{equation}

To conclude the proof, let $\Xi:=\Xi^{\scriptscriptstyle{(1,0)}}+\Xi^{\scriptscriptstyle{(0,1)}}$. It satisfies \eqref{eq:gamma-ast-Xi}, as we see by adding up \eqref{eq:gamma-ast-Xi-01} and \eqref{eq:Theta-law}. Thus, $d+\Xi$ defines a connection $\widetilde{\nabla}\colon E\to E\otimes\Acal^{1}_{X}$, extending $\nabla$. By construction, it is compatible with the holomorphic structure of $E$. It is also harmonic with respect to $h$, because $\Xi^{\scriptscriptstyle{(1,0)}}$ is an antiderivative (in the variable $x$) of $\sum_{j}\omega_{j}\wedge ds_{j}$, and the $\omega_{j}$ are harmonic on fibers. It is normalized too, since $\Xi(\widetilde{p},s)=\Xi^{\scriptscriptstyle{(1,0)}}(\widetilde{p},s)=-\sum_{j}\beta_{j}(s)ds_{j}$, and the $\beta_{j}(s)$ have been chosen to have zero trace. 
\end{proof}

To complete the proof of Theorem \ref{theorem:can-ext}, we still need to show that the coboundary terms $\beta_{j}(s)$ in the period relationship \eqref{eq:period} are smooth in $s\in S$. 
\begin{lemma}\label{lemma:Cell-cocycle}
Let $T$ be a manifold and $\rho\colon\Gamma\to\GL_{r}(\Ccal^{\ell}(T))$ a family of irreducible representations, for some integer $\ell\geq 0$. Let $c(t)=(\Ad(\rho(t))-\id)\alpha(t)\in B^{1}(\Gamma, \Ad(\rho(t)))$ be a coboundary, with traceless $\alpha$. If $c(t)$ is $\Ccal^{\ell}$ in $t\in T$,  then so is $\alpha$.
\end{lemma}
\begin{proof}
We proceed by induction. We begin with the $\Ccal^{0}$ case. Choose a hermitian norm $\|\cdot\|$ on $\slfrak_{r}(\CBbb)$. Suppose for a contradiction that $\alpha$ is not continuous at some $t_{0}$. Then there exists $\varepsilon>0$ and a sequence of points $t_{n}\to t_{0}$, such that $\varepsilon_{n}:=\|\alpha(t_{n})-\alpha(t_{0})\|\geq\varepsilon$. The quotients $M_{n}=(\alpha(t_{n})-\alpha(t_{0}))/\varepsilon_{n}\in\slfrak_{r}(\CBbb)$ are uniformly bounded, and after possibly restricting to a subsequence, we may assume that $M_{n}\to M$, for some traceless matrix $M$ of norm one. From $c=(\Ad(\rho)-\id)\alpha$, we derive
\begin{displaymath}
    \frac{c(t_{n})-c(t_{0})}{\varepsilon_{n}}=\left(\Ad(\rho(t_{n}))-\id\right)M_{n}
    +\left(\Ad(\rho(t_{n}))-\Ad(\rho(t_{0}))\right)\frac{\alpha(t_{0})}{\varepsilon_{n}}.
\end{displaymath}
The left hand side converges to $0$, because $c$ is continuous and $\varepsilon_{n}\geq \varepsilon>0$. On the right hand side, the first term converges to $(\Ad(\rho(t_{0}))-\id)M$, by continuity of $\rho$. The second term converges to 0, because $\rho$ is continuous and $\varepsilon_{n}\geq\varepsilon>0$. Passing to the limit and recalling that $M$ is traceless, we find $M\in H^{0}(\Gamma,\Ad_{0}(\rho_{0}(t_{0}))=0$, which contradicts $\|M\|=1$. We conclude that $\alpha$ is continuous.

Next we assume that the $\Ccal^{\ell}$ regularity has been established for some $\ell\geq 0$, and we tackle the $\Ccal^{\ell+1}$ case. We first check the existence of directional derivatives at some $t_{0}$. After introducing local coordinates around $t_{0}$, we can suppose that $T$ is a ball in $\RBbb^{N}$ centered at $t_{0}=0$. Let $u\in\RBbb^{N}\setminus\lbrace 0\rbrace$ be any direction.  Take a sequence of non-zero real numbers $\varepsilon_{n}\to 0$, with $\varepsilon_{n}u\in T$. We claim that the quotients $M_{n}=(\alpha(\varepsilon_{n}u)-\alpha(0))/\varepsilon_{n}$ are uniformly bounded. Otherwise, after possibly going to a subsequence, we can suppose that $\|M_{n}\|\to\infty$. Restricting to a further subsequence if necessary, we can also suppose that $Q_{n}=M_{n}/\|M_{n}\|$ converges to a traceless matrix $Q$, of norm 1. Proceeding as above, we have the relationship
\begin{displaymath}
    \frac{1}{\|M_{n}\|}\frac{c(\varepsilon_{n}u)-c(0)}{\varepsilon_{n}}=(\Ad(\rho(\varepsilon_{n}u))-\id)Q_{n}+\frac{\Ad(\rho(\varepsilon_{n}u))-\Ad(\rho(0))}{\varepsilon_{n}}\frac{\alpha(0)}{\|M_{n}\|}.
\end{displaymath}
Because $c$ is now differentiable and $\|M_{n}\|\to\infty$, the left hand side converges to 0. The first term on the right hand side converges to $(\Ad(\rho(0))-\id)Q$. The second term converges to 0, because $\rho$ is differentiable and $\|M_{n}\|\to\infty$. Therefore $(\Ad(\rho(0))-\id)Q=0$, which entails $Q=0$ and contradicts that $Q$ has norm 1. Thus, the $M_{n}$ are uniformly bounded. Now we take a subsequence $M_{n_{k}}$. By boundedness, there is a further subsequence $M_{n_{k_{j}}}$ which converges to some limit $M\in\slfrak_{r}(\CBbb)$. By a similar argument as before, we find
\begin{displaymath}
    D_{u}c(0)-(D_{u}\rho(0))\alpha(0)=(\Ad(\rho(0))-\id)M,
\end{displaymath}
where $D_{u}$ denotes the directional derivative. By the irreducibility assumption, the solution $M$ to this equation is uniquely determined by $c(0)$ and $\alpha(0)$, and hence does not depend on the subsequence. We conclude that the $M_{n}$ converge to $M$, and hence $D_{u}\alpha(0)$ exists. Repeating the same argument for the other points of $T$, we see that $D_{u}\alpha$ exists everywhere and satisfies the differential equation
\begin{equation}\label{eq:eq-diff-cocycle}
    D_{u}c(t)-(D_{u}\rho(t))\alpha(t)=(\Ad(\rho(t))-\id)D_{u}\alpha(t).
\end{equation}
The left hand side of this equation is $\Ccal^{\ell}$. Therefore, by the induction hypothesis, $D_{u}\alpha$ is $\Ccal^{\ell}$ too. Since this is true for any direction $u$, we conclude that $\alpha$ is of class $\Ccal^{\ell+1}$.
\end{proof}

\begin{definition}\label{def:quasi-canonical}
Let the setting be as in Theorem \ref{theorem:can-ext}. 
The extension $\widetilde{\nabla}$ is called the canonical extension of $\nabla$ with respect to $h$. If we do not need to specify the metric $h$, we may refer to $\widetilde{\nabla}$ simply as a canonical extension of $\,\nabla$.
\end{definition}

\begin{remark}\label{rmk:quasi-canonical}
Several comments on the theorem and its proof are in order.
\begin{enumerate}
    \item By the theorem of Corlette--Donaldson \cite{Corlette:metrics, Donaldson:harmonic}, the irreducibility assumption guarantees the existence of a unique harmonic metric on $E$, which agrees with the standard metric on $\Ocal_{S}^{\oplus r}$ through the rigidification. Letting $h$ be this choice of metric, the corresponding extension $\widetilde{\nabla}$ is genuinely canonically attached to $\nabla$.
    \item Canonical extensions are compatible with base change. This is because harmonicity and normalization are preserved by base change, and by the characterization of canonical extensions in terms of these.
    \item In rank one, $\End E$ is trivial and the harmonicity notion does not depend on the choice of $h$. Also, the normalization amounts to a rigidification of $\widetilde{\nabla}$. This means that the restriction $\sigma^{\ast}\widetilde{\nabla}$ corresponds to the trivial connection through the trivialization $\sigma^{\ast}E\overset{\sim}{\to}\Ocal_{S}$. This is consistent with \cite[Theorem 1.1 (i)]{Freixas-Wentworth-1}. The construction of the canonical extension in the Section 4 of \emph{op. cit.} guarantees the harmonicity property. Therefore, by the uniqueness part of Theorem \ref{theorem:can-ext}, both constructions agree. Besides, in \cite{Freixas-Wentworth-1} the canonical extension was characterized by a reciprocity law for connections. We do not know of an analogous description in higher rank.
    \item For connections with $\SL_{r}$ monodromies, the normalization condition is automatically fulfilled. 
    \item If the metric $h$ on $E$ is flat on fibers, and $\nabla$ is the vertical projection of the associated Chern connection $\nabla^{\chmini}$, it is possible to see that the canonical extension with respect to $h$ coincides with $\nabla^{\chmini}$. This can be inferred from the characterization and from Fay's explicit computations in the proof of \cite[Theorem 4.1]{Fay} (see in particular equation (4.5) therein).
\end{enumerate}
\end{remark}

\subsection{Curvature of $\nabla^{\ICmini}$}\label{subsubsec:curvature-preliminaries}

The next objective is the curvature of $\nabla^{\ICmini}$. We first place ourselves in the setting of \textsection\ref{subsec:harmonic-ext}--\textsection\ref{subsec:canonical-ext}. We further assume that the genus of the fibers of $f\colon X\to S$ is $g\geq 2$. We are given a rigidified, flat irreducible $\nabla\colon E\to E\otimes\Acal^{1,0}_{X/S}$ and a hermitian metric $h$ on $E$. Since the curvature is a local invariant, we can suppose that $S$ is contractible and admits complex coordinates. We then have the central fiber $X_{0}$, the base point $p=\sigma(0)$ and the family of irreducible representations $\rho\colon\Gamma=\pi_{1}(X_{0},p)\to\GL_{r}(\Ccal^{\infty}(S))$ attached to $\nabla$, as in the proof of Theorem \ref{theorem:can-ext}. Recall the discussion \textsection \ref{subsub:atiyahbottgoldman}, where we studied the $\Ccal^{\infty}$ classifying map associated to $\rho$
\begin{displaymath}
   \nu\colon S\to\Mbold_{\Bet}^{\ir}(X_{0},r).
\end{displaymath}
The differential $d\nu\in H^{1}(\Gamma,\Ad(\rho))\otimes A^{1}(S)$ was described as the class of the cocycle $\kappa=d\rho\ \rho^{-1}\in Z^{1}(\Gamma,\Ad(\rho))\otimes A^{1}(S)$. We adopted an integration convention for the cup products of such classes, and we related them to the Atiyah--Bott--Goldman form in Lemma \ref{lemma:functorial-ABG}. In the theorem below, we denote by  $\det\colon\Mbold_{\Bet}(X_{0},\GL_{r})\to\Mbold_{\Bet}(X_{0},\GL_{1})$ the determinant map.  


\begin{theorem}\label{theorem:curvature-int-conn-holo}
With the assumptions and notation as above, the curvature of $\nabla^{\ICmini}$ is given by 
\begin{equation}\label{eq:curvature-can-ext-IC2}
    \begin{split}
         c_{1}(IC_{2}(E),\nabla^{\ICmini})=&-\frac{1}{8\pi^{2}}\int_{X_{0}}\left(\tr\left(d\nu\cup  d\nu\right)-\tr d\nu\cup\tr d\nu\right)\\
         =&-\frac{1}{4\pi^{2}}\left(\nu^{\ast}\omega_{\scriptscriptstyle{\GL_{r}}}-\nu^{\ast}\mathrm{det}^{\ast}\omega_{\scriptscriptstyle{\GL_{1}}}\right).
    \end{split}
\end{equation} 
Consequently, if $\nabla$ is holomorphic, then $ c_{1}(IC_{2}(E),\nabla^{\ICmini})\in A^{2,0}(S)$ and $\nabla^{\ICmini}$ is a holomorphic connection.
\end{theorem}
\begin{proof}
For the proof, we interpret $\nabla^{\ICmini}$ as being induced by a canonical extension $\widetilde{\nabla}$. By equation \eqref{eq:curv-IC2-jenesaisplus}, the curvature of $\nabla^{\ICmini}$ is expressed in terms of $F_{\widetilde{\nabla}}$. We thus have to relate the latter to $d\nu$. 

In the vertically flat basis $\ubsymb$, $\widetilde{\nabla}$ decomposes as $d+\Xi$, so that for the curvature we find
\begin{displaymath}
    F_{\widetilde{\nabla}}= d\Xi+\Xi\wedge\Xi.
\end{displaymath}
Notice that, while $\Xi$ is only defined on $\widetilde{\Xcal}$, the curvature $F_{\widetilde{\nabla}}$ descends to $\Xcal$. We insert this in \eqref{eq:curv-IC2-jenesaisplus}. Most traces disappear for formal reasons, and since the derivatives in the $S$-direction do not contribute to the integrals we find that
\begin{displaymath}
    c_{1}(IC_{2}(E),\nabla^{\ICmini})=\frac{1}{8\pi^{2}}\left(\int_{X_{0}}\tr\left(d_{x}\Xi\wedge d_{x}\Xi\right)-\int_{X_{0}}\tr(d_{x}\Xi)\wedge\tr(d_{x}\Xi)\right),
\end{displaymath}
where we decomposed $d=d_{x}+d_{s}$ according to the product $\widetilde{\Xcal}=\widetilde{X}_{0}\times S$. We thus have to see that $d_{x}\Xi$ represents the class $d\nu$. 

 From the very definition \eqref{eq:Theta} of $\Xi^{\scriptscriptstyle{(1,0)}}$, we deduce that $d_{x}\Xi^{\scriptscriptstyle{(1,0)}}$ is cohomologous to $\partial\rho\ \rho^{-1}$. Now for $\Xi^{\scriptscriptstyle{(0,1)}}$. Recall from \eqref{eq:gamma-ast-Xi-01} the action of $\gamma\in\Gamma$ on $\Xi^{\scriptscriptstyle{(0,1)}}$, and that $\Xi^{\scriptscriptstyle{(0,1)}}(\widetilde{p},s)=0$. Therefore, applying \eqref{eq:gamma-ast-Xi-01} we find
\begin{displaymath}
    \int_{\widetilde{p}}^{\gamma\widetilde{p}}d_{x}\Xi^{\scriptscriptstyle{(0,1)}}=\gamma^{\ast}\Xi^{\scriptscriptstyle{(0,1)}}(\widetilde{p},\bullet)
    =\Ad(\rho)(\gamma)\ \Xi^{\scriptscriptstyle{(0,1)}}(\widetilde{p},\bullet)+\ov{\partial}\rho(\gamma)\ \rho(\gamma)^{-1}=\ov{\partial}\rho(\gamma)\ \rho(\gamma)^{-1}.
\end{displaymath}
This means that $d_{x}\Xi^{\scriptscriptstyle{(0,1)}}$ is cohomologous to $\ov{\partial}\rho(\gamma)\ \rho(\gamma)^{-1}$. We have thus proven that $d_{x}\Xi$ represents $d\rho\ \rho^{-1}$, that is $d\nu$.

To conclude the first part of the proof, it remains to explain the sign in front of the integral \eqref{eq:curvature-can-ext-IC2}. This comes from the sign convention in the definition of \eqref{eq:int-tr-dnu-cup-dnu}--\eqref{eq:sign-conv-cup}, plus the fact that $d_{x}\Xi$ is an actual 2-form and, contrary to $d\nu$, the usual anticommutativity rule applies to its constituent 1-forms.

If $\nabla$ is holomorphic, then $\nu$ is holomorphic. Therefore, $d\nu\in H^{1}(\Gamma,\Ad(\rho))\otimes A^{1,0}(S)$. This implies that $c_{1}(IC_{2}(E),\nabla^{\ICmini})$ has type $(2,0)$. Since $\nabla^{\ICmini}$ is a compatible connection, we infer that it has to be holomorphic. The proof is complete.

\end{proof}

\begin{remark}\label{rmk:Fay-a-tout-fait}
\begin{enumerate}
    \item For connections with $\SL_{r}$ monodromies, the second factor in the integral \eqref{eq:curvature-can-ext-IC2} vanishes.
    \item For flat unitary connections, the theorem recovers the well-known curvature formula for $\nabla^{\ICmini}$. See Takhtajan--Zograf \cite{ZT:moduli-1, ZT:moduli-2}.
    \item\label{item:Fay-a-tout-fait-3} In the holomorphic case, a crude form of Theorem \ref{theorem:curvature-int-conn-holo} was obtained by Fay in his study of non-abelian theta functions. See \cite[Theorem 5.7]{Fay} and \cite[Section 1]{Fay-2}. The holomorphic differential forms $\Omega(\sigma,\tau)$ and $\omega(s)$ therein correspond to local connection forms of $\nabla^{\ICmini}$. Our theorem clarifies the geometric content of his explicit constructions.
\end{enumerate}
\end{remark}

The statement about holomorphic connections can be extended to non-irreducible connections over general parameter spaces. 

\begin{corollary}\label{cor:extension-int-connection}
Let $X\to S$ be a family of compact Riemann surfaces of genus $g\geq 2$, over a reduced complex analytic space. Let $(E, \nabla)$ be vector bundle with a flat relative holomorphic connection. Then, 
for every desingularization $\mu: S' \to S$,  the intersection connection $\nabla^{\ICmini}$ on $S'$ associated to the base change $(\mu^{\ast} E,\mu^{\ast} \nabla)$ is holomorphic. 

\end{corollary}
\begin{proof}

 We prove the statement when $S$ itself is a smooth analytic space, the statement for general reduced analytic spaces can be deduced from this. Holomorphicity being a local statement, we can suppose that $S$ is contractible and that $X\to S$ admits a section $\sigma$. 
 
 We then consider the statement over the representation space itself. We recall $\Rbold_{\dR}(X/S, \sigma, r)$ is reduced and that the irreducible locus $\Rbold_{\dR}^{\ir}(X/S, \sigma, r)$ is smooth over $S$, and hence smooth (see Proposition \ref{prop:flatness} and \textsection \ref{subsec:complements-moduli-betti}). Let $\widetilde{\Rbold} \to \Rbold_{\dR}(X/S, \sigma, r)$ be a desingularization of the representation space which is an isomorphism over the irreducible locus. The universal vector bundle on $\Rbold_{\dR}(X/S, \sigma, r)$, pulls-back to $\widetilde{\Rbold}$ and admits a relative holomorphic connection. The intersection connection on $IC_2$ can be computed over the irreducible locus using a canonical extension, in which case we know it is holomorphic by Theorem \ref{theorem:curvature-int-conn-holo}. By density of the irreducible locus, it must be holomorphic on all of $\widetilde{\Rbold}$.

Given $(E, \nabla)$ as in the statement of the corollary over a reduced space $S$, with the additional assumptions in the beginning of the proof, this amounts to a holomorphic map $S \to \Rbold_{\dR}(X/S, \sigma, r)$. Denote by $\mu: S' \to S$ a birational map, such that $S'$ is also smooth and the map $S' \to S \to \Rbold_{\dR}(X/S, \sigma, r)$ factors as $S' \to \widetilde{\Rbold} \to \Rbold_{\dR}(X/S, \sigma, r)$ for a holomorphic map $S' \to \widetilde{\Rbold}$. The intersection connection associated to $(\mu^* E, \mu^* \nabla)$ is the intersection connection of the corresponding family on $\widetilde{\Rbold}$, whose connection is holomorphic by the previous argument. By the same token, the intersection connection on $(E, \nabla)$ pulls back to that of $(\mu^{\ast} E, \mu^{\ast} \nabla)$. A smooth connection which is holomorphic on a dense open set must be holomorphic everywhere, from which  we conclude.
\end{proof}

\section{The complex Chern--Simons line bundle}\label{section:CS-theory}
We introduce complex Chern--Simons line bundles on relative moduli spaces of flat vector bundles. The construction builds upon the previous developments on intersection bundles and intersection connections. We establish the main properties and propose a characterization, which in particular involves a  fundamental compatibility with Simpson's Gauss--Manin connection, referred to as crystalline property. We study the behaviour of the complex Chern--Simons line bundles under a change of orientation of the underlying Riemann surfaces. This gives rise to the notion of complex metric, for which we provide explicit expressions in rank 2. 

\subsection{Construction and curvature}\label{subsec:construction-CS}
Throughout  this subsection, $f\colon X\to S$ is a smooth proper morphism of complex manifolds, with one-dimensional connected fibers of genus $g\geq 2$. For simplicity, we suppose that $f$ admits a section $\sigma\colon S\to X$. We will now make use of the notation and results of Section \ref{section:universal-IC2}. We recall that the bulk of the latter is written in the algebraic category, which is then extended to the analytic category in \textsection \ref{subsec:complements-moduli-betti}. We may thus work in this generality.

 Recall the moduli scheme $\Rbold^{\ir}_{\dR}(X/S,\sigma,r)$ of rigidified vector bundles of rank $r$, equipped with irreducible flat relative connections. Let $\Ecal_{\dR}^{\un}$ be the universal vector bundle and $\nabla^{\un}$ the universal relative connection. Because $S$ is non-singular, $\Rbold^{\ir}_{\dR}(X/S, \sigma, r)$ is non-singular as well. The results of \textsection \ref{subsec:intersection-conn-flat} apply and produce a holomorphic intersection connection $\nabla^{\ICmini}$ on $IC_{2}(\Ecal^{\un}_{\dR})$, which lives on $\Rbold^{\ir}_{\dR}(X/S, \sigma, r)$. Here we stress the holomorphic nature of $\nabla^{\ICmini}$, which is granted by Theorem \ref{theorem:curvature-int-conn-holo}. By Proposition \ref{prop:descent-IC2}, we know that $IC_{2}(\Ecal^{\un}_{\dR})$ descends to the coarse moduli $\Mbold_{\dR}^{\ir}(X/S, r)$. The following lemma shows that the connection descends, too.

\begin{lemma}\label{lemma:CS-descends}
The line bundle with holomorphic connection $(IC_{2}(\Ecal^{\un}_{\dR}),\nabla^{\ICmini})$ on $\Rbold^{\ir}_{\dR}(X/S,\sigma, r)$ descends to $\Mbold_{\dR}^{\ir}(X/S, r)$. The descended connection is independent of the choice of section $\sigma\colon S\to X$.
\end{lemma}

\begin{proof}
Recall from \textsection \ref{subsub:coarsemodulidR} that, locally on $\Mbold^{\ir}_{\dR}(X/S, r)$, there is a universal object on $X\times _{S}\Mbold^{\ir}_{\dR}(X/S, r)$,  unique modulo twisting by a line bundle coming from the base $\Mbold^{\ir}_{\dR}(X/S, r)$. By the description \textsection\ref{subsubsec:descent-stable-locus} of the descended line bundle in terms of local universal objects, and by Lemma \ref{lemma:canonical-connection-IC2}, the claim amounts to the following observation. Let $E$ be a holomorphic vector bundle on $X$ with a compatible connection. Let $L$ be a line bundle on $S$, endowed also with a compatible connection. Then the attached intersection connections on $IC_{2}(E\otimes f^{\ast}L)$ and $IC_{2}(E)$ coincide via the canonical isomorphism $IC_{2}(E\otimes f^{\ast}L)\simeq IC_{2}(E)$. This is an immediate consequence of Proposition \ref{prop:properties-nabla-IC2}, and the fact that the intersection connections on the canonically trivial line bundles $\langle\det E, f^{\ast}L\rangle\simeq\Ocal_{S}$ and $\langle f^{\ast}L, f^{\ast}L\rangle\simeq\Ocal_{S}$ (cf. Proposition \ref{Prop:generalpropertiesDeligneproduct}) are the trivial connections, as can be checked from the explicit expression \eqref{eq:gg-10}.
\end{proof}

\begin{definition}
The complex Chern--Simons line bundle on $\Mbold_{\dR}^{\ir}(X/S, r)$ is the line bundle with holomorphic connection  obtained by descent from $(IC_{2}(\Ecal^{\un}_{\dR}),\nabla^{\ICmini})^{\vee}$ on $\Rbold^{\ir}_{\dR}(X/S,\sigma, r)$ (notice the dual). We denote it $(\Lcal_{\CS}(X/S),\nabla^{\CS})$, or more loosely $\Lcal_{\CS}(X/S)$.
\end{definition}

The definition carries over to the case of $\SL_{r}$ monodromies, and results in a complex Chern--Simons line bundle on $\Mbold_{\dR}^{\ir}(X/S,\SL_{r})$. One  checks that it is the pullback of $\Lcal_{\CS}(X/S )$ through the natural map $\Mbold_{\dR}^{\ir}(X/S,\SL_{r})\to \Mbold_{\dR}^{\ir}(X/S, r)$. Therefore, we will employ the same notation for the complex Chern--Simons line bundle on $\Mbold_{\dR}^{\ir}(X/S,\SL_{r})$.

Via the Riemann--Hilbert correspondence, we can transport the complex Chern--Simons line bundle to the Betti moduli space $\Mbold_{\Bet}^{\ir}(X/S,r)$, and similarly in the $\SL_{r}$ case. We will still write $(\Lcal_{\CS}(X/S),\nabla^{\CS})$ or $\Lcal_{\CS}(X/S)$ for the resulting holomorphic line bundle with connection on the Betti space, and occasionally refer to it as the Betti realization. 

\begin{remark}
\begin{enumerate}
    \item By Lemma \ref{lemma:CS-descends}, the construction of $\Lcal_{\CS}(X/S)$ is independent of the section $\sigma$. By a simple descent argument, this allows to generalize the definition in the absence of a section.
    \item The line bundle underlying $\Lcal_{\CS}(X/S)$ is defined on the whole $\Rbold_{\dR}(X/S, \sigma, r)$ and descends to $\Mbold_{\dR}(X/S, r)$, as we proved in Proposition \ref{prop:descent-IC2}. Nevertheless, for the connection to be defined, we need to restrict to the locus of irreducible representations. It would be interesting to know whether the connection extends as a singular connection. For a partial result in this direction, see Corollary \ref{cor:extension-int-connection}. This problem is related to the type of singularities of the moduli spaces. In the $\SL_{r}$ setting, the singularities are known to be canonical for high enough genus \cite{Aizenbud-Avni}. 
\end{enumerate}
\end{remark}

The first property of $\Lcal_{\CS}(X/S)$ we wish to discuss is the base change functoriality. For this, recall from Proposition \ref{prop:flatness} that the formation of $\Mbold_{\dR}(X/S,r)$ commutes with base change.

\begin{proposition}\label{prop:base-chanhge-LCS}
Let $S^{\prime}\to S$ be a morphism of complex manifolds, and $q\colon\Mbold_{\dR}^{\ir}(X^{\prime}/S^{\prime},r)\to\Mbold_{\dR}^{\ir}(X/S,r)$ the natural map. Then there is a canonical isomorphism of line bundles with connections $q^{\ast}\Lcal_{\CS}(X/S)\simeq\Lcal_{\CS}(X^{\prime}/S^{\prime})$.
\end{proposition}
\begin{proof}
By Proposition  \ref{prop:descent-IC2} there is a canonical isomorphism $q^{\ast}IC_{2}(\Ecal^{\un}_{\dR})\simeq IC_{2}(\Ecal^{\un\ \prime}_{\dR})$. In terms of local universal objects, this is given as $q^{\ast}IC_{2}(\Fcal^{\un}_{\dR})\simeq IC_{2}(\Fcal^{\un\ \prime}_{\dR})$, in turn induced by $q^{\ast}\Fcal^{\un}_{\dR}\simeq\Fcal^{\un\ \prime}_{\dR}\otimes\pi^{\prime\ \ast}L$. Here $\pi^{\prime}\colon X^{\prime}\times_{S^{\prime}}\Mbold_{\dR}^{\ir}(X^{\prime}/S^{\prime},r)\to \Mbold_{\dR}^{\ir}(X^{\prime}/S^{\prime},r)$ is the natural projection. See \textsection \ref{subsubsec:descent-stable-locus} for this description. Reasoning as in the proof of Lemma \ref{lemma:CS-descends} to get rid of $\pi^{\prime\ \ast}L$, and then recalling the independence property of Lemma \ref{lemma:canonical-connection-IC2}, we infer that the isomorphism $q^{\ast}IC_{2}(\Fcal^{\un}_{\dR})\simeq IC_{2}(\Fcal^{\un\ \prime}_{\dR})$ preserves the universal intersection connections. 
\end{proof}

There is an obvious functoriality with respect to isomorphisms $X^{\prime}\to X$ of relative curves over $S$:
\begin{lemma}\label{lemma:stupid-lemma-3}
Let $g\colon X^{\prime}\to X$ be an isomorphism of relative curves over $S$. Let $\widetilde{g}\colon\Mbold_{\dR}^{\ir}(X^{\prime}/S,r)\simeq\Mbold_{\dR}^{\ir}(X/S,r)$ be the induced natural isomorphism. Then, there is a canonical isomorphism $\widetilde{g}^{\ast}\Lcal_{\CS}(X/S)\simeq\Lcal_{\CS}(X^{\prime}/S)$.
\end{lemma}
\begin{proof}
Inspecting the construction of the complex Chern--Simons line bundles, this reduces to Lemma \ref{lemma:stupid-lemma-2} and Lemma \ref{lemma:canonical-connection-IC2}. The details are left to the reader.
\end{proof}

\subsubsection{Curvature of $\Lcal_{\CS}(X/S)$}\label{subsubsec:curvature-CS}
To state the curvature theorem for $\Lcal_{\CS}(X/S)$, we will next argue in the Betti realization. Since the curvature is a local invariant, it is enough to describe it after restricting to a simply connected open subset $S^{\circ}\subseteq S$. Fix a base point $0\in S^{\circ}$. Recall from \textsection \ref{subsub:coarseBetti} that $\Mbold{\Bet}(X/S, r)$ is a local system of complex analytic spaces, so there is a natural isomorphism $\Mbold_{\Bet}^{\ir}(X/S, r)_{\mid S^{\circ}}\simeq\Mbold_{\Bet}^{\ir}(X_{0},r)\times S^{\circ}$ over $S^{\circ}$, and a resulting holomorphic retraction to the fiber at $0$:
\begin{displaymath}
    p_{0}\colon \Mbold_{\Bet}^{\ir}(X/S, r)_{\mid S^{\circ}}\to \Mbold_{\Bet}^{\ir}(X_{0},r).
\end{displaymath}

\begin{theorem}\label{theorem:curvature-CS-single-RS}
Let the notation be as above. Then:
\begin{enumerate}
    \item The first Chern form of the complex Chern--Simons line bundle on $\Mbold_{\Bet}^{\ir}(X_{0},r)$ is 
    \begin{equation}\label{eq:c1-CS-1}
        c_{1}(\Lcal_{\CS}(X_{0}),\nabla^{\CS})=\frac{1}{4\pi^{2}}\left(\omega_{\scriptscriptstyle{\GL_{r}}}-\mathrm{det}^{\ast}(\omega_{\scriptscriptstyle{\GL_{1}}})\right).
    \end{equation}
    \item\label{item:c1-CS-2} The first Chern form of the complex Chern--Simons line bundle on $\Mbold_{\Bet}^{\ir}(X_{0},\SL_{r})$ is
    \begin{equation}\label{eq:c1-CS-2}
        c_{1}(\Lcal_{\CS}(X_{0}),\nabla^{\CS})=\frac{1}{4\pi^{2}}\omega_{\scriptscriptstyle{\SL_{r}}}.
    \end{equation}
    \item In general, the first Chern form of $(\Lcal_{\CS}(X/S),\nabla^{\CS})$ is crystalline. That is, it satisfies
\begin{equation}\label{eq:local-curvature-CS}
    c_{1}(\Lcal_{\CS}(X/S),\nabla^{\CS})_{\mid S^{\circ}}=p_{0}^{\ast}c_{1}(\Lcal_{\CS}(X_{0}),\nabla^{\CS}).
\end{equation}
\end{enumerate}
\end{theorem}
\begin{proof}
The first and second formulas are a direct application of Theorem \ref{theorem:curvature-int-conn-holo}. For the third one, because both sides of the identity are holomorphic differential forms, it is enough to check the claim over a non-empty open neighborhood of $0$ contained in $S^{\circ}$. Therefore, we can reduce to the case when $S^{\circ}$ is contractible and admits holomorphic coordinates. Then we are in position to apply Theorem \ref{theorem:curvature-int-conn-holo}, from which one concludes the desired equality by a simple inspection.
\end{proof}
\begin{remark}\label{rmk:change-sign}
While $\Mbold_{\Bet}^{\ir}(X_{0},r)$ depends on $X_{0}$ only as a topological surface, the curvature of $\Lcal_{\CS}(X_{0})$ depends on the orientation of $X_{0}$. A change of the orientation would change the sign of the curvature. See \textsection \ref{subsub:atiyahbottgoldman}.
\end{remark}

\begin{corollary}\label{cor:CS-crystalline}
Except possibly for $g=r=2$, the complex Chern--Simons line bundle on $\Mbold_{\Bet}^{\ir}(X/S,\SL_{r})$ is crystalline. That is, if $(S^{\circ},0)$ is a simply connected open subset of $S$ with a base point, and $p_{0} : \Mbold_{\Bet}^{\ir}(X/S,\SL_{r})_{\mid S^{\circ}}\to\Mbold_{\Bet}^{\ir}(X_{0},\SL_{r})$ is the retraction to the fiber at 0, then there exists a unique holomorphic isomorphism of flat vector bundles
\begin{displaymath}
    (\Lcal_{\CS}(X/S),\nabla^{\CS})_{\mid S^{\circ}}\overset{\sim}{\longrightarrow} p_{0}^{\ast}(\Lcal_{\CS}(X_{0}),\nabla^{\CS})
\end{displaymath}
restricting to the identity on $\Mbold_{\Bet}^{\ir}(X_{0},\SL_{r})$.
\end{corollary}
\begin{proof}
First, the complement of $\Mbold_{\Bet}^{\ir}(X_{0},\SL_{r})$ in $\Mbold_{\Bet}(X_{0},\SL_{r})$ has codimension $\geq 2$, by \cite[Proposition 11.3]{Simpson:moduli-2}. Although this reference is stated for $\GL_{r}$ monodromies, the argument can be adapted to $\SL_{r}$ monodromies. By \cite[Theorem B]{Biswas-Lawton}, the space $\Mbold_{\Bet}(X_{0},\SL_{r})$ is simply connected. We infer that $\Mbold_{\Bet}^{\ir}(X_{0},\SL_{r})$ is simply connected too, and consequently so is $\Mbold_{\Bet}^{\ir}(X/S,\SL_{r})_{\mid S^{\circ}}$. Second, by Theorem \ref{theorem:curvature-CS-single-RS}, $(\Lcal_{\CS}(X/S),\nabla^{\CS})_{\mid S^{\circ}}$ and $p_{0}^{\ast}(\Lcal_{\CS}(X_{0}),\nabla^{\CS})$ have the same curvature. Moreover, by Proposition \ref{prop:base-chanhge-LCS}, $(\Lcal_{\CS}(X/S),\nabla^{\CS})_{\mid S^{\circ}}$ restricts to $(\Lcal_{\CS}(X_{0}),\nabla^{\CS})$ on $\Mbold_{\Bet}^{\ir}(X_{0},\SL_{r})$. Therefore, by parallel transport there exists a unique extension of the identity automorphism of $(\Lcal_{\CS}(X_{0}),\nabla^{\CS})$ into a horizontal isomorphism $\Lcal_{\CS}(X/S)\simeq p_{0}^{\ast}\Lcal_{\CS}(X_{0})$, necessarily holomorphic, as stated. 
\end{proof}


\subsection{Characterization}\label{subsec:characterization-CS}
The setting and notation of the previous subsection are still in force. We  proceed to characterize $\Lcal_{\CS}(X/S)$ by an extension property from the moduli space of stable vector bundles.

Recall the moduli space $\Rbold^{\sst}(X/S,\sigma,r)$ of rigidified stable vector bundles of rank $r$ and degree 0. By the theorem of Narasimhan--Seshadri \cite[Theorem 2]{Narasimhan-Seshadri}, the universal vector bundle $\Ecal^{\un}$ carries a natural fiberwise flat Chern connection, whence an induced intersection connection $\nabla^{\ICmini}$ on $IC_{2}(\Ecal^{\un})$. This intersection connection is itself a Chern connection, by Proposition \ref{prop:int-conn-chern}. Arguing as in Lemma \ref{lemma:CS-descends}, it descends to the coarse moduli space $\Mbold^{\sst}(X/S, r)$, and the result is independent of the choice of section. It will be useful to give it a name and a specific notation.

\begin{definition}\label{def:LCSU}
The unitary Chern--Simons line bundle on $\Mbold^{\sst}(X/S, r)$ is the line bundle with compatible connection obtained by descent from $(IC_{2}(\Ecal^{\un}),\nabla^{\ICmini})^{\vee}$ on $\Rbold^{\sst}(X/S, \sigma, r)$ (notice the dual). We denote it as  $(\Lcal_{\CSU}(X/S),\nabla^{\CSU})$, or more loosely $\Lcal_{\CSU}(X/S)$.
\end{definition}
 There is an analogous construction in the trivial determinant case, which agrees with the pullback of $\Lcal_{\CSU}(X/S)$ by the natural morphism $\Mbold^{\sst}(X/S,\SL_{r})\rightarrow\Mbold^{\sst}(X/S, r)$. Therefore, we still denote it by $\Lcal_{\CSU}(X/S)$. Let us mention that the latter is a relative variant of the Ramadas--Singer--Weitsman line bundle \cite[Theorems 1 \& 2]{RSW}, although this will not be needed.

 The theorem of Narasimhan--Seshadri \cite[Theorem 2]{Narasimhan-Seshadri} can be recast as a closed immersion of $\Ccal^{\infty}$ manifolds $\Mbold^{\sst}(X/S, r)\hookrightarrow\Mbold^{\sst}_{\dR}(X/S, r)$, a section of the forgetful map $\pi: \Mbold^{\sst}_{\dR}(X/S, r)\to \Mbold^{\sst}(X/S, r)$. This is a totally real embedding, as it stems from the interpretation as character varieties. The corresponding fact also holds in the trivial determinant case. Next, recall the open inclusion $\Mbold^{\sst}_{\dR}(X/S, r)\subset \Mbold^{\ir}_{\dR}(X/S, r)$. This allows us to consider $\Mbold^{\sst}(X/S, r) \subseteq \Mbold^{\ir}_{\dR}(X/S, r)$.

\begin{theorem}\label{theorem:extension-CS}
The complex Chern--Simons line bundle $\Lcal_{\CS}(X/S)$ on $\Mbold_{\dR}^{\ir}(X/S, r)$ is a holomorphic extension of $\Lcal_{\CSU}(X/S)$ on $\Mbold^{\sst}(X/S, r)$. Except possibly for $g=r=2$, on $\Mbold_{\dR}^{\ir}(X/S,\SL_{r})$ it is the unique extension with the additional property of being crystalline, up to unique isomorphism.
\end{theorem}
\begin{proof}
Let us first see that $\Lcal_{\CS}(X/S)$ indeed extends $\Lcal_{\CSU}(X/S)$. For this, we can argue at the level of representations spaces and universal $IC_{2}$ bundles, with their natural intersection connections. Denote the forgetful map $\pi\colon\Rbold_{\dR}^{\sst}(X/S,\sigma, r)\to\Rbold^{\sst}(X/S,\sigma, r)$. The universal vector bundle $\Ecal^{\un}$ pulls-back to the universal vector bundle $\Ecal_{\dR}^{\un}$. Therefore, by the compatibility of $IC_{2}$ with base change, we have $IC_{2}(\Ecal^{\un}_{\dR})=\pi^{\ast} IC_{2}(\Ecal^{\un})$. 

Denote by $\nabla^{\chmini}$ the Chern connection on $\Ecal^{\un}$. Pulling it back by $\pi$, we obtain a Chern connection on $\Ecal^{\un}_{\dR}$, still written as $\nabla^{\chmini}$. Let $\widetilde{\nabla}$ be a compatible extension of $\nabla^{\un}$. By the construction of intersection connections, we reduce to show that the Chern--Simons integral $IT(\nabla^{\chmini},\widetilde{\nabla})$ restricts to 0 on $\Rbold^{\sst}(X/S, \sigma, r)$, for $\pi^{\ast}(IC_{2}(\Ecal^{\un}),\nabla^{\ICmini})$ trivially restricts to $(IC_{2}(\Ecal^{\un}),\nabla^{\ICmini})$.

Write $\widetilde{\nabla}=\nabla^{\chmini}+\theta$, with $\theta$ a $(1,0)$-form with values in $\End\Ecal^{\un}_{\dR}$. Proposition \ref{prop:ICS-explicit-formula} provides an expression for $IT(\nabla^{\chmini},\nabla^{\chmini}+\theta)$. We will work out the several terms in that expression, and see that they individually vanish when restricted to $\Rbold^{\sst}(X/S, \sigma, r)$. For this purpose, the key observation is that the restrictions of $\widetilde{\nabla}$ and $\nabla^{\chmini}$ to fibers over $\Rbold^{\sst}(X/S,\sigma, r)$ coincide. Therefore, along $\Rbold^{\sst}(X/S, \sigma, r)$ the form $\theta$ becomes a  section of $\End \Ecal^{\un} \otimes g^{\ast}\Acal^{1}_{\Rbold^{\sst}(X/S, \sigma, r)}$, where $g\colon X\times_{S} \Rbold^{\sst}(X/S, \sigma, r)\to \Rbold^{\sst}(X/S, \sigma, r)$ is the universal curve. Together with the fact that $F^{\chmini}$ vanishes on fibers, we infer that
\begin{displaymath}
    \int_{\Xcal/\Rbold^{\sst}_{\dR}(X/S, \sigma, r)}\tr(F^{\chmini}\wedge\theta)\quad\text{vanishes along}\quad \Rbold^{\sst}(X/S,\sigma, r),
\end{displaymath}
where $\Xcal=X\times_{S}\Rbold^{\sst}_{\dR}(X/S,\sigma, r)$. For type reasons, the next term provided by Proposition \ref{prop:ICS-explicit-formula} can be rewritten as
\begin{displaymath}
    \int_{\Xcal/\Rbold^{\sst}_{\dR}(X/S, \sigma, r)}\tr(\theta\wedge\ov{\partial}\theta)=\int_{\Xcal/\Rbold^{\sst}_{\dR}(X/S,\sigma,  r)}\tr(\theta\wedge\nabla^{\chmini}\theta).
\end{displaymath}
With the latter presentation, it is easier to understand the restriction of the integral to $\Rbold^{\sst}(X/S, \sigma, r)$, since contrary to $\ov{\partial}$, the action of $\nabla^{\chmini}$ commutes with the restriction operation. We use again that the restriction of $\theta$ becomes a smooth section of $\End\Ecal^{\un} \otimes g^{\ast}\Acal^{1}_{\Rbold^{\sst}(X/S,\sigma, r)}$. Since the fiber integral reduces the degree of differential forms by 2, we necessarily have that
\begin{displaymath}
    \int_{\Xcal/\Rbold^{\sst}_{\dR}(X/S, \sigma, r)}\tr(\theta\wedge\nabla^{\chmini}\theta)\quad\text{vanishes along}\quad \Rbold^{\sst}(X/S, \sigma, r).
\end{displaymath}
The rest of the terms in Proposition \ref{prop:ICS-explicit-formula} are treated in the same way.

For the second part of the theorem, we first need to review a couple of facts. The first one is that the complement of $\Mbold_{\dR}^{\sst}(X/S,\SL_{r})$ in $\Mbold_{\dR}^{\ir}(X/S,\SL_{r})$ has codimension $\geq 2$. If $S$ is reduced to a point, the analogous question for Higgs bundles has been treated by Hitchin \cite[pp. 371--373]{Hitchin:flat}, and his method adapts to flat vector bundles. The general case is deduced from the stated property on fibers. The second fact we will need is that if $S^{\circ}\subseteq S$ is a contractible open subset, then $\Mbold^{\sst}(X/S,\SL_{r})_{\mid S^{\circ}}$ and $\Mbold_{\dR}^{\sst}(X/S,\SL_{r})_{\mid S^{\circ}}$ are simply connected. For $\Mbold_{\dR}^{\sst}(X/S,\SL_{r})_{\mid S^{\circ}}$, we use that it has codimension $\geq 2$ in $\Mbold_{\dR}^{\ir}(X/S,\SL_{r})_{\mid S^{\circ}}$, which we already know to be simply connected (proof of Corollary \ref{cor:CS-crystalline}). Next for $\Mbold^{\sst}(X/S,\SL_{r})_{\mid S^{\circ}}$. Given $0\in S^{\circ}$, the space $\Mbold^{\sst}(X/S,\SL_{r})_{\mid S^{\circ}}$ is $\Ccal^{\infty}$ isomorphic to $\Mbold^{\sst}(X_{0},\SL_{r})\times S^{\circ}$, as follows from the existence of Ehresmann trivializations and the interpretation as a character variety provided by the theorem of Narasimhan--Seshadri.\footnote{One can also go through Higgs bundles, and proceed as in Simpson's \cite[Lemma 7.17--7.18]{Simpson:moduli-2}, for stable bundles are stable Higgs bundles with trivial Higgs field, and correspond to unitary representations.} Except for $g=r=2$, the complement of $\Mbold^{\sst}(X_{0},\SL_{r})$ in $\Mbold(X_{0},\SL_{r})$ has complex codimension $\geq 2$, because $\Mbold(X_{0},\SL_{r})$ is normal and the stable locus is exactly the smooth locus. See \cite[Theorem 1]{Narasimhan-Ramanan} for a proof of this fact for semistable vector bundles of degree 0; the trivial determinant case is similar. Finally, $\Mbold(X_{0},\SL_{r})$ is simply connected by \cite[Theorem B]{Biswas-Lawton}, hence so do $\Mbold^{\sst}(X_{0},\SL_{r})$ and $\Mbold^{\sst}(X/S,\SL_{r})_{\mid S^{\circ}}\simeq_{\Ccal^{\infty}}\Mbold^{\sst}(X_{0},\SL_{r})\times S^{\circ}$.

Now for the proof of uniqueness. Let $(\Lcal_{1},\nabla_{1})$ and $(\Lcal_{2},\nabla_{2})$ be two holomorphic and crystalline extensions of $(\Lcal_{\CSU}(X/S),\nabla^{\CSU})$. We are thus given isomorphisms of $\Ccal^{\infty}$ line bundles with connections: $\varphi_{j}\colon (\Lcal_{j},\nabla_{j})_{\mid \Mbold^{\sst}(X/S,\SL_{r})}\simeq (\Lcal_{\CSU}(X/S),\nabla^{\CSU})$, $j=1,2$. Consider the product $\Lcal=\Lcal_{1}\otimes\Lcal_{2}^{\vee}$ with the holomorphic connection $\nabla$ induced by $\nabla_{1}$ and $\nabla_{2}$. The isomorphisms $\varphi_{j}$ induce a $\Ccal^{\infty}$ trivialization of $(\Lcal,\nabla)$ along $\Mbold^{\sst}(X/S,\SL_{r})$, denoted by $\varphi$. We will show that $\varphi$ has an extension to a trivialization of $(\Lcal,\nabla)$. Notice that such an extension will automatically be holomorphic, and then necessarily unique: use the holomorphicity and that the embedding $\Mbold^{\sst}(X/S,\SL_{r})\hookrightarrow\Mbold^{\sst}_{\dR}(X/S,\SL_{r})$ is totally real on fibers, plus the density of $\Mbold^{\sst}_{\dR}(X/S,\SL_{r})$ in $\Mbold^{\ir}_{\dR}(X/S,\SL_{r})$. Let $F_{\nabla}$ be the curvature of $\nabla$, which is a holomorphic 2-form. Let us see that it vanishes. This can be checked locally with respect to $S$, so that we can localize over a contractible open subset $S^{\circ}$. Because $\Lcal$ is crystalline by assumption, the curvature of $\nabla$ restricted to $S^{\circ}$ is completely determined by the restriction to a fiber, say $\Mbold^{\ir}_{\dR}(X_{0},\SL_{r})$ for $0\in S^{\circ}$. And in turn, this is determined by further restricting to the dense open subset $\Mbold^{\sst}_{\dR}(X_{0},\SL_{r})$. But $F_{\nabla}$ restricts to 0 on the totally real submanifold $\Mbold^{\sst}(X_{0},\SL_{r})\subset \Mbold^{\sst}_{\dR}(X_{0},\SL_{r})$. Therefore, because it is holomorphic, $F_{\nabla}$ necessarily vanishes on $\Mbold^{\sst}_{\dR}(X_{0},\SL_{r})$, as required. To sum up, we have shown that $\nabla$ is flat. To extend the trivialization $\varphi$, it is again enough to argue over contractible open subsets $S^{\circ}\subseteq S$. Indeed, such local extensions will  necessarily be holomorphic and unique, and hence they will glue into a global extension. We saw that $\Mbold_{\dR}^{\sst}(X/S,\SL_{r})_{\mid S^{\circ}}$ is simply connected. Therefore, by parallel transport, the trivialization $\varphi$ on $\Mbold^{\sst}(X/S,\SL_{r})_{\mid S^{\circ}}$ extends to a trivialization of $(\Lcal,\nabla)$ on $\Mbold_{\dR}^{\sst}(X/S,\SL_{r})_{\mid S^{\circ}}$. This further extends to $\Mbold_{\dR}^{\ir}(X/S,\SL_{r})_{\mid S^{\circ}}$, because the complement of $\Mbold_{\dR}^{\sst}(X/S,\SL_{r})_{\mid S^{\circ}}$ has codimension $\geq 2$. This concludes the proof.
\end{proof}

Since we have just been dealing with moduli of stable vector bundles, we take the occasion to state one more property of $\Lcal_{\CS}(X/S)$.

\begin{proposition}\label{prop:restriction-CS-fibers}
The restriction of $(\Lcal_{\CS}(X/S),\nabla^{\CS})$ to the fibers of $\,\Mbold_{\dR}^{\sst}(X/S, r)\to\Mbold^{\sst}(X/S, r)$ is trivial.
\end{proposition}
\begin{proof}
The statement can be reformulated as follows. Let $X$ be a compact Riemann surface, and $E$ a stable vector bundle of degree 0 on it. It has a flat Chern connection $\nabla^{\chmini}$. Introduce the affine variety $S$ given by $H^{0}(X,\End E \otimes\Omega^{1}_{X})$. For concreteness, we take a basis of the latter $\theta_{1},\ldots,\theta_{n}$, and the associated holomorphic coordinates $s_{1},\ldots,s_{n}$ on $S$. We form the product $\Xcal=X\times S$, with projections $p_{1},p_{2}$. Then $(p_{1}^{\ast}E,p_{1}^{\ast}\nabla^{\chmini})$ is a holomorphic vector bundle with flat Chern connection on $\Xcal$. Let $\vartheta=\sum_{j}s_{j}\theta_{j}$ be the universal holomorphic differential form on $\Xcal$ with values in $\End p_{1}^{\ast}E$. Finally, set $\widetilde{\nabla}=p_{1}^{\ast}E+\vartheta$, which is a holomorphic connection on $p_{1}^{\ast}E$. If $\nabla^{\ICmini}$ is the associated intersection connection on $IC_{2}(p_{1}^{\ast}E)$, we must show that this defines a trivial line bundle with trivial connection.

The line bundle $IC_{2}(p_{1}^{\ast}E)$ is indeed trivial: by base change functoriality, $IC_{2}(p_{1}^{\ast}E)$ is the trivial line bundle on $S$ with fiber the complex line $IC_{2}(E)$. Similarly, the intersection connection attached to the Chern connection $\nabla^{\chmini, \ICmini}$ is the trivial connection. We are thus led to prove that $IT(p_{1}^{\ast}\nabla^{\chmini},p_{1}^{\ast}\nabla^{\chmini}+\vartheta)$ vanishes. For this, we invoke the explicit formula of Proposition \ref{prop:ICS-explicit-formula}. The vanishing of $IT$ is then automatic, since the curvature of $p_{1}^{\ast}\nabla^{\chmini}$ is $F^{\chmini}=0$ and $\ov{\partial}\vartheta=0$.
\end{proof}

\begin{remark}
As an addendum to Remark \ref{rmk:Fay-a-tout-fait} \eqref{item:Fay-a-tout-fait-3}, a variant of Proposition \ref{prop:restriction-CS-fibers} in the  language of Szeg\"o kernels was observed by Fay in \cite[p. 176]{Fay-2}, after equation $(5)'$ in \emph{loc. cit.}
\end{remark}

\subsection{Complex metrics}\label{subsection:complex-metrics}
\subsubsection{General theory}
Let $X_{0}$ be a compact Riemann surface of genus $g\geq 2$, with a point $p\in X_{0}$. We denote by $\ov{X}_{0}$ the complex conjugate Riemann surface, obtained by changing the almost complex structure from $I$ to $-I$. The underlying differentiable surfaces are thus the same, hence so are their fundamental groups based at $p$. We write $\Gamma=\pi_{1}(X_{0},p)$. Using the notation of \textsection \ref{subsub:coarseBetti} we have identified $\Mbold(\Gamma, r)=\Mbold_{\Bet}(X_{0},r)=\Mbold_{\Bet}(\ov{X}_{0},r)$ with the space of irreducible representations $\Mbold^{\ir}(\Gamma, r)$. We similarly have the spaces $\Mbold(\Gamma,\SL_{r})$ and $\Mbold^{\ir}(\Gamma,\SL_{r})$. 

Consider the complex Chern--Simons line bundles $\Lcal_{\CS}(X_{0})$ and $\Lcal_{\CS}(\ov{X}_{0})$ in their Betti realizations, hence living on $\Mbold^{\ir}(\Gamma, r)$. The underlying line bundles are actually defined on the whole of $\Mbold(\Gamma, r)$. Our goal is to canonically trivialize $\Lcal_{\CS}(X_{0})\otimes_{\Ocal_{\Mbold(\Gamma, r)}}\Lcal_{\CS}(\ov{X}_{0})$ as a holomorphic line bundle, in such a way that on $\Mbold^{\ir}(\Gamma, r)$ the trivialization is horizontal for the tensor product connection.

We need some preparation. We begin with the space $\Rbold^{\sst}(X_{0}, p, r)$. The universal vector bundle $\Ecal^{\un}$ carries a smooth hermitian metric, fiberwise flat, for which the rigidification is orthonormal. We already know that $IC_{2}(\Ecal^{\un})$ and the intersection connection descend to $\Mbold^{\sst}(X_{0}, r)$. We need a bit more:

\begin{lemma}\label{lemma:metric-IC2-descends}
The smooth hermitian metric on $IC_{2}(\Ecal^{\un})$ descends to $\Mbold^{\sst}(X_{0}, r)$, and it does not depend on the choice of base point $p\in X_{0}$.
\end{lemma}
\begin{proof}
This can all be reframed as the rigidification independence of the metric on $IC_{2}$. Let $E$ be a stable vector bundle of degree 0 on $X_{0}$, and $\ebold$ and $\ebold^{\prime}$ two bases of $E_{p}$. Denote by $h_{1}$ and $h_{2}$ the associated flat hermitian metrics, for which $\ebold$ and $\ebold^{\prime}$ are orthonormal, respectively. We must show that the induced metrics on $IC_{2}(E)$ coincide. To see this, we may place ourselves in the universal setting. Namely, let $S$ be the complex variety of bases of $E_{p}$. We form the product $\Xcal=X_{0}\times S$, with projections $p_{1}$ and $p_{2}$. On the vector bundle $\Ecal=p_{1}^{\ast}E$, which is rigidified along the constant section $p$, there is a universal hermitian metric, fiberwise flat. If $\nabla$ is its Chern connection, and $\nabla^{\ICmini}$ the associated intersection connection on $IC_{2}(\Ecal)$, we have to see that $\nabla^{\ICmini}$ is trivial. For this to make sense, as in the proof of Proposition \ref{prop:restriction-CS-fibers}, we first notice that $IC_{2}(\Ecal)$ is the trivial line bundle with fiber $IC_{2}(E)$ on $S$.

Fix an auxiliary flat Chern connection $\nabla_{0}$ on $E$, and equip $\Ecal$ with the pullback connection, still denoted by $\nabla_{0}$. Write $\nabla=\nabla_{0}+\theta$, where $\theta$ is a differential form of type $(1,0)$ with values in the endomorphism bundle. The flat Chern connection on $E$ is independent of the rigidification, and this means that the vertical projection of $\theta$ vanishes. That is, $\theta$ is a smooth section of $\End \Ecal \otimes p_{2}^{\ast}\Acal^{1,0}_{S}$. Now we apply Proposition \ref{prop:ICS-explicit-formula}, and we see that $IT(\nabla_{0},\nabla_{0}+\theta)=0$. Indeed, on the one hand the curvature of $\nabla_{0}$ vanishes. On the other hand, for type reasons, the integral
\begin{displaymath}
    \int_{\Xcal/S}\tr(\theta\wedge\ov{\partial}\theta)=0,
\end{displaymath}
and similarly with the other terms in Proposition \ref{prop:ICS-explicit-formula}. As a result, the intersection connection attached to $\nabla$ on $IC_{2}(\Ecal)$ is the same as the intersection connection attached to $\nabla_{0}$. The latter is trivial, since $\nabla_{0}$ was a pullback from $X$ and intersection connections are compatible with base change. This concludes the proof.
\end{proof}
Let us take up the digression towards our goal. Recall the unitary counterpart of the Chern--Simons line bundle, $\Lcal_{\CSU}(X_{0})$ on $\Mbold^{\sst}(X_{0}, r)$ (Definition \ref{def:LCSU}). By Lemma \ref{lemma:metric-IC2-descends}, it carries a smooth hermitian metric whose Chern connection is then $\nabla^{\CSU}$. Similarly, we consider $\Lcal_{\CSU}(\ov{X}_{0})$ on $\Mbold^{\sst}(\ov{X}_{0}, r)$, endowed with its natural metric and Chern connection. By the Narasimhan--Seshadri theorem, the varieties $\Mbold^{\sst}(X_{0}, r)$ and $\Mbold^{\sst}(\ov{X}_{0}, r)$ are both naturally $\Ccal^{\infty}$ isomorphic to the space of conjugacy classes of irreducible unitary representations of $\Gamma$, $\Mbold^{\ir}(\Gamma,\U_{r})$. Actually, they are complex conjugate, as can be easily inferred from the explicit description of the complex structure in \cite[Theorem 4.1]{Fay}. In particular, $\Lcal_{\CSU}(X_{0})$ and $\Lcal_{\CSU}(\ov{X}_{0})$ live on the same $\Ccal^{\infty}$ manifold.

\begin{lemma}
As $\Ccal^{\infty}$ line bundles, $\Lcal_{\CSU}(X_{0})$ and $\Lcal_{\CSU}(\ov{X}_{0})$ are complex conjugate and isometric.
\end{lemma}
\begin{proof}
The proof rests on the following observations. First, consider a locally defined universal object $\Fcal^{\un}$ on $\Mbold^{\sst}(X_{0}, r)$. Then, in terms of the complex conjugate $\ov{\Fcal}^{\un}$, we see that $\ov{\Fcal}^{\un\ \vee}$ is a local universal object on $\Mbold^{\sst}(\ov{X}_{0},r)$. Second, we have the isomorphism $IC_{2}(\Fcal^{\un})\simeq IC_{2}(\Fcal^{\un\ \vee})$ of Proposition \ref{prop:dual}, which is an isometry by Proposition \ref{lemma:IC2-dual-metric}. Finally, expressing the universal $IC_{2}$ bundles in terms of the determinant of the cohomology \eqref{def:DelIC2}, we see that indeed there is a natural isomorphism
\begin{displaymath}
    \ov{IC_{2}(\Fcal^{\un})}=\ov{IC_{2}(\Fcal^{\un\ \vee})}\simeq IC_{2}(\ov{\Fcal}^{\un\ \vee}).
\end{displaymath}
\end{proof}

After the lemma, the hermitian metric on $\Lcal_{\CSU}(X_{0})$ can be interpreted as providing a $\Ccal^{\infty}$ trivialization of $\Lcal_{\CSU}(X_{0})\otimes\Lcal_{\CSU}(\ov{X}_{0})$ on $\Mbold^{\ir}(\Gamma,\U_{r})$. By definition, this trivialization is horizontal for the tensor product of the built-in Chern connections. We will refer to the trivialization of $\Lcal_{\CSU}(X_{0})\otimes\Lcal_{\CSU}(\ov{X}_{0})$ simply as the hermitian metric on $\Lcal_{\CSU}(X_{0})$.

\begin{theorem}\label{theorem:complex-metric}
Except possibly for $g=r=2$, the hermitian metric on $\Lcal_{\CSU}(X_{0})$ uniquely extends to a holomorphic trivialization of $\Lcal_{\CS}(X_{0})\otimes_{\Ocal_{\Mbold(\Gamma, r)}}\Lcal_{\CS}(\ov{X}_{0})$. On $\Mbold^{\ir}(\Gamma, r)$, the trivialization is horizontal for the tensor product of the complex Chern--Simons connections.
\end{theorem}

\begin{proof}
During the proof, we will write $\Lcal=\Lcal_{\CS}(X_{0})\otimes_{\Ocal_{\Mbold(\Gamma, r)}}\Lcal_{\CS}(\ov{X}_{0})$ and $\nabla$ for its (holomorphic) connection on $\Mbold^{\ir}(\Gamma, r)$. 

If such a holomorphic extension exists, then it is necessarily unique, because $\Mbold^{\ir}(\Gamma,\U_{r})\subset\Mbold^{\ir}(\Gamma, r)$ is totally real, and $\Mbold^{\ir}(\Gamma, r)$ is dense in $\Mbold(\Gamma, r)$. 

We remark that it suffices to establish the extension property from $\Mbold^{\ir}(\Gamma,\U_{r})$ to $\Mbold^{\ir}(\Gamma, r)$, since $\Lcal$ is a holomorphic line bundle on the whole $\Mbold(\Gamma, r)$, and the latter is a normal complex analytic space, with  $\codim(\Mbold(\Gamma, r)\setminus\Mbold^{\ir}(\Gamma, r))\geq 2$ by \cite[Proposition 11.3]{Simpson:moduli-2}.

Let us address the extension property. First, we restrict to the setting of $\Mbold^{\ir}(\Gamma,\SU_{r})\subset\Mbold^{\ir}(\Gamma,\SL_{r})$. For the first Chern form of $(\Lcal,\nabla)$, we have
\begin{displaymath}
     c_{1}(\Lcal,\nabla)=c_{1}(\Lcal_{\CS}(X_{0}),\nabla^{\CS})+c_{1}(\Lcal_{\CS}(\ov{X}_{0}),\nabla^{\CS})=0,
\end{displaymath}
by Theorem \ref{theorem:curvature-CS-single-RS} \eqref{item:c1-CS-2} and because $X_{0}$ and $\ov{X}_{0}$ have opposite orientations. See Remark \ref{rmk:change-sign}. This means that $(\Lcal,\nabla)$ is flat on $\Mbold^{\ir}(\Gamma,\SL_{r})$. By Theorem \ref{theorem:extension-CS}, we know that $(\Lcal,\nabla)$ extends the $\Ccal^{\infty}$ bundle with connection $\Lcal_{\CSU}(X_{0})\otimes\Lcal_{\CSU}(\ov{X}_{0})$ on $\Mbold^{\ir}(\Gamma,\SU_{r})$. The hermitian metric, seen as a trivialization, is horizontal for the connection on $\Lcal_{\CSU}(X_{0})\otimes\Lcal_{\CSU}(\ov{X}_{0})$. From the simple connectivity of $\Mbold^{\ir}(\Gamma,\SL_{r})$, via parallel transport by $\nabla$ we can extend the hermitian metric to a horizontal holomorphic trivialization. 

For $\GL_{r}$ monodromies,  simple connectivity is no longer available. But we can reduce to the case of $\SL_{r}$ monodromies and Deligne pairings of flat line bundles, treated in \cite{Freixas-Wentworth-2}. We provide the main lines of the argument and leave the details to the reader. If $E$ is a vector bundle of rank $r$ on $X_{0}$, and $L$ is an $r$-th root of $\det E$, then by Proposition \ref{prop:tensorlinebundleiso} we have a canonical functorial isomorphism
\begin{equation}\label{eq:dec-IC2-moduli}
    IC_{2}(E)=IC_{2}(E\otimes L^{-1}\otimes L)\simeq IC_{2}(E\otimes L^{-1})\otimes \langle L,L\rangle^{\binom{r}{2}}.
\end{equation}
If $E$ is endowed with a flat hermitian metric, so is $L$ and all the terms carry naturally induced metrics. The isomorphism is an isometry by Proposition \ref{prop:isometry_IC2-tensor}. A change of $L$ by an $r$-torsion line bundle leaves both sides invariant by construction, including with metrics.

A universal version of \eqref{eq:dec-IC2-moduli} can be realized on moduli spaces. Precisely, consider a Cartesian diagram 

\begin{displaymath}
   \xymatrix{
        \widetilde{\Mbold}^{\ir}(\Gamma,r)\ar[d]\ar[r]     &\Mbold(\Gamma,1)\ar[d]^{[r]}\\
        \Mbold^{\ir}(\Gamma,r)\ar[r]^{\det}    &\Mbold(\Gamma,1),
   }
\end{displaymath}
where $[r]$ is the $r$-power map, which is finite and \'etale, with Galois group $\Jac(X_{0})[r]$. A universal version of \eqref{eq:dec-IC2-moduli} exists on $\widetilde{\Mbold}^{\ir}(\Gamma,r)$. The piece corresponding to $IC_{2}(E\otimes L^{-1})$ comes from a moduli space $\Mbold^{\ir}(\Gamma,\SL_{r})$. The piece corresponding to $\langle L,L\rangle$ comes from $\Mbold(\Gamma,1)$. The universal version of \eqref{eq:dec-IC2-moduli} is compatible with the natural intersection connections. The holomorphic extension of the metric for the term of the form $IC_{2}(E\otimes L^{-1})$ has been treated above. For Deligne pairings of flat line bundles, this was addressed by Freixas--Wentworth \cite[Corollary 4.7 \& \textsection 4.4.2]{Freixas-Wentworth-2}. The conclusion is that on $\widetilde{\Mbold}^{\ir}(\Gamma,r)$ there is a holomorphic extension of the metric on $\Lcal_{\CSU}(X_{0})$. To show that it descends to $\Mbold^{\ir}(\Gamma,r)$, we need to justify that it is invariant under the action of $\Jac(X_{0})[r]$ on $\widetilde{\Mbold}^{\ir}(\Gamma,r)$. But we already saw that this is the case for the hermitian metrics. By uniqueness of the holomorphic extension, the invariance holds for the latter too. 
\end{proof}

\begin{definition}\label{def:complex-metric:LCS}
The holomorphic trivialization provided by Theorem \ref{theorem:complex-metric} is called the complex metric of $\Lcal_{\CS}(X_{0})\otimes_{\Ocal_{\Mbold(\Gamma, r)}}\Lcal_{\CS}(\ov{X}_{0})$, or simply $\Lcal_{\CS}(X_{0})$.
\end{definition}

\begin{remark}
\begin{enumerate}
    \item For Deligne pairings of flat line bundles, the holomorphic extension of the metric of \cite{Freixas-Wentworth-2} was obtained by an explicit construction, in terms of the symbols defining the Deligne pairings. This seems unavoidable, since $\Mbold(\Gamma,1)\simeq (\CBbb^{\times})^{2g}$ is not simply connected.
    \item We do not have a good understanding of the complex metric restricted to the boundary $\Mbold(\Gamma,r)\setminus\Mbold^{\ir}(\Gamma,r)$. It is natural to expect a relationship with the complex metrics of the Chern--Simons line bundles of the moduli spaces of lower rank bundles stratifying the boundary. 
    \item The complex Chern--Simons line bundles and the complex metric can be pulled back to $\Rbold(\Gamma,r)$. 
\end{enumerate}
\end{remark}

\subsubsection{Example: explicit formulas in rank 2}\label{subsub:explicit-formulas-rk-2}
We wish to describe the complex metrics of Theorem \ref{theorem:complex-metric} explicitly. For concreteness, we restrict to vector bundles of rank $2$. The discussion is further simplified by arguing on the level of representation spaces, rather than the coarse moduli. Finally, since the case of Deligne pairings was treated in \cite{Freixas-Wentworth-2}, as in \textsection \ref{subsec:IC2-metric-flat-unitary} we may tensor our vector bundles by a sufficiently positive hermitian line bundle, in order to achieve global generation. We notice that \emph{op. cit.} provides explicit formulas for Deligne pairings of flat rigidified line bundles with non-necessarily flat, hermitian line bundles.

Let $X_{0}=\Gamma\backslash\HBbb$ be a compact Riemann surface of genus $g\geq 2$, where $\Gamma\subset\PSL_{2}(\RBbb)$ is a torsion-free Fuchsian subgroup and $\HBbb$ is the upper half-plane. On $X_{0}$ we fix the hyperbolic metric of curvature $-1$. Hence, if $\tau=x+iy$ is the variable on $\HBbb$, then the Riemannian tensor is $|d\tau|^{2}/(\Imag\tau)^{2}=(dx^{2}+dy^{2})/y^{2}$. This induces a hermitian metric on the canonical sheaf $\omega_{X_{0}}$. We realize $\ov{X}_{0}$ as  $\Gamma\backslash\ov{\HBbb}$, where $\ov{\HBbb}$ is the lower half-plane, and endow $\omega_{\ov{X}_{0}}$ with the dual of the hyperbolic metric too. 

Given a representation $\rho\colon\Gamma\to\GL_{2}(\CBbb)$, we will denote by $\Ecal_{\rho}$ the corresponding flat vector bundle on $X_{0}$, and $\Ecal_{\rho}^{c}$ the flat vector bundle associated to $\rho^{\vee}$ on $\ov{X}_{0}$. Thus, these are fibers of universal vector bundles $\Ecal$ and $\Ecal^{c}$. Concretely, $\Ecal_{\rho}$ is obtained as $\HBbb\times_{\Gamma}\CBbb^{2}$, where $\Gamma$ acts on $\CBbb^{2}$ through $\rho$. Similarly $\Ecal_{\rho}^{c}$ is given by $\ov{\HBbb}\times_{\Gamma}\CBbb^{2}$, with $\Gamma$ acting through $\rho^{\vee}$ on $\CBbb^{2}$.

Let $\rho_{0}\in\Rbold^{\ir}(\Gamma,\U(2))$ be a unitary, irreducible representation. Then $\Ecal_{\rho_{0}}$ and $\Ecal_{\rho_{0}}^{c}$ are mutually conjugate, stable vector bundles of degree 0. Since stability is an open condition, there exists an open neighborhood $U$ of $\rho_{0}$ in $\Rbold^{\ir}(\Gamma,2)$ such that, for $\rho\in U$, both $\Ecal_{\rho}$ and $\Ecal_{\rho}^{c}$ are stable. Notice that, in general, $\Ecal_{\rho}^{c}$ is not complex conjugate to $\Ecal_{\rho}$.

By \cite[Lemme 20]{Seshadri}, for $\rho\in U$, the vector bundles $\Ecal_{\rho}\otimes \omega_{X_{0}}^{2}$ and $\Ecal_{\rho}^{c}\otimes \omega_{\ov{X}_{0}}^{2}$ are globally generated, with vanishing $H^{1}$. Possibly shrinking $U$, we can suppose that there exists a global section $s=s(\rho)$ of $\Ecal_{\rho}$, depending holomorphically on $\rho$, such that $\Ecal/s\Ocal_{X_{0}\times U}$ is a line bundle. We can as well suppose that there exists a section $t$ of $\Ecal_{\rho}^{c}$ with the corresponding property. We may arrange so that $t(\rho_{0})=\ov{s(\rho_{0})}$ are conjugate. In particular, we are in the setting of \textsection \ref{subsec:IC2-metric-flat-unitary}. We form the induced section $\T{s}\otimes\T{t}$ of $IC_{2}(\Ecal_{\rho}\otimes \omega_{X_{0}}^{2})\otimes IC_{2}(\Ecal_{\rho}^{c}\otimes \omega_{X_{0}}^{2})$. At unitary $\rho\in U$, Proposition \ref{prop:explicit-metric} describes the logarithm of the norm of $\T{s}\otimes\T{t}$. Further restricting $U$ if necessary, we shall extend this to a holomorphic function on $U$.

Pulling back to $\HBbb$, we can interpret $s$ as a holomorphic quadratic differential with multiplier $\rho$. This is given as a vector $(f_{1},f_{2})^{t}\otimes d\tau^{2}$, holomorphic in $\tau\in\HBbb$ and $\rho\in U$, with 
\begin{displaymath}
    f(\gamma \tau,\rho)\gamma^{\prime}(\tau)^{2}=\rho(\gamma)f(\tau,\rho),\quad f=\left(\begin{array}{c} f_{1}\\ f_{2}\end{array} \right). 
\end{displaymath}
Similarly, we can interpret $t$ as a vector $(g_{1},g_{2})^{t}\otimes d\ov{\tau}^{2}$, where the $g_{i}$ are holomorphic in $\ov{\tau}\in\ov{\HBbb}$ and $\rho\in U$.

We introduce 
\begin{displaymath}
    \langle\langle s,t\rangle\rangle=\left(f_{1}(\tau,\rho)g_{1}(\ov{\tau},\rho)+f_{2}(\tau,\rho)g_{2}(\ov{\tau},\rho)\right)y^{4}.
\end{displaymath}
It is readily checked that this expression is $\Gamma$-invariant. It is manifestly $\Ccal^{\infty}$ on $X_{0}\times U$, and holomorphic in $\rho$. At $\rho=\rho_{0}$, it coincides with the pointwise squared norm of $s(\rho_{0})$ on $X_{0}$, which is strictly positive. Then, possibly shrinking $U$ around $\rho_{0}$, we can suppose that $\Real\langle\langle s,t\rangle\rangle>0$ on $X_{0}\times\ U$. For this, it is enough to observe that, since $X_{0}$ is compact, the function $\rho\mapsto \inf_{q\in X_{0}}\Real\langle\langle s,t\rangle\rangle(q,\rho)$ is continuous. Therefore, considering the principal branch of the logarithm, the expression $\log\langle\langle s,t\rangle\rangle$ is defined on $X_{0}\times U$, and is holomorphic in $\rho$. Since we seek an extension of equation \eqref{eq:metric-on-Ts} in Proposition \ref{prop:explicit-metric}, we define
\begin{equation}\label{eq:explicit-complex-metric}
    \mathrm{LOG}(\T{s}\otimes\T{t})=2-2g+\int_{X_{0}}\log \langle\langle s,t\rangle\rangle\frac{dx\wedge dy}{2\pi y^{2}}
    -\frac{1}{\pi}\int_{X_{0}}\frac{\partial}{\partial\tau}\log \langle\langle s,t\rangle\rangle \frac{\partial}{\partial\ov{\tau}}\log \langle\langle s,t\rangle\rangle dx\wedge dy.
\end{equation}
This is a holomorphic function on $U$, whose exponential describes the complex metric of $IC_{2}(\Ecal_{\rho}\otimes\omega_{X_{0}}^{2})\otimes IC_{2}(\Ecal_{\rho}^{c}\otimes\omega_{\ov{X}_{0}}^{2})$ on $U$. 

\begin{remark}
Variants of the construction above were already envisioned by Fay. See for instance \cite[Theorem 5.7 \& Remark 1, p.105]{Fay}. This method can only produce valid expressions in a neighborhood of unitary irreducible representations. The non-trivial content of Theorem \ref{theorem:complex-metric} is the existence of an analytical continuation to the whole of $\Rbold(\Gamma,r)$. The key ingredient is the flat holomorphic connection provided by our complex Chern--Simons theory.
\end{remark}


\subsection{Complex Chern--Simons for $\PSL_{r}$}\label{subsection:complements-CS-PSL}
On the relative moduli space $\Mbold_{\Bet}^{\ir}(X/S,\PSL_{r})$ discussed in \textsection \ref{subsubsec:Betti-spaces}--\textsection \ref{subsubsec:uni-end}, and with the notation therein, there is also a counterpart of the complex Chern--Simons line bundle, obtained by the same method as for $\Lcal_{\CS}(X/S)$, but departing from the universal endomorphism bundle with connection $(\Ucal,\nabla^{\un})$. In applications, we will focus on the locus of liftable representations $\Mbold_{\Bet}^{\ir}(X/S,\PSL_{r})_{\ell}$. In this case, the complex Chern--Simons line bundle can be obtained by descent from $\Lcal_{\CS}(X/S)^{2r}$ on $\Mbold_{\Bet}^{\ir}(X/S,\SL_{r})$. We proceed to prove this fact.

\begin{proposition}\label{prop:descend-IC2-PSLr}
The universal intersection bundle with connection $(IC_{2}(\Ecal_{\Bet}^{\un}),\nabla^{\ICmini})^{\otimes (2r)}$ on $\Mbold_{\Bet}^{\ir}(X/S,\SL_{r})$ descends to $\Mbold_{\Bet}^{\ir}(X/S,\PSL_{r})$, and the descended object is canonically isomorphic to $(IC_{2}(\Ucal),\nabla^{\ICmini})$.
\end{proposition}
\begin{proof}
We first address the descent of the line bundle $IC_{2}(\Ecal^{\un}_{\Bet})^{2r}$, and we will incorporate the connection datum afterwards.

The statement is local with respect to $S$. Possibly localizing, we may suppose that $S$ is contractible and Stein. After base changing to a suitable open subset  $S^{\prime}\subset \Mbold_{\Bet}^{\ir}(X/S,\PSL_{r})_{\ell}$, we reduce to the discussion below, where we put $f^{\prime}\colon X^{\prime}=X\times_{S} S^{\prime}\to S^{\prime}$.

Let $E$ be a rank $r$ vector bundle on $X^{\prime}$, with trivial determinant on fibers. 
By Corollary \ref{cor:IC2-End}, there is a canonical, functorial isomorphism
\begin{displaymath}
    IC_{2}(\End E)=IC_{2}(E\otimes E^{\vee})\simeq IC_{2}(E)^{2r}\otimes\langle\det E,\det E\rangle^{1-r}.
\end{displaymath}
Since $\det E$ is trivial on fibers, there is a canonical isomorphism $\det E\simeq f^{\ast}f_{\ast}\det E$, and then by $(4')$ (c) of Proposition \ref{Prop:generalpropertiesDeligneproduct}, the Deligne pairing $\langle\det E,\det E\rangle$ is canonically trivial. 

Now introduce $L$ a line bundle on $X^{\prime}$, fiberwise of degree $0$. We form the diagram
\begin{equation}\label{eq:diagramtrivialKCSbis}
    \xymatrix{IC_2(E \otimes E^\vee)                \ar[r]^{\alpha} \ar[d] &   IC_2(E)^{2r} \ar[d] \\ 
        IC_2((E \otimes L) \otimes (E\otimes L)^{\vee}) \ar[r]& IC_2(E\otimes L)^{2r} \otimes \langle \det(E \otimes L), \det(E \otimes L)  \rangle^{1-r} 
    }
\end{equation}
where the horizontal arrows are defined as above by applying Proposition \ref{prop:whitneyproductwithbundle} and Proposition \ref{prop:dual}, and the rightmost arrow is obtained by developing the products according to Proposition \ref{prop:tensorlinebundleiso} and comparing the two expressions. We claim the diagram commutes. To prove this, denote the automorphism obtained by going around \eqref{eq:diagramtrivialKCSbis} by $\alpha_L$. By functoriality, $\alpha_L \in \Ocal_{S'}^{\times}$  only depends on the isomorphism class of $L$. Moreover, if $M$ is a line bundle on $S'$, then $\alpha_{L \otimes {f'}^{\ast} M} = \alpha_L$. This is a local statement, and we can assume that $M$ is trivial, in which case it is obvious. This argument can be done for any base change of $X' \to S'$ along any morphism $S'' \to S'$, and implies that $L \mapsto \alpha_L$ is an invertible holomorphic function on $\Jac(X^{\prime}/S^{\prime})$. Since $\Jac(X^{\prime}/S^{\prime}) \to S'$  is proper and smooth the invertible functions are constant along fibers. To find $\alpha_L$ we evaluate at $L$ being the trivial bundle, and find that $\alpha_L=1$, \emph{i.e.} the diagram commutes.

We specialize to the case of $L$ being $r$-torsion. Consider the natural diagram of isomorphisms
\begin{displaymath}
\xymatrix{ 
    \langle \det(E \otimes L), \det(E \otimes L)  \rangle \ar[r] \ar[d] & \langle (\det E) \otimes L^{r}, (\det E) \otimes L^{r}  \rangle \ar[d] \\ 
     \Ocal_{S'} & \ar[l] \langle \det E, \det E\rangle \otimes \langle \det E, L^{r} \rangle^{2} \otimes \langle L^{r}, L^{r} \rangle }
\end{displaymath}
where the left vertical arrow is obtained by developing $\det (E \otimes L)\simeq (\det E) \otimes L^{\otimes r}$ and then using the trivializations related to that of the first step and  $L^{r} \simeq \Ocal_{S'}$. We recognize that the top left term is the additional term in \eqref{eq:diagramtrivialKCSbis} in the comparison between between $IC_2(E)^{2r}$ and $IC_2(E\otimes L)^{2r}$. This diagram commutes, by an argument which amounts to an application of $(4')$ (c) of Proposition \ref{Prop:generalpropertiesDeligneproduct}.

Now that we know that $IC_{2}(\Ecal^{\un}_{\Bet})^{2r}$ descends, the corresponding statement for connections is a consequence of Proposition \ref{prop:properties-nabla-IC2} and Lemma \ref{lemma:canonical-connection-IC2}. 
\end{proof}

After the proposition, the following notation is justified.
\begin{notation}\label{notation:LCS-2r}
The complex Chern--Simons line bundle on $\Mbold_{\Bet}^{\ir}(X/S,\PSL_{r})_{\ell}$ is denoted by $\Lcal_{\CS}(X/S)^{2r}$.
\end{notation}

\begin{proposition}\label{prop:complements-CS-PSLr}
\begin{enumerate}
    \item The functoriality with respect to base change and isomorphisms of relative curves, the curvature and the crystalline properties of $\Lcal_{\CS}(X/S)$ on $\Mbold_{\Bet}^{\ir}(X/S,\SL_{r})$ descend to corresponding features for $\Lcal_{\CS}(X/S)^{2r}$ on $\Mbold_{\Bet}^{\ir}(X/S,\PSL_{r})_{\ell}$.
    \item In the case of a single Riemann surface $X_{0}$, we have 
        \begin{equation}\label{eq:curvature-LCS-PSLr}
             c_{1}(\Lcal_{\CS}(X_{0})^{2r})=\frac{r}{2\pi^{2}}\omega_{\scriptscriptstyle{\PSL_{r}}}.
        \end{equation}
    \item Except possibly for $g=r=2$, the complex metric on $\Lcal_{\CS}(X_{0})^{2r}$ descends to  $\Mbold_{\Bet}^{\ir}(X_{0},\PSL_{r})_{\ell}$.
\end{enumerate}
\end{proposition}
\begin{proof}
The first item is formal to check if we interpret $\Lcal_{\CS}(X/S)^{2r}$ in terms of the universal endomorphism bundle. The second property follows from the fact that $\omega_{\scriptscriptstyle{\PSL_{r}}}$ is obtained by descent of $\omega_{\scriptscriptstyle{\SL_{r}}}$ (see \textsection \ref{subsubsec:Betti-spaces}) and by Theorem \ref{theorem:curvature-CS-single-RS}. The third property is addressed in a similar vein as the second part of the proof of Theorem \ref{theorem:complex-metric}, when reducing from $\GL_{r}$ to $\SL_{r}$ monodromies.
\end{proof}

\section{Applications to projective structures}\label{section:CS-applications}
We apply the general theory of Section \ref{section:canonical-extensions} and Section \ref{section:CS-theory} to spaces of projective structures on families of Riemann surfaces. We build complex Chern--Simons line bundles on those, and show they are naturally isomorphic to Deligne pairings of canonical sheaves. This leads to the definition of the Chern--Simons transform of a family of projective structures, as an induced connection on the Deligne pairing. We then study its properties in detail. Over the Teichm\"uller space, the Chern--Simons transform establishes an equivalence between relative projective structures and compatible connections on the Deligne pairing. A similar result has been independently obtained by Biswas--Favale--Pirola--Torelli \cite{BFPT}, with a different, somewhat \emph{ad hoc} method. We go on to consider the transform of classical families of projective structures. For Fuchsian uniformizations, we recover Wolpert's result on the relationship between the first tautological class and the Weil--Petersson form on Teichm\"uller space. The proof is an example of the application of Higgs bundles to the practical computation of intersection connections. For Schottky and quasi-Fuchsian groups, we propose a conceptual and simple construction of potentials of the Weil--Petersson forms, first studied by Takhtajan--Zograf \cite{Zograf-Takhtajan} and Takhtajan--Teo \cite{Takhtajan-Teo}. Finally, we compare our results to the work of Guillarmou--Moroianu \cite{Guillarmou-Moroianu} on complex Chern--Simons line bundles on Teichm\"uller spaces of cocompact convex hyperbolic 3-manifolds. In this section we entirely work in the analytic category.

\subsection{Background on projective structures} We review the theory of families of projective structures. We divide the exposition in five parts. For an easier reading, we provide references at the beginning of each part.

\subsubsection{Projective structures and their holonomies}\label{subsubsec:Back-Proj-1} Cf. \cite[Section 2]{Dumas}, \cite[Section 1.3]{GKM}, \cite[Section 1]{Hubbard}.

Let $X_{0}$ be a compact Riemann surface of genus $g\geq 2$. A \emph{projective structure} on $X_{0}$, subordinate to the complex structure of $X_{0}$, is a maximal atlas of holomorphic charts on $X_{0}$, with values in $\PBbb^{1}$ and transition maps induced by complex M\"obius transformations. These are called projective charts. The set of projective structures on $X_{0}$ is denoted by  $\Pcal(X_{0})$. 

Let $\pi\colon\widetilde{X}_{0}\to X_{0}$ be a universal cover of $X_{0}$, with fundamental group $\Gamma$. Given a projective structure $\alpha\in\Pcal(X_{0})$, there exists a developing map for $\alpha$. This is a local biholomorphism $h\colon\widetilde{X}_{0}\to\PBbb^{1}$, such that for every contractible open subset $U$ of $X_{0}$, the morphism $(h\circ\pi^{-1})_{\mid U}$ is a projective chart of $\alpha$. Any other developing map $g$ relates to $h$ by $g=A\circ h$, for a unique M\"obius transformation $A$. 

Let $\alpha\in\Pcal(X_{0})$ and $h\colon\widetilde{X}_{0}\to \PBbb^{1}$ be a developing map for $\alpha$ as above. For every $\gamma\in\Gamma$, acting on $\widetilde{X}_{0}$ by deck transformations, the composition $h\circ\gamma$ is another developing map for $\alpha$, hence of the form $\rho(\gamma)\circ h$ for a unique $\rho(\gamma)\in\PSL_{2}(\CBbb)$. This defines a representation $\rho\colon\Gamma\to\PSL_{2}(\CBbb)$, called the holonomy, or monodromy representation of $h$. If $A$ is a M\"obius transformation, the holonomy representation of $A\circ h$ is $A\rho A^{-1}$. The $\PSL_{2}(\CBbb)$-orbit of $\rho$ only depends on $\alpha$, and may be loosely called the holonomy representation of $\alpha$. Holonomy representations of projective structures are known to be irreducible. Thus, there is a well-defined \emph{holonomy map}
\begin{equation}\label{eq:hol-proj-absolute}
    \hol\colon\Pcal(X_{0})\to\Mbold_{\Bet}^{\ir}(X_{0},\PSL_{2}). 
\end{equation}

\subsubsection{Schwarzian equations: local theory}\label{subsubsec:Back-Proj-2} Cf. \cite[Chapter I, \textsection 4--\textsection 5]{Deligne:connections}, \cite[Section 3]{Dumas}, \cite[Section 2]{Hawley-Schiffer}, \cite[Section 4]{Hejhal:monodromy}.

Projective structures on $X_{0}$ can be reframed in terms of Schwarzian differential equations. We first review the local aspects.

Let $\Omega$ be a domain in $\CBbb$ with holomorphic coordinate $z$. If $h\colon\Omega\to\CBbb$ is a local biholomorphism, then its Schwarzian holomorphic quadratic differential is defined as $\Scal(h,z)=\lbrace h,z\rbrace dz^{2}$, where
\begin{displaymath}
    \lbrace h,z\rbrace=6\frac{\partial^{2}}{\partial z\partial w}\log\left(\frac{h(z)-h(w)}{z-w}\right)\Bigg|_{z=w}
\end{displaymath}
is the Schwarzian derivative. The Schwarzian differential fulfills two fundamental properties:
\begin{enumerate}
    \item if $z=g(t)$ is a change of variable, then there is a cocycle relation
    \begin{equation}\label{eq:cocycle-Schwarz}
        \Scal(h\circ g,t)=g^{\ast}\Scal(h,z)+\lbrace z,t\rbrace dt^{2}.
    \end{equation}
    The reason of writing $\lbrace z,t\rbrace dt^{2}$ instead of $\Scal(g,t)$ will be apparent in the global theory \textsection\ref{subsubsec:Back-Proj-3} below, when we define the notion of projective connection.  
    \item if $g\colon\Omega\to\CBbb$ is another local biholomorphism, then $\Scal(g)=\Scal(h)$ if, and only if $g=A\circ h$ for some $A\in\PSL_{2}(\CBbb)$.
\end{enumerate}

Suppose now that $\Omega$ is simply connected. Let $q(z)dz^{2}$ be a holomorphic quadratic differential on $\Omega$. The problem of solving the Schwarzian equation $\Scal(h)=q(z)dz^{2}$ is equivalent to a second order linear differential equation
\begin{equation}\label{eq:Schwarz-1}
    u^{\prime\prime}(z)+\frac{1}{2}q(z)u(z)=0.
\end{equation}
The space of solutions is a 2-dimensional $\CBbb$-vector space. If $u_{1}$ and $u_{2}$ are two independent solutions, one derives from \eqref{eq:Schwarz-1} that their Wronskian is a non-zero constant, and it can be normalized to be 1. The meromorphic function $h=u_{1}/u_{2}$ solves the Schwarzian equation. If $z=g(t)$ is a change of coordinates on $\Omega$ and $g^{\prime}(t)^{1/2}$ is a fixed square root of $g^{\prime}(t)$, setting $v_{j}(t)=u_{j}(g(t))g^{\prime}(t)^{-1/2}$ produces a basis of solutions of the Schwarzian equation in the variable $t$. 

In order to globalize the previous discussion to Riemann surfaces, we are required  to express equation \eqref{eq:Schwarz-1} intrinsically as a second order differential operator $D\colon\kappa^{\vee}\to\kappa^{\vee}\otimes\omega^{2}$, where $\omega$ is the canonical sheaf on $\Omega$ and $\kappa$ is a choice of theta characteristic; namely, a holomorphic line bundle with an isomorphism $\kappa^{2}\simeq\omega$. The space of solutions of \eqref{eq:Schwarz-1} is $V=\ker D$. The usual trick transforming \eqref{eq:Schwarz-1} into a system of linear differential equations of order 1 is recast as a holomorphic connection, in particular flat,  $\nabla$ on $\Jcal^{1}(\kappa^{\vee})$, the first jet bundle of $\kappa^{\vee}$.  Then $V=\ker\nabla$ can be seen as a trivial local system on $\Omega$, and the inclusion of sheaves $V\hookrightarrow\Jcal^{1}(\kappa^{\vee})$ induces an isomorphism $V\otimes_{\CBbb}\Ocal_{\Omega}\overset{\sim}{\to}\Jcal^{1}(\kappa^{\vee})$. The filtration of $\Jcal^{1}(\kappa^{\vee})$ provides an exact sequence
\begin{equation}\label{eq:fil-Jet}
    0\to \kappa^{\vee}\otimes\omega\overset{i}{\to}\Jcal^{1}(\kappa^{\vee})\overset{p}{\to}\kappa^{\vee}\to 0,
\end{equation}
so that $\det\Jcal^{1}(\kappa^{\vee})\simeq\kappa^{-2}\otimes\omega\simeq\Ocal_{\Omega}$. The structure of the Schwarzian equation has the following implications for $(\Jcal^{1}(\kappa^{\vee}),\nabla)$:
\begin{enumerate}
    \item[(S1)] $\nabla$ induces the trivial connection on $\det\Jcal^{1}(\kappa)\simeq\Ocal_{\Omega}$. This amounts to the fact that the Wronskian of a basis of solutions can be normalized to being one;
    \item[(S2)]the $\Ocal_{\Omega}$-linear morphism
\begin{displaymath}
    \kappa^{\vee}\otimes\omega\overset{i}\hookrightarrow\Jcal^{1}(\kappa^{\vee})\overset{\nabla}{\to}\Jcal^{1}(\kappa^{\vee})\otimes\omega\overset{p}{\twoheadrightarrow}\kappa^{\vee}\otimes\omega
\end{displaymath}
is the identity. 
\end{enumerate}

We have thus naturally transformed the Schwarzian equation into a flat vector bundle satisfying properties (S1)--(S2).

\subsubsection{Schwarzian equations: global theory}\label{subsubsec:Back-Proj-3}
The references given in \textsection \ref{subsubsec:Back-Proj-1}--\ref{subsubsec:Back-Proj-2} also cover the following discussion. We add \cite[Section 5]{Faltings:projective}, especially for the penultimate paragraph of the ongoing part.

We go back to our compact Riemann surface $X_{0}$ of genus $g\geq 2$. A \emph{projective connection} on $X_{0}$ consists in giving, locally on coordinate open subsets of $X_{0}$, holomorphic quadratic differentials $\Scal(z)$ transforming like \eqref{eq:cocycle-Schwarz} under changes of coordinates $z=g(t)$. The difference between two projective connections is a well-defined global holomorphic quadratic differential. The obstruction to the existence of projective connections lives in $H^{1}(X_{0},\omega^{2}_{X_{0}})$, which vanishes for $g\geq 2$.

Given a projective connection on $X_{0}$, local independent solutions of the Schwarzian equation define a projective atlas on $X_{0}$. Conversely, the Schwarzian differentials of a projective atlas constitute a projective connection on $X_{0}$. After \textsection\ref{subsubsec:Back-Proj-1}--\textsection\ref{subsubsec:Back-Proj-2}, both notions are equivalent. It follows that $\Pcal(X_{0})$ has a natural structure of an affine space under $H^{0}(X_{0},\omega^{2}_{X_{0}})$, and is thus a complex manifold. With respect to this structure, the holonomy map \eqref{eq:hol-proj-absolute} is a morphism of complex manifolds. 

Let $(\kappa,\iota\colon\kappa^{2}\simeq\omega_{X_{0}})$ be a theta characteristic on $X_{0}$. As in the local theory, a projective connection on $X_{0}$ is equivalent to a differential operator $D\colon\kappa^{\vee}\to\kappa^{\vee}\otimes\omega_{X_{0}}^{2}$, locally of the form \eqref{eq:Schwarz-1}, as well as to a holomorphic connection $\nabla$ on $\Jcal^{1}(\kappa^{\vee})$ satisfying the global counterparts of (S1)--(S2) above. Moreover, if $(\Ecal,\nabla)$ is a flat vector bundle on $X_{0}$, fitting into an exact sequence of the form \eqref{eq:fil-Jet} and satisfying the analogues of (S1)--(S2), then $(\Ecal,\nabla)$ is isomorphic to $(\Jcal^{1}(\kappa^{\vee}),\nabla)$, compatibly with the filtrations. We will say they are \emph{gauge equivalent}.

Looking at the holonomy of the connections associated to projective structures, we get a lift of the holonomy map \eqref{eq:hol-proj-absolute} to
\begin{displaymath}
    \widetilde{\hol}\colon \Pcal(X_{0})\to\Mbold_{\Bet}^{\ir}(X_{0},\SL_{2}).
\end{displaymath}
Two lifts associated to two different theta characteristics differ by multiplication by some $\chi\colon\pi_{1}(X_{0})\to\lbrace\pm 1\rbrace$, corresponding to the 2-torsion line bundle comparing the theta characteristics.

\subsubsection{Relative projective structures}\label{subsubsec:Back-Proj-4} Cf. \cite[pp. 686--687]{GKM}, \cite[Sections 6--8]{Hejhal:monodromy}, \cite[Sections 3--6]{Hubbard}, \cite[Chapter 3]{Tjurin}. 

Now let $f\colon X\to S$ be a submersion of complex manifolds whose fibers are compact Riemann surfaces of genus $g\geq 2$. Similar to the absolute setting of \textsection\ref{subsubsec:Back-Proj-1}, there is a notion of holomorphic relative projective structure for $f$, also called holomorphic family of projective structures. It consists in giving a maximal atlas of relative holomorphic charts $(U_{\alpha},z_{\alpha})$ on $X$, whose changes of coordinates are of the form $z_{\alpha}=A_{\alpha\beta}\circ z_{\beta}$ for holomorphic families M\"obius transformations $A_{\alpha\beta}\colon f(U_{\alpha}\cap U_{\beta})\to\PSL_{2}(\CBbb)$. Over a Stein base, holomorphic families of projective structures always exist.

There exists a universal space for holomorphic relative projective structures, with the following properties. It is a complex manifold $\Pcal(X/S)$, equipped with a holomorphic submersion $\pi\colon\Pcal(X/S)\to S$ which makes it an $S$-torsor under $\VBbb(f_{\ast}\omega^{2}_{X/S})$, the affine bundle of relative quadratic differentials. On the space $X\times_{S}\Pcal(X/S)$ there is a holomorphic relative projective structure over $\Pcal(X/S)$, which is a universal object for holomorphic relative projective structures on base changes $X_{T}\to T$ of $f$. Equivalently, $\Pcal(X/S)$ represents a functor of holomorphic relative projective structures, and in particular holomorphic relative projective structures over $S$ correspond to holomorphic sections of $\Pcal(X/S)\to S$. Also, the formation of $\Pcal(X/S)$ is compatible with base change. Notice that over a Stein base, the existence of a holomorphic family of projective structures is equivalent to the fact that the torsor structure of $\Pcal(X/S)$ can be trivialized. We also observe that for an isomorphism $X^{\prime}\to X$ of families of curves over $S$, there is an induced obvious isomorphism $\Pcal(X^{\prime}/S)\simeq\Pcal(X/S)$.

The notion of a smooth family of projective structures can also be developed. For the purposes of this article, we simply define this as a $\Ccal^{\infty}$ section of $\Pcal(X/S)\to S$. Two smooth families of projective structures differ by a $\Ccal^{\infty}$-section of $f_{\ast}\omega_{X/S}^{2}$.

There is a relative counterpart of the holonomy map \eqref{eq:hol-proj-absolute}. For this, assume for a moment that $S$ is contractible and Stein, and $f$ has a section $\sigma$. Suppose $\alpha\colon S\to\Pcal(X/S)$ is a holomorphic relative projective structure. Fix base points $0\in S$ and $p=\sigma(0)$, and set $\Gamma=\pi_{1}(X,p)\simeq\pi_{1}(X_{0},p)$. Let $\widetilde{X}$ be the universal cover of $X$ based at $p$. Then there exists a holomorphic relative developing map $h\colon\widetilde{X}\to\PBbb^{1}_{S}$ for $\alpha$. Similar to the absolute case, associated to $h$ there is a holomorphic family of representations $\rho\colon\Gamma\to\PSL_{2}(\Gamma(S,\Ocal_{S}))$. Changing the developing map results in conjugating $\rho$ by some $A\in \PSL_{2}(\Gamma(S,\Ocal_{S}))$. Over a general base, this discussion can be carried out locally over $S$, and results in the relative holonomy map
\begin{equation}\label{eq:rel-hol-map-PSL}
    \hol\colon \Pcal(X/S)\to\Mbold_{\Bet}^{\ir}(X/S,\PSL_{2}).
\end{equation}
This is a morphism of complex manifolds. It has natural functoriality properties with respect to base change and isomorphisms of relative curves $X^{\prime}\to X$.

There is also a relative version of the Schwarzian interpretation of projective structures of \textsection\ref{subsubsec:Back-Proj-2}--\textsection\ref{subsubsec:Back-Proj-3}. We will need the formulation in terms of jet bundles. We could not find an appropriate reference, but it is a formal adaptation of the absolute case reviewed in \textsection \ref{subsubsec:Back-Proj-3}. Suppose that $S$ is contractible and Stein, and $f\colon X\to S$ admits a section $\sigma$. We may then introduce a relative theta characteristic $(\kappa,\iota\colon\kappa^{2}\simeq\omega_{X/S})$. Given a holomorphic family of projective structures, we build a collection of local relative Schwarzian equations, with holomorphic dependence on $s\in S$. These organize into a relative differential operator $D\colon\kappa^{\vee}\to\kappa^{\vee}\otimes\omega_{X/S}^{2}$, equivalent to a relative holomorphic connection $\nabla\colon\Jcal^{1}(\kappa^{\vee})\to\Jcal^{1}(\kappa^{\vee})\otimes\Omega^{1}_{X/S}$, satisfying the relative versions of (S1)--(S2). As in the absolute case, $(\Jcal^{1}(\kappa^{\vee}),\nabla)$ is unique with these properties, up to gauge equivalence. Looking at the holonomy of such connections, we obtain a lift of the relative holonomy map
\begin{equation}\label{eq:rel-hol-map}
    \widetilde{\hol}\colon\Pcal(X/S)\to\Mbold_{\Bet}^{\ir}(X/S,\SL_{2}).
\end{equation}
Two lifts differ by a character $\chi\colon\pi_{1}(X)\to\lbrace\pm 1\rbrace$, corresponding to the 2-torsion line bundle comparing to relative theta characteristics.

Similarly, if $S$ is contractible, Stein and $f$ admits a section, smooth families of projective structures can be reinterpreted as flat relative connections $\nabla\colon\Jcal^{1}(\kappa^{\vee})\to\Jcal^{1}(\kappa^{\vee})\otimes\Acal^{1,0}_{X/S}$. This can be seen by comparing to a holomorphic relative projective structure via a $\Ccal^{\infty}$ section of $f_{\ast}\omega_{X/S}^{2}$. The latter provides a section of $\End(\Jcal^{1}(\kappa^{\vee}))\otimes\Acal^{1,0}_{X/S}$, which accounts for the comparison of the associated flat relative connections.

Over a general base $S$, the discussion surrounding \eqref{eq:rel-hol-map} shows that $\hol$ actually lands in the space of (locally) liftable representations $\Mbold_{\Bet}^{\ir}(X/S,\PSL_{2})_{\ell}$ (see \textsection \ref{subsubsec:Betti-spaces}). 

\begin{remark}
If $S$ is contractible and Stein, and $f$ admits a section $\sigma$, $\hol$ actually lifts to the representation space $\Rbold_{\Bet}^{\ir}(X/S,\sigma, \SL_{2})$. This is achieved by introducing appropriate initial conditions for the solutions of the relative Schwarzian equations, which induces a rigidification of $(\Jcal^{1}(\kappa^{\vee}),\nabla)$ along $\sigma$.
\end{remark}

\subsection{Complex Chern--Simons and projective structures}
Until the end of this section on projective structures, $f\colon X\to S$ denotes a submersion between connected complex manifolds, whose fibers are compact Riemann surfaces of genus $g\geq 2$.

\subsubsection{Complex Chern--Simons line bundle}\label{subsubsec:projective-construction}
 We consider the space of holomorphic relative projective structures $\pi\colon\Pcal(X/S)\to S$ and the relative holonomy map $\hol\colon\Pcal(X/S)\to\Mbold_{\Bet}^{\ir}(X/S,\PSL_{2})_{\ell}$. We saw in Proposition \ref{prop:descend-IC2-PSLr} that $\Lcal_{\CS}(X/S)^{4}$ on $\Mbold_{\Bet}(X/S,\SL_{2})$ descends to $\Mbold_{\Bet}^{\ir}(X/S,\PSL_{2})_{\ell}$, and following Notation \ref{notation:LCS-2r} we simply write $\Lcal_{\CS}(X/S)^{4}$ for the descended object. 

\begin{definition}
The complex Chern--Simons line bundle on $\Pcal(X/S)$ is defined as the holomorphic line bundle with holomorphic connection $\hol^{\ast}\Lcal_{\CS}(X/S)^{4}$. We denote it $(\Kcal_{\CS}(X/S),\nabla^{\CS})$, or simply $\Kcal_{\CS}(X/S)$.
\end{definition}

\begin{proposition}\label{prop:stupid-proposition}
The formation of $\Kcal_{\CS}(X/S)$ is functorial with respect to base change and isomorphisms of relative curves $X^{\prime}\to X$ over $S$.
\end{proposition}
\begin{proof}
The functoriality of $\Lcal_{\CS}(X/S)$ and $\Pcal(X/S)$ easily entail that the formation of $\Kcal_{\CS}(X/S)$ is compatible with base change. See Proposition \ref{prop:base-chanhge-LCS} and Proposition \ref{prop:complements-CS-PSLr}. The functoriality with respect to isomorphisms $X^{\prime}\to X$ is similarly addressed with the help of Lemma \ref{lemma:stupid-lemma-3} and Proposition \ref{prop:complements-CS-PSLr}. 
\end{proof}

In practice, we will have use for an explicit construction of $\Kcal_{\CS}(X/S)$ in terms of relative theta characteristics. Suppose momentarily that $S$ is contractible and Stein, and $f$ admits a section. We introduce a relative theta characteristic $(\kappa,\iota\colon\kappa^{2}\simeq\omega_{X/S})$. Associated to this choice, as in \eqref{eq:rel-hol-map} there is a lift $\widetilde{\hol}\colon\Pcal(X/S)\to\Mbold_{\Bet}^{\ir}(X/S,\SL_{2})$ of $\hol$. By Proposition \ref{prop:descend-IC2-PSLr}, there is a canonical isomorphism of line bundles with connections
\begin{equation}\label{eq:iso-Kcal-Lcal}
    \varphi\colon\Kcal_{\CS}(X/S)\overset{\sim}{\longrightarrow}\widetilde{\hol}^{\ast}\Lcal_{\CS}(X/S)^{4}.
\end{equation}
In general, relative theta characteristics as above exist locally with respect to $S$, and the locally defined \eqref{eq:iso-Kcal-Lcal} patch together. 

\subsubsection{Relationship to Deligne pairings}
\begin{theorem}\label{prop:iso-Kcal-Deligne}
On $\Pcal(X/S)$ there is a natural isomorphism of holomorphic line bundles
\begin{equation}\label{eq:iso-KCS-Deligne}
    \Kcal_{\CS}(X/S)\overset{\sim}{\longrightarrow}\pi^{\ast}\langle\omega_{X/S},\omega_{X/S}\rangle,
\end{equation}
where $\pi\colon\Pcal(X/S)\to S$ is the structure map. It is functorial in the sense of Proposition \ref{prop:stupid-proposition}.
\end{theorem}
\begin{proof}
The proof proceeds in five steps. The background expounded in \textsection\ref{subsubsec:Back-Proj-3}--\textsection\ref{subsubsec:Back-Proj-4} is relevant here. \bigskip

\noindent\emph{Step 1: local construction.} We begin by localizing on $S$, and we suppose first that it is contractible, Stein, and that $f\colon X\to S$ admits a section. We can then introduce a relative theta characteristic $(\kappa,\iota\colon\kappa^{2}\simeq\omega_{X/S})$. We set $S^{\prime}=\Pcal(X/S)\simeq\VBbb(f_{\ast}\omega^{2}_{X/S})$, which is contractible and Stein too. We write $f^{\prime}\colon X^{\prime}\to S^{\prime}$ for the base change of $f$, so that
\begin{displaymath}
    \langle\omega_{X^{\prime}/S^{\prime}},\omega_{X^{\prime}/S^{\prime}}\rangle\simeq \pi^{\ast}\langle\omega_{X/S},\omega_{X/S}\rangle.
\end{displaymath}
The base change of the theta characteristic is still denoted by $(\kappa,\iota)$.

On $X^{\prime}$ there is a universal relative projective structure, which gives rise to a flat relative connection $\nabla\colon\Jcal^{1}(\kappa^{\vee})\to\Jcal^{1}(\kappa^{\vee})\otimes\Omega^{1}_{X^{\prime}/S^{\prime}}$, with trivial determinant. The natural exact sequence \eqref{eq:fil-Jet} on $\Jcal^{1}(\kappa^{\vee})$ induces a Whitney isomorphism
\begin{displaymath}
    \alpha\colon IC_{2}\left(\Jcal^{1}(\kappa^{\vee})\right)\overset{\sim}{\longrightarrow}\langle\kappa^{\vee}\otimes\omega_{X^{\prime}/S^{\prime}},\kappa^{\vee}\rangle.
\end{displaymath}
So in fact we find natural isomorphisms
\begin{equation}\label{eq:iso-IC2-U-Deligne}
         \Psi(\kappa,\iota)\colon\left(\hol^{\ast}\Lcal_{\CS}(X/S)^{4}\right)^{\vee}\simeq IC_{2}\left(\Jcal^{1}(\kappa^{\vee})\right)^{4}\overset{\alpha}{\simeq}
         \langle\kappa^{\vee}\otimes\omega_{X^{\prime}/S^{\prime}},\kappa^{\vee}\rangle^{4}\overset{\beta}{\simeq}  \langle\omega_{X^{\prime}/S^{\prime}},\omega_{X^{\prime}/S^{\prime}}\rangle^{\vee},
\end{equation}
where $\beta$ is induced by $\iota\colon\kappa^{2}\simeq\omega_{X/S}$. \bigskip

In order to globalize this construction, we need to check that $\Psi(\kappa,\iota)$ is gauge independent, as well as independent of $(\kappa,\iota)$.\bigskip 

\noindent\emph{Step 2: gauge independence.} Recall that $(\Jcal^{1}(\kappa^{\vee}),\nabla)$ with the filtration \eqref{eq:fil-Jet} and the relative counterpart of properties (S1)--(S2), is unique up to gauge equivalence. We show that an automorphism of the extension \eqref{eq:fil-Jet} induces the identity on $\langle\omega_{X^{\prime}/S^{\prime}},\omega_{X^{\prime}/S^{\prime}}\rangle$. Such an automorphism $\phi$ fits in a commutative diagram of exact sequences
\begin{displaymath}
\xymatrix{
        0\ar[r]      &\kappa^{\vee}\otimes\omega_{X^{\prime}/S^{\prime}}\ar[r]\ar[d]^{\id}       &\Jcal^{1}(\kappa^{\vee})\ar[r]\ar[d]^{\phi}     &\kappa^{\vee}\ar[r]\ar[d]^{\id}      &0\\
        0\ar[r]      &\kappa^{\vee}\otimes\omega_{X^{\prime}/S^{\prime}}\ar[r]        &\Jcal^{1}(\kappa^{\vee})\ar[r]     &\kappa^{\vee}\ar[r]        &0.
    }
\end{displaymath}
By (a) in \emph{(IC1)} in Theorem \ref{thm:ic2properties}, this induces a commutative diagram
\begin{displaymath}
    \xymatrix{
        IC_{2}\left(\Jcal^{1}(\kappa^{\vee})\right)\ar[r]^-{\alpha}\ar[d]^{IC_{2}(\phi)}                  &\langle\kappa^{\vee}\otimes\omega_{X^{\prime}/S^{\prime}},\kappa^{\vee}\rangle \ar[d]^{\id}          \\
        IC_{2}\left(\Jcal^{1}(\kappa^{\vee})\right) \ar[r]^-{\alpha}          &\langle\kappa^{\vee}\otimes\omega_{X^{\prime}/S^{\prime}},\kappa^{\vee}\rangle,
    }
\end{displaymath}
and the expected result follows by taking the fourth power.\bigskip

\noindent\emph{Step 3: independence of $\iota$.} Let $\lambda$ be an invertible holomorphic function on $S$. We claim that $\Psi(\kappa,\iota)=\Psi(\kappa,\lambda\iota)$. The dependence on $\iota$ is all captured by the isomorphism $\beta$ in \eqref{eq:iso-IC2-U-Deligne}. By Proposition \ref{Prop:generalpropertiesDeligneproduct} $(4)$, scaling $\iota$ to $\lambda\iota$ changes $\beta$ into $\lambda^{\delta}\beta$, with $\delta=\deg\kappa^{\vee}+\deg(\kappa^{\vee}\otimes\omega_{X^{\prime}/S^{\prime}})=0$. This settles the claim, and it is thus legitimate to write $\Psi(\kappa)$ instead of $\Psi(\kappa,\iota)$.\bigskip

\noindent\emph{Step 4: independence of $\kappa$}.  Consider the effect of changing $\kappa$ to $\kappa\otimes L$, where $L$ is a line bundle endowed with $L^{\otimes 2}\simeq\Ocal_{X}$. We set $\kappa^{\prime}=\kappa\otimes L$. Then there is a natural isomorphism $\Jcal^{1}(\kappa^{\prime \vee})\simeq \Jcal^{1}(\kappa^{\vee})\otimes L^{\vee}$, as filtered vector bundles with connections. The descent Proposition \ref{prop:descend-IC2-PSLr} together with Proposition \ref{prop:tensorlinebundleiso} imply $\Psi(\kappa) = \Psi(\kappa')$. \bigskip

\noindent\emph{Step 5: functoriality.} To prove that the previous construction is functorial, we notice that the intermediate isomorphisms between the several intersection bundles involved in the construction are themselves functorial. 
\end{proof}

\begin{corollary}\label{cor:CS-Deligne-pairing}
If $\sigma\colon S\to\Pcal(X/S)$ is a $\Ccal^{\infty}$ (resp. holomorphic) section, then there is a canonical, functorial isomorphism
\begin{displaymath}
    \sigma^{\ast}\Kcal_{\CS}(X/S)\overset{\sim}{\longrightarrow}\langle\omega_{X/S},\omega_{X/S}\rangle.
\end{displaymath}
In particular, $\nabla^{\CS}$ induces a $\Ccal^{\infty}$ (resp. holomorphic) connection on $\langle\omega_{X/S},\omega_{X/S}\rangle$.
\end{corollary}\qed
\bigskip

We introduce the following terminology for the induced connections of the corollary.

\begin{definition}\label{def:CS-transform}
Given a $\Ccal^{\infty}$ section $\sigma\colon S\to\Pcal(X/S)$, the induced connection on $\langle\omega_{X/S},\omega_{X/S}\rangle$ is called de Chern--Simons transform of $\sigma$, and is denoted by $\nabla^{\sigma}$.
\end{definition}

\begin{remark}
\begin{enumerate}
    \item By the functoriality of $\Kcal_{\CS}(X/S)$ and of the isomorphism \eqref{eq:iso-KCS-Deligne}, the Chern--Simons transform is compatible with base change.
    \item At this stage, it is still not clear that Chern--Simons transforms are compatible connections. The proof of this fact will be given in Proposition \ref{prop:int-conn-Deligne-compatible}. Nevertheless, we already know by Corollary \ref{cor:CS-Deligne-pairing} that the Chern--Simons transform preserves holomorphicity.
\end{enumerate}
\end{remark}

\subsection{Properties of Chern--Simons transforms}\label{subsec:CS-transform}
 We study in  detail the structure of the Chern--Simons transforms of Definition \ref{def:CS-transform}. Most notably, we show these are actually compatible connections. Over the Teichm\"uller space, we prove that the Chern--Simons transform provides an equivalence between smooth families of projective structures and compatible connections on Deligne pairings. For the rudiments of Teichm\"uller theory needed in the current and the forthcoming subsections, we refer the reader to Ahlfors--Bers \cite{Ahlfors-Bers}, Ahlfors \cite{Ahlfors} and Bers \cite{Bers, Bers:uniformization}. Also Wolpert's article \cite{Wolpert:Chern} contains a thorough account. For the relationship with projective structures, an appropriate reference for our purposes is Loustau \cite{Loustau}.

\subsubsection{Analytic description of jet bundles}\label{subsubsec:analytic-description-jet}
We first deliver an explicit description of the holomorphic structure of jet bundles of relative theta characteristics, whose relationship to relative projective structures was described in \textsection\ref{subsubsec:Back-Proj-4} and utilized in Theorem \ref{prop:iso-Kcal-Deligne}.

Suppose that $S$ is contractible, Stein and that $f\colon X\to S$ admits a section $S\to X$. We are thus in the setting of the proof of Theorem \ref{prop:iso-Kcal-Deligne}. In particular, we can choose a relative theta characteristic $\kappa$. We set $V:=\Jcal^{1}(\kappa^{\vee})$ and  $E:=\kappa\oplus\kappa^{\vee}$. Since the fibers of $f\colon X\to S$ have genus $g\geq 2$, we can equip $\omega_{X/S}$ with the dual of the hyperbolic metric, and $\kappa$, $\kappa^{\vee}$ with the induced metrics and Chern connections. Also, $E$ is considered with the orthogonal sum metric.

By Atiyah's interpretation of connections in terms of jet bundles, the vertical projection of the Chern connection of $\kappa^{\vee}$ provides a $\Ccal^{\infty}$ splitting of the extension \eqref{eq:fil-Jet}
\begin{equation}\label{eq:splitting}
    V\simeq_{\Ccal^{\infty}}\ \kappa^{\vee}\otimes\omega_{X/S}\ \oplus\ \kappa^{\vee} \simeq E.
\end{equation}
In the representation $E$, the vertical projection of the $\ov{\partial}$ operator of $V$ can be written as
\begin{equation}\label{eq:del-bar-V}
    \left(\begin{array}{cc}
        \ov{\partial}_{\kappa}    &\beta\\
        0          &\ov{\partial}_{\kappa^{\vee}}
    
    \end{array}\right),
\end{equation}
where $\beta$ is a relative $(0,1)$-form with values in $\Hom(\kappa^{\vee},\kappa)\simeq\omega_{X/S}$. The form $\beta$ is $\ov{\partial}$-closed, and in Dolbeault cohomology it represents the relative extension class of \eqref{eq:fil-Jet}. This class defines a section of $R^{1}f_{\ast}\omega_{X/S}\simeq\Ocal_{S}$, and corresponds to the constant function $1-g=\deg\kappa^{\vee}$ (Atiyah class). Actually, Hitchin's discussion in \cite[pp. 122--123]{Hitchin:self-duality} shows that the restriction of $\beta$ to a fiber $X_{s}$ can be taken to be a fixed multiple of the K\"ahler form of the hyperbolic metric on $X_{s}$ (it depends on the normalization of the hyperbolic metric only).

We now observe that $\beta$ extends to a closed $(1,1)$-form on $X$. It is enough to consider an appropriate multiple of the first Chern form of $\omega_{X/S}$ with the dual hyperbolic metric. So, we may suppose that $\beta$ is already a $(1,1)$-form on $X$.

We project $\beta$ to a $\widetilde{\beta}\in A^{0,1}(X,\omega_{X/S})$. We notice that $\widetilde{\beta}$ is still $\ov{\partial}$-closed, since the projection $\Omega^{1}_{X}\to\omega_{X/S}$ is holomorphic. We can then construct a $\ov{\partial}$-operator on $\kappa\oplus\kappa^{\vee}$, as a $\Ccal^{\infty}$ bundle on $X$, by 
\begin{equation}\label{eq:del-bar-V-bis}
        \left(\begin{array}{cc}
        \ov{\partial}_{\kappa}    &\widetilde{\beta}\\
        0          &\ov{\partial}_{\kappa^{\vee}}
    
    \end{array}\right).
\end{equation}
It defines an extension $V^{\prime}$ of $\kappa^{\vee}$ by $\kappa$. By construction, the extensions $V$ and $V^{\prime}$ are isomorphic on fibers. Because $S$ is Stein, this suffices to guarantee that $V$ and $V^{\prime}$ are isomorphic as extensions on $X$. Therefore, we can suppose that \eqref{eq:del-bar-V-bis} is the $\ov{\partial}$-operator of $V$ in the $\Ccal^{\infty}$ representation given by $E$.

\subsubsection{Linearity and compatibility of the Chern--Simons transform}\label{subsubsec:lin-comp-CF-tr}
Let $\sigma_{1},\sigma_{2}\colon S\to\Pcal(X/S)$ be two $\Ccal^{\infty}$ sections. Since $\Pcal(X/S)$ is an affine bundle under $\VBbb(f_{\ast}\omega_{X/S}^{2})$, there exists a $\Ccal^{\infty}$ section $q$ of $f_{\ast}\omega_{X/S}^{2}$ such that $\sigma_{2}=\sigma_{1}+q$. We proceed to compare the Chern--Simons transforms $\nabla^{\sigma_{1}}$ and $\nabla^{\sigma_{2}}$ in terms of $q$. 

Suppose momentarily that we are in the setting of \textsection\ref{subsubsec:analytic-description-jet}, and adopt the notation therein. The sections $\sigma_{1}$ and $\sigma_{2}$ induce flat relative connections $\nabla_{1}$ and $\nabla_{2}$ on $V$. There is a relationship $\nabla_{2}=\nabla_{1}+\theta(q)$, where $\theta(q)$ is a $\Ccal^{\infty}$ section of $\End(V)\otimes\Omega^{1}_{X/S}$. In terms of the representation $V\simeq_{\Ccal^{\infty}} E$ (see equation \eqref{eq:splitting}), $\theta(q)$ can be expressed as a matrix: 
\begin{equation}\label{eq:matrix-Phi-q}
   \theta(q)\leftrightsquigarrow \left(\begin{array}{cc}
            0   &q\\
            0   &0
        \end{array}\right).
\end{equation}
Notice that this matrix is killed by the operator \eqref{eq:del-bar-V-bis} if, and only if, $q$ is holomorphic, as expected. 

Now we look at the associated intersection connections on $IC_{2}(V)$. We will have a comparison $\nabla_{2}^{\ICmini}=\nabla_{1}^{\ICmini}+\omega(q)$, for some $(1,0)$-form $\omega(q)$ on $S$. The latter is given by a Chern--Simons integral. Precisely, let $\widetilde{\nabla}_{1}$ be a compatible extension of $\nabla_{1}$. Choose a lift $\widetilde{q}\in A^{1,0}(X,\omega_{X/S})$ of $q$, which exists because sheaves of $\Ccal^{\infty}$-modules have vanishing higher cohomology. We define a corresponding $\Ccal^{\infty}$ section $\theta(\widetilde{q})$ of $\End(V)\otimes\Omega^{1}_{X}$, by a matrix representation analogous to \eqref{eq:matrix-Phi-q}, with $\widetilde{q}$ in place of $q$. Then $\widetilde{\nabla}_{1}+\theta(\widetilde{q})$ is an extension of $\nabla_{2}$. Therefore, $\omega(q)=IT(V,\widetilde{\nabla}_{1},\widetilde{\nabla}_{1}+\theta(\widetilde{q}))$, which can be evaluated by Proposition \ref{prop:ICS-explicit-formula}. We claim that
\begin{equation}\label{eq:linear-vartheta-q}
    \omega(q)=\frac{1}{2\pi i}\int_{X/S}\tr(\widetilde{F}_{1}\wedge\theta(\widetilde{q})),
\end{equation}
where $\widetilde{F}_{1}$ is the curvature of $\widetilde{\nabla}_{1}$. The only term in Proposition \ref{prop:ICS-explicit-formula} whose vanishing requires some justification is $\tr(\theta(\widetilde{q})\wedge\ov{\partial}\theta(\widetilde{q}))$. But this is immediate after inspection, since the operator $\ov{\partial}$ of $V$ is given by \eqref{eq:del-bar-V-bis}, and then $\theta(\widetilde{q})\wedge\ov{\partial}\theta(\widetilde{q})$ is seen to be strictly upper triangular, hence with zero trace. We infer from \eqref{eq:linear-vartheta-q} that $\omega(q)$ is $\Ccal^{\infty}(S)$-linear in $q$, since $\theta(\widetilde{q})$ is $\Ccal^{\infty}(S)$-linear in $\widetilde{q}$ and the expression \eqref{eq:linear-vartheta-q} is independent of the chosen extensions, by Lemma \ref{lemma:canonical-connection-IC2}.

After these preliminaries, we are in position to establish the following proposition, whose first part improves Corollary \ref{cor:CS-Deligne-pairing}.

\begin{proposition}\label{prop:int-conn-Deligne-compatible}
\begin{enumerate}
    \item\label{item:prop-compat-conn-Pcal-1} Chern--Simons transforms are compatible with the holomorphic structure of $\langle\omega_{X/S},\omega_{X/S}\rangle$. 
    \item\label{item:prop-compat-conn-Pcal-2} If $q$ is a $\Ccal^{\infty}$ section of $f_{\ast}\omega_{X/S}^{2}$, then $\nabla^{\sigma+q}=\nabla^{\sigma}+\vartheta(q)$, where $\vartheta(q)$ is a $(1,0)$-form whose dependence on $q$ is $\Ccal^{\infty}(S)$-linear.
\end{enumerate}
\end{proposition}
\begin{proof}
We can place ourselves in the setting of \textsection\ref{subsubsec:ext-unif-Higgs}. It is enough to establish the  properties analogous to \eqref{item:prop-compat-conn-Pcal-1}--\eqref{item:prop-compat-conn-Pcal-2} for $IC_{2}(V)\simeq\langle\kappa^{\vee},\kappa^{\vee}\otimes\omega_{X/S}\rangle$. 

Instead of \eqref{item:prop-compat-conn-Pcal-1}, we will establish a finer universal property. During the discussion, we write $\pi\colon S^{\prime}=\Pcal(X/S)\to S$ and $X^{\prime}\to S^{\prime}$ for the base change of $X\to S$ by $\pi$. We denote by $\pi^{\ast}V$ the pullback of $V$ to $X^{\prime}$. Let $\nabla$ be the flat relative connection on $V$ corresponding to $\sigma$. We will write $\pi^{\ast}\nabla$ for the pullback of $\nabla$ to $X^{\prime}$. Likewise, the universal relative projective structure corresponds to a universal flat relative connection on $\pi^{\ast}V$, say $\nabla^{\un}$. We form the associated intersection connections $\nabla^{\ICmini}$ on $IC_{2}(V)$ and $(\nabla^{\un})^{\ICmini}$ on $IC_{2}(\pi^{\ast}V)\simeq\pi^{\ast}IC_{2}(V)$. We shall prove that $\sigma^{\ast}(\nabla^{\un})^{\ICmini}=\nabla^{\ICmini}$, which is a compatible connection on $IC_{2}(V)=\sigma^{\ast}\pi^{\ast}IC_{2}(V)$. 

The argument goes along the same lines as the first part of the proof of Theorem \ref{theorem:extension-CS}. We return to the discussion preceding the ongoing proposition, with the following changes. Instead of $X\to S$, we are now dealing with the family $X^{\prime}\to S^{\prime}$. For the connection $\nabla_{2}$ we take $\nabla^{\un}$. For $\nabla_{1}$ we take $\pi^{\ast}\nabla$. We write $\nabla^{\un}=\pi^{\ast}\nabla+\theta(q)$, for a smooth section $q$ of $\End(\pi^{\ast}V) \otimes \Omega^{1}_{X^{\prime}/S^{\prime}}$. Since the restriction of $\nabla^{\un}$ along $\sigma$ gives back $\nabla$, the restriction of $\theta(q)$ along $\sigma$ has to vanish. We introduce a compatible extension $\widetilde{\nabla}$ of $\nabla$, so that $\pi^{\ast}\widetilde{\nabla}$ is an extension of $\pi^{\ast}\nabla$. We let $\widetilde{q}\in A^{1,0}(X^{\prime},\omega_{X^{\prime}/S^{\prime}})$ be an extension of $q$. For the intersection connections, we have $(\nabla^{\un})^{\ICmini}=\pi^{\ast}\nabla^{\ICmini}+\omega(q)$, where $\omega(q)$ is of the form \eqref{eq:linear-vartheta-q}. We have to show that $\sigma^{\ast}\omega(q)$ vanishes. We already know that the restriction of $\theta(q)$ along $\sigma$ vanishes, so that the restriction of $\theta(\widetilde{q})$ along $\sigma$ must become a section of $\End(V)\otimes f^{\ast}\Acal^{1}_{S}$. Because the curvature of $\widetilde{\nabla}$ vanishes on fibers, we conclude that $\sigma^{\ast}\omega(q)=0$, as desired. 

The second claim of the proposition follows from the fact that \eqref{eq:linear-vartheta-q} depends $\Ccal^{\infty}(S)$-linearly in $q$, as we already observed after that equation. 
\end{proof}

\subsubsection{Chern--Simons transform on the Teichm\"uller space}\label{subsubsec:fam-proj-str-Teich}
Let $X_{0}$ be a compact Riemann surface of genus $g\geq 2$. Let $\Tcal=\Tcal(X_{0})$ be the Teichm\"uller space of $X_{0}$, and $f\colon\Ccal\to\Tcal$ the universal Teichm\"uller curve of Bers.\footnote{The concrete construction involves the choice of a Fuchsian uniformization of $X_{0}$, which is unique up to conjugation in $\PSL_{2}(\RBbb)$. This is irrelevant for the purposes of this article and may henceforth be ignored.} We will consider families of projective structures parametrized by $\Tcal$, and the corresponding Chern--Simons transforms. A particularly relevant instance is the holomorphic section $\beta\colon\Tcal\to\Pcal(\Ccal/\Tcal)$ provided by Bers' simultaneous uniformizations of pairs of Riemann surfaces $(X,\ov{X}_{0})$ by quasi-Fuchsian groups. Quasi-Fuchsian uniformizations will be dealt with in greater generality in \textsection\ref{subsec:QF} below. For the ongoing discussion, it will suffice to recall the properties of $\beta$ that we need and use them in a formal manner.

The holomorphic cotangent bundle of $\Tcal$ is naturally isomorphic to $f_{\ast}\omega_{\Ccal/\Tcal}^{2}$. Thus, it is justified to identify smooth sections of $f_{\ast}\omega_{\Ccal/\Tcal}^{2}$ with $(1,0)$-forms on $\Tcal$. In particular, $\Pcal(\Ccal/\Tcal)$ has the structure of a torsor under $\VBbb(\Omega^{1}_{\Tcal})$, the total space of the holomorphic cotangent bundle. Therefore, given a smooth section $\sigma\colon\Tcal\to\Pcal(\Ccal/\Tcal)$, there is a corresponding $\Ccal^{\infty}$ trivialization of the torsor structure $\Pcal(\Ccal/\Tcal)\simeq_{\Ccal^{\infty}}\VBbb(\Omega^{1}_{\Tcal})$. The canonical symplectic form on $\VBbb(\Omega^{1}_{\Tcal})$ can then be transported to $\Pcal(\Ccal/\Tcal)$, thus defining a symplectic form denoted by  $\omega_{\sigma}$. Tautologically, $\sigma^{\ast}\omega_{\sigma}=0$. The  case of the Bers section deserves special consideration: the attached symplectic form $\omega_{\beta}$ is related to the Atiyah--Bott--Goldman form, in the following manner. First, since $\Tcal$ is simply connected, we can compose the relative holonomy map \eqref{eq:rel-hol-map-PSL} with the retraction to the fiber over the point corresponding to $X_{0}$:
\begin{displaymath}
    \Pcal(\Ccal/\Tcal)\overset{\hol}{\longrightarrow}\Mbold_{\Bet}^{\ir}(\Ccal/\Tcal,\PSL_{2})_{\ell}\simeq
    \Mbold_{\Bet}^{\ir}(X_{0},\PSL_{2})_{\ell}\times\Tcal\longrightarrow\Mbold_{\Bet}^{\ir}(X_{0},\PSL_{2})_{\ell}.
\end{displaymath}
The pullback of the Atiyah--Bott--Goldman form $\omega_{\scriptscriptstyle{\PSL}_{2}}$ on $\Mbold_{\Bet}^{\ir}(X_{0},\PSL_{2})_{\ell}$ to $\Pcal(\Ccal/\Tcal)$ is denoted by $\omega_{G}$. Then
\begin{equation}\label{eq:Kawai-1}
    \omega_{\beta}=-i\omega_{G}.
\end{equation}
This is a theorem of Kawai \cite{Kawai}, revisited by Loustau \cite[Theorem 6.10]{Loustau} with an alternative proof. In particular, the Bers section is Lagrangian for $\omega_{G}$. Also, notice that by Theorem \ref{theorem:curvature-CS-single-RS} and Proposition \ref{prop:complements-CS-PSLr}, the first Chern form of $\Kcal_{\CS}(\Ccal/\Tcal)$ is 
\begin{equation}\label{eq:Kawai-2}
    c_{1}(\Kcal_{\CS}(\Ccal/\Tcal))=\frac{1}{\pi^{2}}\omega_{G}.
\end{equation}

We are now ready to state and prove the main theorem of this subsection.
\begin{theorem}\label{thm:rel-proj-str-rel-conn}
Let $\sigma\colon\Tcal\to\Pcal(\Ccal/\Tcal)$ be a $\Ccal^{\infty}$ section. Then, for every smooth section $q$ of $f_{\ast}\omega_{\Ccal/\Tcal}^{2}\simeq\Omega_{\Tcal}^{1}$, we have $\nabla^{\sigma+q}=\nabla^{\sigma}+\frac{2}{\pi} q$.
\end{theorem}
\begin{proof}
Let us write $\nabla^{\sigma+q}=\nabla^{\sigma}+\vartheta(q)$. By Proposition \ref{prop:int-conn-Deligne-compatible}, the form $\vartheta(q)$ has type $(1,0)$ and depends $\Ccal^{\infty}(\Tcal)$-linearly on $q$. By the torsor structure of $\Pcal(\Ccal/\Tcal)$ and the linearity of $\vartheta(q)$ in $q$, we can reduce to the case $\sigma=\beta$. By \cite[Proposition 3.3]{Loustau}, we have the relationship
\begin{displaymath}
    \omega_{\beta+q}-\omega_{\beta}=-\pi^{\ast}dq,
\end{displaymath}
where $\pi\colon\Pcal(\Ccal/\Tcal)\to\Tcal$ is the structure map. Therefore, if $\sigma^{\prime}=\beta+q$, we derive
\begin{displaymath}
    \sigma^{\prime\ast}\omega_{\beta}=dq,
\end{displaymath}
because $\sigma^{\prime\ast}\omega_{\sigma^{\prime}}=0$. Combining \eqref{eq:Kawai-1}--\eqref{eq:Kawai-2}, we find for the first Chern form of $\sigma^{\prime\ast}\Kcal_{\CS}(\Ccal/\Tcal)$
\begin{displaymath}
    c_{1}(\sigma^{\prime\ast}\Kcal_{\CS}(\Ccal/\Tcal))=\frac{i}{\pi^{2}}dq.
\end{displaymath}
Equivalently, the curvature of $\nabla^{\beta+q}$ is
\begin{displaymath}
    F_{\beta+q}=\frac{2}{\pi}dq.
\end{displaymath}
If follows that
\begin{displaymath}
    d\vartheta(q)=F_{\beta+q}-F_{\beta}=\frac{2}{\pi}dq.
\end{displaymath}
Because $q$ and $\vartheta(q)$ are $(1,0)$-forms, this equality entails that $\pi\vartheta(q)-2q$ is holomorphic. But this is true for all such $q$. Therefore, for any $\Ccal^{\infty}(\Tcal)$-function $h$, the form $\pi\vartheta(hq)-2hq$ is holomorphic too. By the $\Ccal^{\infty}(\Tcal)$-linearity of $\vartheta$, we thus find that $h(\pi\vartheta(q)-2q)$ is holomorphic for every $\Ccal^{\infty}(\Tcal)$-function $h$. It is an exercise to check that this is possible only if $\pi\vartheta(q)-2q$ vanishes identically. That is, $\vartheta(q)=\frac{2}{\pi}q$ as asserted. This concludes the proof of the theorem.
\end{proof}

\begin{corollary}\label{cor:BFPT}
 The Chern--Simons transform establishes a canonical, bijective, $A^{1,0}(\Tcal)$-linear correspondence between smooth families of projective structures over $\Tcal$ and compatible connections on $\langle\omega_{\Ccal/\Tcal},\omega_{\Ccal/\Tcal}\rangle$, such that holomorphic families of projective structures exactly correspond to holomorphic connections on the Deligne pairing.
\end{corollary}
\begin{proof}
Because $\Tcal$ is Stein, the torsor $\Pcal(\Ccal/\Tcal)\to\Tcal$ can be holomorphically trivialized, and hence admits a holomorphic section. By Corollary \ref{cor:CS-Deligne-pairing}, the Chern--Simons transform of a holomorphic section is a holomorphic connection on the Deligne pairing. With this understood, the statement is an immediate consequence of Theorem \ref{thm:rel-proj-str-rel-conn}.
\end{proof}

\subsection{Fuchsian uniformization}\label{subsec:Fuchsian-unif}
The Fuchsian uniformization of the fibers of $f\colon X\to S$ defines a section $\sigma_{\Fu}\colon S\to\Pcal(X/S)$. It is  $\Ccal^{\infty}$, as can  be inferred, for instance, from \cite{Hejhal:variational}. An independent argument will be  provided shortly in \textsection\ref{subsubsec:ext-unif-Higgs}. We denote by $\nabla^{\Fu}$ the Chern--Simons transform $\nabla^{\sigma_{\Fu}}$. 

 The hyperbolic metric on the fibers of $f$ defines a hermitian metric on $\omega_{X/S}$. There is an associated intersection metric on the Deligne pairing. It depends on the normalization of the hyperbolic metric only through a scaling factor.\footnote{It might be convenient to think of the usual curvature $-1$ convention, but most of the time we can ignore this normalization.} The corresponding Chern connection is unambiguously defined. We shall call it the \emph{natural Chern connection}.

\begin{theorem}\label{thm:fuchsian-chern}
The natural Chern connection on $\langle\omega_{X/S},\omega_{X/S}\rangle$ coincides with $\nabla^{\Fu}$.
\end{theorem}

The proof of the theorem requires some preparation. 

\subsubsection{Extension of relative uniformizing Higgs bundles}\label{subsubsec:ext-unif-Higgs}

We shall need the analytic description of jet bundles. We thus pick up the discussion \textsection\ref{subsubsec:analytic-description-jet}, and we follow the setting and notation therein. 

Consider the adjoint $\widetilde{\beta}^{\ast}$ of $\widetilde{\beta}$, taken with respect to the hermitian structures on $\kappa$ and $\kappa^{\vee}$ induced by the hyperbolic metric, and extended conjugate linearly to differential form coefficients on $X$. Hence, $\widetilde{\beta}^{\ast}$ is a $(1,0)$-form on $X$ with values in $\Hom(\kappa,\kappa^{\vee})$. We define
\begin{displaymath}
    \widetilde{\Phi}=\left(\begin{array}{cc}
        0    &0\\
        \widetilde{\beta}^{\ast}          &0
    \end{array}\right)\in \End(E)\otimes\Acal^{1,0}(X),
\end{displaymath}
so that
\begin{displaymath}
     \widetilde{\Phi}^{\ast}=\left(\begin{array}{cc}
        0    &\widetilde{\beta}\\
        0          &0
    \end{array}\right)\in \End(E)\otimes\Acal^{0,1}(X).
\end{displaymath}
The restriction of $(E,\widetilde{\Phi})$ to a fiber $X_{s}$ is a uniformizing Higgs bundle as in \cite[pp. 122--123]{Hitchin:self-duality}. This motivates the introduction of a $\Ccal^{\infty}$ connection on $V\simeq_{\Ccal^{\infty}}E$ defined by
\begin{equation}\label{eq:ext-fuchsian-connection}
    \widetilde{\nabla}=\nabla^{\chmini}_{E}+\widetilde{\Phi}+\widetilde{\Phi}^{\ast}.
\end{equation}
Here, $\nabla_{E}^{\chmini}=\partial_{E}+\ov{\partial}_{E}$ is the Chern connection of $E$. By looking at the $\ov{\partial}$-operator \eqref{eq:del-bar-V-bis}, we see that $\widetilde{\nabla}$ is a compatible connection on $V$. After Hitchin \emph{loc. cit.}, the vertical projection of $\widetilde{\nabla}$ is the flat relative connection on $V$ induced by the Fuchsian uniformization of the fibers and the chosen theta characteristic. Incidentally, the family of Fuchsian projective structures $\sigma_{\Fu}\colon S\to\Pcal(X/S)$ is indeed $\Ccal^{\infty}$. See the background on projective structures \textsection \ref{subsubsec:Back-Proj-4}. Although it is not clear whether $\widetilde{\nabla}$ is a canonical extension in the sense of Section \ref{section:canonical-extensions}, it can nevertheless be used to compute intersection connections, by Lemma \ref{lemma:canonical-connection-IC2}. 

\subsubsection{Proof of Theorem \ref{thm:fuchsian-chern}}
Because the statement is of local nature in $S$, we can argue within the framework of \textsection\ref{subsubsec:analytic-description-jet}--\textsection\ref{subsubsec:ext-unif-Higgs}. We introduce the hermitian metric on $V$ induced by the isomorphism $V\simeq_{\Ccal^{\infty}} E$, and denote by $\nabla_{V}^{\chmini}$ the associated Chern connection.

By a similar argument as in Proposition \ref{prop:int-conn-Deligne-compatible}, the theorem reduces to the following. On the one hand, associated to the exact sequence \eqref{eq:fil-Jet} with the chosen hermitian structures, there is a Bott--Chern secondary class $\widetilde{c}_{2}$. It measures the possible lack of isometry of the Whitney isomorphism
\begin{displaymath}
    IC_{2}(V)\simeq\langle\kappa^{\vee},\kappa^{\vee}\otimes\omega_{X/S}\rangle.
\end{displaymath}
See \textsection \ref{subsubsec:ic2metraxiom} \emph{(MIC3)}. On the other hand, by Definition \ref{def:intersection-connection-general}, Lemma \ref{lemma:canonical-connection-IC2} and Definition \ref{def:intersection-connection}, the intersection connection on $IC_{2}(V)$ for the Fuchsian uniformization is computed as
\begin{displaymath}
    \widetilde{\nabla}^{\ICmini}=(\nabla_{V}^{\chmini})^{\ICmini}+IT(V,\nabla_{V}^{\chmini},\widetilde{\nabla}).
\end{displaymath}
Here $\widetilde{\nabla}$ is the extension \eqref{eq:ext-fuchsian-connection}. It is enough to show that
\begin{displaymath}
    IT(V,\nabla_{V}^{\chmini},\widetilde{\nabla})=-2\partial\int_{X/S}\widetilde{c}_{2}=0.
\end{displaymath}

By the definition of the hermitian metric on $V$, the isomorphism $V\simeq_{\Ccal^{\infty}} E$ is an isometry. In this case, the Bott--Chern secondary form has the following expression:
\begin{displaymath}
    \widetilde{c}_{2}=\frac{1}{2\pi i}\tr(\widetilde{\beta}^{\ast}\wedge\widetilde{\beta}).
\end{displaymath}
See for instance \cite[Th\'eor\`eme 10.2]{Deligne-determinant} or \cite[Section 4]{Soule:Bourbaki}. Recall that the restriction of $\widetilde{\beta}$ to a fiber $X_{s}$ is a fixed multiple of the hyperbolic volume form of $X_{s}$. A simple explicit computation shows that the same holds for the $(1,1)$-form $\tr(\widetilde{\beta}^{\ast}\wedge\widetilde{\beta})$. Therefore, the fiber integral of $\widetilde{c}_{2}$ is constant and
\begin{equation}\label{eq:fiber-int-c2-0}
    \partial\int_{X/S}\widetilde{c}_{2}=0.
\end{equation}

For $IT(V,\nabla_{V}^{\chmini},\widetilde{\nabla})$, we use the explicit expression of Proposition \ref{prop:ICS-explicit-formula}. Before, we notice that the definition \eqref{eq:ext-fuchsian-connection} can equivalently be written as 
\begin{displaymath}
    \widetilde{\nabla}=\nabla_{V}^{\chmini}+2\widetilde{\Phi}
\end{displaymath}
and the curvature of $\nabla_{V}^{\chmini}$ is given by
\begin{displaymath}
    F_{\nabla_{V}^{\chmini}}=\left(
        \begin{array}{cc}
            F_{\kappa}-\widetilde{\beta}^{\ast}\wedge\widetilde{\beta}      &\partial\widetilde{\beta}\\
            -\ov{\partial}\widetilde{\beta}^{\ast}      &F_{\kappa^{\vee}}-\widetilde{\beta}\wedge\widetilde{\beta}^{\ast}
        \end{array}
    \right).
\end{displaymath}
In this matrix, $\partial$ denotes the $(1,0)$ part of the Chern connection (on $X$) of $\omega_{X/S}$, acting on the corresponding piece of $\widetilde{\beta}$. With this understood, using that $\widetilde{\beta}$ is $\ov{\partial}$-closed, for the first term in Proposition \ref{prop:ICS-explicit-formula}, we obtain
\begin{displaymath}
    \tr(F_{\nabla_{V}^{\chmini}}\wedge 2\widetilde{\Phi})=-2\partial\tr(\widetilde{\beta}^{\ast}\wedge\widetilde{\beta}).
\end{displaymath}
The second term $\tr(\widetilde{\Phi}\wedge\ov{\partial}\widetilde{\Phi})$ is easily seen to be 0. We find
\begin{displaymath}
    IT(V,\nabla_{V}^{\chmini},\widetilde{\nabla})=-2\partial\int_{X/S}\widetilde{c}_{2},
\end{displaymath}
which vanishes by \eqref{eq:fiber-int-c2-0}. 

All in all, we conclude that the intersection connection $\widetilde{\nabla}^{\ICmini}$ on $IC_{2}(V)$, attached to the Fuchsian uniformization of the fibers, coincides with the Chern connection on the Deligne pairing $\langle\kappa^{\vee}\otimes\omega_{X/S},\kappa^{\vee}\rangle$. This concludes the proof. \qed

\subsubsection{Wolpert's curvature formula} We fix a compact Riemann surface $X_{0}$ of genus $g\geq 2$ and form $\Tcal=\Tcal(X_{0})$ the Teichm\"uller space of $X_{0}$, with $\Ccal\to\Tcal$ the universal Teichm\"uller curve. Denote by $\omega_{\WP}$ the Weil--Petersson K\"ahler form on $\Tcal$. For the sake of clarity, we recall that in local coordinates, the expression of $\omega_{\WP}$ in terms of the Weil--Petersson hermitian pairing is
\begin{displaymath}
    \omega_{\WP}=\frac{i}{2}\sum_{j,k}\left\langle\frac{\partial}{\partial z_{j}},\frac{\partial}{\partial z_{k}}\right\rangle_{\WP} dz_{j}\wedge d\ov{z}_{k}.
\end{displaymath}
If $\mu,\nu\in A^{-1,1}(X)$ are harmonic Beltrami differentials on a compact Riemann surface $X$, representing holomorphic tangent vectors of $\Tcal$ at the point corresponding to $X$, then their Weil--Petersson pairing is given by
\begin{displaymath}
    \langle\mu,\nu\rangle_{\WP}=\int_{X}\mu\ov{\nu}dA,
\end{displaymath}
where $dA$ is the area element of the hyperbolic metric of curvature $-1$.

The following corollary recovers Wolpert's curvature formula \cite[Corollary 5.11]{Wolpert:Chern}. To ease the comparison with \emph{loc. cit.}, we emphasize that Wolpert works with twice the usual K\"ahler form as described above.

\begin{corollary}\label{cor:Wolpert-curvature}
The curvature of the natural Chern connection on $\langle\omega_{\Ccal/\Tcal},\omega_{\Ccal/\Tcal}\rangle$ is $\frac{1}{\pi^{2}}\omega_{\WP}$.
\end{corollary}
\begin{proof}
By Theorem \ref{thm:fuchsian-chern}, we have to compute the curvature of $\sigma_{\Fu}^{\ast}\Kcal_{\CS}(\Ccal/\Tcal)$. Thus, the setting is as in \textsection\ref{subsubsec:fam-proj-str-Teich}. By equation \eqref{eq:Kawai-2}, we have 
\begin{displaymath}
    c_{1}(\sigma_{\Fu}^{\ast}\Kcal_{\CS}(\Ccal/\Tcal))=\frac{1}{\pi^{2}}\sigma_{\Fu}^{\ast}\omega_{G},
\end{displaymath}
where we recall that $\omega_{G}$ is the pullback of the Atiyah--Bott--Goldman form on $\Mbold_{\Bet}^{\ir}(X_{0},\PSL_{2})_{\ell}$ to $\Pcal(\Ccal/\Tcal)$. By Goldman's \cite[Proposition 2.5]{Goldman}, we have $\sigma_{\Fu}^{\ast}\omega_{G}=\omega_{\WP}$. See Loustau \cite[Theorem 4.2]{Loustau} for a formulation in the same terms and conventions as ours. This concludes the proof.
\end{proof}
\begin{remark}
The line bundle $\langle\omega_{\Ccal_{g}/\Mcal_{g}},\omega_{\Ccal_{g}/\Mcal_{g}}\rangle$ is a lift of Mumford's first tautological class $\kappa_{1}$. The corollary above entails Wolpert's result that $\kappa_{1}$ is represented by $\frac{1}{\pi^{2}}\omega_{\WP}$.
\end{remark}

\subsection{Schottky uniformization}\label{subsec:Schottky}
We now study the case of the projective structures induced by Schottky uniformizations. We refer to Bers \cite{Bers:uniformization, Bers:Schottky} for basics on the theory.

\subsubsection{Chern--Simons transform on the Schottky space}\label{subsubsec:CS-transform-Schottky}
Let $\Sfrak_{g}$ be the Schottky space of genus $g$. It carries a universal curve $\Xcal\to\Sfrak_{g}$. By \cite[Theorem 5.1]{Hejhal:variational}, the projective structures given by Schottky uniformizations define a holomorphic map $\sigma_{\Sch}\colon\Sfrak_{g}\longrightarrow\Pcal(\Xcal/\Sfrak_{g})$. Let $\nabla^{\Sch}$ be the Chern--Simons transform of $\sigma_{\Sch}$. By Corollary \ref{cor:CS-Deligne-pairing}, this is a holomorphic connection on $\langle\omega_{\Xcal/\Sfrak_{g}},\omega_{\Xcal/\Sfrak_{g}}\rangle$. 

\begin{theorem}\label{thm:trivial-S-CS}
If $g\geq 3$, then $\langle\omega_{\Xcal/\Sfrak_{g}},\omega_{\Xcal/\Sfrak_{g}}\rangle$ has a holomorphic trivialization which is flat for $\nabla^{\Sch}$. If $g=2$, the same holds for the tenth power $\langle\omega_{\Xcal/\Sfrak_{g}},\omega_{\Xcal/\Sfrak_{g}}\rangle^{10}$.
\end{theorem}
\begin{proof}
Let $X_{0}$ be a marked\footnote{A marking consists in the choice of a base point $p$ and a system of generators $\alpha_{1},\ldots,\alpha_{g},\beta_{1},\ldots,\beta_{g}$ of $\pi_{1}(X_{0},p)$ with relation $\prod_{j}[\alpha_{j},\beta_{j}]=1$, and with the associated intersection matrix $\tiny{\left(\begin{array}{cc} 0 & \id_{g}\\ -\id_{g} & 0 \end{array}\right) }$ in homology.} compact Riemann surface of genus $g$, and $\Tcal=\Tcal(X_{0})$ the Teichm\"uller space of $X_{0}$. The normalized Schottky uniformization of marked Riemann surfaces defines a holomorphic map $p\colon\Tcal\to\Sfrak_{g}$. We claim that $\nabla^{\Sch}$ induces a flat connection on $p^{\ast}\langle\omega_{\Xcal/\Sfrak_{g}},\omega_{\Xcal/\Sfrak_{g}}\rangle$. We form the commutative diagram
\begin{displaymath}
    \xymatrix{
        \Tcal\ar[r]\ar[d]_{p}       &\Pcal(\Xcal/\Sfrak_{g})\times_{\Sfrak_{g}}\Tcal\ar[d]\ar[r]     &\Mbold_{\Bet}^{\ir}(\Xcal/\Sfrak_{g},\PSL_{2})_{\ell}\times_{\Sfrak_{g}}\Tcal\ar[d]\ar[r]       &\Mbold_{\Bet}^{\ir}(X_{0},\PSL_{2})_{\ell} \\
        \Sfrak_{g}\ar[r]^-{\sigma_{\Sch}}     &\Pcal(\Xcal/\Sfrak_{g})\ar[r]        &\Mbold_{\Bet}^{\ir}(\Xcal/\Sfrak_{g},\PSL_{2})_{\ell}.        &
    }
\end{displaymath}
At the top right corner, we used the retraction granted by the simply connected nature of $\Tcal$. We consider the composition of the upper arrows $\alpha\colon\Tcal\to \Mbold_{\Bet}^{\ir}(X_{0},\PSL_{2})_{\ell}$. By the crystalline property of the curvature of Chern--Simons line bundles, we have the realationship
\begin{displaymath}
    c_{1}(p^{\ast}\sigma_{\Sch}^{\ast}\Kcal_{\CS}(\Xcal/\Sfrak_{g}))=\alpha^{\ast}c_{1}(\Lcal_{\CS}(X_{0})^{4})=\frac{1}{\pi^{2}}\alpha^{\ast}\omega_{\scriptscriptstyle{\PSL_2}},
\end{displaymath}
By \cite[Theorem 4.3]{Loustau}, the morphism $\alpha\colon\Tcal\to \Mbold_{\Bet}^{\ir}(X_{0},\PSL_{2})_{\ell}$ lands in a Lagrangian subspace, meaning in particular that the pullback of the holomorphic symplectic form vanishes. This establishes the claim.

Since $\Tcal$ is simply connected and $\nabla^{\Sch}$ is flat, we can find a flat holomorphic trivialization of $p^{\ast}\langle\omega_{\Xcal/\Sfrak_{g}},\omega_{\Xcal/\Sfrak_{g}}\rangle$. Let us call it $m$. We need to show that $m$ descends to $\Sfrak_{g}$. Let $\gamma\in\Aut(\Tcal/\Sfrak_{g})\simeq\pi_{1}(\Sfrak_{g})$. Because $\nabla^{\Sch}$ is already defined on $\Sfrak_{g}$, pullback by $\gamma$ commutes with $\nabla^{\Sch}$ on $\Tcal$. This implies that $\gamma^{\ast}m$ is also a flat trivialization of $p^{\ast}\langle\omega_{\Xcal/\Sfrak_{g}},\omega_{\Xcal/\Sfrak_{g}}\rangle$, and therefore $\gamma^{\ast}m=\chi(\gamma)m$ for some $\chi(\gamma)\in\CBbb^{\times}$. We thus obtain a character of $\pi_{1}(\Sfrak_{g})$. The latter is a quotient of the mapping class group $\Gamma_{g}$. In genus $g\geq 3$, the abelianization $\Gamma_{g}^{\scriptscriptstyle{ab}}$ is trivial \cite{Powell}, and in this case we deduce that $\chi=1$. Therefore $m$ is invariant under the automorphism group, and must descend. In genus $2$, $\Gamma_{2}^{\scriptscriptstyle{ab}}\simeq\ZBbb/10\ZBbb$ by \emph{op. cit.}, and similarly $m^{10}$ descends. This concludes the proof.
\end{proof}

\subsubsection{Potential of the Weil--Petersson form on $\Sfrak_{g}$}\label{subsubsec:potential-Schottky} We now make use of the intersection metric $\|\cdot\|$ on $\langle\omega_{\Xcal/\Sfrak_{g}},\omega_{\Xcal/\Sfrak_{g}}\rangle$, induced by the choice of the hyperbolic metric. If $g\geq 3$, let $\tau_{\Sch}$ be a flat trivialization as in Theorem \ref{thm:trivial-S-CS}. It is unique, up to multiplication by a constant. Then, by the definition of the first Chern form and Wolpert's curvature formula (Corollary \ref{cor:Wolpert-curvature}), we have
\begin{equation}\label{eq:Liouville-Schottky}
    dd^{c}\log\|\tau_{\Sch}\|^{-2}=\frac{1}{\pi^{2}}\omega_{\WP}.
\end{equation}
Therefore, $\log\|\tau_{\Sch}\|^{-2}$ is a potential of the Weil--Petersson form on the Schottky space, which is canonical up to the addition of a constant. In genus $2$, we take instead $\tau_{\Sch}$ to be a trivialization of $\langle\omega_{\Xcal/\Sfrak_{g}},\omega_{\Xcal/\Sfrak_{g}}\rangle^{10}$, and then $\frac{1}{10}\log\|\tau_{\Sch}\|^{-2}$ is a potential of $\omega_{\WP}$.

A potential for the Weil--Petersson form on Schottky space was constructed by Takhtajan--Zograf \cite{Zograf-Takhtajan}, by studying a suitable Lagrangian on the space of conformal metrics on a Riemann surface (the Liouville action). Let $S_{\TZ}\colon\Sfrak_{g}\to\RBbb$ be the function of Takhtajan--Zograf, defined in \textsection 1.4 and \textsection 3 of \emph{op. cit.}

\begin{proposition}\label{prop:potential-WP-TZ}
If $g\geq 3$, we have an equality $\log\|\tau_{\Sch}\|=\frac{1}{2\pi} S_{\TZ}$, up to the addition of a constant. If $g=2$, we have $\frac{1}{10}\log\|\tau_{\Sch}\|=\frac{1}{2\pi} S_{\TZ}$, up to the addition of a constant.
\end{proposition}
\begin{proof}
For brevity, we only treat the case $g\geq 3$. It is enough to check the equality after passing to the Teichm\"uller space $\Tcal$. By \cite[Remark 2, p. 310]{Zograf-Takhtajan}, the difference between the Fuchsian and Schottky structures is $\sigma_{\Fu}-\sigma_{\Sch}=\frac{1}{2}\partial S_{\TZ}$. By Theorem \ref{thm:rel-proj-str-rel-conn}, we infer
\begin{displaymath}
    \nabla^{\Fu}=\nabla^{\Sch}+\frac{1}{\pi}\partial S_{\TZ}.
\end{displaymath}
We apply this identity to $\tau_{\Sch}$, which is flat for $\nabla^{\Sch}$ by construction. By Theorem \ref{thm:fuchsian-chern}, we know that $\nabla^{\Fu}\tau_{\Sch}=\partial\log\|\tau_{\Sch}\|^{2}\otimes\tau_{\Sch}$. We conclude 
\begin{displaymath}
    \partial\log\|\tau_{\Sch}\|^{2}=\frac{1}{\pi}\partial S_{\TZ}.
\end{displaymath}
The claim follows.
\end{proof}
\begin{remark}
Independently of the work of Takhtajan--Zograf, the very definition of $\tau_{\Sch}$ and Theorem \ref{thm:rel-proj-str-rel-conn} show that $\sigma_{\Fu}-\sigma_{\Sch}=\pi\partial\log\|\tau_{\Sch}\|$ if $g\geq 3$, and $\sigma_{\Fu}-\sigma_{\Sch}=\frac{\pi}{10}\partial\log\|\tau_{\Sch}\|$ if $g=2$.
\end{remark}

\subsection{Quasi-Fuchsian uniformization}\label{subsec:QF}
We now provide applications to relative projective structures arising from quasi-Fuchsian uniformizations. For the theory of quasi-Fuchsian groups, we recall our sources \cite{Ahlfors-Bers, Ahlfors, Bers, Bers:uniformization}. 

\subsubsection{Chern--Simons transform on quasi-Fuchsian space}\label{subsubsec:CS-tr-qF} Fix a reference compact Riemann surface $X_{0}$ of genus $g\geq 2$, with a base point $p$, and set $\Gamma=\pi_{1}(X_{0},p)$. We take $\Tcal=\Tcal(X_{0})$ the Teichm\"uller space of $X_{0}$. The Teichm\"uller space of $\ov{X}_{0}$ is naturally isomorphic to the complex conjugate of $\Tcal$: $\Tcal(\ov{X}_{0})\simeq\ov{\Tcal}$. We henceforth identify them. For later use, we recall that $\Tcal$ and $\ov{\Tcal}$ share the same underlying $\Ccal^{\infty}$ manifold, but their almost complex structures differ by a sign.

We define the \emph{quasi-Fuchsian space} of $X_{0}$ as $\Qcal=\Tcal\times\ov{\Tcal}$. Given a point in $\Qcal$ represented by a couple of Riemann surfaces $(X,Y)$,  the Bers simultaneous uniformization of $(X,Y)$ produces a faithful representation $\rho\colon \Gamma\to\PSL_{2}(\CBbb)$, unique up to conjugation, realizing $\Gamma$ as a quasi-Fuchsian subgroup $\Gamma^{\prime}\subset\PSL_{2}(\CBbb)$. The group $\Gamma^{\prime}$ acts on $\PBbb^{1}(\CBbb)$, and its limit set $\Lambda$ is a Jordan curve. The complement $\Omega=\PBbb^{1}(\CBbb)\setminus\Lambda$ is a disjoint union of two domains $\Omega^{+}$ and $\Omega^{-}$ biholomorphic to open disks, with $X\simeq \Omega^{+}/\Gamma^{\prime}$ and $Y\simeq \Omega^{-}/\Gamma^{\prime}$. The subspace $\Tcal\times\lbrace\ov{X}_{0}\rbrace\subset\Qcal$ is called the Bers slice. On $\Qcal$ there are two universal curves $\Xcal^{\pm}\to\Qcal$, whose fibers are of the form $\Omega^{\pm}/\Gamma^{\prime}$ as above, with $\Gamma^{\prime}\subset\PSL_{2}(\CBbb)$ a holomorphic family of quasi-Fuchsian groups, parametrized by $\Qcal$. The restriction of $\Xcal^{+}$ to the Bers slice is the universal Teichm\"uller curve over $\Tcal$, previously denoted by $\Ccal\to\Tcal$.

The Bers uniformization of the fibers of $\Xcal^{+}\sqcup\Xcal^{-}\to\Qcal$ gives rise to sections
\begin{equation}\label{eq:def-QF-sections}
    \sigma_{\QF}^{+}\colon\Qcal\to\Pcal(\Xcal^{+}/\Qcal),\quad \sigma_{\QF}^{-}\colon\Qcal\to\Pcal(\Xcal^{-}/\Qcal). 
\end{equation}
These are holomorphic, since the uniformizing quasi-Fuchsian groups $\Gamma^{\prime}\subset\PSL_{2}(\CBbb)$ of the fibers depend holomorphically on parameters. See also \cite[Lemma 4.6]{Takhtajan-Teo} for an essentially equivalent justification.

We introduce the Chern--Simons transforms of $\sigma^{\pm}_{\QF}$. These are holomorphic connections $\nabla^{\QF,\pm}$ on the Deligne pairings $\langle\omega_{\Xcal^{\pm}/\Qcal},\omega_{\Xcal^{\pm}/\Qcal}\rangle$. To lighten notation, below we write 
\begin{equation}\label{eq:Npm}
    \Ncal^{\pm}=\langle\omega_{\Xcal^{\pm}/\Qcal},\omega_{\Xcal^{\pm}/\Qcal}\rangle
\end{equation}
and 
\begin{equation}\label{eq:N-prod-del-prod}
    \Ncal=\langle\omega_{\Xcal^{+}/\Qcal},\omega_{\Xcal^{+}/\Qcal}\rangle\otimes\langle\omega_{\Xcal^{-}/\Qcal},\omega_{\Xcal^{-}/\Qcal}\rangle.
\end{equation}
The latter is equipped with the tensor product connection $\nabla^{\QF}$.

\subsubsection{The complex metric}\label{subsub:complex-metric-QF} Because $\Qcal$ is simply connected, we can proceed as in the proof of Corollary \ref{cor:Wolpert-curvature} and Theorem \ref{thm:trivial-S-CS}, and obtain holomorphic maps
\begin{equation}\label{eq:varphi-plus-min}
    \varphi^{+}\colon\Qcal\to\Mbold_{\Bet}^{\ir}(X_{0},\PSL_{2})_{\ell},\quad \varphi^{-}\colon\Qcal\to\Mbold_{\Bet}^{\ir}(\ov{X}_{0},\PSL_{2})_{\ell}.
\end{equation}
The Betti spaces in \eqref{eq:varphi-plus-min} are equal to $\Mbold^{\ir}(\Gamma,\PSL_{2})_{\ell}$ and the morphisms $\varphi^{\pm}$ agree, for they both send a point $(X,Y)$ to the class of the corresponding quasi-Fuchsian representation $\Gamma\to\PSL_{2}(\CBbb)$. Thus, from now on we shall write $\varphi=\varphi^{+}=\varphi^{-}$. With this understood, we have for the first Chern forms
\begin{displaymath}
    c_{1}(\Ncal^{\pm},\nabla^{\QF,\pm})=\pm\frac{1}{\pi^{2}}\varphi^{\ast}\omega_{\scriptscriptstyle{\PSL_2}},\quad c_{1}(\Ncal,\nabla^{\QF})=0,
\end{displaymath}
so that $(\Ncal,\nabla^{\QF})$ is flat. Since $\Qcal$ is simply connected, there exists a holomorphic horizontal trivialization of $\Ncal$, unique up to scaling. 

\begin{definition}\label{def:complex-metric-KCS}
A holomorphic horizontal trivialization of $(\Ncal,\nabla^{\QF})$ as above is called a complex metric on $\Ncal$. We introduce the notation $\tau_{\De}$ for these complex metrics.
\end{definition}
\begin{remark}\label{rmk:crystalline-QF}
In the particular case that $g\geq 3$, we can invoke the crystalline property of Corollary \ref{cor:CS-crystalline} and Proposition \ref{prop:complements-CS-PSLr}. We deduce that there are canonical isomorphisms of line bundles with connections
\begin{displaymath}
    \Ncal^{+}\overset{\sim}{\longrightarrow}\varphi^{\ast}\Lcal_{\CS}(X_{0})^{4},\quad \Ncal^{-}\overset{\sim}{\longrightarrow}\varphi^{\ast}\Lcal_{\CS}(\ov{X}_{0})^{4},
\end{displaymath}
and
\begin{displaymath}
    \Ncal\overset{\sim}{\longrightarrow}\varphi^{\ast}(\Lcal_{\CS}(X_{0})^{4}\otimes\Lcal_{\CS}(\ov{X}_{0})^{4}).
\end{displaymath}
The complex metric of $\Lcal_{\CS}(X_{0})^{4}$ on $\Mbold_{\Bet}^{\ir}(X_{0},\PSL_{2})_{\ell}$ can be pulled back to $\Ncal$, thus providing a complex metric in the sense of Definition \ref{def:complex-metric-KCS}.
\end{remark}

For the next statement, we introduce $j\colon\Tcal\hookrightarrow\Qcal$ the totally real embedding sending a point represented by $X$ to the couple $(X,\ov{X})$ of conjugate Riemann surfaces. Notice that the quasi-Fuchsian group uniformizing $X$ and $\ov{X}$ is actually Fuchsian. The image $j(\Tcal)$ is thus called the \emph{Fuchsian locus}. Recall that the pairing $\langle\omega_{\Ccal/\Tcal},\omega_{\Ccal/\Tcal}\rangle$ carries an intersection metric induced by the choice of a hyperbolic metric, whose Chern connection is $\nabla^{\Fu}$ by Theorem \ref{thm:fuchsian-chern}. We denote by $\ov{\nabla}^{\Fu}$ the associated connection on the conjugate line bundle $\ov{\langle\omega_{\Ccal/\Tcal},\omega_{\Ccal/\Tcal}\rangle}$, defined by imposing $\ov{\nabla}^{\Fu}\ov{v}=\ov{\nabla^{\Fu}v}$ for a local section $v$ of $\langle\omega_{\Ccal/\Tcal},\omega_{\Ccal/\Tcal}\rangle$. Below, we interpret the metric on the Deligne pairing as providing a $\Ccal^{\infty}$ trivialization of $\langle\omega_{\Ccal/\Tcal},\omega_{\Ccal/\Tcal}\rangle\otimes\ov{\langle\omega_{\Ccal/\Tcal},\omega_{\Ccal/\Tcal}\rangle}$.

\begin{theorem}\label{thm:CS-Fuchsian}
Let the assumptions be as above.
\begin{enumerate}
    \item There are unique isomorphisms of $\Ccal^{\infty}$ complex line bundles with connections on $\Tcal$
    \begin{displaymath}
        j^{\ast}(\Ncal^{+},\nabla^{\QF,+})\simeq(\langle\omega_{\Ccal/\Tcal},\omega_{\Ccal/\Tcal}\rangle,\nabla^{\Fu}),\quad j^{\ast}(\Ncal^{-},\nabla^{\QF,-})\simeq(\ov{\langle\omega_{\Ccal/\Tcal},\omega_{\Ccal/\Tcal}\rangle},\ov{\nabla}^{\Fu}),
    \end{displaymath}
    restricting to the identity at $X_{0}$ and $\ov{X}_{0}$, respectively.
    \item For the hyperbolic intersection metrics, the isomorphisms of the first point are also compatible with the Chern connections, and isometric.
    \item Let $\tau_{\De}$ be a complex metric on $\Ncal$. Then, via the isomorphisms of the first point, $j^{\ast}\tau_{\De}$ is identified with a constant multiple of the trivialization associated to the hyperbolic intersection metric on $\langle\omega_{\Ccal/\Tcal},\omega_{\Ccal/\Tcal}\rangle$.
    \item For the hyperbolic intersection metrics, the norm of $\tau_{\QF}$ is constant along the Fuchsian locus.
\end{enumerate}
\end{theorem}

\begin{proof}
We begin with a local version of the first point, say for $\Ncal^{+}$. Let $p_{1}\colon\Qcal\to\Tcal$ be the projection to the first factor, and denote by $\widetilde{\Ccal}\to\Qcal$ the base change of $\Ccal\to\Tcal$ by $p_{1}$. Because $\Ccal\to\Tcal$ is a local universal family of Riemann surfaces, the curves $\Xcal^{+}$ and $\widetilde{\Ccal}$ are isomorphic locally with respect to $\Qcal$. We can thus cover the Fuchsian locus by open subsets of the form $V=U\times\ov{U}$, so that the restricted curves $\Xcal^{+}_{V}\to V$ and $\widetilde{\Ccal}_{V}\to V$ are isomorphic over $V$. For such $V$, we look at the commutative diagram
\begin{displaymath}
    \xymatrix{
                    &\Pcal(\Xcal^{+}_{V}/V)\ar[d]\ar[r]^{\widetilde{p}_{1}}       &\Pcal(\Ccal_{U}/U)\ar[d]\\
    U\ar[r]^{j}\ar@/_1pc/[rr]_{\id}      &V \ar[r]^{p_{1}}\ar@/^1pc/[u]^{\sigma_{\QF}^{+}}                           &U\ar@/^1pc/[u]^{\sigma_{\Fu}},
    }
\end{displaymath}
where the square is cartesian with $\widetilde{p}_{1}$ defined as the composition
\begin{displaymath}
    \Pcal(\Xcal^{+}_{V}/V)\simeq\Pcal(\widetilde{\Ccal}_{V}/ V)\simeq\Pcal(\Ccal_{U}/U)\times_{U}V\to\Pcal(\Ccal_{U}/U).
\end{displaymath}
We have the relationship $\widetilde{p}_{1}\sigma_{\QF}^{+}j=\sigma_{\Fu}$, which translates the tautological fact that the quasi-Fuchsian uniformization agrees with the Fuchsian uniformization along the Fuchsian locus. By the functoriality of complex Chern--Simons line bundles, we have an isomorphism
\begin{displaymath}
    \Kcal_{\CS}(\Xcal^{+}_{V}/V)\simeq \widetilde{p}_{1}^{\ast}\Kcal_{\CS}(\Ccal_{U}/_{U})
\end{displaymath}
of holomorphic line bundles with connections. We derive a string of isomorphisms of $\Ccal^{\infty}$ bundles with connections
\begin{equation}\label{eq:isomorphism-head-ache}
    \begin{split}
         j^{\ast}(\Ncal^{+},\nabla^{\QF,+})_{\mid U}\simeq &j^{\ast}(\sigma_{\QF}^{+})^{\ast}\Kcal_{\CS}(\Xcal^{+}_{V}/V)\\
         \simeq &j^{\ast}(\sigma_{\QF}^{+})^{\ast}\widetilde{p}_{1}^{\ast}\Kcal_{\CS}(\Ccal_{U}/U)\\
         \simeq &\sigma_{\Fu}^{\ast}\Kcal_{\CS}(\Ccal_{U}/U)\simeq (\langle\omega_{\Ccal/\Tcal},\omega_{\Ccal/\Tcal}\rangle,\nabla^{\Fu})_{\mid U},
    \end{split}
\end{equation}
 where we applied Theorem \ref{prop:iso-Kcal-Deligne} in the last isomorphism. 
 
 Next, consider an open cover $\lbrace U_{i}\rbrace_{i}$ of $\Tcal$, such that the previous argument applies to each $V_{i}=U_{i}\times\ov{U}_{i}$. We may suppose that the intersections $U_{i}\cap U_{j}$ are connected. We denote the corresponding isomorphisms of the form \eqref{eq:isomorphism-head-ache} by $\varphi_{i}$. On overlaps $U_{ij}=U_{i}\cap U_{j}$, we can write $\varphi_{i}=c_{ij}\varphi_{j}$, for some invertible functions $c_{ij}\in\Ccal^{\infty}(U_{ij})$. Since the isomorphisms $\varphi_{i}$ and $\varphi_{j}$ are compatible with the same connections $\nabla^{\QF,+}$ and $\nabla^{\Fu}$, the functions $c_{ij}$ are necessarily constant. These constants constitute a 1-cocyle with values in $\CBbb^{\times}$. Because $\Tcal$ is contractible, $H^{1}(\Tcal,\CBbb^{\times})=\lbrace 1\rbrace$ and the cocycle $\lbrace c_{ij}\rbrace_{ij}$ is a coboundary. Therefore, after possibly scaling the isomorphisms $\varphi_{i}$ by constants, we can suppose that $\varphi_{i}=\varphi_{j}$ on $U_{ij}$, and hence glue them together into a global isomorphism of line bundles over $\Tcal$. The latter still preserves the connections, since the scaling involves only constants. Such an isomorphism preserving the connections is unique up to a constant, and can be normalized to restrict to the identity at the origin $X_{0}$ of $\Tcal$. This is the sought canonical isomorphism of the statement. By the same token, if $\ov{j}\colon\ov{\Tcal}\hookrightarrow\Qcal$ sends $\ov{X}$ to $(X,\ov{X})$, then there is a canonical isomorphism
 \begin{equation}\label{eq:j-ast-N-minus}
    \ov{j}^{\ast}(\Ncal^{-},\nabla^{\QF,-})\simeq (\langle\omega_{\ov{\Ccal}/\ov{\Tcal}},\omega_{\ov{\Ccal}/\ov{\Tcal}}\rangle,\widetilde{\nabla}^{\Fu}),
 \end{equation}
where $\widetilde{\nabla}^{\Fu}$ is the natural Chern connection on $\langle\omega_{\ov{\Ccal}/\ov{\Tcal}},\omega_{\ov{\Ccal}/\ov{\Tcal}}\rangle$. The underlying $\Ccal^{\infty}$ manifolds of $\Tcal$ and $\ov{\Tcal}$ are equal, and in this interpretation the $\Ccal^{\infty}$ morphisms $j$ and $\ov{j}$ are just the diagonal embedding. Therefore, the $\Ccal^{\infty}$ bundle $\ov{j}^{\ast}\Ncal^{-}$ is the same as $j^{\ast}\Ncal^{-}$. It remains to justify that $\langle\omega_{\ov{\Ccal}/\ov{\Tcal}},\omega_{\ov{\Ccal}/\ov{\Tcal}}\rangle$ is naturally isomorphic to the conjugate of $\langle\omega_{\Ccal/\Tcal},\omega_{\Ccal/\Tcal}\rangle$, and $\widetilde{\nabla}^{\Fu}$ corresponds to $\ov{\nabla}^{\Fu}$. Below, we address these facts in sequence.

Consider a local holomorphic trivialization $\langle s,t\rangle$ of $\langle\omega_{\Ccal/\Tcal},\omega_{\Ccal/\Tcal}\rangle$, where $s$ and $t$ are meromorphic sections of $\omega_{\Ccal/\Tcal}$ (after possibly replacing $\Tcal$ by an open subset). We denote by $\ov{\langle s,t\rangle}$ the same section seen in $\ov{\langle\omega_{\Ccal/\Tcal},\omega_{\Ccal/\Tcal}\rangle}$. The complex conjugates $\ov{s}$ and $\ov{t}$ make sense as relative meromorphic forms on $\ov{\Ccal}$, and the symbol $\langle\ov{s},\ov{t}\rangle$ is a local holomorphic trivialization of $\langle\omega_{\ov{\Ccal}/\ov{\Tcal}},\omega_{\ov{\Ccal}/\ov{\Tcal}}\rangle$. We claim that sending $\langle\ov{s},\ov{t}\rangle$ to $\ov{\langle s,t\rangle}$ for any such symbols, induces a $\Ccal^{\infty}$ isomorphism of complex line bundles. It is enough to show that the relations defining the Deligne pairing are preserved, compatibly with the action of the sheaf $\Ccal^{\infty}_{\Tcal}$ on these line bundles. For instance, let us change $s$ by a meromorphic function $f$, such that the symbol $\langle fs,t\rangle$ is still defined. On the one hand,
\begin{displaymath}
    \langle\ov{f}\ \ov{s}, \ov{t}\rangle=N_{\Div \ov{t}/\ov{\Tcal}}(\ov{f})\ \langle\ov{s},\ov{t}\rangle=\ov{N_{\Div t/\Tcal}(f)}\ \langle\ov{s},\ov{t}\rangle.
\end{displaymath}
On the other hand, by the very definition of the action of $\Ccal^{\infty}_{\Tcal}$ on $\ov{\langle\omega_{\Ccal/\Tcal},\omega_{\Ccal/\Tcal}\rangle}$, we have
\begin{displaymath}
    \ov{N_{\Div t/\Tcal}(f)}\ \ov{\langle s, t\rangle}=\ov{N_{\Div t/\Tcal}(f)\langle s, t\rangle}=\ov{\langle fs,t\rangle}. 
\end{displaymath}
The symmetric argument is valid for the second entry of the Deligne pairing. This settles the claim. In particular, the Chern connection $\widetilde{\nabla}^{\Fu}$ can be seen as acting on $\ov{\langle\omega_{\Ccal/\Tcal},\omega_{\Ccal/\Tcal}\rangle}$.

The action of the Chern connections on local sections are made explicit as follows. If $\langle s,t\rangle$ is a symbol as above, then
\begin{displaymath}
    \nabla^{\Fu}\langle s,t\rangle=\partial\log\|\langle s,t\rangle\|^{2}\otimes\langle s,t\rangle.
\end{displaymath}
Since the $\partial$ operator on $\ov{\Tcal}$ is the $\ov{\partial}$ operator on $\Tcal$, we similarly have
\begin{displaymath}
    \widetilde{\nabla}^{\Fu}\ \ov{\langle s,t\rangle}=\ov{\partial}\log\|\ov{\langle s,t\rangle}\|^{2}\otimes\ov{\langle s,t\rangle}.
\end{displaymath}
We notice that the norms of $\langle s,t\rangle$ and $\ov{\langle s,t\rangle}$ coincide. Extending the actions of $\Ccal^{\infty}_{\Tcal}$ on $\langle\omega_{\Ccal/\Tcal},\omega_{\Ccal/\Tcal}\rangle$ and $\ov{\langle\omega_{\Ccal/\Tcal},\omega_{\Ccal/\Tcal}\rangle}$ to the algebra of complex differential forms in the usual manner, we conclude
\begin{equation}\label{eq:nabla-Fu-nabla-Fu}
    \widetilde{\nabla}^{\Fu}\ \ov{\langle s,t\rangle}=\ov{\nabla^{\Fu}\langle s,t\rangle}.
\end{equation}
Namely, $\widetilde{\nabla}^{\Fu}=\ov{\nabla}^{\Fu}$ by definition of the latter. This completes the proof of the first point.

For the second claim of the statement, the Fuchsian uniformization of the fibers of the curves $\Xcal^{\pm}\to\Qcal$ induces $\Ccal^{\infty}$ sections $\sigma_{\Fu}^{\pm}\colon\Qcal\to\Pcal(\Xcal^{\pm}/\Qcal)$. These satisfy $\sigma_{\Fu}^{\pm}j=\sigma_{\QF}^{\pm}j$, since the quasi-Fuchsian uniformization agrees with the Fuchsian uniformization along the Fuchsian locus. We know by Theorem \ref{thm:fuchsian-chern} that the Chern--Simons transforms of the $\sigma_{\Fu}^{\pm}$ are the natural Chern connections on $\Ncal^{\pm}$, denoted by $\nabla^{\Fu,\pm}$. Therefore we find
\begin{displaymath}
    \begin{split}
         j^{\ast}(\Ncal^{\pm},\nabla^{\QF,\pm})=&j^{\ast}(\sigma_{\QF}^{\pm})^{\ast}\Kcal_{\CS}(\Xcal^{\pm}/\Qcal)\\
         =&j^{\ast}(\sigma_{\Fu}^{\pm})^{\ast}\Kcal_{\CS}(\Xcal^{\pm}/\Qcal)=j^{\ast}(\Ncal^{\pm},\nabla^{\Fu,\pm}).
    \end{split}
\end{displaymath}
Thus, the isomorphisms of the first point are also compatible with the Chern connections $\nabla^{\Fu,\pm}$ on $\Ncal^{\pm}$, and necessarily isometric, up to constants. The normalization condition at $X_{0}$ and $\ov{X}_{0}$ ensures they are actually isometric everywhere.

Now for the third point. The trivialization $j^{\ast}\tau_{\De}$ is flat for the connection $j^{\ast}\nabla^{\QF}$, which is the tensor product connection of $j^{\ast}\nabla^{\QF,+}$ and $j^{\ast}\nabla^{\QF,-}$. For simplicity of notation, let us identify these objects with the corresponding ones via the isomorphisms of the first point. In particular, the connections are identified with the connections $\nabla^{\Fu}$ and $\ov{\nabla}^{\Fu}$. We write the flatness condition of the dual trivialization $j^{\ast}\tau^{\vee}_{\De}$ in terms of its action on $\langle\omega_{\Ccal/\Tcal},\omega_{\Ccal/\Tcal}\rangle\otimes\ov{\langle\omega_{\Ccal/\Tcal},\omega_{\Ccal/\Tcal}\rangle}$. If $u\otimes\ov{v}$ is a local section of the latter, we find 
\begin{equation}\label{eq:d-j-ast-tau}
    j^{\ast}\tau^{\vee}_{\De}(\nabla^{\Fu}u\otimes\ov{v})+j^{\ast}\tau^{\vee}_{\De}(u\otimes\ov{\nabla^{\Fu}v})=d(j^{\ast}\tau^{\vee}_{\De}(u\otimes\ov{v})),
\end{equation}
where we used that $\ov{\nabla}^{\Fu}\ov{v}=\ov{\nabla^{\Fu}v}$ by definition of $\ov{\nabla}^{\Fu}$. The differential equation \eqref{eq:d-j-ast-tau} is the same as the one expressing that $\nabla^{\Fu}$ is unitary for the hermitian metric on the Deligne pairing. Since two flat trivializations of a line bundle differ by a constant, we conclude that $j^{\ast}\tau^{\vee}_{\De}$ and the metric are proportional as in the statement.

The last assertion is a direct consequence of the second and third points of the theorem.
\end{proof}

\begin{remark}\label{rmk:complex-metric-QF}
\begin{enumerate}
    \item Theorem \ref{thm:CS-Fuchsian} can be loosely restated by saying that the hermitian metric on the Deligne pairing uniquely extends to a holomorphic trivialization of $\Ncal$. This justifies the terminology "complex metric". 
    \item The proof of the theorem overcomes the fact that the embedding $j\colon\Tcal\to\Qcal$ is not holomorphic, and hence the base change functoriality of intersection bundles does not apply. Furthermore, the base changes of the Bers' curves by $j$ are not holomorphic or anti-holomorphic families, but only $\Ccal^{\infty}$. We will encounter similar complications later in \textsection \ref{subsec:DRR-quasi-Fuchsian}. See also the next remark.
    \item It is possible to establish a variant of the theorem, where the Bers' curves $\Xcal^{\pm}$ are replaced by the base changes to $\Qcal$ of the curves $\Ccal\to\Tcal$ and $\ov{\Ccal}\to\ov{\Tcal}$. The corresponding statements are then easier to justify. However, it is unnatural to work with these families of curves.
\end{enumerate}
\end{remark}

\subsubsection{Potential of the Weil--Petersson form on the quasi-Fuchsian space}\label{subsub:potential-WP-QF} Since the line bundle $\Ncal$ is the product of Deligne pairings in \eqref{eq:N-prod-del-prod}, we can also equip it with the hermitian metric associated to a hyperbolic metric on the fibers of $\Xcal^{\pm}\to\Qcal$. For any complex metric $\tau_{\De}$ on $\Ncal$, the function $\log\|\tau_{\De}\|^{-2}$ is then a potential of the Weil--Petersson form on $\Qcal$. By Theorem \ref{thm:CS-Fuchsian}, the function $\log\|\tau_{\De}\|^{-2}$ is constant along the Fuchsian locus. It is also real analytic, since the Weil--Petersson metric is real analytic. Finally, it is easily seen to be invariant by the natural anti-holomorphic involution of $\Qcal$, sending a point $(X,Y)$ to $(\ov{Y},\ov{X})$. These are the same conditions satisfied by the Liouville action of Takhtajan--Teo \cite{Takhtajan-Teo}. See also McIntyre--Teo \cite[Section 2.6]{McIntyre-Teo}. In Corollary \ref{eq:coincidence-Liouville} below, we will prove that both potentials actually agree, up to the addition of a constant. The proof is based on holomorphic factorization formulas for determinants of Laplacians. An alternative method, along the lines of the proof of Proposition \ref{prop:potential-WP-TZ}, relies on an extension of Theorem \ref{thm:rel-proj-str-rel-conn} to the quasi-Fuchsian space. A more general setting is considered in the forthcoming subsection, where this argument is sketched.

\subsection{Comparison to the work of Guillarmou--Moroianu}\label{subsection:convex-hyp}
In the previous subsections, we focused on spaces of Schottky and quasi-Fuchsian groups. These are particular instances of Kleinian groups, uniformizing convex cocompact hyperbolic 3-manifolds. In \cite{Guillarmou-Moroianu}, Guillarmou--Moroianu developed a theory of complex Chern--Simons line bundles on the Teichm\"uller spaces of such groups. Taking \textsection\ref{subsec:CS-transform}--\textsection \ref{subsec:QF} as a guide, we propose an analogous construction from our perspective, and outline the comparison between both approaches.

We follow the exposition by Loustau \cite[Section 2.3]{Loustau} regarding convex cocompact hyperbolic manifolds, and the references therein. Let $M$ be an oriented complete hyperbolic 3-manifold, isometric to a quotient $\HBbb^{3}/\Gamma$, where $\HBbb^{3}$ is the hyperbolic 3-space and $\pi_{1}(M)\simeq\Gamma\subset\PSL_{2}(\CBbb)$ is a Kleinian group. We suppose that $M$ admits a compactificaiton by adding a conformal boundary $\partial M$, consisting of a finite number of compact Riemann surfaces of genus at least 2. The boundary can be uniformized as $\Omega/\Gamma$, where $\Omega\subset\PBbb^{1}(\CBbb)$ is the domain of discontinuity of $\Gamma$. Let us decompose $\partial M=X_{1}\sqcup\ldots\sqcup X_{n}$, where the $X_{j}$ are connected Riemann surfaces. We consider the associated Teichm\"uller space $\Tcal:=\Tcal(\partial M)=\Tcal(X_{1})\times\ldots\times\Tcal(X_{n})$, with universal curves $\Xcal_{j}\to\Tcal$, for $j=1,\ldots,n$. We write $\Pcal(\partial M)=\Pcal(\Xcal_{1}/\Tcal)\times\ldots\times\Pcal(\Xcal_{n}/\Tcal)$. This is the space of all the complex projective structures on $\partial M$, seen as an oriented topological surface. We define the complex Chern--Simons line bundle on $\Pcal(\partial M)$ as the product
\begin{displaymath}
    \Kcal_{\CS}(\partial M)=\bigotimes_{j}p_{j}^{\ast}\Kcal_{\CS}(\Xcal_{j}/\Tcal),
\end{displaymath}
where $p_{j}$ is the projection to the $j$-th component. If $\pi\colon\Pcal(\partial M)\to\Tcal(\partial M)$ is the structure map, then Theorem \ref{prop:iso-Kcal-Deligne} provides a canonical isomorphism
\begin{displaymath}
    \Kcal_{\CS}(\partial M)\simeq\bigotimes_{j}\pi^{\ast}\langle\omega_{\Xcal_{j}/\Tcal},\omega_{\Xcal_{j}/\Tcal}\rangle.
\end{displaymath}
We will look at two particular sections of $\pi\colon\Pcal(\partial M)\to\Tcal(\partial M)$, and the corresponding induced connections on $\otimes_{j}\langle\omega_{\Xcal_{j}/\Tcal},\omega_{\Xcal_{j}/\Tcal}\rangle$. 

The first one is a holomorphic section $\beta$, obtained by varying the convex cocompact hyperbolic structure on $M$. This is a generalization of Bers' simultaneous uniformization. We refer to Loustau \emph{loc. cit.}, and more precisely Theorem 2.1, Proposition 2.2. By his Theorem 4.3, and proceeding as in \textsection\ref{subsec:QF} above for the quasi-Fuchsian case, the connection on $\beta^{\ast}\Kcal_{\CS}(\partial M)$ is flat. By the simple connectivity of $\Tcal(\partial M)$, we can find a flat trivialization $\tau_{\scriptscriptstyle{\mathsf{K}}}$, unique up to constant.

The second one is the Fuchsian section $\sigma_{\Fu}$, obtained as in \textsection \ref{subsec:Fuchsian-unif}. By Theorem \ref{thm:fuchsian-chern}, the connection on $\sigma^{\ast}_{\Fu}\Kcal_{\CS}(\partial M)$ is identified with the product of the natural Chern connections on the Deligne pairings, denoted by $\nabla^{\Fu}$. The curvature is given in terms of the Weil--Petersson form on $\Tcal(\partial M)$, by Corollary \ref{cor:Wolpert-curvature}. A potential of the latter is then $\log\|\tau_{\scriptscriptstyle{\mathsf{K}}}\|^{-2}$, where $\|\cdot\|$ is the metric on the product of Deligne pairings, associated to the choice of the hyperbolic metric on the fibers of the curves $\Xcal_{j}\to\Tcal$.

Now we review the main constructions of Guillarmou--Moroianu \cite{Guillarmou-Moroianu}, and adopt their notation for an easier comparison. The authors define a Chern--Simons line bundle $\Lcal$ on $\Tcal$. It is a holomorphic line bundle endowed with a hermitian metric, whose curvature is also expressed in terms of the Weil--Petersson form \cite[Theorem 2]{Guillarmou-Moroianu}. The Chern connection is denoted by $\nabla^{\Lcal}$. They consider as well the pullback of $(\Lcal,\nabla^{\Lcal})$ to (the total space of) the cotangent bundle $T^{\ast}\Tcal$, where they modify the connection $\nabla^{\Lcal}$ by adding a suitable multiple of the $(1,0)$ part of the Liouville form. This produces a new compatible connection $\nabla^{\mu}$. By means of a renormalized complex Chern--Simons action, they also exhibit an explicit holomorphic section of $\Lcal$ on $\Tcal$, denoted by $e^{2\pi i\CS^{\PSL_{2}(\CBbb)}}$. Up to a constant, it is characterized by being flat for $\sigma^{\ast}\nabla^{\mu}$, where $\sigma\colon\Tcal\to T^{\ast}\Tcal$ is a section described in terms of hyperbolic funnels.

To compare both constructions, we need to identify $T^{\ast}\Tcal$ and $\Pcal(\partial M)$. For this, we use the $\Ccal^{\infty}$ section $\sigma_{\Fu}$ and the affine bundle structure of $\Pcal(\partial M)$. Under this identification, we can see $\Kcal_{\CS}(\partial M)$ as living on $T^{\ast}\Tcal$. We will now compare the curvatures of $\Kcal_{\CS}(\partial M)$ and $\nabla^{\mu}$. For this, a slight generalization of \cite[Corollary 6.13]{Loustau} to $\Tcal(\partial M)$ is needed. The proof follows the same lines as \emph{loc. cit.}, by applying \cite[Theorem, bottom of p. 1769]{Loustau} and a straightforward extension of Proposition 3.3 in \emph{op. cit.}. One can then deduce that the curvature of $\Kcal_{\CS}(\partial M)$ is minus twice the curvature of $\nabla^{\mu}$. Because $\Tcal$ is simply connected, we infer that there is an isomorphism of line bundles with connections  $\psi\colon\Kcal_{\CS}(\partial M)\simeq (\Lcal,\nabla^{\mu})^{\otimes (-2)}$ on $T^{\ast}\Tcal$, unique up to constant. This isomorphism is necessarily holomorphic, for the holomorphic structures of the involved bundles.

Via the identification $\Pcal(\partial M)\simeq T^{\ast}\Tcal$ above, the Fuchsian section $\sigma_{\Fu}$ corresponds to the zero section. Pulling back $\Kcal_{\CS}(\partial M)$ and $(\Lcal,\nabla^{\mu})$ by this section, $\psi$ induces an isomorphism of line bundles with connections
\begin{displaymath}
    \psi_{0}\colon\left(\otimes_{j}\langle\omega_{\Xcal_{j}/\Tcal},\omega_{\Xcal_{j}/\Tcal}\rangle,\nabla^{\Fu}\right)\overset{\sim}{\longrightarrow} (\Lcal,\nabla^{\Lcal})^{\otimes (-2)}.
\end{displaymath}
We deduce that this isomorphism is holomorphic and isometric, up to a constant. 

It remains to compare the holomorphic sections $\tau_{\scriptscriptstyle{\mathsf{K}}}$ and $e^{2\pi i\CS^{\PSL_{2}(\CBbb)}}$. It suffices to compare their norms. For this purpose, one first establishes a direct extension of Theorem \ref{thm:rel-proj-str-rel-conn} to sections of $\Pcal(\partial M)\to\Tcal(\partial M)$, formally with the same argument. Then one can argue as in the proof of Proposition \ref{prop:potential-WP-TZ}, but invoking Takhtajan--Teo \cite[Theorem 6.10]{Takhtajan-Teo} in order to express $\sigma_{\Fu}-\beta=\frac{1}{2}\partial S_{\TT}$, where $S_{\TT}$ is their potential for the Weil--Petersson form in the Kleinian case. This entails that $\log\|\tau_{\scriptscriptstyle{\mathsf{K}}}\|=\frac{1}{2\pi} S_{\TT}$ up to the addition of a constant, or $\|\tau_{\scriptscriptstyle{\mathsf{K}}}\|=\exp(\frac{1}{2\pi} S_{\TT})$ up to scaling. The latter equals the norm of $(e^{2\pi i\CS^{\PSL_{2}(\CBbb)}})^{-2}$, up to scaling. This is a consequence of the very construction \cite[Proposition 1]{Guillarmou-Moroianu} and the fundamental relationship between $S_{\TT}$ and the renormalized volume of convex cocompact hyperbolic 3-manifolds \cite[Theorem 5.3]{Takhtajan-Teo}. We thus conclude that $\tau_{\scriptscriptstyle{\mathsf{K}}}$ and $(e^{2\pi i\CS^{\PSL_{2}(\CBbb)}})^{-2}$ correspond via $\psi_{0}$, up to scaling.

As an application of their work, Guillarmou--Moroianu establish an explicit isomorphism between the Chern--Simons line bundle on the Schottky space and the determinant bundle \cite[Theorem 5]{Guillarmou-Moroianu}. We postpone an approach in our terms to Remark \ref{rmk:isom-GM} \eqref{item:isom-GM-2} in \textsection\ref{subsubsec:relative-bergman-schottky}.

\section{Applications to analytic torsions}\label{section:CM-ARR}
In this section, the theory of Section \ref{section:CS-theory} and Section \ref{section:CS-applications} is applied to the Deligne--Riemann--Roch isomorphism. For the universal object of a moduli space of flat vector bundles on a Riemann surface, our study of the complex Chern--Simons line bundle provides complex metrics on the intersection bundles involved in Deligne's isomorphism. The question is whether there is a counterpart for the determinant of cohomology, which admits a spectral description as in the theory of the Quillen metric. A candidate is the holomorphic Cappell--Miller torsion of flat vector bundles \cite{Cappell-Miller}. In the introduction of \emph{op. cit.}, the authors conjectured that their torsion element indeed satisfies properties akin to those of the Quillen metric in the work of Bismut and coworkers. For flat line bundles on Riemann surfaces, this was settled by Freixas--Wentworth \cite{Freixas-Wentworth-2}, stating that the Cappell--Miller torsion and the complex metrics correspond via Deligne's isomorphism. The results of Section \ref{section:CS-theory} are the key to the extension to arbitrary rank. A similar program can be envisioned on spaces of projective structures on Riemann surfaces, in connection with Chern--Simons transforms and complex metrics for Deligne pairings of canonical bundles. Notably, we consider Bergman and quasi-Fuchsian structures, and we show that our theory subsumes the holomorphic factorization theorems for determinants of Laplacians proven by Kim \cite{Kim} and McIntyre--Teo \cite{McIntyre-Teo}. This leads to the proof of a conjecture of Bertola--Korotkin--Norton \cite{Bertola} on the comparison between Bergman and quasi-Fuchsian projective structures, in terms of Kim's holomorphic extension of the determinant of the Laplacian.

\subsection{Deligne--Riemann--Roch and holomorphic analytic torsion}
Let $f\colon X\to S$ be a proper submersion of complex manifolds, whose fibers are Riemann surfaces of genus $g\geq 2$. Let $E$ be a holomorphic vector bundle on $X$. Recall from Theorem \ref{DRR:iso-general} the Deligne--Riemann--Roch isomorphism

\begin{displaymath}
   \mathsf{DRR}(X/S,E)\colon\lambda(E)^{12}\simeq\langle\omega_{X/S},\omega_{X/S}\rangle^{\rk E}\otimes \langle\det E,\det E \otimes \omega_{X/S}^{-1}\rangle^{6}\otimes IC_{2}(E)^{-12},
\end{displaymath}
which commutes with base change and is functorial in $E$. Suppose that we are given a hermitian metric on $\omega_{X/S}$, and a hermitian metric on a holomorphic vector bundle $E$ on $X$. We equip the intersection bundles on the right hand side of $\mathsf{DRR}(X/S,E)$ with the corresponding intersection metrics. On the determinant of the cohomology, we consider the Quillen metric, that we will shortly review for completeness. As a consequence of the work of Deligne \cite{Deligne-determinant}, Bismut--Freed \cite{freed1, freed2} Bismut--Gillet--Soul\'e \cite{BGS1,BGS2,BGS3} and Bismut-Lebeau \cite{Bismut-Lebeau}, Deligne's isomorphism is an isometry up to an explicit topological constant, depending only on the genus of the fibers of $f$ and the rank of $E$. The precise value of the constant can be extracted from the arithmetic Grothendieck--Riemann--Roch theorem of Gillet--Soul\'e \cite[Theorem 7]{GS:ARR}, but we will not need it in this article. 

Let us briefly recall the construction of the Quillen metric. It is enough to discuss the case when $S$ is a point, and thus $X$ is a compact Riemann surface. First, we introduce the Dolbeault complex
\begin{displaymath}
    A^{0,0}(X,E)\overset{\ov{\partial}_{E}}{\longrightarrow} A^{0,1}(X,E).
\end{displaymath}
Depending on the hermitian metrics on $T_{X}$ and $E$, the spaces $A^{0,p}(X,E)$ carry $L^{2}$ hermitian products. Let $\ov{\partial}^{\ast}_{E}$ be the formal adjoint of $\ov{\partial}_{E}$ for the $L^{2}$-product. There are canonical isomorphisms
\begin{displaymath}
    H^{0}(X,E)\simeq\ker\ov{\partial}_{E}\subset A^{0,0}(X,E),\quad H^{1}(X,E)\simeq\ker\ov{\partial}^{\ast}_{E}\subset A^{0,1}(X,E).
\end{displaymath}
The cohomology spaces inherit the $L^{2}$ hermitian products, and we denote by $h_{\Ltwo}$ the induced metric on $\lambda(E)$. Finally, let $\Delta_{\ov{\partial}_{E}}=\ov{\partial}_{E} \ov{\partial}^{\ast}_{E}$ be the Dolbeault Laplacian acting on $A^{0,1}(X,E)$. It is an elliptic differential operator of order 2, which is self-adjoint and positive. Let $\lbrace\lambda_{k}\rbrace_{k\geq 1}$ be the strictly positive eigenvalues of $\Delta_{\ov{\partial}_{E}}$, repeated according to multiplicities. Then the associated spectral zeta function
\begin{displaymath}
    \zeta_{\ov{\partial}_{E}}(s)=\sum_{k}\frac{1}{\lambda_{k}^{s}},\quad\Real(s)>1,
\end{displaymath}
is absolutely, and locally uniformly convergent, thus defining a holomorphic function on the half-plane $\Real(s)>1$. By the theory of the heat operator, it is shown that $\zeta_{\ov{\partial}_{E}}(s)$ has a meromorphic continuation to $\CBbb$, which is holomorphic at $s=0$. The zeta regularized determinant of $\Delta_{\ov{\partial}_{E}}$ is then defined as
\begin{displaymath}
    \mathrm{det}^{\prime}\Delta_{\ov{\partial}_{E}}=\exp\left(-\zeta_{\ov{\partial}_{E}}^{\prime}(0)\right),
\end{displaymath}
and the Quillen metric is
\begin{displaymath}
    h_{\Qu}=(\mathrm{det}^{\prime}\Delta_{\ov{\partial}_{E}})^{-1}h_{\Ltwo}.
\end{displaymath}

A last word regarding the normalization of the $L^{2}$-metric on $A^{0,p}(X,E)$. In the work of Bismut--Gillet--Soul\'e, it is taken to be $\frac{1}{2\pi}$ times the usual $L^{2}$-pairing in Hodge theory. This distinction is irrelevant for our purposes, but must be borne in mind if one is willing to evaluate the norm of $\mathsf{DRR}(X/S,E)$ exactly, and not only up to a universal topological constant.

\subsection{The Cappell--Miller torsion}
We review the construction of the Cappell--Miller torsion of holomorphic vector bundles with connections of type $(1,1)$ on Riemann surfaces. We mostly follow their original paper \cite{Cappell-Miller}. Further details can be found in Liu--Yu \cite{Liu-Yu}, Huang \cite{Huang} and Su \cite{Su}. 

\subsubsection{Non-self-adjoint Laplacians} Let $X$ be a compact Riemann surface and $E$ a $\Ccal^{\infty}$ complex vector bundle on $X$. Let $\nabla$ be a $\Ccal^{\infty}$ connection on $E$. We suppose that the curvature of $E$ has type $(1,1)$. Then the $(0,1)$ part (resp.\ $(1,0)$ part) of $\nabla$ defines a holomorphic structure (resp.\ anti-holomorphic structure) on $E$. Notice that an anti-holomorphic structure on $E$ is the same thing as a holomorphic structure for $E$ seen as a vector bundle on $\ov{X}$. We write $\ov{\partial}_{E}$ and $\partial_{E}$ for $\nabla^{0,1}$ and $\nabla^{1,0}$. 

We fix a hermitian structure on $X$, and let $\ast$ be the corresponding Hodge star operator on differential forms, which is conjugate linear. Complex conjugation acts on differential forms, and it makes sense to conjugate $\ast$. The resulting operator $\widehat{\ast}$ is complex linear. This can be coupled to $\id_{E}$, and produces a complex linear map
\begin{displaymath}
    \widehat{\ast}\colon A^{p,q}(X,E)\longrightarrow A^{1-q,1-p}(X,E).
\end{displaymath}
We then define
\begin{displaymath}
    \ov{\partial}_{E}^{\natural}=-\widehat{\ast}\partial_{E}\widehat{\ast}\colon A^{p,q}(X,E)\to A^{q,p-1}(X,E),
\end{displaymath}
and 
\begin{displaymath}
    \Delta^{\natural}_{\ov{\partial}_{E}}=(\ov{\partial}_{E}+\ov{\partial}_{E}^{\natural})^{2}\colon A^{p,q}(X,E)\to A^{p,q}(X,E).
\end{displaymath}
If we introduce an auxiliary hermitian metric on $E$, then we can also take the formal adjoint $\ov{\partial}_{E}^{\ast}$ of $\ov{\partial}_{E}$ and the usual Dolbeault Laplacian $\Delta_{\ov{\partial}_{E}}=(\ov{\partial}_{E}+\ov{\partial}_{E}^{\ast})^{2}$. We may then compare
\begin{displaymath}
    \ov{\partial}_{E}^{\natural}=\ov{\partial}_{E}+\varphi,\quad \Delta^{\natural}_{\ov{\partial}_{E}}=\Delta_{\ov{\partial}_{E}}+D,
\end{displaymath}
where $\varphi\in A^{0}(X,\End E)$ and $D=\ov{\partial}_{E}\varphi+\varphi\ov{\partial}_{E}$ is a first order differential operator. We thus see that $\Delta^{\natural}_{\ov{\partial}_{E}}$ is an elliptic differential operator of order 2, with the same principal symbol as $\Delta_{\ov{\partial}_{E}}$. Contrary to the Dolbeault Laplacian, $\Delta^{\natural}_{\ov{\partial}_{E}}$ is in general non-self-adjoint. Nevertheless, for the purpose of constructing the Cappell--Miller torsion, it has as good spectral properties as $\Delta_{\ov{\partial}_{E}}$ for the construction of the Quillen metric \cite[Section 4 \& 11]{Cappell-Miller}. See also M\"uller \cite[Section 2]{Muller} for a compendium of the relevant spectral theory, and Shubin \cite[Chapter 2]{Shubin} for a thorough treatment.

\begin{remark}\label{rmk:super-confusing}
Suppose that $(F,h)$ is a holomorphic hermitian vector bundle on $X$, with Chern connection $\partial_{F}+\ov{\partial}_{F}$. For later use, we clarify a point regarding the structure of the holomorphic vector bundle $(F^{\vee}, \partial_{F})$ on $\ov{X}$. 

 Denote by $\ov{F}$ the complex conjugate vector bundle on $\ov{X}$.\footnote{This is not to be confused with the notation in Section \ref{sec:met-int-bund} for hermitian vector bundles.} For a local holomorphic section $f$ of $F$, there is a corresponding holomorphic section $\ov{f}$ of $\ov{F}$, and the correspondence $f\mapsto\ov{f}$ is conjugate linear. From the condition of Chern connection, it is readily checked that the assignment $\ov{f}\mapsto h(\ \cdot \ , f)$ establishes an isomorphism $\ov{F}\simeq F^{\vee}$ of holomorphic vector bundles on $\ov{X}$, where $F^{\vee}$ is equipped with the holomorphic structure induced by $\partial_{F}$.
\end{remark}

\subsubsection{Construction of the Cappell--Miller torsion}\label{subsubsec:CM-construction} For $(E,\nabla)$ as above, the Cappell--Miller torsion is a non-trivial element $\tau(E,\nabla)$ of the product of determinant lines $\det H^{\bullet}_{\ov{\partial}_{E}}(X,E)\otimes \det H^{\bullet}_{\partial_{E}}(\ov{X},E^{\vee})$. When $\nabla$ is the Chern connection of a hermitian metric on $E$, $\tau(E,\nabla)$ is identified with the Quillen metric, seen as an element of the product of conjugate lines $\det H^{\bullet}_{\ov{\partial}_{E}}(X,E)\otimes\ov{\det H^{\bullet}_{\ov{\partial}_{E}}(X,E)}$ (see Remark \ref{rmk:super-confusing}). We list the main steps of the construction of $\tau(E,\nabla)$.

For the discussion below, recall the notion of determinant of a complex of vector bundles in \textsection \ref{subsec:determinantofcoh}. Let $b\in\RBbb_{>0}$ be such that no generalized eigenvalue of $\Delta^{\natural}_{\ov{\partial}_{E}}$ has real part equal to $b$. We denote by $A^{0,p}_{<b}(X,E)$ the finite dimensional subspace of $A^{0,p}(X,E)$ spanned by generalized eigenvectors, whose eigenvalues $\lambda$ have $\Real\lambda<b$. The Dolbeault complex restricts to 
\begin{displaymath}
    A^{0,0}_{<b}(X,E)\overset{\ov{\partial}_{E}}{\longrightarrow} A^{0,1}_{<b}(X,E).
\end{displaymath}
One can prove that the inclusion $(A^{0,\bullet}_{<b}(X,E),\ov{\partial}_{E})$ into the full Dolbeault complex is a quasi-isomorphism \cite[p. 151]{Cappell-Miller}. Hence, we infer from \eqref{def:detcomplex2} that
\begin{equation}\label{eq:CM-1}
    \det(A^{0,\bullet}_{<b}(X,E),\ov{\partial}_{E})\simeq\det H^{\bullet}_{\ov{\partial}_{E}}(X,E).
\end{equation}
Similarly, the operator $\ov{\partial}_{E}^{\natural}$ provides a homological complex (going in the opposite direction)
\begin{displaymath}
    A^{0,1}_{<b}(X,E)\overset{\ov{\partial}_{E}^{\natural}}{\longrightarrow} A^{0,0}_{<b}(X,E).
\end{displaymath}
By the definition of the determinant of a complex in \eqref{def:detcomplex1}, we trivially have
\begin{equation}\label{eq:CM-2}
    \det (A^{0,\bullet}_{<b}(X,E),\ov{\partial}_{E}^{\natural})\simeq \det(A^{0,\bullet}_{<b}(X,E),\ov{\partial}_{E})^{\vee}.
\end{equation}
We will now observe that by construction, $(A^{0,\bullet}_{<b}(X,E),\ov{\partial}_{E}^{\sharp})$ is isomorphic to a restricted Dolbeault complex for $F:=E\otimes\omega_{\ov{X}}$ on $\ov{X}$. Before, let us stress that $(p,q)$ forms on $X$ are seen as $(q,p)$ forms on $\ov{X}$. With this understood, we have a  commutative diagram
\begin{displaymath}
    \xymatrix{
        A^{0,1}_{<b}(X,E)\ar[r]^{\ov{\partial}_{E}^{\natural}}\ar[d]_{\widehat{\ast}}^{\wr}    &A^{0,0}_{<b}(X,E)\ar[d]_{-\widehat{\ast}}^{\wr}\\
        A^{0,1}_{<b}(X,E)\ar[r]^{\partial_{E}}\ar@{=}[d]     & A^{1,1}_{<b}(X,E)\ar@{=}[d]\\
        A^{0,0}_{<b}(\ov{X},F)\ar[r]^{\partial_{F}}        &A^{0,1}_{<b}(\ov{X},F).
    }
\end{displaymath}
We infer
\begin{equation}\label{eq:CM-3}
    \det (A^{0,\bullet}_{<b}(X,E),\ov{\partial}_{E}^{\natural})\simeq \det H^{\bullet}_{\partial_{F}}(\ov{X},F)\simeq\det H^{\bullet}_{\partial_{E}}(\ov{X},E^{\vee}),
\end{equation}
where the rightmost isomorphism is Serre's duality on $\ov{X}$. Concatenating \eqref{eq:CM-1}--\eqref{eq:CM-3}, we arrive at an isomorphism
\begin{displaymath}
    \det H^{\bullet}_{\ov{\partial}_{E}}(X,E)\simeq\det H^{\bullet}_{\partial_{E}}(\ov{X},E^{\vee})^{\vee}.
\end{displaymath}
Equivalently, we have found a non-trivial element $\tau_{<b}(E,\nabla)\in\det H^{\bullet}_{\ov{\partial}_{E}}(X,E)\otimes\det H^{\bullet}_{\partial_{E}}(\ov{X},E^{\vee})$.

The previous construction depends on the choice of the cut $b$. After introducing an Agmon angle for the definition of the complex powers of $\Delta^{\natural}_{\ov{\partial}_{E}}$, one has a well-defined spectral zeta function
\begin{displaymath}
    \zeta_{>b}(s)=\tr\left(\Pi_{>b}(\Delta^{\natural}_{\ov{\partial}_{E}})^{-s}\right), \quad\Real s>1,
\end{displaymath}
where $\Pi_{>b}=1-\Pi_{<b}$ is the spectral projector to the space of generalized eigenfunctions, whose generalized eigenvalues are $\lambda$ with $\Real\lambda>b$. By Lidskii's theorem, the zeta function has the form
\begin{displaymath}
    \zeta_{>b}(s)=\sum_{k}\frac{1}{\lambda_{k}^{s}},
\end{displaymath}
where the $\lambda_{k}$ are the generalized eigenvalues with $\Real\lambda_{k}>b$, counted with algebraic multiplicities. The zeta function converges absolutely and locally uniformly for $\Real s>1$, and  has a meromorphic continuation to $\CBbb$, which is holomorphic in a neighborhood of $s=0$. We define
\begin{equation}\label{eq:truncated-det}
    \det\Delta^{\natural}_{\ov{\partial}_{E},\ >b}=\exp(-\zeta_{>b}^{\prime}(0))\in\CBbb^{\times},
\end{equation}
which is seen to be independent of the choice of Agmon angle. Finally, define
\begin{displaymath}
    \tau(E,\nabla):=\left(\det\Delta^{\sharp}_{\ov{\partial}_{E},\ >b}\right)^{-1}\tau_{<b}(E,\nabla)\in\det H^{\bullet}_{\ov{\partial}_{E}}(X,E)\otimes\det H^{\bullet}_{\partial_{E}}(\ov{X},E^{\vee}).
\end{displaymath}
This element does not depend on the choice of $b$, and we call it the \emph{Cappell--Miller (holomorphic) torsion}.

\subsection{Cappell--Miller torsion and Deligne--Riemann--Roch}
We will now establish the compatibility of the Cappell--Miller torsion and the complex metrics on intersection bundles, through Deligne's isomorphism.

The setting is as follows. Let $X$ be a compact Riemann surface of genus $g\geq 2$ and $p\in X$ a base point. Set $\Gamma=\pi_{1}(X,p)$. We let $\Xcal=X\times\Rbold^{\ir}(\Gamma,r)\to\Rbold^{\ir}(\Gamma,r)$ be the universal curve. We also write $\Ecal=\Ecal_{\Bet}^{\un}$ for the universal vector bundle. We will make no notational distinction between vector bundles on $X$ and their pullbacks to $\Xcal$. For instance, $\omega_{X}$ is abusively identified with $\omega_{\Xcal/\Rbold^{\ir}(\Gamma,r)}$. We fix a holomorphic vector bundle $F$ of rank $f$ on $X$, and we form Deligne's isomorphism $\mathsf{DRR}(\Xcal/\Rbold^{\ir}(\Gamma,r),\Ecal\otimes F)$:
\begin{equation}\label{eq:Del-iso-1}
    \Dcal\colon\lambda(\Ecal\otimes F)^{12}\simeq\langle\omega_{X},\omega_{X}\rangle^{r \cdot f}
    \otimes\langle\det(\Ecal\otimes F),\det(\Ecal\otimes F)\otimes\omega_{X}^{-1}\rangle^{ 6}\otimes IC_{2}(\Ecal\otimes F)^{-12}.
\end{equation}

We consider corresponding objects for the conjugate Riemann surface $\ov{X}$. Recall that the complex structure of $\Rbold^{\ir}(\Gamma,r)$ does not depend on the complex structure of $X$. We will use the notation $\Xcal^{c}=\ov{X}\times\Rbold^{\ir}(\Gamma,r)\to\Rbold^{\ir}(\Gamma,r)$, and $\Ecal^{c}$ for the dual of the universal object. Hence, the $\Ccal^{\infty}$ vector bundles underlying $\Ecal^{\vee}$ and $\Ecal^{c}$ are the same, as are their flat relative connections, but the holomorphic structures are different. We also take $\ov{F}$ the complex conjugate to $F$. Deligne's isomorphism for $\Ecal^{c}\otimes\ov{F}$ can then be written as
\begin{equation}\label{eq:Del-iso-2}
    \Dcal^{c}\colon\lambda(\Ecal^{c}\otimes \ov{F})^{12}\simeq\langle\omega_{\ov{X}},\omega_{\ov{X}}\rangle^{r \cdot f}
    \otimes\langle\det(\Ecal^{c}\otimes \ov{F}),\det(\Ecal^{c}\otimes \ov{F})\otimes\omega_{\ov{X}}^{-1}\rangle^{ 6}\otimes IC_{2}(\Ecal^{c}\otimes \ov{F})^{-12}.
\end{equation}

\subsubsection{Holomorphic regularity of the Cappell--Miller torsion}\label{subsubsec:CM-holo} We first study the left hand side of the isomorphism $\Dcal\otimes\Dcal^{c}=$\eqref{eq:Del-iso-1}$\otimes$\eqref{eq:Del-iso-2}. Introduce hermitian metrics on $X$ and $F$. The Chern connection of $F$ decomposes as $\partial_{F}+\ov{\partial}_{F}$. We identify $F^{\vee}$ and $\ov{F}$ by means of the hermitian metric as in Remark \ref{rmk:super-confusing}. With this identification, the holomorphic structure of $\ov{F}$ corresponds to the holomorphic structure induced by $\partial_{F}$ on $F^{\vee}$, as in the Cappell--Miller construction. Therefore we can use the Cappell--Miller torsion to obtain a pointwise trivialization of $\lambda(\Ecal\otimes F)\otimes\lambda(\Ecal^{c}\otimes\ov{F})$ on $\Rbold^{\ir}(\Gamma,r)$. The input datum here is the chosen hermitian metric on $X$, the universal connection on $\Ecal$ and the Chern connection on $F$. If $\rho\in\Rbold^{\ir}(\Gamma,r)$ is a given representation, we denote the associated Cappell--Miller element
\begin{displaymath}
    \tau_{\CM}(\rho)\in \lambda(\Ecal_{\rho}\otimes F)\otimes\lambda(\Ecal^{c}_{\rho}\otimes\ov{F}).
\end{displaymath}
We establish that $\rho\mapsto\tau_{\CM}(\rho)$ is holomorphic, with respect to the natural holomorphic structure of the determinant line bundles of Knudsen--Mumford \cite{KnudsenMumford}, briefly described in \textsection \ref{subsec:determinantofcoh}.

\begin{theorem}
With the assumptions as above, the pointwise Cappell--Miller torsion $\tau_{\CM}$ defines a holomorphic trivialization of $\lambda(\Ecal\otimes F)\otimes\lambda(\Ecal^{c}\otimes\ov{F})$ on $\Rbold^{\ir}(\Gamma,r)$.
\end{theorem}

\begin{proof}
If $\Ecal$ has rank one and $F$ is trivial, we already established the holomorphicity property in \cite[Section 5.1]{Freixas-Wentworth-2}. The proof in the general case follows the same lines. The introduction of $F$ does not pose any particular difficulty. Only the analogue of Lemma 5.1 (i) in \emph{op. cit.} requires an explanation. Namely, we need to justify that the non-self-adjoint Laplacians $\Delta^{\natural}_{\rho}$ associated to $\Ecal_{\rho}\otimes F$ in the Cappell--Miller construction, constitute a holomorphic family of closed operators of type (A) in the sense of Kato \cite[Chapter VII, Section 2]{Kato}. We now address this point. The approach is similar to the starting point of Fay's \cite[proof of Theorem 4.8]{Fay}, in the unitary setting.

We need a local holomorphic parametrization of $\Rbold^{\ir}(\Gamma,r)$. For this, we can rely on the local parametrization of $\Mbold_{\dR}^{\ir}(X,r)$ by the Hodge slices, introduced by Collier--Wentworth \cite[Section 3]{Brian-Richard}. Alternatively, we can also use de Rham slices as in Ho--Wilkin--Wu \cite[Section 2]{Wilkin}. Both approaches provide the following description. Fix a representation $\rho_{0}\in\Rbold^{\ir}(\Gamma,r)$, which corresponds to the rigidified flat vector bundle $\Ecal_{\rho_{0}}$ on $X$, with connection $D$. We write $E$ for the $\Ccal^{\infty}$ vector bundle underlying $\Ecal_{\rho_{0}}$. Then there exists a finite dimensional complex submanifold $\Ucal\subset A^{1}(X,\End E)$, containing the origin, such that for every $\mu\in\Ucal$, the $\Ccal^{\infty}$ connection $D+\mu$ on $E$ is flat. In particular, $D+\mu$ defines a new holomorphic structure on $E$. We can then take the point corresponding to $(E,D+\mu)$ in $\Mbold_{\dR}^{\ir}(X,r)$. Possibly shrinking $\Ucal$ around $\rho_{0}$, this assignment defines a holomorphic open immersion $\Ucal\to\Mbold_{\dR}^{\ir}(X,r)$, which parametrizes an open neighborhood of $(E,D)$.

Let $\widetilde{X}$ be the universal cover of $X$ based at $p$, and $\widetilde{p}\in X$ a lift of $p$. By definition, $\Ecal_{\rho_{0}}=\widetilde{X}\times_{\Gamma}\CBbb^{r}$, where the action of $\Gamma$ on $\CBbb^{r}$ is via the representation $\rho_{0}\colon\Gamma\to\GL_{r}(\CBbb)$. On $\widetilde{X}$, the connection $D+\mu$ becomes an operator $d+\widetilde{\mu}$ on vectors of $\rho_{0}$-equivariant functions, where $\widetilde{\mu}$ is now a matrix of holomorphic differential forms with $\Ad\rho_{0}$ multiplier. Because the integrability condition $d\widetilde{\mu}+\widetilde{\mu}\wedge\widetilde{\mu}=0$ is fulfilled, by flatness of both $D$ and $D+\mu$, we can uniquely solve the following differential equation in the variable $z\in\widetilde{X}$:
\begin{displaymath}
    A(\mu,z)^{-1}d_{z}A(\mu,z)=-\widetilde{\mu},\quad A(\mu,\widetilde{p})=\id_{r}.
\end{displaymath}
The matrix $A(\mu,z)$ depends holomorphically on $\mu\in\Ucal$. It satisfies a transformation law of the form
\begin{displaymath}
    A(\mu,\gamma z)=\rho(\mu)(\gamma)A(\mu,z)\rho(\mu)(\gamma)^{-1},
\end{displaymath}
with $\rho(\mu)\colon\Gamma\to\GL_{r}(\CBbb)$ a holomorphic family of representations parametrized by $\mu\in\Ucal$, and $\rho(0)=\rho_{0}$. For this family, the classifying morphism $\Ucal\to\Mbold^{\ir}(\Gamma,r)$ is a holomorphic open embedding, corresponding to the above $\Ucal\to\Mbold_{\dR}^{\ir}(X,r)$, under the Riemann--Hilbert correspondence. For simplicity, let us identify $\Ucal$ with its image in $\Mbold^{\ir}(\Gamma,r)$. Since the $\rho(\mu)$ are actual representations and not just conjugacy classes, we even have a holomorphic map $\Ucal\to\Rbold^{\ir}(\Gamma,r)$, and hence a local section of the quotient map  $\pi\colon\Rbold^{\ir}(\Gamma,r)\to\Mbold^{\ir}(\Gamma,r)$. Since this morphism has the structure of a $\PSL_{r}(\CBbb)$-torsor, we conclude that $\pi^{-1}(\Ucal)\simeq\Ucal\times\PSL_{r}(\CBbb)$, where the action of $\PSL_{r}(\CBbb)$ on $\Ucal$ is by conjugation on the representations $\rho(\mu)$. Because the projection $\SL_{r}(\CBbb)\to\PSL_{r}(\CBbb)$ is \'etale, we can actually use $\Ucal\times\SL_{r}(\CBbb)$ as a local parametrization of $\Rbold^{\ir}(\Gamma,r)$. Given $B\in\SL_{r}(\CBbb)$, we shall write $\rho(\mu,B)=B\rho(\mu)B^{-1}$.

We now examine the Laplacian $\Delta^{\natural}_{\rho(\mu,B)}$ acting on $A^{0,p}(X,\Ecal_{\rho(\mu,B)}\otimes F)$. We lift it to the universal cover. If we denote by $\widetilde{F}$ the pullback of $F$ to $\widetilde{X}$, then $\Delta^{\natural}_{\rho(\mu,B)}$ lifts to an operator on $A^{0,p}(\widetilde{X},\CBbb^{r}\otimes\widetilde{F})$. From the very construction, it is seen to coincide with the Dolbeault Laplacian of $\widetilde{F}^{\oplus r}=\CBbb^{r}\otimes\widetilde{F}$, associated to the pulled back hermitian metrics on $T_{\widetilde{X}}$ and $\widetilde{F}$. The latter does not depend on $\rho(\mu,B)$, and we denote it by $\widetilde{\Delta}$. We now conjugate $\widetilde{\Delta}$ by the function $BA(\mu,\bullet)$ constructed in the previous paragraph. We obtain a new non-self-adjoint Laplacian $\widetilde{\Delta}_{(\mu,B)}=(BA(\mu,\bullet))^{-1}\widetilde{\Delta}\ BA(\mu,\bullet)$. Let us summarize the situation:
\begin{displaymath}
    \begin{split}
        \widetilde{\Delta}\curvearrowright A^{0,p}(\widetilde{X},\CBbb^{r}\otimes\widetilde{F})&\longrightarrow  A^{0,p}(\widetilde{X},\CBbb^{r}\otimes\widetilde{F})\curvearrowleft\widetilde{\Delta}_{(\mu,B)}\\
        \gbsymb&\longmapsto      (BA(\mu,\bullet))^{-1}\gbsymb.\\
        \rho(\mu,B)\text{-equivariant}  &\longmapsto \rho_{0}\text{-equivariant}.
    \end{split}
\end{displaymath}
The operators $\widetilde{\Delta}_{(\mu,B)}$ acting on $\rho_{0}$-equivariant elements in $A^{0,p}(\widetilde{X},\CBbb^{r}\otimes\widetilde{F})$ are closed and share the same domain. They are obtained by conjugating the fixed operator $\widetilde{\Delta}$ by $BA(\mu,\bullet)$, whose dependence on the parameters $(\mu,B)\in\Ucal\times\SL_{r}(\CBbb)$ is holomorphic. These are the conditions required for a holomorphic family of closed operators of type (A). 

For the convenience of the reader, we briefly review the remaining steps of the proof. With the notation as in \textsection\ref{subsubsec:CM-construction}, Kato's condition (A) ensures that the restricted Dolbeault complexes of the form $A^{0,p}_{<b}(X,\Ecal_{\rho}\otimes F)$ organize into holomorphic vector bundles, for small perturbations of $\rho$ in $\Rbold^{\ir}(\Gamma,r)$. One then needs to check that the corresponding isomorphisms \eqref{eq:CM-1} and \eqref{eq:CM-3} are holomorphic, where the determinants of Dolbeault cohomologies are equipped with the holomorphic structure coming from the Knudsen--Mumford construction. The compatibility of both holomorphic structures is easily established after the description of the latter offered by Bismut--Gillet--Soul\'e \cite[p. 346]{BGS3}. For our purpose, the key point is that in the family $X\times\Rbold^{\ir}(\Gamma,r)\to\Rbold^{\ir}(\Gamma,r)$, the complex structure of the fibers is kept constant. The argument in rank one and trivial $F$ is elaborated in \cite[Proposition 5.3]{Freixas-Wentworth-2}, and the general case is formally the same. To conclude, it remains to justify that the determinants $\det\Delta^{\natural}_{\rho,\ >b}$ as in \eqref{eq:truncated-det} are holomorphic in $\rho$, under small deformations of $\rho$. This is a consequence of Greiner's parametrix construction and asymptotic expansions for heat equations of elliptic operators \cite[Section 1]{Greiner}. Alternatively, one can directly invoke Kontsevich--Vishik \cite[Corollary 4.2]{Kontsevich-Vishik}. 
\end{proof}

\subsubsection{Compatibility with the complex metrics} We maintain the setting of \textsection \ref{subsubsec:CM-holo} and exclude the case $g=r=2$. We now look at the right hand side of $\Dcal\otimes\Dcal^{c}$. We first apply Proposition \ref{prop:whitneyproductwithbundle} to expand the $IC_{2}$ terms. After this manipulation, on the right hand side of the isomorphism we find products of intersection bundles of the following form:
\begin{itemize}
    \item $\langle L,M\rangle\otimes\langle\ov{L},\ov{M}\rangle$, where $L$ and $M$ are fixed line bundles on $X$ endowed with hermitian metrics, and $\ov{L}$, $\ov{M}$ are their conjugates on $\ov{X}$. The hermitian metric is interpreted as a trivialization of the tensor product. Notice here there is no variation in the horizontal (\emph{i.e.} $\Rbold^{\ir}(\Gamma,r)$) directions, so that the trivialization is constant (hence holomorphic) in the family.
    \item $\langle L,\det\Ecal\rangle\otimes\langle\ov{L},\det\Ecal^{c}\rangle$, where $L$ is fixed on $X$ and has a hermitian metric. For such products we defined complex metrics in \cite[Section 4]{Freixas-Wentworth-2}. These will provide holomorphic trivializations over $\Rbold^{\ir}(\Gamma,r)$, recovering the hermitian metric on the Deligne pairings at unitary representations. This complex metric depends on the rigidification of $\det\Ecal$, induced by the rigidification of $\Ecal$. But it actually does not depend on the metric on $L$.
    \item  $\langle\det\Ecal,\det\Ecal\rangle\otimes \langle\det\Ecal^{c},\det\Ecal^{c}\rangle$, which also carries a complex metric by \cite[Section 4]{Freixas-Wentworth-2}, recovering the natural hermitian metric at unitary representations. It does not depend on the rigidification.
    \item $IC_{2}(F)\otimes IC_{2}(\ov{F})$, where $F$ is endowed with a hermitian metric and the intersection metric on the $IC_{2}$ is seen as a trivialization of $IC_{2}(F)\otimes IC_{2}(\ov{F})$. Here there is no variation in the horizontal directions.
    \item $IC_{2}(\Ecal)\otimes IC_{2}(\Ecal^{c})$, which is naturally isomorphic to $IC_{2}(\Ecal)\otimes IC_{2}(\Ecal^{c\ \vee})$, and thus carries the holomorphic trivialization obtained by pulling back to $\Rbold^{\ir}(\Gamma,r)$ the complex metric on the dual of $\Lcal_{\CS}(X)\otimes\Lcal_{\CS}(\ov{X})$ on $\Mbold^{\ir}(\Gamma,r)$ (Theorem \ref{theorem:complex-metric}). By construction, it recovers the hermitian metric along the unitary locus.
\end{itemize}
Combining all these, we obtain a holomorphic trivialization of the right hand side of $\Dcal\otimes\Dcal^{c}$, defined over $\Rbold^{\ir}(\Gamma,r)$.

\begin{theorem}\label{theorem:DRR-isom-flat}
Excluding the case $g=r=2$, the isomorphism $\Dcal\otimes\Dcal^{c}$ on $\Rbold^{\ir}(\Gamma,r)$ sends the Cappell--Miller torsion to the complex metric on the combination of intersection bundles, up to an explicit universal topological constant derived from the arithmetic Riemann--Roch theorem of Gillet--Soul\'e. 
\end{theorem}

\begin{proof}
We already know that the  trivializations of $\Dcal\otimes\Dcal^{c}$ considered on both sides   are holomorphic on $\Rbold^{\ir}(\Gamma,r)$ and, along unitary representations, recover the Quillen metric and the natural metric on the intersection bundles. For these, Deligne's isomorphism is an isometry, up to an explicit topological constant determined by the arithmetic Riemann--Roch theorem of Gillet--Soul\'e. The locus of unitary representations in $\Rbold^{\ir}(\Gamma,r)$ is the preimage of the unitary locus in $\Mbold^{\ir}(\Gamma,r)$ by the quotient map $\Rbold^{\ir}(\Gamma,r)\to\Mbold^{\ir}(\Gamma,r)$. The unitary locus in $\Mbold^{\ir}(\Gamma,r)$ is a totally real submanifold, and the quotient map is a holomorphic submersion. This is enough to ensure that the Cappell--Miller torsion and the complex metric have to correspond everywhere via Deligne's isomorphism, up to the same explicit topological constant as in the hermitian case.
\end{proof}

\begin{remark}
\begin{enumerate}
    \item For the proof of the theorem, it is fundamental that the complex structure on $\Rbold^{\ir}(\Gamma,r)$ does not depend on the complex structure of $X$. This is why we need to work with representation spaces instead of de Rham spaces.
    \item Because $\Rbold(\Gamma,r)$ is a normal irreducible space, the Cappell--Miller torsion on $\Rbold^{\ir}(\Gamma,r)$ uniquely extends to the whole of $\Rbold(\Gamma,r)$. This we already knew for the complex metrics on intersection bundles. The isomorphism $\Dcal\otimes\Dcal^{c}$ on $\Rbold(\Gamma,r)$ will necessarily establish a correspondence between these extensions. The meaning of this on the reducible locus is not clear. For instance, we do not know if the trivialization defined by extending the Cappell--Miller torsion to the reducible locus, coincides with the Cappell--Miller torsion itself (whose definition does not require any irreducibility assumption).
    \item In consideration of the explicit expressions in \cite[Section 3 \& 4]{Freixas-Wentworth-2} for the complex metrics on Deligne pairings, one can see that Theorem \ref{theorem:DRR-isom-flat} is compatible with the anomaly formula for the holomorphic Cappell--Miller torsion of Liu--Yu \cite[Theorem 2.5]{Liu-Yu}. Similarly, our theorem is compatible with the asymptotic expansion established by Su \cite[Theorem 5.5]{Su}.
\end{enumerate}
\end{remark}

The following variant of the theorem may suffice for some applications, \emph{e.g.} \textsection\ref{subsubsec:CM-quasi-Fuchsian} below. The proof is similar, and we leave the details to the reader. 

\begin{theorem}\label{thm:variant-DRR-flat}
Except for $g=r=2$, via the isomorphism
\begin{displaymath}
    \begin{split}
          IC_{2}(\Ecal\otimes F)\otimes IC_{2}(\Ecal^{c}\otimes\ov{F})\simeq&\lambda(\Ecal\otimes F)^{-1}\otimes\lambda(\Ecal^{c}\otimes \ov{F})^{-1}
         \\ 
         &\otimes \lambda(\det(\Ecal\otimes F))\otimes\lambda(\det(\Ecal^{c}\otimes\ov{F}))\otimes\lambda(\Ocal_{X})^{r\cdot f-1}\otimes\lambda(\Ocal_{\ov{X}})^{r\cdot f-1}
     \end{split}
\end{displaymath}
deduced from \eqref{def:DelIC2}, the complex metric on the left hand side corresponds to the combination of Cappell--Miller torsions on the right hand side.
\end{theorem}\qed

Notice that the previous theorems entail corresponding versions on $\Rbold^{\ir}(\Gamma,\SL_{r})$, by restricting through the closed immersion $\Rbold^{\ir}(\Gamma,\SL_{r})\hookrightarrow\Rbold^{\ir}(\Gamma,r)$.

\begin{corollary}\label{cor:CM-descends}
In the $\SL_{r}$ case, and except for $g=r=2$, the line bundle $\lambda(\Ecal\otimes F)\otimes\lambda(\Ecal^{c}\otimes \ov{F})$ together with the Cappell--Miller torsion descend to $\Mbold^{\ir}(\Gamma,\SL_{r})$.
\end{corollary}
\begin{proof}
This is a simple consequence of Theorem \ref{thm:variant-DRR-flat}, Proposition \ref{prop:Deligne-iso-descends} and the fact that the complex metric on $IC_{2}(\Ecal)\otimes IC_{2}(\Ecal^{c})$ descends by construction, see \textsection \ref{subsection:complex-metrics}. 
\end{proof}

\subsubsection{Example: explicit formulas in rank 2}
This is a continuation of the example \textsection \ref{subsub:explicit-formulas-rk-2}. We place ourselves in the setting therein, and adopt the same notation. Furthermore, we assume that $g\geq 3$ and instead of working on $\Rbold^{\ir}(\Gamma,2)$, we restrict to $\Rbold^{\ir}(\Gamma,\SL_{2})$. 

By Theorem \ref{thm:variant-DRR-flat} and the $\SL_{2}$ assumption, the natural isomorphism
\begin{displaymath}
    \begin{split}
          IC_{2}(\Ecal\otimes \omega_{X_{0}}^{2})\otimes IC_{2}(\Ecal^{c}\otimes\omega_{\ov{X}_{0}}^{2})\simeq&\lambda(\Ecal\otimes \omega_{X_{0}}^{2})^{-1}\otimes\lambda(\Ecal^{c}\otimes \omega_{\ov{X}_{0}}^{2})^{-1}
         \\ 
         &\otimes \lambda(\omega_{X_{0}}^{4})\otimes\lambda(\omega_{\ov{X}_{0}}^{4})\otimes\lambda(\Ocal_{X_{0}})\otimes\lambda(\Ocal_{\ov{X}_{0}})
     \end{split}
\end{displaymath}
sends the complex metric on the $IC_{2}$ bundles to a combination of Cappell--Miller torsions. The determinant bundles in the second line of the isomorphism are constant $\CBbb$-vector spaces of dimension one, endowed with Quillen metrics. Hence, upon choosing a basis of those, the Cappell--Miller torsion of $\lambda(\Ecal\otimes \omega_{X_{0}}^{2})\otimes\lambda(\Ecal^{c}\otimes \omega_{\ov{X}_{0}}^{2})$ can be identified with the inverse of the complex metric on the $IC_{2}$ bundles, up to a constant. In the vicinity of the unitary locus, the latter is determined by the explicit construction \eqref{eq:explicit-complex-metric}. 

\subsection{Deligne--Riemann--Roch and Bergman projective structures}\label{subsec:DRR-Bergman}
In Section \ref{section:CS-applications} we considered classical families of projective structures on Riemann surfaces. Yet, the relevant case of Bergman projective structures was omitted. We are now in position to study the latter, from the point of view of Chern--Simons transforms and Deligne's isomorphism. This will be important in \textsection \ref{subsubsec:BKN} below, where we address a conjecture of Bertola--Korotkin--Norton.

\subsubsection{Relative Bergman projective structures and Chern--Simons transforms}\label{subsubsec:CS-transform-Bergman}
Let $X$ be a marked compact Riemann surface. The marking induces a canonical basis $\lbrace A_{j}, B_{j}\rbrace_{j}$ of the first homology $H_{1}(X,\ZBbb)$. Following Hawley--Schiffer \cite[p. 202--203]{Hawley-Schiffer}, there is an associated projective structure on $X$, extracted from a meromorphic differential on $X\times X$ characterized by: \emph{i)} being symmetric with a double pole along the diagonal; \emph{ii)} having biresidue 1; and \emph{iii)} having vanishing $A$-periods in any of the two variables. Locally around a point $p$ in the diagonal, the coordinate expression of this meromorphic form is
\begin{displaymath}
    \left(\frac{1}{(x-y)^{2}}+\frac{1}{6}S_{\Be}(p)+\text{higher order terms}\right)dx\otimes dy.
\end{displaymath}
The term $S_{\Be}(p)$ is seen to provide the local expression of a projective connection on $X$. The construction can be carried out in holomorphic families of marked Riemann surfaces, giving rise to holomorphic relative projective structures. See for instance Hejhal \cite[Section 3]{Hejhal:variational}. We refer to the projective structures so obtained as \emph{Bergman projective structures}.

Let $X_{0}$ be a marked compact Riemann surface of genus $g\geq 2$. Introduce the Teichm\"uller space $\Tcal=\Tcal(X_{0})$ and the universal curve $f\colon\Ccal\to\Tcal$. The fibers of $f$ inherit a marking from $X_{0}$, and thus we can form the relative Bergman projective connection. This defines a holomorphic section $\sigma_{\Be}\colon\Tcal\to\Pcal(\Ccal/\Tcal)$, whose Chern--Simons transform we denote by $\nabla^{\Be}$. We proceed to characterize the latter in terms of the determinant of the cohomology and the normalized abelian differentials.

Consider the Hodge bundle $\lambda(\Ocal_{\Ccal})\simeq\det f_{\ast}\omega_{\Ccal/\Tcal}$ on $\Tcal$. It has a holomorphic trivialization, depending on the $A$-cycles only, of the form $\omega_{1}\wedge\ldots\wedge\omega_{g}$, where $\lbrace\omega_{j}\rbrace_{j}$ is the basis of relative holomorphic differentials with normalized $A$-periods:
\begin{displaymath}
    \int_{A_{j}}\omega_{k}=\delta_{jk}.
\end{displaymath}
Let $\tau_{\Be}$ be the trivialization of the Deligne pairing corresponding to $(\omega_{1}\wedge\ldots\wedge\omega_{g})^{\otimes 12}$ via Deligne's isomorphism $\lambda(\Ocal_{\Ccal})^{12}\simeq\langle\omega_{\Ccal/\Tcal},\omega_{\Ccal/\Tcal}\rangle$. 

\begin{theorem}\label{thm:Bergman-connection}
The connection $\nabla^{\Be}$ is the unique connection on the Deligne pairing such that $\nabla^{\Be}\tau_{\Be}=0$. In particular, $\nabla^{\Be}$ is flat.
\end{theorem}
\begin{proof}
During the proof, we denote by $\nabla^{\Be,\lambda}$ the connection on $\lambda(\Ocal_{\Ccal})$ deduced by transporting $\nabla^{\Be}$ through Deligne's isomorphism. It is enough to show that $\nabla^{\Be,\lambda}\omega_{1}\wedge\ldots\wedge\omega_{g}=0$. First of all, recall from \textsection\ref{subsec:Fuchsian-unif} the connection $\nabla^{\Fu}$ on the Deligne pairing, arising from the Fuchsian uniformization. By Theorem \ref{thm:fuchsian-chern}, it coincides with the Chern connection associated to the hyperbolic metric on $\omega_{\Ccal/\Tcal}$. We transport it to the Hodge bundle via Deligne's isomorphism. By the isometry property of Deligne's isomorphism, the resulting connection is the Chern connection of the Quillen metric on $\lambda(\Ocal_{\Ccal})$, denoted by $\nabla^{\Qu}$. The connection form of the latter is given by
\begin{equation}\label{eq:Quillen-connection-Selberg-zeta}
    \frac{\nabla^{\Qu}\omega_{1}\wedge\ldots\wedge\omega_{g}}{\omega_{1}\wedge\ldots\wedge\omega_{g}}=\partial\log\|\omega_{1}\wedge\ldots\wedge\omega_{g}\|^{2}_{\Qu}=\partial\log\frac{\det\Imag\Omega}{Z^{\prime}(1)}.
\end{equation}
In this expression, $\Omega$ is the matrix of $B$-periods $(\int_{B_{j}}\omega_{k})$, and $Z^{\prime}(1)$ is the smooth function on $\Tcal$ which, to a point represented by a Riemann surface $X$, associates the derivative at one of the Selberg zeta function of $X$. See \cite[Proposition 6.4]{Freixas:ARR} for the statement of the Riemann--Roch isometry in terms of the Selberg zeta function. Secondly, a result of Takhtajan--Zograf \cite[Proof of Theorem 2]{Zograf-Takhtajan:Bergman} describes the difference between the Bergman and Fuchsian connections:
\begin{displaymath}
    \sigma_{\Be}-\sigma_{\Fu}=6\pi\partial\log\frac{Z^{\prime}(1)}{\det\Imag\Omega}.
\end{displaymath}
After Theorem \ref{thm:rel-proj-str-rel-conn}, this relationship is equivalent to
\begin{equation}\label{eq:difference-nabla-be-nabla-fu}
    \nabla^{\Be}-\nabla^{\Fu}=12\partial\log\frac{Z^{\prime}(1)}{\det\Imag\Omega}. 
\end{equation}
For the connections induced on $\lambda(\Ocal_{\Ccal})$ via Deligne's isomorphism, the last equality yields
\begin{displaymath}
    \nabla^{\Be,\lambda}=\nabla^{\Qu}+\partial\log\frac{Z^{\prime}(1)}{\det\Imag\Omega}.
\end{displaymath}
Notice here that the factor $12$ in equation \eqref{eq:difference-nabla-be-nabla-fu} compensates with the power $12$ in Deligne's isomorphism. Finally, evaluating the last equality on $\omega_{1}\wedge\ldots\wedge\omega_{g}$ and combining with equation \eqref{eq:Quillen-connection-Selberg-zeta}, we conclude with the desired property: $\nabla^{\Be,\lambda}\omega_{1}\wedge\ldots\wedge\omega_{g}=0$.
\end{proof}

\subsubsection{Relative Bergman projective structure on the Schottky space}\label{subsubsec:relative-bergman-schottky}
Let us maintain the previous setting. The Schottky space $\Sfrak_{g}$ is a quotient of $\Tcal$, in a way that the fibers of the universal curve $\Xcal\to\Sfrak_{g}$ still inherit well-defined $A$-cycles from the marking of $X_{0}$. Therefore, the bases of relative abelian differentials and the Bergman connection descend to  $\Sfrak_{g}$. Below, we will use the same notation as in \textsection\ref{subsubsec:CS-transform-Bergman} for the descended objects. Our aim now is to compare $\nabla^{\Be}$ and $\nabla^{\Sch}$, where we recall that $\nabla^{\Sch}$ is the Chern--Simons transform of the relative Schottky projective structure \textsection\ref{subsubsec:CS-transform-Schottky}. To that end, we will combine our understanding of the Bergman projective structures with the results of \textsection \ref{subsec:Schottky} on the Schottky space.

Before proceeding, we introduce the zeta function of a Schottky group studied by Zograf and McIntyre--Takhtajan. We follow the latter authors \cite[Section 2.1 \& 5.2]{McIntyre-Takhtajan}, to which we refer for details. For a Schottky group $\Gamma\subset\PSL_{2}(\CBbb)$ and $\gamma\in\Gamma\setminus\lbrace 1\rbrace$, let $q_{\gamma}$ be the multiplier of $\gamma$. It is a complex number with $0<|q_{\gamma}|<1$, and it only depends on the conjugacy class of $\gamma$. We form the infinite double product
\begin{displaymath}
    F(\Gamma)=\prod_{[\gamma]}\prod_{k=1}^{\infty}(1-q_{\gamma}^{k}),
\end{displaymath}
where the first index runs over the primitve conjugacy classes in $\Gamma$, distinct from the identity. This product may not absolutely converge, but it does for Schottky groups with exponent of convergence $\delta<1$. These constitute a non-empty open subset of  $\Sfrak_{g}$. The function $F$ is holomorphic on this subset.

\begin{proposition}\label{prop:comp-proj-str-Be-Sch}
On the region of absolute convergence of the function $F$, we have $\nabla^{\Be}-\nabla^{\Sch}=12\partial\log F$. Consequently, $\sigma_{\Be}-\sigma_{\Sch}=6\pi\partial\log F$.
\end{proposition}
\begin{proof}
The proof is a variant of Theorem \ref{thm:Bergman-connection}. As in the proof of \emph{loc. cit.}, we let $\nabla^{\Be,\lambda}$ and $\nabla^{\Sch,\lambda}$ be the connections on $\lambda(\Ocal_{\Xcal})$ deduced from $\nabla^{\Be}$ and $\nabla^{\Sch}$ via Deligne's isomorphism. The first claim reduces to showing $\nabla^{\Be,\lambda}-\nabla^{\Sch,\lambda}=\partial\log F$. On the one hand, by the proof of Proposition \ref{prop:potential-WP-TZ} and Deligne's isomorphism, we know that 
\begin{equation}\label{eq:nabla-Be-Sch-1}
    \nabla^{\Qu}-\nabla^{\Sch,\lambda}=\frac{1}{12\pi}S_{\TZ},
\end{equation}
where $\nabla^{\Qu}$ is the Chern connection of the Quillen metric on $\lambda(\Ocal_{\Xcal})$, for the hyperbolic metric. On the other hand, by the very definition of the Quillen metric and Theorem \ref{thm:Bergman-connection}, we have
\begin{equation}\label{eq:nabla-Be-Sch-2}
    \nabla^{\Qu}-\nabla^{\Be,\lambda}=\partial\log\frac{\|\omega_{1}\wedge\ldots\wedge\omega_{g}\|^{2}_{L^{2}}}{\mathrm{\det}^{\prime}\Delta_{\hyp}}.
\end{equation}
In this equation, $\det^{\prime}\Delta_{\hyp}$ is the smooth function on $\Sfrak_{g}$ which, to a Schottky group $\Gamma$, associates the determinant of the hyperbolic Laplacian on the corresponding Riemann surface.\footnote{The Quillen metric is defined in terms of the determinant of the Laplacian on $(0,1)$-forms, but this in turn equals the determinant of the Laplacian on functions.} The first assertion is then a result of combining \eqref{eq:nabla-Be-Sch-1}--\eqref{eq:nabla-Be-Sch-2} and Zograf's holomorphic factorization formula \cite[Theorem 1]{McIntyre-Takhtajan}:
\begin{displaymath}
    \frac{\mathrm{det}^{\prime}\Delta_{\hyp}}{\|\omega_{1}\wedge\ldots\wedge\omega_{g}\|^{2}_{L^{2}}}=C\exp\left(-\frac{1}{12\pi}S_{\TZ}\right)|F|^{2},
\end{displaymath}
for some constant $C$.

The second claim follows from Theorem \ref{thm:rel-proj-str-rel-conn}.
\end{proof}

\begin{remark}\label{rmk:isom-GM}
\begin{enumerate}
    \item The conclusion of the proposition regarding the comparison of projective structures was stated in \cite[Proposition 5.3]{Bertola}, but the details of the proof were not provided. 
    \item\label{item:isom-GM-2} Proposition \ref{prop:comp-proj-str-Be-Sch} is equivalent to the following explicit description of Deligne's isomorphism for the Hodge bundle. We assume for simplicity that $g\geq 3$. Recall from \textsection\ref{subsubsec:potential-Schottky} the flat trivialization $\tau_{\Sch}$ of $\langle\omega_{\Xcal/\Sfrak_{g}},\omega_{\Xcal/\Sfrak_{g}}\rangle$, for the connection $\nabla^{\Sch}$. Then we have
\begin{displaymath}
    \begin{split}   
        \mathsf{DRR}(\Xcal/\Sfrak_{g},\Ocal_{\Xcal})\colon\lambda(\Ocal_{\Xcal})^{12}&\overset{\sim}{\longrightarrow}\langle\omega_{\Xcal/\Sfrak_{g}},\omega_{\Xcal/\Sfrak_{g}}\rangle\\
        (F\cdot \omega_{1}\wedge\ldots\wedge\omega_{g})^{12} &\longmapsto \tau_{\Sch},
    \end{split}
\end{displaymath}
after possibly scaling $\tau_{\Sch}$ by a constant. In \cite[Theorem 5]{Guillarmou-Moroianu}, Guillarmou--Moroianu built a similar isometric isomorphism by hand, but the relation to Deligne's isomorphism was not addressed. Since both their theory and ours can be compared according to \textsection \ref{subsub:potential-WP-QF}, we conclude that the power $-2$ of the isomorphism of Guillarmou--Moroianu coincides with Deligne's isomorphism, up to a constant.
\end{enumerate}
\end{remark}

\subsection{Deligne--Riemann--Roch on quasi-Fuchsian space}\label{subsec:DRR-quasi-Fuchsian}
We place ourselves in the setting of \textsection \ref{subsec:QF}, and we adopt the same notation. In particular, $X_{0}$ is a fixed compact Riemann surface of genus $g\geq 2$, and $\Qcal=\Tcal(X_{0})\times\Tcal(\ov{X}_{0})$ is the associated quasi-Fuchsian space. We recall the various Bers' curves $\Ccal\to\Tcal$, $f^{\pm}\colon\Xcal^{\pm}\to\Qcal$, and the notation $\Ncal^{\pm}$ and $\Ncal$ for the Deligne pairings in \eqref{eq:Npm}--\eqref{eq:N-prod-del-prod}. Also, we use the terminology ``Fuchsian locus'' for the Teichm\"uller space diagonally embedded in $\Qcal$, denoted by $j\colon\Tcal\hookrightarrow\Qcal$.

Let $k\geq 1$ be an integer. The Deligne--Riemann--Roch isomorphism applied to $\omega_{\Xcal^{\pm}/\Qcal}^{k}$ provides isomorphisms
\begin{equation}\label{eq:Deligne-iso-powers}
    \Dcal^{\pm}_{k}\colon\lambda(\omega_{\Xcal^{\pm}/\Qcal}^{k})^{12}\simeq (\Ncal^{\pm})^{6k^{2}-6k+1}. 
\end{equation}
In a similar vein as in Theorem \ref{theorem:DRR-isom-flat}, we wish to produce natural holomorphic trivializations of both sides of $\Dcal^{+}_{k}\otimes\Dcal^{-}_{k}$, which correspond under this isomorphism. For the product of the Deligne pairings $\Ncal$, we have the complex metrics considered in \textsection \ref{subsub:complex-metric-QF}, generally denoted by $\tau_{\De}$. See in particular Definition \ref{def:complex-metric-KCS}, Theorem \ref{thm:CS-Fuchsian} and Remark \ref{rmk:complex-metric-QF}. Henceforth we tackle $\lambda(\omega_{\Xcal^{+}/\Qcal}^{k})\otimes\lambda(\omega_{\Xcal^{-}/\Qcal}^{k})$.

We will make use of hermitian metrics on Deligne pairings, and $L^{2}$ and Quillen metrics on determinant bundles. All these will be associated to the hyperbolic metric. For simplicity of language, the Chern connection of the Quillen metric will be called the Quillen connection, generally denoted $\nabla^{\Qu}$ or simple variants thereof. We will also need the isometry property of Deligne's isomorphism. 

\subsubsection{Restriction along the Fuchsian locus}
We study the restriction of the isomorphisms \eqref{eq:Deligne-iso-powers} along the Fuchsian locus. We begin by considering the determinant of the cohomology, and we furnish an analogue of the first part of Theorem \ref{thm:CS-Fuchsian}.

\begin{proposition}\label{prop:restriction-Deligne-QF}
For every $k\geq 1$, there exist unique isomorphisms of $\Ccal^{\infty}$ line bundles with connections on $\Tcal$
\begin{displaymath}
    j^{\ast}(\lambda(\omega_{\Xcal^{+}/\Qcal}^{k}),\nabla^{\Qu,+})\simeq(\lambda(\omega_{\Ccal/\Tcal}^{k}),\nabla^{\Qu}),
    \quad j^{\ast}(\lambda(\omega_{\Xcal^{-}/\Qcal}^{k}),\nabla^{\Qu,-})\simeq(\ov{\lambda(\omega_{\Ccal/\Tcal}^{k})},\ov{\nabla}^{\Qu}),
\end{displaymath}
with the property of coinciding with the identity at $X_{0}$, resp. $\ov{X}_{0}$. Furthermore, they are isometric for the $L^{2}$ and Quillen metrics.
\end{proposition}
\begin{proof}
We just give the outline of the argument, since it proceeds along the same lines as the first part of the proof of Theorem \ref{thm:CS-Fuchsian}. We place ourselves in the same situation, and adopt the notation therein. We first treat the case of $\Xcal^{+}\to\Qcal$. The isomorphism of relative curves $\Xcal^{+}_{V}\simeq\widetilde{\Ccal}_{V}$ is isometric for the fiberwise hyperbolic metrics. Therefore, there are induced isometries for the $L^{2}$ and Quillen metrics
\begin{displaymath}
    \lambda(\omega_{\Xcal^{+}/\Qcal}^{k})_{\mid V}\simeq\lambda(\omega_{\widetilde{\Ccal}/\Qcal}^{k})_{\mid V}
    \simeq p_{1}^{\ast}(\lambda(\omega_{\Ccal/\Tcal}^{k})_{\mid U}).
\end{displaymath}
In particular, such isomorphisms preserve the Quillen connections. Pulling back by $j$, we obtain isomorphisms
\begin{equation}\label{eq:iso-lambda-fuchsian-locus}
    (j^{\ast} \lambda(\omega_{\Xcal^{+}/\Qcal}^{k}))_{\mid U}\simeq \lambda(\omega_{\Ccal/\Tcal}^{k})_{\mid U},
\end{equation}
preserving the metrics and the Quillen connection. Accordingly, on the open subsets of a suitable cover $\lbrace U_{i}\rbrace_{i}$ of $\Tcal$, we have isomorphisms $\psi_{i}$ of the form \eqref{eq:iso-lambda-fuchsian-locus}, which differ by constants on overlaps $U_{i}\cap U_{j}$. These constants have modulus one, since the isomorphisms preserve the metrics. Reasoning as in Theorem \ref{thm:CS-Fuchsian}, because $H^{1}(\Tcal,S^{1})=\{1\}$, after possibly scaling the isomorphisms by constants of modulus one, we can suppose they glue together. The resulting global isomorphism still preserves the metrics and the connection, since the scaling involves only constants of modulus one. Further scaling by a constant of modulus one, we can suppose it to be the identity at $X_{0}$. This is the desired isomorphism. For $\Xcal^{-}\to\Qcal$ we proceed analogously, but we first find an isometry 
\begin{displaymath}
    j^{\ast}(\lambda(\omega_{\Xcal^{-}/\Qcal}^{k}),\nabla^{\Qu,-})\simeq(\lambda(\omega_{\ov{\Ccal}/\ov{\Tcal}}^{k}),\widetilde{\nabla}^{\Qu}),
\end{displaymath}
where $\widetilde{\nabla}^{\Qu}$ is the Quillen connection on $\lambda(\omega_{\ov{\Ccal}/\ov{\Tcal}}^{k})$. Proceeding as in the proof of Theorem \ref{thm:CS-Fuchsian}, the right hand side of the latter is naturally identified with $(\ov{\lambda(\omega_{\Ccal/\Tcal}^{k})},\ov{\nabla}^{\Qu})$, compatibly with the $L^{2}$ and Quillen metrics. The details are left as an exercice to the reader.
\end{proof}

\begin{corollary}\label{cor:restriction-Deligne-is-Deligne}
Via the isomorphisms of Theorem \ref{thm:CS-Fuchsian} and Proposition \ref{prop:restriction-Deligne-QF}, the restriction of $\Dcal^{+}_{k}$ along the Fuchsian locus it identified with the Deligne isomorphism 
\begin{equation}\label{eq:deligne-omega-C-T}
        \lambda(\omega_{\Ccal/\Tcal}^{k})^{12}\simeq\langle\omega_{\Ccal/\Tcal},\omega_{\Ccal/\Tcal}\rangle^{6k^{2}-6k+1}.
\end{equation}
Likewise, the restriction of $\Dcal^{-}_{k}$ along the Fuchsian locus is identified with the conjugate of \eqref{eq:deligne-omega-C-T}.
\end{corollary}
\begin{proof}
We just treat the case of $\Dcal^{+}_{k}$, because $\Dcal^{-}_{k}$ is dealt with similarly. By the isomorphisms of Theorem \ref{thm:CS-Fuchsian} and Proposition \ref{prop:restriction-Deligne-QF}, the pullback of $\Dcal^{+}_{k}$ by $j$ yields an isomorphism of $\Ccal^{\infty}$ line bundles of the form \eqref{eq:deligne-omega-C-T}. Via this isomorphism, the Quillen connection is sent to the Chern connection of the hyperbolic intersection metric on the Deligne pairing. Deligne's isomorphism has this property too, because it is an isometry of hermitian holomorphic line bundles, up to a constant. We deduce that both isomorphisms coincide up to a constant. The normalization condition at $X_{0}$ ensures they actually agree everywhere.
\end{proof}
\begin{remark}
The proof of the corollary is based on the fact that two $\Ccal^{\infty}$ isomorphisms of line bundles with connections $(\Lcal_{1},\nabla_{1})\rightrightarrows (\Lcal_{2},\nabla_{2})$ differ by a constant. This is no longer true for isometries instead of horizontal isomorphisms. Consequently, since we only know that the $j^{\ast}\Dcal^{\pm}_{k}$ are $\Ccal^{\infty}$ isomorphisms, the proof can not resort to the isometry statements in Theorem \ref{thm:CS-Fuchsian} and Proposition \ref{prop:restriction-Deligne-QF}.
\end{remark}

\subsubsection{Trivializations of $\lambda(\omega_{\Xcal^{+}/\Qcal})\otimes\lambda(\omega_{\Xcal^{-}/\Qcal})$}
The trivialization we ultimately seek is a complex metric version of the Quillen metric. Our approach is based on the existence and properties of the complex metrics $\tau_{\QF}$ and the previous developments on Deligne's isomorphisms. We obtain a new conceptual proof of Kim's theorem \cite[Theorem 2.8]{Kim} on the existence of a holomorphic extension of the determinant of the hyperbolic Laplacian, from the Teichm\"uller space to the quasi-Fuchsian space. 

\begin{proposition}\label{prop:complex-L2-metric}
There exists a unique isomorphism of holomorphic line bundles
\begin{displaymath}
   \tau_{\Ltwo}\colon \lambda(\omega_{\Xcal^{+}/\Qcal})\otimes\lambda(\omega_{\Xcal^{-}/\Qcal})\overset{\sim}{\longrightarrow}\Ocal_{\Qcal},
\end{displaymath}
which coincides with the $L^{2}$-metric along the Fuchsian locus.
\end{proposition}
\begin{proof}
First of all, we notice that $R^{1}f_{\ast}^{\pm}\omega_{\Xcal^{\pm}/\Qcal}$ is canonically isomorphic to the trivial line bundle. Via this trivialization, the $L^{2}$-metric on the Fuchsian locus is a constant multiple of the trivial metric. The proportionality factor is given by the hyperbolic volume of the fibers. Therefore, we reduce to constructing an extension of the $L^{2}$-metric on $\det f_{\ast}^{+}\omega_{\Xcal^{+}/\Qcal}\otimes \det f_{\ast}^{-}\omega_{\Xcal^{-}/\Qcal}$.

Let $(X,Y)$ be a couple of Riemann surfaces, corresponding to a point of $\Qcal$. Let $\omega_{1}^{+},\ldots, \omega_{g}^{+}$ be a basis of $H^{0}(X,\omega_{X})$. Via the Hodge filtration $H^{0}(X,\omega_{X})\subseteq H^{1}(X,\CBbb)$, we obtain cohomology classes in the latter. Because the local system $R^{1}f_{\ast}^{+}\CBbb$ is trivial with fiber $H^{1}(X_{0},\CBbb)$, we can view the $\omega_{j}^{+}$ as in $H^{1}(X_{0},\CBbb)$, and we will abusively use the same notation for these classes. We argue similarly for a basis $\omega_{1}^{-},\ldots,\omega_{g}^{-}$ of $H^{0}(Y,\omega_{Y})$, from which we obtain classes in $H^{1}(\ov{X}_{0},\CBbb)=H^{1}(X_{0},\CBbb)$. If $\cup$ is the cohomological cup-product on $H^{1}$, the assignment
\begin{equation}\label{eq:complex-metric-Hodge}
    \omega_{1}^{+}\wedge\ldots\wedge\omega_{g}^{+}\otimes\omega_{1}^{-}\wedge\ldots\wedge\omega_{g}^{-}\mapsto\det\left(\frac{i}{2}\int_{X_{0}}\omega_{j}^{+}\cup\omega_{k}^{-}\right)_{jk}\in\CBbb
\end{equation}
is well-defined. Notice that the integral depends on the orientation of $X_{0}$ determined by its complex structure. Let us consider the effect of a small variation of $(X,Y)$ in $\Qcal$ and holomorphic deformations of the bases $\lbrace\omega_{j}^{+}\rbrace_{j}$ and $\lbrace\omega_{j}^{-}\rbrace_{j}$. The latter exist since the $f^{\pm}_{\ast}\omega_{\Xcal^{\pm}/\Qcal}$ are locally free. The corresponding cohomology classes in $H^{1}(X_{0},\CBbb)$ vary holomorphically, since the Hodge filtrations $f^{\pm}_{\ast}\omega_{\Xcal^{\pm}/\Qcal}\subset (R^{1}f^{\pm}_{\ast}\CBbb)\otimes\Ocal_{\Qcal}$ are holomorphic on $\Qcal$, and the local systems $R^{1}f^{\pm}_{\ast}\CBbb$ are trivial with fiber $H^{1}(X_{0},\CBbb)=H^{1}(\ov{X}_{0},\CBbb)$. We infer that \eqref{eq:complex-metric-Hodge} varies holomorphically as well. Equivalently, we have constructed a morphism of holomorphic line bundles
\begin{displaymath}
    \det f_{\ast}^{+}\omega_{\Xcal^{+}/\Qcal}\otimes \det f_{\ast}^{-}\omega_{\Xcal^{-}/\Qcal}\longrightarrow\Ocal_{\Qcal}.
\end{displaymath}
By construction, along the Fuchsian locus it reproduces the $L^{2}$-metric. Since the Fuchsian locus is totally real in $\Qcal$, such an extension is unique. 

It remains to prove that $\tau_{\Ltwo}$ is an isomorphism. Introduce a marking of $X_{0}$, which induces a canonical homology basis $\lbrace A_{j},B_{j}\rbrace_{j}$ of $X_{0}$, and $\lbrace A_{j},-B_{j}\rbrace_{j}$ of $\ov{X}_{0}$. Let $(X,Y)$ be a point of $\Qcal$. The surfaces $X,Y$ inherit canonical homology bases from $X_{0}$ and $\ov{X}_{0}$. We take $\lbrace\omega_{j}^{+}\rbrace_{j}$ and $\lbrace\omega_{j}^{-}\rbrace_{j}$ the bases of normalized abelian differentials for $X$ and $Y$, respectively. We denote by $\Omega^{+}$ and $\Omega^{-}$ the corresponding matrices of $B$-periods. We compute
\begin{equation}\label{eq:Kim-formula-L2}
    \left(\frac{i}{2}\int_{X_{0}}\omega_{j}^{+}\cup\omega_{k}^{-}\right)_{jk}=\frac{1}{2i}(\Omega^{+}+\Omega^{-}).
\end{equation}
Since the $\Omega^{\pm}$ are symmetric with positive definite imaginary part, so is $\Omega^{+}+\Omega^{-}$. Thus, it is invertible. The proof is complete. 

\end{proof}

\begin{theorem}\label{cor:Kim}
\begin{enumerate}
    \item There exists a unique invertible holomorphic function $\widetilde{\det}\ \Delta_{\hyp}\colon\Qcal\to\CBbb$ whose restriction to the Fuchsian locus agrees with the determinant of the hyperbolic Laplacian, with the zero eigenvalue removed.
    \item The expression $\tau_{\Qu}:=(\widetilde{\det}\ \Delta_{\hyp})^{-1}\tau_{\Ltwo}$ induces a holomorphic trivialization of $\lambda(\omega_{\Xcal^{+}/\Qcal})\otimes\lambda(\omega_{\Xcal^{-}/\Qcal})$, such that $\tau_{\Qu}^{12}$ corresponds to a complex metric $\tau_{\De}$ on $\Ncal$, via Deligne's isomorphism. 
\end{enumerate}
In particular, $\tau_{\Qu}$ coincides with the Quillen metric along the Fuchsian locus.
\end{theorem}
\begin{proof}
We will establish all the statements simultaneously. Let $\tau_{\QF}$ be a complex metric on $\Ncal$. Recall that $j^{\ast}\tau_{\QF}$ is identified with the intersection metric under the isomorphisms of Theorem \ref{thm:CS-Fuchsian}, up to a constant. We transport $\tau_{\QF}$ to a holomorphic trivialization of $\lambda(\omega_{\Xcal^{+}/\Qcal})^{12}\otimes\lambda(\omega_{\Xcal^{-}/\Qcal})^{12}$, via Deligne's isomorphism. Since $\Qcal$ is contractible and Stein, we can write this trivialization in the form $\tau_{\Qu}^{12}$, for a holomorphic trivialization $\tau_{\Qu}$ of $\lambda(\omega_{\Xcal^{+}/\Qcal})\otimes\lambda(\omega_{\Xcal^{-}/\Qcal})$. Notice that $\tau_{\Qu}$ is well defined up to a 12-th root of unity, which is irrelevant for our purposes. In the sequel, we will see $\tau_{\Qu}$ as an isomorphism $\lambda(\omega_{\Xcal^{+}/\Qcal})\otimes\lambda(\omega_{\Xcal^{-}/\Qcal})\to\Ocal_{\Qcal}$.

Define a nowhere vanishing holomorphic function $h\colon\Qcal\to\CBbb$ by the relationship $\tau_{\Qu}=h^{-1}\cdot \tau_{\Ltwo}$. We claim that $h$ is a holomorphic extension of the determinant of the hyperbolic Laplacian on the Fuchsian locus, up to a constant. Notice that the uniqueness property of the extension is guaranteed by the fact that $j$ is a totally real embedding.

Let us examine $j^{\ast}\tau_{\Qu}$. We identify it with a trivialization of $\lambda(\omega_{\Ccal/\Tcal})\otimes\ov{\lambda(\omega_{\Ccal/\Tcal})}$ via the isomorphism of Proposition \ref{prop:restriction-Deligne-QF}. Let us see that, up to a constant, $j^{\ast}\tau_{\Qu}$ is identified with the trivialization furnished by the Quillen metric. Indeed, $\tau_{\Qu}^{12}$ corresponds to $\tau_{\QF}$ via Deligne's isomorphism, and by Theorem \ref{thm:CS-Fuchsian} the section $j^{\ast}\tau_{\QF}$ is identified with the trivialization associated to the hyperbolic intersection metric on $\langle\omega_{\Ccal/\Tcal},\omega_{\Ccal/\Tcal}\rangle$, up to a constant. Via Deligne's isomorphism again, this corresponds to the section provided by the Quillen metric on $\lambda(\omega_{\Ccal/\Tcal})^{12}$, up to a constant. By Corollary \ref{cor:restriction-Deligne-is-Deligne}, this is exactly the  section that $j^{\ast}\tau_{\Qu}^{12}$ identifies with. This confirms our expectation. By scaling by a constant, we can thus suppose that $j^{\ast}\tau_{\Qu}$ corresponds to the Quillen metric.

For the trivialization $j^{\ast}\tau_{\Ltwo}$, we readily derive from Proposition \ref{prop:restriction-Deligne-QF} and Proposition \ref{prop:complex-L2-metric} that it corresponds to the $L^{2}$ metric on $\lambda(\omega_{\Ccal/\Tcal})$. 

To conclude the proof, we now write $j^{\ast}\tau_{\Qu}=j^{\ast}h^{-1}\cdot j^{\ast}\tau_{L^{2}}$. We saw that $j^{\ast}\tau_{\Qu}$ and $j^{\ast}\tau_{L^{2}}$ correspond to the Quillen and $L^{2}$ metrics on $\lambda(\omega_{\Ccal/\Tcal})$. By definition of the Quillen metric, we infer that $j^{\ast}h$ equals the determinant of the hyperbolic Laplacian, as was to be shown.
\end{proof}


\subsubsection{Trivializations of $\lambda(\omega_{\Xcal^{+}/\Qcal}^{k})\otimes\lambda(\omega_{\Xcal^{-}/\Qcal}^{k})$, for $k\geq 2$} We begin with a brief reformulation of results of McIntyre--Teo \cite[Section 5]{McIntyre-Teo}, providing a trivialization of $\lambda(\omega_{\Xcal^{+}/\Qcal}^{k})\otimes\lambda(\omega_{\Xcal^{-}/\Qcal}^{k})$ which extends the $L^{2}$-metric. We refer to \emph{loc. cit.} for the details of the construction.

Let $(X,Y)$ be a couple of Riemann surfaces representing a point in $\Qcal$. Let $\theta_{1}^{+},\ldots,\theta_{r}^{+}$ be a basis of $H^{0}(X,\omega_{X}^{k})$. Then, after Bers \cite{Bers:Acta}, there is an associated basis $\theta_{1}^{-},\ldots,\theta_{r}^{-}$ of $H^{0}(Y,\omega_{Y}^{k})$. This basis is obtained from $\lbrace\theta_{j}^{+}\rbrace$ by applying an integral operator denoted by $K_{-}$ in \cite[Section 5]{McIntyre-Teo}. It is an essential feature of $K_{-}$ that, for a small variation of $(X,Y)$ in $\Qcal$, and holomorphically varying  $\lbrace\theta_{j}^{+}\rbrace_{j}$, the corresponding $\lbrace\theta_{j}^{-}\rbrace_{j}$ also vary holomorphically. Furthermore, for couples $(X,\ov{X})$, the $\lbrace\theta_{j}^{-}\rbrace_{j}$ are conjugate to the $\lbrace\theta_{j}^{+}\rbrace_{j}$. In general, the construction is such that the expression $\theta_{1}^{+}\wedge\ldots\wedge\theta_{r}^{+}\otimes \theta_{1}^{-}\wedge\ldots\wedge\theta_{r}^{-}$ does not depend on the choice of the $\lbrace\theta_{j}^{+}\rbrace_{j}$. It defines a holomorphic trivialization of $\lambda(\omega_{\Xcal^{+}/\Qcal}^{k})\otimes\lambda(\omega_{\Xcal^{-}/\Qcal}^{k})$ on $\Qcal$. Along the Fuchsian locus, the $L^{2}$-norm of this trivialization is 1. We conclude that the $L^{2}$-metric can be extended to a holomorphic trivialization:
\begin{proposition}
There exists an isomorphism of holomorphic line bundles
\begin{displaymath}
    \tau_{\Ltwo, k}\colon\lambda(\omega_{\Xcal^{+}/\Qcal}^{k})\otimes\lambda(\omega_{\Xcal^{-}/\Qcal}^{k})\overset{\sim}{\longrightarrow}\Ocal_{\Qcal},
\end{displaymath}
which coincides with the $L^{2}$-metric along the Fuchsian locus.
\end{proposition}\qed

Following \cite[Section 4]{McIntyre-Teo}, we introduce a quasi-Fuchsian version of the Selberg zeta function at $s=k$. For a quasi-Fuchsian group $\Gamma$, set
\begin{displaymath}
    F(k,\Gamma)=\prod_{[\gamma]}\prod_{n=0}^{\infty}(1-q_{\gamma}^{n+k}).
\end{displaymath}
Here, the first product runs over the conjugacy classes of primitive elements in $\Gamma\setminus\lbrace 1\rbrace$, and $q_{\gamma}$ is the notation for the multiplier of $\gamma$. Since $k\geq 2$, this product converges absolutely and defines a holomorphic function on $\Qcal$, denoted by $F(k)$. If $\Gamma$ is a Fuchsian group, then $F(k,\Gamma)$ coincides with $Z(k,\Gamma)$, the value at $s=k$ of the Selberg zeta function of $\Gamma$. We refer to \emph{loc. cit.} for details. We recall that in the Fuchsian case, $Z(k,\Gamma)$ coincides with the determinant of the hyperbolic Laplacian acting on forms of order $k$, up to an explicit constant depending only on $k$ and $g$. See \cite[Proposition 6.2]{Freixas:AHS}.

\begin{proposition}\label{prop:McIntyre-Teo}
The expression $\tau_{\Qu,k}:=F(k)^{-1}\tau_{\Ltwo,k}$ induces a holomorphic trivialization of $\lambda(\omega_{\Xcal^{+}/\Qcal}^{k})\otimes\lambda(\omega_{\Xcal^{-}/\Qcal}^{k})$, which coincides with the Quillen metric along the Fuchsian locus and such that $\tau_{\Qu,k}^{12}$ corresponds to the complex metric $\tau^{6k^{2}-6k+1}_{\De}$ on $\Ncal^{6k^{2}-6k+1}$ via Deligne's isomorphism, up to scaling by a constant.
\end{proposition}
\begin{proof}
The proof is an easy variant of the argument of Theorem \ref{cor:Kim}, and is left to the reader.
\end{proof}

Recall now the potential of the Weil--Petersson form on $\Qcal$, constructed in \textsection \ref{subsub:potential-WP-QF}. If we introduce the hyperbolic intersection metric on the product of Deligne pairings $\Ncal$, the potential is defined as $\log\|\tau_{\De}\|^{-2}$. Let $S_{\scriptscriptstyle{\mathsf{TT}}}\colon\Qcal\to\RBbb$ be the Liouville action of Takhtajan--Teo \cite{Takhtajan-Teo}. The relationship between both potentials is described by the following corollary.
\begin{corollary}\label{eq:coincidence-Liouville}
The function $\log\|\tau_{\De}\|$ coincides with $\frac{1}{2\pi}S_{\scriptscriptstyle{\mathsf{TT}}}$, up to the addition of a constant.
\end{corollary}
\begin{proof}
Introduce the product of Quillen metrics on $\lambda(\omega_{\Xcal^{+}/\Qcal}^{2})\otimes\lambda(\omega_{\Xcal^{-}/\Qcal}^{2})$. Up to a  constant depending only on the genus, the norms of $\tau_{\Qu,2}^{12}$ and $\tau^{13}_{\De}$ coincide, because these sections correspond via Deligne's isomorphism, by Proposition \ref{prop:McIntyre-Teo}. Comparing the content of this equality to the main theorem of McIntyre--Teo \cite[Section 6, Theorem]{McIntyre-Teo}, it is readily seen that $\log\|\tau_{\De}\|$ coincides with $\frac{1}{2\pi}S_{\scriptscriptstyle{\mathsf{TT}}}$, up to the addition of a constant depending only on the genus.
\end{proof}

\begin{remark}
\begin{enumerate}
    \item The proof of the corollary works more generally with any $\tau_{\Qu,k}$, $k\geq 2$.
    \item In \textsection \ref{subsection:convex-hyp} we sketched an alternative proof of the corollary. Taking this for granted, we can reverse the argument provided above and derive the holomorphic factorization formula of McIntyre--Teo from Proposition \ref{prop:McIntyre-Teo}.
\end{enumerate}
\end{remark}

\subsubsection{On a conjecture of Bertola--Korotkin--Norton}\label{subsubsec:BKN}
As an application of our results on Chern--Simons transforms and complex metrics, we will now settle a conjecture of Bertola--Korotkin--Norton \cite[Conjectutre 1.1]{Bertola} on the comparison between Bergman and quasi-Fuchsian projective structures. 

The setting is as follows. We consider the quasi-Fuchsian section $\sigma_{\QF}^{+}\colon\Qcal\to\Pcal(\Xcal^{+}/\Qcal)$ defined in \eqref{eq:def-QF-sections}, and restrict it to the Bers slice $\Tcal(X_{0})\times\lbrace\ov{X}_{0}\rbrace$. This gives a holomorphic section $\Tcal\to\Pcal(\Ccal/\Tcal)$, where we recall that $\Ccal\to\Tcal$ is the universal Teichm\"uller curve. This restricted section is denoted by $\sigma_{\Bers}$. After fixing a marking for $X_{0}$, we also have $\sigma_{\Be}$, the section defined by the Bergman projective structures and studied in \textsection \ref{subsec:DRR-Bergman}. We wish to compare $\sigma_{\Bers}$ and $\sigma_{\Be}$.

Recall the invertible holomorphic function $\widetilde{\det}\ \Delta_{\hyp}$ of Theorem \ref{cor:Kim}, and restrict it to the Bers slice. We still use the same notation for the restriction. Let $\omega_{1}^{\pm},\ldots,\omega_{g}^{\pm}$ be the holomorphic trivializations of $f^{\pm}_{\ast}\omega_{\Xcal^{\pm}/\Qcal}$ of normalized abelian differentials, with respect to the given marking. Restricting to the Bers slice, the $\omega_{j}^{+}$ induce a basis of normalized abelian differentials of $f_{\ast}\omega_{\Ccal/\Tcal}$. Similarly, the $\omega_{j}^{-}$ restrict to a basis of the $\CBbb$-vector space $H^{0}(\ov{X}_{0},\Omega^{1}_{\ov{X}_{0}})$. Again, we maintain the notation for these restrictions. Evaluating the complex $L^{2}$-metric provided by Proposition \ref{prop:complex-L2-metric} (see in particular \eqref{eq:complex-metric-Hodge}--\eqref{eq:Kim-formula-L2}), we have
\begin{displaymath}
    \tau_{\Ltwo}(\omega_{1}^{+}\wedge\ldots\wedge\omega_{g}^{+}\otimes\omega_{1}^{-}\wedge\ldots\wedge\omega_{g}^{-})=\det\left(\frac{i}{2}\int_{X_{0}}\omega_{j}^{+}\cup\omega_{k}^{-}\right)_{jk}=\det\left(\frac{1}{2i}(\Omega-\ov{\Omega}_{0})\right),
\end{displaymath}
where $\Omega$ is the matrix of $B$-periods of the fibers of $\Ccal\to\Tcal$, and $\Omega_{0}$ is the matrix of $B$-periods of $X_{0}$. Notice here that $-\ov{\Omega}_{0}$ is the matrix of $B$-periods of $\ov{X}_{0}$. The function $\det(\Omega-\ov{\Omega}_{0})$ is nowhere vanishing, as we saw that $\tau_{\Ltwo}$ is an isomorphism. Thus, 
\begin{displaymath}
    \frac{\widetilde{\det}\ \Delta_{\hyp}}{\det(\Omega-\ov{\Omega}_{0})}
\end{displaymath}
is a well-defined, nowhere vanishing holomorphic function on $\Tcal$.

The following statement confirms the conjecture of Bertola--Korotkin--Norton \cite[Conjecture 1.1]{Bertola}.
\begin{theorem}\label{thm:Bertola-conjecture}
The difference of the Bergman and Bers projective connections on $\Tcal$ is given by
\begin{displaymath}
    \sigma_{\Be}-\sigma_{\Bers}=6\pi\partial\log\left(\frac{\widetilde{\det}\ \Delta_{\hyp}}{\det(\Omega-\ov{\Omega}_{0})} \right).
\end{displaymath}
\end{theorem}

\begin{proof}
By Theorem \ref{thm:rel-proj-str-rel-conn}, we have to establish an analogous formula for the Chern--Simons transforms $\nabla^{\Be}$ and $\nabla^{\Bers}$ on $\langle\omega_{\Ccal/\Tcal},\omega_{\Ccal/\Tcal}\rangle$. We transport these connections to the Hodge bundle $\lambda(\Ocal_{\Ccal})$ via Deligne's isomorphism. As usual in this subsection, the resulting connections are denoted by  $\nabla^{\Be,\lambda}$ and $\nabla^{\Bers,\lambda}$. By Theorem \ref{thm:Bergman-connection}, we already know that $\nabla^{\Be,\lambda}\omega_{1}^{+}\wedge\ldots\wedge\omega_{g}^{+}=0$. To determine the connection $\nabla^{\Bers,\lambda}$, we first argue with $\nabla^{\QF}$ on the whole quasi-Fuchsian space $\Qcal$. See the discussion around \eqref{eq:N-prod-del-prod} for the definition of $\nabla^{\QF}$. The connection $\nabla^{\QF}$ is such that the complex metric $\tau_{\QF}$ is flat for $\nabla^{\QF}$. Therefore, by Theorem \ref{cor:Kim}, $\nabla^{\QF,\lambda}$ is the connection on $\lambda(\omega_{\Xcal^{+}/\Qcal})\otimes\lambda(\omega_{\Xcal^{-}/\Qcal})$ for which $\tau_{\Qu}$ is flat. Now, the trivialization of this product of Hodge bundles corresponding to $\tau_{\Qu}$ is
\begin{equation}\label{eq:tau-Q-1}
    \tau_{\Qu}^{-1}(1)=\frac{\widetilde{\det}\ \Delta_{\hyp}}{\det\left(\frac{1}{2i}(\Omega^{+}+\Omega^{-})\right)}\omega_{1}^{+}\wedge\ldots\wedge\omega_{g}^{+}\otimes\omega_{1}^{-}\wedge\ldots\wedge\omega_{g}^{-},
\end{equation}
with $\Omega^{\pm}$ denoting the matrices of $B$-periods of the fibers of $\Xcal^{\pm}\to\Qcal$. Hence, this trivialization is flat for $\nabla^{\QF,\lambda}$. But the connection $\nabla^{\Bers}$ is basically the restriction of $\nabla^{\QF}$ to the Bers slice; they just differ by the trivial connection on the constant line bundle with fiber $\langle\omega_{\ov{X}_{0}},\omega_{\ov{X}_{0}}\rangle$. Then, restricting \eqref{eq:tau-Q-1} to the Bers slice, we deduce that the trivialization of $\lambda(\Ocal_{\Ccal})$ given by
\begin{displaymath}
    \frac{\widetilde{\det}\ \Delta_{\hyp}}{\det(\Omega-\ov{\Omega}_{0})}\omega_{1}^{+}\wedge\ldots\wedge\omega_{g}^{+}
\end{displaymath}
is flat for $\nabla^{\Bers,\lambda}$ on $\Tcal$. Equivalently,
\begin{displaymath}
    \frac{\nabla^{\Bers,\lambda}\omega_{1}^{+}\wedge\ldots\wedge\omega_{g}^{+}}{\omega_{1}^{+}\wedge\ldots\wedge\omega_{g}^{+}}
    =-\partial\log\left(\frac{\widetilde{\det}\ \Delta_{\hyp}}{\det(\Omega-\ov{\Omega}_{0})} \right).
\end{displaymath}
All in all, we conclude
\begin{displaymath}
    \nabla^{\Be,\lambda}-\nabla^{\Bers,\lambda}=\partial\log\left(\frac{\widetilde{\det}\ \Delta_{\hyp}}{\det(\Omega-\ov{\Omega}_{0})} \right).
\end{displaymath}
We finally transport this connection back to the Deligne pairing, taking care of the $12$ power, and achieving our goal. 
\end{proof}

\subsubsection{Cappell--Miller torsion of quasi-Fuchsian representations}\label{subsubsec:CM-quasi-Fuchsian}
The Cappell--Miller torsion is difficult to determine, since it is a vector of a generally non-trivial complex line, rather than a complex number. We now propose a reinterpretation in the case of flat vector bundles arising from quasi-Fuchsian representations, which to our knowledge is the first non-trivial example of computation of the Cappell--Miller torsion. We maintain the setting and notation of \textsection \ref{subsec:QF}. We further suppose that the reference Riemann surface $X_{0}$ has genus $g\geq 3$. We endow it with the hyperbolic metric of constant curvature $-1$.

We begin with the commutative diagram
\begin{displaymath}
    \xymatrix{
            &\Mbold^{\ir}(\Gamma,\SL_{2})\ar[d]\\
        \Qcal\ar[r]_-{\varphi}\ar[ru]^{\widetilde{\varphi}}       &\Mbold^{\ir}(\Gamma,\PSL_{2})_{\ell},
    }
\end{displaymath}
where $\Gamma=\pi_{1}(X_{0},p)$ and $\varphi$ is the composition of the relative holonomy map and the retraction. 
The lift $\widetilde{\varphi}$ exists for topological reasons: $\Qcal$ is simply connected and, since $g\geq 3$, the space $\Mbold^{\ir}(\Gamma,\SL_{2})$ is necessarily the universal cover of $\Mbold^{\ir}(\Gamma,\PSL_{2})_{\ell}$. 
Alternatively, the Riemann divisor provides a global choice of theta characteristic depending on the choice of homology basis. 
By Corollary \ref{cor:CM-descends}, the line bundle $\lambda(\Ecal)\otimes\lambda(\Ecal^{c})$ and the Cappell--Miller torsion on $\Rbold^{\ir}(\Gamma,\SL_{2})$ descend to $\Mbold^{\ir}(\Gamma,\SL_{2})$. We use the same notation for the descended objects, and we pull them back by $\widetilde{\varphi}$. By Proposition \ref{prop:Deligne-iso-descends}, we have a natural isomorphism
\begin{displaymath}
    \widetilde{\varphi}^{\ast}(\lambda(\Ecal)\otimes\lambda(\Ecal^{c}))\simeq \widetilde{\varphi}^{\ast}(IC_{2}(\Ecal)\otimes IC_{2}(\Ecal^{c}))^{-1}\otimes(\lambda(\Ocal_{X_{0}})\otimes\lambda(\Ocal_{\ov{X}_{0}}))^{2}
\end{displaymath}
which, after Theorem \ref{thm:variant-DRR-flat}, is compatible with the Cappell--Miller torsions and the complex metric on the $IC_{2}$ bundles. We will now take the fourth power of this isomorphism. Before, recall from Proposition \ref{prop:descend-IC2-PSLr} that $\Lcal_{\CS}(X_{0})^{4}$ descends to $\Mbold^{\ir}(\Gamma,\PSL_{2})_{\ell}$, with the same notation for the descended object. Similarly for $\ov{X}_{0}$. The complex metric descends as well, by Proposition \ref{prop:complements-CS-PSLr}. We find
\begin{equation}\label{eq:iso-lambda-lambda-LCS-LCS}
    \widetilde{\varphi}^{\ast}(\lambda(\Ecal)\otimes\lambda(\Ecal^{c}))^{4}\simeq \varphi^{\ast}(\Lcal_{\CS}(X_{0})^{4}\otimes\Lcal_{\CS}(\ov{X}_{0})^{4})\otimes(\lambda(\Ocal_{X_{0}})\otimes\lambda(\Ocal_{\ov{X}_{0}}))^{8},
\end{equation}
compatibly with the Cappell--Miller torsions and the complex metrics. Incidentally, we realize that the left hand side of the isomorphism, together with the Cappell--Miller torsion, does not depend on the lift $\widetilde{\varphi}$. As in Remark \ref{rmk:crystalline-QF}, by the crystalline property of Chern--Simons line bundles, we can rewrite \eqref{eq:iso-lambda-lambda-LCS-LCS} as
\begin{equation}\label{eq:phi-ast-product-det-E-Ec}
    \widetilde{\varphi}^{\ast}(\lambda(\Ecal)\otimes\lambda(\Ecal^{c}))^{4}\simeq \Ncal\otimes(\lambda(\Ocal_{X_{0}})\otimes\lambda(\Ocal_{\ov{X}_{0}}))^{8}.
\end{equation}
The complex metric on $\Lcal_{\CS}(X_{0})^{4}$ induces a complex metric on $\Ncal$ which, by Theorem \ref{thm:CS-Fuchsian}, coincides with the holomorphic extension of the metric on the Deligne pairing, up to a universal constant. This constant depends only on the normalization of the hyperbolic metric. Next, we take the $6k^{2}-6k+1$ power of \eqref{eq:phi-ast-product-det-E-Ec} and then apply Theorem \ref{cor:Kim} if $k=1$, or Proposition \ref{prop:McIntyre-Teo} if $k\geq 2$. We find a natural isomorphism

\begin{equation}\label{eq:funny-iso-QF}
    \widetilde{\varphi}^{\ast}(\lambda(\Ecal)\otimes\lambda(\Ecal^{c}))^{24k^{2}-24k+4}\simeq (\lambda(\omega_{\Xcal^{+}/\Qcal}^{k})\otimes \lambda(\omega_{\Xcal^{-}/\Qcal}^{k}))^{12}\otimes(\lambda(\Ocal_{X_{0}})\otimes\lambda(\Ocal_{\ov{X}_{0}}))^{48k^{2}-48k+8},
\end{equation}
relating the Cappell--Miller torsion to $\tau_{\Qu,k}$ and the Quillen metric on $\lambda(\Ocal_{X_{0}})$, with the convention $\tau_{\Qu,1}=\tau_{\Qu}$, and up to a universal constant. Since the Quillen metric on $\lambda(\Ocal_{X_{0}})$ is constant on $\Qcal$, we conclude that the content of the Cappell--Miller torsion of quasi-Fuchsian representations is essentially equivalent to the complex metrics $\tau_{\Qu,k}$.

To conclude, notice that an isomorphism such as \eqref{eq:funny-iso-QF}, relating the Cappell--Miller torsion and the complex Quillen metric, exists for trivial reasons: all the involved line bundles are trivial. The point of the discussion above is that \eqref{eq:funny-iso-QF} can be obtained as a succession of natural explicit isomorphisms, and in particular the crystalline property of complex Chern--Simons line bundles.

\section*{Appendix}

\subsection*{Proof of Proposition \ref{prop:flatness}}
For the first claim, the compatibility of $\Rbold_{\dR}(X/S, \sigma, r)\to S$ with base change is clear, since $\Rbold_{\dR}(X/S, \sigma, r)$ represents a moduli functor. For $\Mbold_{\dR}(X/S,r)\to S$, the base change property holds since it is a universal categorical quotient. Indeed, if $T\to S$ is any morphism, let $\Mbold_{\dR}(X/S,r)_{T}$ denote the base change to $T$. Then, the base change of the morphism $\Rbold_{\dR}(X/S,\sigma, r)\to \Mbold_{\dR}(X/S,r)$ by $\Mbold_{\dR}(X/S,r)_{T}$ is a categorical quotient. That is, the latter is a categorical quotient of $\Rbold_{\dR}(X/S,\sigma, r)_{T}\simeq \Rbold_{\dR}(X_{T}/T,\sigma_{T}, r)$. This entails that there exists a canonical isomorphism $\Mbold_{\dR}(X_{T}/T,r)\simeq \Mbold_{\dR}(X/S,r)_{T}$, by uniqueness of categorical quotients. 


For the second claim, notice that the morphisms $\Rbold_{\dR}(X/S, \sigma, r)\to S$ and $\Mbold_{\dR}(X/S, r)\to S$ are quasi-projective, and $S$ is a scheme of finite type over $\CBbb$. Therefore, we can apply \cite[Expos\'e XII, Proposition 3.1]{SGA1}, and it is enough to check that the morphisms $\Rbold_{\dR}(X/S, \sigma, r)^{\an}\to S^{\an}$ and $\Mbold_{\dR}(X/S, r)^{\an}\to S^{\an}$ are flat. By the Riemann--Hilbert correspondence, we may equivalently reason for $\Rbold_{\Bet}(X^{\an}/S^{\an},\sigma^{\an}, r)\to S^{\an}$ and $\Mbold_{\Bet}(X^{\an}/S^{\an}, r)\to S^{\an}$. These are locally trivial fibrations over $S^{\an}$, hence automatically flat. The smoothness over $S$ of the loci of irreductible representations is established in a similar manner.

Now for the third claim. Suppose first that $S$ is irreducible. Because irreducibility is preserved by GIT quotients, it is enough to treat $\Rbold_{\dR}(X/S, \sigma, r)$. We already know that is $\Rbold_{\dR}(X/S, \sigma, r)\to S$ is flat. Furthermore, the fibers over closed points in $S$ are irreducible, and necessarily so does the fiber over the generic point (see \cite[\href{https://stacks.math.columbia.edu/tag/0553}{0553}]{stacks-project}; closed points are dense in $S$, which is also irreducible). From \cite[Corollaire 2.3.5 (iii)]{EGAIV2}, we deduce that $\Rbold_{\dR}(X/S,\sigma, r)$ is irreducible. Suppose next that $S$ is reduced. We claim that $\Rbold_{\dR}(X/S, \sigma, r)$ is reduced, by \cite[Expos\'e XII, Proposition 3.1]{SGA1} we may equivalently proceed for $\Rbold_{\dR}(X/S,\sigma,  r)^{\an}\simeq\Rbold_{\Bet}(X^{\an}/S^{\an}, \sigma^{\an}, r)$. The latter is locally trivial over $S^{\an}$. Now, $S^{\an}$ is reduced, and for a compact marked Riemann surface $(X,p)$, $\Rbold_{\Bet}(X,p,r)$ is reduced as well. The product of reduced spaces is reduced, thus concluding the proof of the claim. Hence, $\Rbold_{\dR}(X/S, \sigma, r)$ is reduced, and this property is inherited by any GIT quotient, such as $\Mbold_{\dR}(X/S, r)$.

The determinant one case is proven along the same lines.\qed

\subsection*{Proof of Proposition \ref{prop:flatness-2}}
The first point is analogous to the first point of Proposition \ref{prop:flatness}.

We next address the irreducibility properties. Assume that $S$ is irreducible. Then it is enough to prove that $\Rbold(X/S, \sigma, r)$ is irreducible. Actually, it is enough to prove that $\Rbold^{\sst}(X/S,\sigma, r)$ is irreducible. Indeed, $\Rbold^{\sst}(X/S, \sigma,r)$ is dense in $\Rbold(X/S,\sigma,r)$, since this is already the case for the closed fibers over $S$. Consider $\Rbold^{\sst}_{\dR}(X/S, \sigma,r)$ the open subscheme of slope stable vector bundles and morphism $\Rbold^{\sst}_{\dR}(X/S, \sigma,r)\to\Rbold^{\sst}(X/S, \sigma,r)$ of forgetting the connection, introduced in \textsection\ref{subsub:repspacesdeRham}. Because $\Rbold^{\sst}_{\dR}(X/S,\sigma, r)$ is irreducible (Proposition \ref{prop:flatness}), we just need to show that the forgetful morphism is surjective. Since our schemes are of finite type over $\CBbb$, surjectivity can be checked at the level of $\CBbb$-points. But a $\CBbb$-point $x\in\Rbold^{\sst}(X/S, \sigma,r)(\CBbb)$ corresponds to a stable vector bundle on the complex projective curve $X_{f(x)}$; the associated Riemann surface admits a holomorphic (automatically algebraizable) flat unitary connection by the theorem of Narasimhan--Seshadri \cite[Theorem 2]{Narasimhan-Seshadri}. Thus, $x$ is in the image of $\Rbold^{\sst}_{\dR}(X/S, \sigma,r)(\CBbb)$. 


Now we study the smoothness of $\Rbold(X/S,\sigma,r)\to S$ and the flatness of $\Mbold(X/S,r)\to S$. We first assume that $S$ is integral and geometrically unibranch \cite[\textsection 6.15]{EGAIV2}, for instance non-singular. We claim that $\Rbold(X/S,\sigma,r)\to S$ is universally open. Because $S$ is geometrically unibranch, by Chevalley's criterion \cite[Corollaire 14.4.4]{EGAIV3} we are lead to check that the morphism is equidimensional in the sense of \cite[Definition 13.3.2]{EGAIV3}. This definition applies here since $\Rbold(X/S,\sigma,r)$ is irreducible. We know that the fibers over closed points are smooth irreducible, of constant dimension $e\geq 1$. Because closed points are dense in schemes of finite type over $\CBbb$, we derive from \cite[Th\'eor\`eme 13.1.3]{EGAIV3} that all the fibers have constant and pure dimension, as required. Next, the fibers of $\Rbold(X/S,\sigma,r)\to S$ over closed points are smooth, hence geometrically reduced, and we can apply \cite[Corollaire 15.2.3]{EGAIV3} to conclude that $\Rbold(X/S,\sigma,r)\to S$ is flat at closed points (here we use that $S$ is reduced). In our context, flatness is an open condition on the source \cite[Proposition 11.3.1]{EGAIV3}, and closed points are dense. We infer that the morphism $\Rbold(X/S,\sigma,r)\to S$ is flat everywhere. Under the flatness condition, we can apply \cite[Th\'eor\`eme 12.1.6]{EGAIV3} to see that the smoothness of fibers over closed points implies the smoothness of all the fibers. Hence, $\Rbold(X/S,\sigma,r)\to S$ is smooth, as was to be shown. For $\Mbold(X/S,r)\to S$, since $S$ is reduced, by \cite[Corollaire 15.2.3]{EGAIV3} flatness follows if we can verify that $\Mbold(X/S,r)\to S$ is universally open and that the fibers are geometrically reduced. The morphism is universally open, since $\Rbold(X/S,\sigma,r)\to S$ is universally open and $\Rbold(X/S,\sigma,r)\to\Mbold(X/S,r)$ is surjective, as for any GIT quotient. The assertion on the fibers can be verified at closed points in $S$. For $s\in S(\CBbb)$, we have $\Mbold(X/S,r)_{s}=\Mbold(X_{s}, r)$, by the first point of the proposition. Since the latter is integral and we work over $\CBbb$, it is hence geometrically reduced, and this provides the second requirement of the criterion.  

Suppose that $S$ is general, and we want to show that $\Rbold(X/S,\sigma,r)\to S$ is smooth and $\Mbold(X/S,r)\to S$ is flat. These are local conditions with respect to $S$. Since $f\colon X\to S$ has genus $g\geq 2$, possibly restricting $S$, we may assume that the locally free sheaf $f_{\ast}(\omega_{X/S}(\sigma)^{\otimes 3})$ is trivial, and that $f$ factors through a closed embedding $X\hookrightarrow\PBbb(f_{\ast}(\omega_{X/S}(\sigma)^{\otimes 3}))\simeq\PBbb_{S}^{5g-2}$. Here we recall that $\sigma\colon S\to X$ is the implicit section. Hence, we may assume that $X\to S$ is a tri-canonicaly embedded one-pointed smooth curve \cite[p. 210]{Knudsen:proj-III}. Then, there is a classifying map $S\to H_{g,1}$, where $H_{g,1}$ is the Hilbert scheme of tri-canonically embedded one-pointed smooth curves of genus $g$ over $\CBbb$. This is known to be irreducible and smooth. The proof is not given in \cite{Knudsen:proj-III}, but it is an immediate generalization of the non-pointed case \cite[Proposition 5.3]{GIT}. Let $\Xcal\to H_{g,1}$ be the universal curve, and $\xi$ the universal section. We recover the family $X\to S$ with it section by base changing the universal pointed curve by the classifying map $S\to H_{g,1}$. By the previous study over integral and geometrically unibranch bases, the structure morphism $\Rbold(\Xcal/ H_{g,1},\xi, r)\to H_{g,1}$ is smooth and $\Mbold(\Xcal/H_{g,1}, r)\to H_{g,1}$ is flat. Since the formation of the relative moduli schemes is compatible with base change, and smoothness and flatness are preserved by base change, we conclude.

For the smoothness of $\Mbold^{\sst}(X/S, r)\to S$, we already know it to be flat. Moreover, the fiber of $\Mbold^{\sst}(X/S, r)$ over a closed point $s$ is $\Mbold^{\sst}(X_{s}, r)$, as follows from the construction of the relative moduli space (see \cite[Theorem 1.21 (4) \& Lemma 1.13]{Simpson:moduli-1}). The scheme $\Mbold^{\sst}(X_{s}, r)$ is known to be non-singular (see \cite[Theorem 1.21 (5)]{Simpson:moduli-1} and \cite[Proposition 23]{Seshadri}). This entails the smoothness property.

Finally, if $S$ is reduced, then $\Rbold(X/S,\sigma,r)$ is reduced, because $\Rbold(X/S,\sigma,r)\to S$ is smooth. Since reucibility is preserved under GIT quotient, we deduce that $\Mbold(X/S,\sigma,r)$ is reduced too.

The determinant one case is tackled along the same lines.\qed

\bibliographystyle{amsplain}

\end{document}